\documentclass[11pt,a4paper,reqno]{amsart}
\usepackage[a4paper,inner=2.5cm,outer=2.5cm,top=2.5cm,bottom=2.5cm]{geometry}
\usepackage{amsmath,amssymb,amsthm,enumerate,mathtools,stmaryrd}
\SetSymbolFont{stmry}{bold}{U}{stmry}{m}{n} 
\usepackage{hyperref}
\hypersetup{colorlinks=true,linkcolor=blue,citecolor=teal,filecolor=magenta,urlcolor=cyan}

\usepackage{makecell}
\usepackage{graphicx}
\usepackage{multirow}

\usepackage{float}
\usepackage[justification=centering]{caption}

\usepackage{extarrows}

\usepackage{xcolor}

\newcommand{\C}{\mathbb{C}}

\colorlet{light-gray}{gray!10}
\colorlet{lightish-gray}{gray!30}
\colorlet{light-cyan}{cyan!20}

\usepackage[style=alphabetic,maxalphanames=5,giveninits,sorting=nyt,%
maxbibnames=99]{biblatex}
\renewbibmacro{in:}{}
\usepackage[noabbrev]{cleveref}
\expandafter\def\csname ver@etex.sty\endcsname{3000/12/31}

\crefname{definition}{Definition}{Definitions}
\Crefname{definition}{Definition}{Definitions}
\crefname{theorem}{Theorem}{Theorems}
\Crefname{theorem}{Theorem}{Theorems}
\crefname{introthm}{Theorem}{Theorems}
\Crefname{introthm}{Theorem}{Theorems}
\crefname{proposition}{Proposition}{Propositions}
\Crefname{proposition}{Proposition}{Propositions}
\crefname{lemma}{Lemma}{Lemmata}
\Crefname{lemma}{Lemma}{Lemmata}
\crefname{corollary}{Corollary}{Corollary}
\Crefname{corollary}{Corollary}{Corollary}
\crefname{section}{Section}{Sections}
\Crefname{section}{Section}{Sections}
\crefname{subsection}{Subsection}{Subsections}
\Crefname{subsection}{Subsection}{Subsections}
\crefname{conjecture}{Conjecture}{Conjectures}
\Crefname{conjecture}{Conjecture}{Conjectures}
\crefname{question}{Question}{Questions}
\Crefname{question}{Question}{Questions}
\crefname{warning}{Warning}{Warnings}
\Crefname{warning}{Warning}{Warnings}
\crefname{figure}{Figure}{Figures}
\crefname{appendix}{Appendix}{Appendices}
\Crefname{appendix}{Appendix}{Appendices}
\Crefname{figure}{Figure}{Figures}
\crefname{table}{Table}{Tables}
\Crefname{table}{Table}{Tables}
\crefname{example}{Example}{Examples}
\Crefname{example}{Example}{Examples}

\makeatletter
\def\subsection{\@startsection{subsection}{2}%
  \z@{\linespacing\@plus.7\linespacing}{.3\linespacing}%
  {\normalfont\bfseries}}
\makeatother

\makeatletter
\def\subsubsection{\@startsection{subsubsection}{2}%
  \z@{\linespacing\@plus.7\linespacing}{.3\linespacing}%
  {\normalfont\itshape}}
\makeatother

\usepackage[textsize=footnotesize]{todonotes}
\presetkeys%
		{todonotes}%
		{}{}%
\theoremstyle{plain}
\newtheorem{theorem}{Theorem}[section]

\newtheorem{lemma}[theorem]{Lemma}
\newtheorem{proposition}[theorem]{Proposition}

\theoremstyle{definition}
\newtheorem{definition}[theorem]{Definition}
\newtheorem{remark}[theorem]{Remark}

\newtheorem{example}[theorem]{Example}

\AddToHook{env/corollary/begin}{\crefalias{theorem}{corollary}}
\AddToHook{env/lemma/begin}{\crefalias{theorem}{lemma}}
\AddToHook{env/proposition/begin}{\crefalias{theorem}{proposition}}
\AddToHook{env/definition/begin}{\crefalias{theorem}{definition}}
\AddToHook{env/remark/begin}{\crefalias{theorem}{remark}}
\AddToHook{env/conjecture/begin}{\crefalias{theorem}{conjecture}}
\AddToHook{env/example/begin}{\crefalias{theorem}{example}}

\newcommand{\mc}[1]{\mathcal{#1}}
\newcommand{\mb}[1]{\mathbb{#1}}
\newcommand{\mf}[1]{\mathfrak{#1}}
\newcommand{\Z}{\mathbb{Z}}
\newcommand{\cD}{\mathcal{D}}

\renewcommand{\P}{\mb{P}}

\newcommand{\dd}{\mathrm{d}}
\newcommand{\Res}{\mathop{\mathrm{Res}}}
\title{Taking limits in topological recursion}

\author[G. Borot]{Ga\"etan Borot}
\address{Institut f\"ur Mathematik \& Institut f\"ur Physik, Humboldt-Universit\"at zu Berlin, Unter den Linden 6, 10099 Berlin, Germany.}
\email{gaetan.borot@hu-berlin.de}

\author[V. Bouchard]{Vincent Bouchard}
\address{Department of Mathematical \& Statistical Sciences, University of Alberta, 632 CAB, Edmonton, Alberta, Canada, T6G 2G1.}
\email{vincent.bouchard@ualberta.ca}

\author[N.K. Chidambaram]{Nitin Kumar Chidambaram}
\address{Max Planck Institut f\"ur Mathematik, Vivatsgasse 7, 53111 Bonn, Germany.}

\address{School of Mathematics, University of Edinburgh, James Clerk Maxwell Building, Peter Guth\-rie Tait Rd, Edinburgh EH9 3FD, U.K.}
\email{nitin.chidambaram@ed.ac.uk}

\author[R. Kramer]{Reinier Kramer}
\address{Department of Mathematical \& Statistical Sciences, University of Alberta, 632 CAB, Edmonton, Alberta, Canada, T6G 2G1.}

\address{Universit\`a di Milano-Bicocca, Dipartimento di Matematica e Applicazioni, Via Roberto Cozzi, 55, Milan, 20125, Italy.}
\email{reinier.kramer@unimib.it}

\author[S. Shadrin]{Sergey Shadrin}
\address{Korteweg-de Vries Institute for Mathematics, University of Amsterdam, Postbus 94248, 1090 GE Amsterdam, the Netherlands.}
\email{s.shadrin@uva.nl}


\begin{document}

\begin{abstract}
	
	When does topological recursion applied to a family of spectral curves commute with taking limits? This problem is subtle, especially when the ramification structure of the spectral curve changes at the limit point. We provide sufficient (straightforward-to-use) conditions for checking when the commutation with limits holds, thereby closing a gap in the literature where this compatibility has been used several times without justification. This takes the form of a stronger result of analyticity of the topological recursion along suitable families. To tackle this question, we formalise the notion of global topological recursion and provide sufficient conditions for its equivalence with local topological recursion. The global version  facilitates the study of analyticity and limits.  For nondegenerate algebraic curves, we reformulate these conditions purely in terms of the structure of its underlying singularities. Finally, we apply this to study deformations of $ (r,s) $-spectral curves, spectral curves for weighted Hurwitz numbers, and provide several other examples and non-examples (where the commutation with limits fails).
	
\end{abstract}

\maketitle

\tableofcontents

\section{Introduction}

\subsection{Incipit}

To a spectral curve (roughly speaking, a branched covering of complex curves with additional data) topological recursion associates a module for a $\mathcal{W}$-algebra, generated by a unique normalised vector, called a partition function. This partition function can be thought of as encoding all genera B-model invariants, and topological recursion organises its computation. The coefficients of this partition function are indexed by $g \in \mathbb{Z}_{\geq 0}$ (genus)  and $n \in \mathbb{Z}_{> 0}$ (number of insertions/punctures) and can be efficiently encoded as meromorphic differentials $\omega_{g,n}$ on the $n$-th product of the spectral curve. The data of the spectral curve includes $\omega_{0,1}$ and $\omega_{0,2}$, and topological recursion provides a way to compute all the $\omega_{g,n}$, recursively in $2g - 2 +  n > 0$. The partition functions constructed by topological recursion have numerous applications. Without being exhaustive, let us mention some in random matrix theory (see e.g. \cite{CEO06,EO09,BEO13}), in enumerative geometry (see e.g. \cite{DN11,EynardICM,EO15,E14,L17,FLZ17,FLZ20,BDKS20a,GKL21,BCCGF22,ABCDGLW23}), in supersymmetric gauge theories of class $ \mathcal S $ in the context of the Alday--Gaiotto--Tachikawa proposal \cite{BBCC21,Osu21}, in integrable systems and semiclassical expansions of solutions of ODEs (see e.g. \cite{BE12,BE17,IMS18,DM18a,EGMO21}), etc.

Topological recursion was first introduced by Eynard and Orantin for smooth branched coverings with simple ramifications \cite{EO07}. The definition was later extended to allow higher ramifications \cite{BHLMR14,BE13}, see also \cite{Prats} for an early instance.  Kontsevich and Soibelman then proposed an algebraic framework leading to topological recursion, called Airy structures \cite{KS17}. A large class of Airy structures can be constructed from vertex operator algebras having a free field representation -- in particular $\mathcal{W}(\mathfrak{gl}_r)$-algebras -- and the computation of their partition function can in turn be recast as a residue computation on a spectral curve \cite{BBCCN18}. The type of ramifications of the spectral curve (among other things) specifies the underlying $\mathcal{W}$-algebra and module. This perspective has led to an extension of the definition of topological recursion to a much larger class of (not necessarily smooth) spectral curves. The advantage of working with spectral curves versus Airy structures is that it allows us to exploit analytic continuation and use the global geometry of the curve, instead of just dealing with germs at a finite set of points.

\subsection{The problem}
\label{sec:problem}
The central question driving this article is:
\begin{center}
\textit{ Given a family of spectral curves, is the topological recursion procedure compatible with taking limits (i.e. approaching special fibres) in the family?}
\end{center}
Of particular interest are the cases where the type of the spectral curve changes at the limit,  for instance, when ramification points collide with each other or with singularities of the curve. Compatibility with such limits has been used in several papers, either without justification or a reference, or with a reference to \cite[Section 3.5]{BE13} where the core idea necessary to study compatibility with limits is introduced but the argument for compatibility itself is merely sketched. Here are some relevant examples:
\begin{itemize}
\item In \cite{DNOPS19}, Theorem 7.3 states that topological recursion on the $ r $-Airy curve (associated to the $A_{r - 1}$-singularity) computes the descendant integrals of the Witten $r$-spin class. Its proof invokes compatibility with limits.
\item In \cite[Section 2.1]{BDS20}, \cite[Section 2.2]{BDKLM22} and \cite[Proof of Theorems 5.3 and 5.4]{BDKS20a}. In these articles it is shown that (weighted) double Hurwitz numbers satisfy topological recursion in case the spectral curve has simple ramifications (this is the case for generic weights), and compatibility with limits is invoked to extend the result to the non-generic case where higher ramifications occur. \cite{BDKS20a} refers to the present article.
\end{itemize}

It turns out that this question is highly non-trivial, and  the assertion ``topological recursion commutes with limits" is only conditionally true. In fact, a better formulation of the problem is to ask for sufficient conditions that guarantee the \emph{analyticity} of the correlators of topological recursion in families. We investigate this problem in detail and identify the crux of the matter. We provide sufficient conditions for analyticity/compatibility with limits which are easy to test and cover many cases of practical use.  We also provide examples where commutation with limits fails, which illustrates  the relevance of our results. In particular, our results apply to families of algebraic curves of constant genus.

Using ideas of \cite{BE13}, Milanov studied in \cite{Milanov} a similar question  
in the family of deformations of the $A_{r - 1}$-singularity and used it to establish Givental's conjecture asserting the analytic extension of the total ancestor potential to non-semisimple points. Prior to this work, Charbonnier, Chidambaram, Garcia-Failde and Giacchetto \cite{CCGG22}  justified that topological recursion commutes with the limit $t \rightarrow 0$ for families of genus $0$ spectral curves of the form
\begin{equation}
\label{negrcurve} x = \frac{P(y,t)}{Q(y,t)}\,,
\end{equation}
where $P$ and $Q$ are bivariate polynomials, assuming that the ramification points of $x$ are simple for $t \neq 0$, that there is a single ramification point above each branch point,  and that the $t \rightarrow 0$ limit of the $\omega_{g,n}(y_1,\ldots,y_n)$ (for fixed $y_i$) exist. Using this limit procedure, they reproved that topological recursion governs the intersection indices of Witten $r$-spin class (filling the gap in \cite[Theorem 7.3]{DNOPS19}) and gave another proof of the Witten $r$-spin conjecture first proved in~\cite{FSZ10}. Furthermore,  \cite{CGG22} proved that the descendant integrals of the $ \Theta^r $-class (a negative spin version of the Witten $r$-spin class) are computed by topological recursion on the spectral curve
$$
xy^{r} = 1, \qquad \omega_{0,1} = y \dd x,\qquad \omega_{0,2} = \frac{ \dd y_1\, \dd y_2}{(y_1 - y_2)^2}\,.
$$
For $r = 2$, this resolved a conjecture of Norbury \cite{Nor17} giving an intersection-theoretic interpretation to the Br\'ezin--Gross--Witten $ \tau $-function. Moreover, again for $ r=2 $,  commutation of the topological recursion with the limit implies the vanishing of descendant integrals of certain combinations of $\kappa$-classes, which was first noticed by Kazarian and Norbury \cite{KN21}. 
Similar to the observation of \cite{KN21}, we expect our results to have consequences for the intersection theory on $\overline{\mathcal{M}}_{g,n}$, but this line of investigation remains beyond the scope of the present article.

Our results about existence of limits and compatibility with topological recursion include and significantly extend the cases treated in \cite{Milanov,CCGG22}. We discuss many examples in \cref{sec:EX}. We will also provide examples of a different nature in \cref{sec:Chiodo} and \cref{GWP111} , where limits can be taken on the intersection-theoretic side but are difficult to study on the topological recursion side. This makes the interplay between enumerative geometry and topological recursion particularly interesting and indicates that there are often non-trivial phenomena underlying compatibility with limits.

 The question \emph{Is topological recursion compatible with limits?} is in fact closely related to other problems in different contexts.

\medskip

\noindent \textit{(Variants of) cohomological field theories:} Can one express the $\omega_{g,n}$s of topological recursion for any admissible spectral curve in terms of intersection indices of classes on $\overline{\mathcal{M}}_{g,n}$ satisfying CohFT-like axioms?  Such a statement is known for certain types of spectral curves (in terms of the Witten $r$-spin class and the $ \Theta^r $-class  \cite{CGG22}) but is expected in general (see \cite[Conjecture 7.23]{BKS20}). It is conceivable that this question can be adressed by taking non-trivial limits of known cases.

\medskip

\noindent \textit{Tautological relations:} When topological recursion commutes with limits, we expect the existence of tautological relations in the cohomology ring of $\overline{\mathcal{M}}_{g,n}$ --- in the spirit of \cite{KN21,CGG22}. It is an interesting question to investigate and prove these relations and to identify whether they are new (i.e., do not fall in the set of Pixton's relations \cite{Pix13}).

\medskip

\noindent \textit{Frobenius manifolds:} In the context of the correspondence between Hurwitz--Frobenius manifolds and topological recursion, cf.~\cite{DNOPS19,DNOPS18}, it would be interesting to see whether the compatibility with the limits and the necessary admissibility conditions can be matched in some regular cases to the invariants and local analysis of semisimple Frobenius coalescent structures studied in~\cite{CottiDubGuz} --- that includes, for instance, the Maxwell strata of the discriminants.

\medskip

\noindent \textit{Representation theory:} Topological recursion on a spectral curve can be reformulated in terms of modules of $\mathcal{W}$-algebras \cite{BBCCN18}. To each fibre in a family of spectral curves one can associate a module of a certain $\mathcal{W}$-algebra. When the ramification type of the spectral curve changes at a special fibre in a family, the type of $\mathcal{W}$-algebra also changes. In cases where topological recursion commutes, the limit procedure gives us a way to construct new modules out of old ones, by a sort of analytic continuation --- cf. the schematics in \cref{S1ex}. This calls for a representation-theoretic understanding of the limit procedure.

\subsection{Outline and main results}
\label{S1ex}
For us, a spectral curve will be the data of a disjoint union of Riemann surfaces $\Sigma$, a holomorphic map $x: \Sigma \to \mathbb{P}^1$, a meromorphic 1-form $\omega_{0,1}$ on $\Sigma$, and a symmetric meromorphic bidifferential $\omega_{0,2}$ on $\Sigma \times \Sigma$, satisfying a few conditions stated in \Cref{de:sc}. If $p \in \Sigma$, let $U_p$ be an (arbitrarily) small neighbourhood of $p$ and $\mathfrak{f}(p) = x^{-1}(x(p))$ the set of points in the same $x$-fibre as $p$.

\medskip

\noindent \underline{\textit{\Cref{S2}}.} We review the theory of topological recursion for spectral curves. Our assumptions are slightly weaker than those previously considered in the literature (see the discussion below \cref{de:localadm}), and this generalisation is useful for our arguments later on. We define the type of a spectral curve as specified by the local data at each point $p \in \Sigma$ (\cref{de:localparameters}), which are, roughly speaking:
\begin{itemize}
\item $r_p \in \mathbb{Z}_{> 0}$, the ramification order of $x$ at $p$;
\item $\bar s_p \in \mathbb{Z}$, the minimal exponent of the local expansion of $\omega_{0,1}$ near $p$;
\item $s_p \in \mathbb{Z} \cup \{\infty\}$, which is similar to $\bar s_p$, but this time considered modulo pullbacks via $x$ of $1$-forms locally defined near $x(p) \in \mathbb{P}^1$.
\end{itemize}
In the literature, ramification points with $s_p = r_p + 1$ are sometimes called ``regular", while those with $s_p = r_p - 1$ are called ``irregular'' \cite{NorDo,NorChe}. The $s_p = 1$ case may be called ``maximally irregular''. The highlight of this section is
\begin{itemize}
\item[$\star$] \cref{th:ale}: Topological recursion (equation~\eqref{eq:TR}) provides the unique solution of the \emph{abstract loop equations} (\cref{de:localloop}) that satisfies the \emph{projection property} (\cref{de:projection}).
\end{itemize}
This existence and uniqueness result is valid only for certain types of spectral curves, which we call \emph{locally admissible} (\cref{de:localadm}).  In order to define topological recursion by the formula \eqref{eq:TR}, local admissibility is the minimal requirement that ensures that the recursion kernel of \cref{de:recker} is well-defined and that the correlators produced by \eqref{eq:TR} are symmetric in all their variables, as they ought to be. The only new aspect of \cref{th:ale} compared to previously known results is its weaker set of assumptions.

\medskip

\noindent \underline{\textit{\Cref{S3}}.} The key idea to understand limits in topological recursion is to transform the \textit{local formula} \eqref{eq:TR} which defines it, in two steps. \Cref{S3} describes these transformations (\cref{de:globalTR}) and examines their conditions of validity. The local formula (rewritten as in \eqref{eq:TRbp}) involves the sum over $q \in x(\Sigma) \subseteq\mathbb{P}^1$ of terms of the form
\begin{equation}
\label{localsk}\sum_{p \in x^{-1}(q)} \Res_{z = p} \sum_{Z \subseteq (\mathfrak{f}(z) \cap U_p) \setminus \{z\}} \Omega_p(Z)\,,
\end{equation}
Here $\Omega_p(Z)$ is a meromorphic $1$-form defined in $U_p$ whose definition depends\footnote{A clean way to define the sum over $Z$ of $\Omega_p(Z)$ is by push-pull of a certain $1$-form via $x_{|U_p}$. While the writing \eqref{localsk} is more suggestive, we refer to \cref{S2} where its meaning is explained.} on the choice of subsets $Z $ of points in the fibre of $x(z)$ that remain near $p$. For each $q \in \mathbb{P}^1$, the first transformation is to ``globalise vertically''. This means that we enlarge the subset  $Z$ in the sum to include  all points in $\mathfrak{f}'(z)$, not just points that remain near $ p $. In other words, we replace \eqref{localsk} by:
$$
\sum_{p \in x^{-1}(q)} \Res_{z = p} \sum_{Z \subseteq \mathfrak{f}(z) \setminus \{z\}} \Omega_p(Z)\,.
$$
Whether or not this expression is equal to \eqref{localsk} for a given $q$ depends on the local behaviour of the spectral curve above $q$. When the two expressions are equal, we say that topological recursion can be globalised above $q$. The main result of this section is
\begin{itemize}
		\item[$\star$]  \cref{th:rewriting}: A non-trivial criterion for vertical globalisation at a point $q \in \mathbb{P}^1$, pertaining to the local behaviour at points and pairs of points in $x^{-1}(q)$.  If $q$ is not a branchpoint, this criterion takes the simpler form of \cref{th:unramified}.
\end{itemize}
The second step is to ``globalise horizontally', which means rewriting the sum over residues as integration over contours that may surround several ramification points at the same time (in the same fibre of $x$ or not). The technical tool here is \cref{pr:globalTRdomain}. In general, there are topological constraints pertaining to the groups of ramification points that can be formed. However, these constraints do not appear if one considers homologically trivial contours of integration or if the spectral curve has genus $0$.

\medskip

\noindent \underline{\textit{\Cref{S4}}.} Fundamental examples of (families of) spectral curves come from algebraic curves, which play a central role in the Hitchin integrable system on the moduli space of (meromorphic) Higgs bundles. The genus $0$ sector of topological recursion encodes the special K\"ahler geometry on the Hitchin base \cite{BH19}, and this relation can be understood in the more general context of deformations of compact curves in symplectic surfaces \cite{CNST20}. The ``type'' of spectral curves corresponds to strata on moduli of deformations of algebraic spectral curves, and it is desirable to understand how topological recursion behaves when changing strata. Building on the results of \cref{S3}, the theory of globalisation for algebraic curves is given a detailed treatment in \cref{S4}.

More precisely, we start by explaining how affine curves give rise to non-compact spectral curves (\cref{S4noncom}) and, more interestingly, how after projectivisation in $\mathbb{P}^1 \times \mathbb{P}^1$ and normalisation, they give rise to compact spectral curves (\cref{S4com}). We introduce a notion of \emph{global admissibility} for compact spectral curves obtained in this way (\cref{de:globaladm}), and prove
\begin{itemize}
\item[$\star$] \cref{th:ga}: If a spectral curve is globally admissible, then it is locally admissible and topological recursion can be vertically globalised above all $q \in \mathbb{P}^1$.
\end{itemize}
 Global admissibility is a condition describing the allowed singularities of the projective curve (before its normalisation) and the property of its Newton--Puiseux series, as proved in \cref{pr:ndga}. In the case of nondegenerate curves, the analysis can be pushed further and our main results are
 \begin{itemize}
\item[$\star$] \cref{th:globsimp}: A criterion for globalisation of irreducible nondegenerate algebraic curves, pertaining to the slopes of the Newton polygon;
\item[$\star$] \cref{th:gapairs}: Global admissibility is preserved along families of irreducible nondegenerate algebraic curves of constant geometric genus.
\end{itemize}
We show in \cref{pr:Qdelt} that, up to a few pathological cases, convex integral polygons inscribed in a fixed rectangle can be enlarged into a unique largest convex integral polygon while keeping the set of integral interior points. This implies the existence of maximal families of irreducible nondegenerate algebraic curves of constant geometric genus, which is interesting in view of \cref{th:gapairs}.

This section relies on basic singularity theory for curves and  the properties of Farey sequences, and is logically independent of \cref{S5} where limits in families are studied. Yet, we indicate in \cref{S5} the extra simplifications arising in the algebraic setting of \cref{S4}.

\medskip

\noindent \underline{\textit{\Cref{S5}}.} We harvest the fruits of \cref{S3} and explain that, whenever the global formula for topological recursion holds, with integration contours whose $x$-projection can be chosen independently of the parameters of the family, commutation of topological recursion with limits is automatic under reasonable assumptions. The main result of the article and the answer we have reached to the central question is
\begin{itemize}
\item[$\star$] \cref{th:TRLimits}: An analyticity result of the correlators $\omega_{g,n}$ with respect to the parameters of a globally admissible family of spectral curves -- the latter notion is introduced in \cref{de:admfamily}.
\end{itemize}
To avoid difficulties due to the topological constraints in the globalisation procedure, we can restrict ourselves to families of spectral curves of ``constant genus'' (cf. \cref{le:admproper}). Although the result of \cref{th:TRLimits} could hold under weaker assumptions, the  impossibility of globalising on a compact domain where all singularities or ramification points remain while moving in the family is often a footprint for a lack of commutation with limits.
\medskip

\noindent \underline{\textit{\Cref{sec:EX}}.} We illustrate our results with a series of examples (where topological recursion commutes with limits) and non-examples (where the commutation fails). Many of the examples are covered by our main result \cref{th:TRLimits}, while some of the examples require an \textit{ad hoc} adaptation of our general arguments. There are also a few examples where commutation is known to hold due to indirect means but our results cannot be applied directly, and we view this as an illustration of the current shortcomings of our method that deserve further investigation. 

In \cref{sec:rsexamples} we give a detailed treatment of the important case of the $(r,s)$-spectral curves, i.e. compact spectral curves obtained from the equation $x^{r - s}y^r = 1$ with $r \geq 2$ and $s \in [r + 1]$ coprime. We exhibit them as the central fibre in a (in some sense, maximal) family $\mathcal{S}_{T}^{(r,s)}$ of compact spectral curves, where $T$ is the base of this family. Moving away from the central fibre generically splits the ramification points of order $r$ above $x = 0$ and $\infty$ into ramification points of lower orders, which can be located above $0$, or $\infty$, or elsewhere. The complete description is given in \cref{Sec532}, but the main fact about them is
\begin{itemize}
\item[$\star$] \cref{le:analrs}: The family $\mathcal{S}_{T}^{(r,s)}$ is globally admissible if and only if the $(r,s)$-curve is locally admissible. This condition amounts here to $r = \pm 1\,\,{\rm mod}\,\,s$.
\end{itemize}
The splitting of ramification points along these deformations is described in \cref{assidesec}. \Cref{th:TRLimits} then tells us that the topological recursion in analytic (thus commutes with limits) in those families, which we formulate as:
\begin{itemize}
\item[$\star$] \Cref{th:lim}: Topological recursion is analytic (in particular commutes with limits) in the family $\mathcal{S}^{(r,s)}_{T}$ if and only if $r = \pm 1\,\,{\rm mod}\,\,s$.
\end{itemize}
This largely generalises the situations where commutation of limits was proved in  \cite{CCGG22}, as announced in \cref{sec:problem}. In particular, we prove commutation with limits for the Chebyshev family, which was a question raised in \cite[Remark 4.8 and 5.5]{CCGG22}.

In \cref{sec:rsexamples2} we examine certain deformations of the $(r,s)$-curve that do not fit in the previous general discussion, including examples with singular curves, curves with logarithmic cuts related to the Euler class of the Chiodo bundle (\cref{pr:Chiodo}), and curves with an irreducible component on which $y = \infty$ (\cref{pr:Norbury}). In the last case we can extend our method to show commutation with limits, while in the first two cases commutation with limits typically does not hold, leading to a number of interesting questions in relation to intersection theory on the moduli space of $r$-spin curves.

In \cref{sec:Hurwitz}, we show that families of spectral curves that govern weighted double Hurwitz numbers and similar enumeration problems in the symmetric group, are globally admissible. Hence, we
\begin{itemize}
\item[$\star$] explain how to extend the validity of the topological recursion results of \cite{BDKS20a,bychkov2023symplectic,alexandrov2023topological} for this class of enumerative problems  to non-generic values of parameters at which the spectral curve develops non-simple ramification (see~\cref{pr:PQRHur}).
\end{itemize}
These papers explicitly refer to the present paper for this step; it also closes the earlier gaps in the literature  and incorrect references that were mentioned in \cref{sec:problem}.

In \cref{sec:moreexamples}, we discuss the possibility or impossibility to handle further examples that do not fit the general scheme: the large radius and the non-equivariant limits in the mirror curve governing equivariant Gromov--Witten theory of $\mathbb{P}^1$ according to \cite{FLZ17}; a spectral curve developing a component on which $x$ is constant.

\medskip

Finally, \cref{ap:appendix} explains how considering the action of certain symplectomorphisms of $\mathbb{C}^2$ on spectral curves makes the congruence condition $r = \pm 1\,\,{\rm mod}\,\,s$ --- which is crucial for the topological recursion to be well defined, but was so far geometrically mysterious --- appear naturally.

\cref{fig:sumdiag} can help visualise part of the logical organisation of some definitions and results of the body of the text.

\begin{figure}[h!]
\begin{center}
\includegraphics[width=\textwidth]{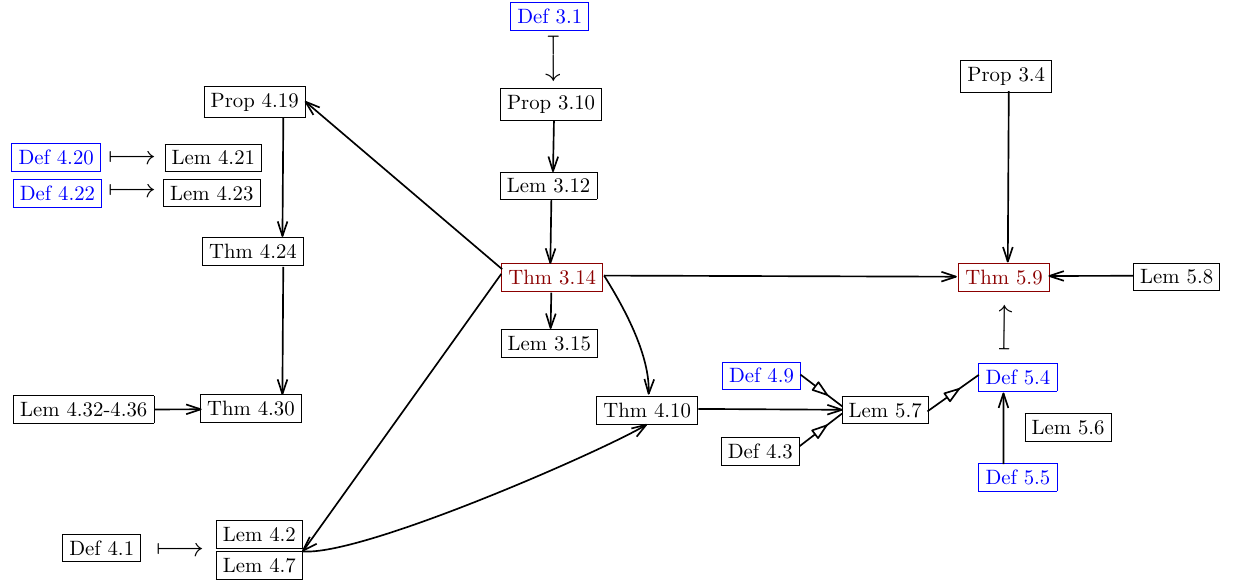}
\caption[Figure explaining relationship between various concepts introduced in this paper]{\label{fig:sumdiag} The main goal of the article is to prove \cref{th:TRLimits} (conditional analyticity of TR, using the basic tool \cref{le:LimitsOfContours}) involving the notion of global admissibility in families given in \cref{de:admfamily} or the simpler finite-degree setting of \cref{de:admfamilyfinite}. Before working in families, we study global admissibility (\cref{de:globaladm}) for the compact algebraic spectral curves of \cref{de:irredpc}. \cref{th:ga} shows that global admissibility implies local admissibility of \cref{de:localadm}. The results  concerning the non-compact algebraic spectral curves of \cref{de:redac} are \cref{le:redacla} and \cref{le:singularities}. An important tool to analyse the admissibility properties of any spectral curve (algebraic or not) is \cref{th:rewriting} (sufficient condition for vertical globalisation), \cref{th:unramified} is its special case for unramified points. This tool comes from a careful analysis of the vertical globalisation of \cref{de:globalTR}, first in terms of correlators (\cref{pr:sufficient2}), then in terms of the spectral curve (\cref{le:pairs}).
	\vspace{.4 cm}
	
	 Results that help to check global admissibility are listed in the left of the diagram: \cref{pr:ndga} (criterion for global admissibility of compact algebraic spectral curves), and its formulation in terms of Newton polygons \cref{th:globsimp} which is proved using \cref{de:edgeadm} and \cref{de:edgeadm2} (admissibility of edges) and the corresponding \cref{le:coincm} and \cref{le:coincm2}, while \cref{th:gapairs} describes allowed deformations of Newton polygons that preserve global admissibility (proved via \cref{le:localad}-\cref{le:muiglob} thanks to the properties of Farey sequences).  \cref{le:admproper} describes the relationship between global admissibility for families of spectral curves  and that for individual compact algebraic spectral curves. }
\end{center}
\end{figure}

\subsection{Splitting the ramification of the \texorpdfstring{$(r,s)$}{(r,s)}-spectral curve}

\label{assidesec}

As an application of our results, we here give an overview of the way ramification points split in the generic fibre of the maximal family $\mathcal{S}_T^{(r,s)}$ of deformations of the $(r,s)$-spectral curve for locally admissible values of $(r,s) $.  \cref{le:rsmax} provides an explicit description of a generic fibre of  $\mathcal{S}_T^{(r,s)}$ in terms of the associated Newton polygon, and \cref{Sec532} gives an explicit rational parametrisation of this generic fibre. From this rational parametrization, we can calculate the branching profile of a generic fibre, which is given as follows.

\noindent $\bullet$ If $s = r \pm 1$, we have
\begin{flalign}
(r,r\pm 1)_0\,\, & \rightsquigarrow\,\, (2,3)^{\boxtimes(r-1)} \,\,\textcolor{gray}{ \boxtimes\,\, (2,3)^{\boxtimes (r-1)}}\,. &&
\end{flalign}
\noindent $\bullet$ If $s \in \{2,\ldots,r-2\}$ with $r = 1\,\,{\rm mod}\,\,s$, we write $r = r's + 1$ with $r' \geq 2$ and we have 
\begin{flalign}
(r,s)_0 \,\,& \rightsquigarrow\,\,(r', 1)_0^{\boxtimes s} \boxtimes (2,3)^{\boxtimes s}\,\,\textcolor{gray}{\boxtimes\,\, (2,3)^{\boxtimes s}}\,. &&
\end{flalign} 
\noindent $\bullet$
If $r \geq 3$ and $s = 1$:
\begin{flalign}
(r,1)_0 \,\,& \rightsquigarrow\,\,(2,3)^{\boxtimes (r-1)}\,\,\textcolor{gray}{\boxtimes\,\,(2,3)^{\boxtimes (r - 1)}}\,. &&
\end{flalign}
\noindent $\bullet$
If $s \in \{3,\ldots,r-2\}$ and $r = -1\,\,{\rm mod}\,\,s$, we write $r = (r - r')s - 1$ and
\begin{flalign} 
(r,s) \,\,& \rightsquigarrow\,\,(r - r' - 1,1)_0 \boxtimes (r - r',1)_0^{\boxtimes (s - 1)} \boxtimes (2,3)^{\boxtimes (s -1)} \,\,\textcolor{gray}{\boxtimes\,\, (2,3)^{\boxtimes {(s - 1)}}}\,, &&
\end{flalign}
where the first factor can be omitted if $r - r' = 2$.

The notation is in part intuitive but let us explain it in detail. The sequence
$$
(r_{p_1},s_{p_1}) \boxtimes \cdots \boxtimes (r_{p_\ell},s_{p_\ell})
$$
encodes the ramification data: each factor represents a ramification point $p$ with local parameters $(r_p,s_p)$ which contributes to topological recursion. This means that unramified points and ramified points with $s_p \leq -1$ are omitted. The $\boxtimes$-notation is compatible with the representation-theoretic content of the system of correlators of the topological recursion \cite{BBCCN18}: they generate a module for the $\prod_{p \in \mathsf{Ram}} \mathcal{W}(\mathfrak{gl}_{r_p})$ algebra, which is  isomorphic to the exterior tensor product of modules $M_p$ associated to the factor labelled $p$, and the isomorphism type of $M_p$ is specified by the local parameters $(r_p,s_p)$. If $p$ does not contribute to topological recursion, the corresponding module comes from the free-field module for the Heisenberg algebra of $\mathbb{C}^{r_p} \subset \mathfrak{gl}_{r_p}$ (generated by a constant), and we have omitted it from the exterior tensor product description.

We add an index $(r_i,s_i)_0$ to indicate the ramification points above $x = 0$. Apart from those, the other ramification points are located above pairwise disjoint branch points. The ramification data of the $(r,s)$-curve appears on the left,  the one of the generic fibre of $\mathcal{S}_T^{(r,s)}$ on the right. The $(r,s)$-curve has a ramification point above $x = \infty$ that does not contribute to topological recursion. In gray we indicate the ramification points of the generic fibre that collide when approaching the central fibre to this ramification point above $x = \infty$ (and thus disappear from the notation in the left-hand side).

\medskip

We observe that the ramification point above $x=0$ in the $(r,s)$-spectral curve generically splits into a collection of the simplest regular ramification points, i.e., of type $ (2,3) $, and maximally irregular ramification points, all having lower ramification order. It is interesting to note that all the maximally irregular ramification points remain above $ x = 0 $, while the points of type $ (2,3) $ move away to generic points.

\medskip

The type of ramification for non-generic fibres can be different. For instance, the well-known deformation of the $A_{r - 1}$-singularity $x = y^{r} + \sum_{k = 0}^{r - 2} t_k y^k$ is a deformation of the $(r,r+1)$-curve which generically has $ (r-1) $ simple regular ramification points and one irrelevant ramification point of order $r$ above $x = \infty$, while the generic fibre of  $ \mathcal S^{(r,r+1)}_T $ has $2(r - 1)$ simple regular ramification points and no ramification point above $x = \infty$. The latter can be presented as
\[
x = \frac{Q(y)}{R(y)}\,,\qquad  \omega_{0,1}(y) = y \dd x\,,\qquad \omega_{0,2}(y_1,y_2) = \frac{\dd y_1\,\dd y_2}{(y_1 - y_2)^2}\,.
\]
where $Q,R$ are polynomials such that $x$ is a map of degree $r$: it contains the families \eqref{negrcurve} studied in \cite{CCGG22}.

\subsection{Notations}

For $n \in \mathbb{Z}_{> 0}$ we denote $[n] = \{1,\ldots,n\}$ and $(n] = \{2,\ldots,n\}$. If $x \in \mathbb{R}$ we denote $\lfloor x \rfloor$ the largest integer $\leq x$ and $\lceil x \rceil$ the smallest integer $\geq x$.

For a tuple of variables $z_1,\dots,z_n$ and a subset $ I\subseteq [n]$, let $z_I = (z_i)_{i\in I}$. As such tuples will be inserted in symmetric functions/differentials, the order of the variables in the tuple will not matter. A partition $ \mathbf{L} $ of $ [n]$, denoted $\mathbf{L} \vdash [n]$, is a tuple of pairwise disjoint non-empty subsets (called parts) of $[n]$ whose union is $[n]$. Then $L \in \mathbf{L}$ denotes a part of the partition $\mathbf{L}$. We denote disjoint union by $\sqcup$ and write $\bigsqcup_{L \in \mathbf{L}} M_L = [m]$ to mean that $(M_L)_{L \in \mathbf{L}}$ is a tuple (indexed by the parts of $\mathbf{L}$) of pairwise disjoint (possibly empty) subsets of $[m]$ whose union is $[m]$. We write $\lambda \vdash n$ to say that $\lambda$ is a partition of $n$, i.e. a weakly decreasing tuple of positive integers summing up to $n$. We write $\mathbb{Y}_{\lambda} = \{(i,j) \in \mathbb{Z}_{> 0}^2\,\,|\,\,i \leq \lambda_j\}$ for the associated Young diagram.

By degree $i$ differential on a complex curve $S$, we mean a meromorphic section of $K_S^{\otimes i}$, i.e. an object of the form $f(t) \dd t^i$ with respect to a local coordinate $t$ for some meromorphic function $f$. Unless said otherwise, $\dd t^i$ should always be read as $(\dd t)^i$. By $n$-differential on a product $S = \prod_{i = 1}^n S_i$ of complex curves, we mean a meromorphic section of $\bigotimes_{i = 1}^n p_i^*K_{S_i}$ where $p_i \colon S \rightarrow S_i$ is the projection on the $i$-th factor, i.e. an object of the form $f(t_1,\ldots,t_n) \dd t_1 \cdots \dd t_n$ where $t_i$ is a local coordinate on $S_i$ and $f$ is a meromorphic function.

If $\phi$ is degree $i$ differential which is not identically zero on a complex curve $S$ and $\zeta$ a local coordinate centered at a point $p \in S$, we mean by $\phi(z) \approx \sum_{l \geq {\rm ord}_p \phi} \phi_l \zeta^{l} \dd \zeta^i$ that the right-hand side is the Laurent series expansion of the left-hand side when $z \rightarrow p$.

\subsection{Recent work}
\label{s:recent_work}

In this paper, by topological recursion we always mean the topological recursion of Eynard and Orantin \cite{EO07} and its generalization to higher ramification from \cite{BHLMR14,BE13}, sometimes known as ``Bouchard-Eynard'' topological recursion: see 
\cref{de:TR}. However, after the first version of this paper appeared, a new formulation of topological recursion was proposed, dubbed ``generalized topological recursion'' \cite{ABDKS25}. The idea of the new framework is to take symplectic duality --- which was a cornerstone of the original proposal of Eynard and Orantin --- as a defining axiom. The result is a recursive framework that coincides with the Eynard-Orantin topological recursion for spectral curves that are ``locally Airy'' (with  local parameters $r_p = 2$ and $s_p = 3$ at all ramification points, in the language of \cref{de:localparameters}), but does not always coincide with the generalization of \cite{BHLMR14,BE13} for arbitrary locally admissible spectral curves. In fact, generalized topological recursion is defined even for spectral curves that do not satisfy the local admissibility condition in \cref{de:localadm}.

The relevance of this new construction to the current paper is that generalized topological recursion is in general better behaved with respect to limits than the generalization of \cite{BHLMR14,BE13}. In a nutshell, if the input data of generalized topological recursion varies analytically with respect to a parameter $t$ (see \cite{BCGS25} for more details), then the differentials produced by generalized topological recursion also vary analytically with respect to $t$. This implies that the $t \to 0$ limit of the differentials produced by generalized topological recursion coincide with the differentials on the limiting spectral curve. This limit procedure was used in \cite{BCGS25} to show that generalized topological recursion on the $(r,s)$-spectral curves computes the descendant integrals of the so-called $\Theta^{r,s}$-classes, which are (non-semisimple) cohomological field theories obtained as top degree pieces of Chiodo classes. However, it is imoprtant to note that for the $(r,s)$-spectral curves, generalized topological recursion coincides with the recursion of \cite{BHLMR14,BE13} only when $s=r-1$ and $s=r+1$. Thus a geometric interpretation of what the recursion of \cite{BHLMR14,BE13} computes for other choices of $s$ remains a mystery, and the work in the current paper may help address this question.

\subsection*{Acknowledgments} We thank Elba Garcia-Failde, Alessandro Giacchetto, Paolo Gregori, Kohei Iwaki, and Paul Norbury for useful discussions and for suggesting relevant examples of limit curves, and Alessandro Chiodo and Ran Tessler for discussions about the Chiodo bundle. \par
V.B. and R.K. acknowledge support from the National Science and Engineering Research Council of Canada. N.K.C. was partially supported  by the Max Planck Institute for Mathematics Bonn and partially by the European Research Council (ERC) under the European Union’s Horizon 2020 research and innovation programme under grant agreement No 948885. R.K. acknowledges support from  the Pacific Institute for the Mathematical Sciences. The research and findings may not reflect those of these institutions. The University of Alberta respectfully acknowledges that they are situated on Treaty 6 territory, traditional lands of First Nations and Métis people. S.S was supported by the Netherlands Organisation for Scientific Research.

\section{Spectral curves and topological recursion}
\label{S2}
We review the standard framework of topological recursion, based on the notion of a spectral curve. Topological recursion was initially proposed by Chekhov, Eynard and Orantin for spectral curves with simple ramification in \cite{CEO06,EO07,EO09}, and generalised to spectral curves with arbitrary ramification in \cite{BHLMR14,BE13}. 

In fact, in this section we extend the existing theory a little bit, by: 1) allowing arbitrary base points in the recursion kernel; 2) allowing the first possibility in (lA2) and the last possibility in (lA3) of \cref{de:localadm} for locally admissible spectral curves; 3) including poles of $x$ as ramification points in the definition of topological recursion. While some of these possibilities were already considered in the literature, we will need a systematic investigation of the theory of topological recursion in this context. The main result of this section is \cref{th:ale}, which extends the known foundational result to this slightly more general setting. This turns out to be important in the study of limits of spectral curves.

\begin{remark}
In this work, when we talk about topological recursion we always refer to the version for spectral curves with arbitrary ramification formulated in \cite{BHLMR14,BE13} (sometimes known as ``Bouchard-Eynard topological recursion''). It is worth noting that there exist other generalizations of topological recursion; see for instance \cref{s:recent_work} for a discussion of the ``generalized topological recursion'' framework of \cite{ABDKS25} and its relation to the current work.
\end{remark}

\subsection{Spectral curves}

{}

\begin{definition}[Spectral curve]\label{de:sc}
	A \emph{spectral curve} is a quadruple $ \mathcal{S} = (\Sigma, x,\omega_{0,1},\omega_{0,2})$, where:
	\begin{itemize}
		\item $\Sigma = \bigsqcup_{i=1}^{N} \Sigma_i$ is a disjoint union of Riemann surfaces,  for some $N \in \mathbb{Z}_{>0}$;
		\item $x \colon \Sigma \rightarrow \mathbb{P}^1$ is a holomorphic map, whose restriction to each component $\Sigma_i$ is non-constant;
		\item $\omega_{0,1}$ is a meromorphic $1$-form on $\Sigma$, whose restriction to each component $\Sigma_i$ is not identically zero;
		\item $\omega_{0,2}$ is a \emph{fundamental bidifferential}, i.e. a symmetric meromorphic bidifferential on $\Sigma \times \Sigma$, whose only poles consist of a double pole on the diagonal with biresidue $1$.
		\end{itemize}
	The $1$-form $\omega_{0,1}$ is equivalently specified by a holomorphic map $y : \Sigma \rightarrow \mathbb{P}^1$ such that $\omega_{0,1} = y \dd x$. We say that the spectral curve is:
	\begin{itemize}
		\item  \emph{finite} if, for all $q \in \mathbb{P}^1$, $x^{-1}(q) \subset \Sigma$ is a finite subset;
		\item  \emph{connected} if $\Sigma$ is connected (i.e. $N=1$);
		\item \emph{compact} if $\Sigma$ is compact.
	\end{itemize}
	  For a compact connected spectral curve $ \Sigma$, the data of a \emph{Torelli marking}, i.e. a symplectic basis $(\mathcal{A}_i,\mathcal{B}_i)_{i}$ of $ H_1 (\Sigma ; \Z )$, determines a unique $ \omega_{0,2}$ by the normalisation condition $ \int_{\mathcal{A}_i} \omega_{0,2}(w_0,\cdot) = 0 $ for all $i$.
	  If $\Sigma$ is of genus $0$, there is a unique fundamental bidifferential $\omega_{0,2}^{{\rm std}}$ (``std" standing for standard), namely
	  \[
	 	 \omega_{0,2}^{{\rm std}}(w_1,w_2) = \frac{\dd w_1\,\dd w_2}{(w_1 - w_2)^2}
	  \]
	  in terms of any uniformising coordinate $w$ on $\Sigma \simeq \mathbb{P}^1$, and we always choose this one implicitly.
	  \end{definition}

Locally, every non-constant holomorphic map between Riemann surfaces looks like a power map. More precisely, let $p \in \Sigma$ and $x(p) \in \mathbb{P}^1$. Then there exists a local coordinate $\zeta$ on the component $\Sigma_i$ to which $p$ belongs, and a local coordinate on $\mathbb{P}^1$, centered respectively at $p$ and $x(p)$, in which the map $x \colon \Sigma \rightarrow \mathbb{P}^1$ takes the local normal form $\zeta \mapsto \zeta^{r_p}$ for some $r_p \in \mathbb{Z}_{> 0}$. Specifically, if $x(p) \in \mathbb{C}$ we can write $x = x(p) + \zeta^{r_p}$, and if $p$ is a pole of $x$, we can write $x= \zeta^{-r_p}$. The integer $r_p$ is uniquely defined and is called the \emph{ramification order} of $x$ at $p$. We call $\zeta$ a \emph{standard coordinate} at $p$. It is defined up to multiplication by a $r_p$-th root of unity. We say that $p \in \Sigma$ is a \emph{ramification point} of $x$ if $r_p \geq 2$. Let $\mathsf{Ram} \subset \Sigma$ be the set of ramification points, and $\mathsf{Br} = x (\mathsf{Ram}) $ be the set of \emph{branch points}; both sets are discrete. The ramification points correspond to the poles of $x$ of order $\geq 2$ and the zeros of its differential $\dd x$.

For $z \in \Sigma$, we define $\mathfrak{f}(z) = x^{-1}(x(z))$ and $\mathfrak{f}'(z) = \mathfrak{f}(z) \setminus \{ z \}$. If the spectral curve is finite, $\mathfrak{f}(z)$ and $\mathfrak{f}'(z)$ are finite subsets, but in general they may not be finite.

Let $q \in x(\Sigma) \subseteq \mathbb{P}^1$. Since branch points are isolated, there exists an open neighbourhood $V$ of $q$ and pairwise disjoint open neighbourhoods $U_{p}$ for each $p \in x^{-1}(q)$, such that $x^{-1}(V) = \bigsqcup_{p \in x^{-1}(q)} U_p$. We can choose $V$ and $U_p$ to be simply-connected and arbitrarily small. For any $z \in x^{-1}(V\setminus \{q\})$ and $p \in x^{-1}(q)$, we write $\mathfrak{f}_{p}(z) = \mathfrak{f}(z) \cap U_{p}$, which consists of the preimages of $x(z) \in V$ that remain in the neighbourhood $U_p$ of $p$. Even though $\mathfrak{f}(z)$ may not be finite, $\mathfrak{f}_p(z)$ always is, and $| \mathfrak{f}_{p}(z)|= r_{p}$. For $z \in U_p \setminus \{p\}$, we also write $\mathfrak{f}_{p}'(z) = \mathfrak{f}_{p}(z) \setminus \{ z \}$.  If $p$ is unramified, then $r_p = 1$ and $\mathfrak{f}'_p(z) = \emptyset$. Given a subset $P \subseteq x^{-1}(q)$, we will  use the notation $\mathfrak{f}_P(z) = \mathfrak{f}(z) \cap (\bigsqcup_{p \in P} U_p)$ for the preimages of $x(z)$ remaining near $P$, and the notation $\mathfrak{f}'_P(z) = \mathfrak{f}_P(z) \setminus \{z\}$.

\begin{remark}
	If the spectral curve is compact, $x \colon \Sigma \rightarrow \mathbb{P}^1$ is a branched covering of finite degree $d \in \mathbb{Z}_{>0}$. Then for all $z \notin \mathsf{Ram}$, we have $|\mathfrak{f}(z)| = d$. Besides, for any $q \in \mathbb{P}^1$, we have $\sum_{p \in x^{-1}(q)} r_p = d$.
\end{remark}

On $U_p$, the meromorphic $1$-form $\omega_{0,1}$ can be expanded in a standard coordinate $\zeta$. Since $\omega_{0,1}$ is assumed to be not identically zero on the component in which $p$ resides, we get
\begin{equation}\label{eq:expomega}
	\omega_{0,1} \approx \sum_{l \geq \bar s_p} \tau_{p,l}\, \zeta^{l} \frac{\dd \zeta}{\zeta}\,,
\end{equation}
for some $\bar s_p \in \mathbb{Z}$ and $\tau_{p,l} \in \mathbb{C}$, with $ \tau_p \coloneqq \tau_{p,\bar{s}_p} \neq 0$. We also define:
\begin{equation}
	s_{p} \coloneqq \min\{ l \in \mathbb{Z}~|~ \tau_{p,l} \neq 0 \,\, {\rm and} \,\, r_p \nmid l \}\,.
\end{equation}
When all exponents with non-vanishing coefficients in the expansion \eqref{eq:expomega} are divisible by $r_p$ (in particular, if $p$ is unramified, i.e. $r_p = 1$), there is no such minimum and we set $s_p = \infty$. In any case, we have $s_p \geq \bar s_p$. To sum up:

\begin{definition}[Local parameters and type]\label{de:localparameters}
	Let  $ \mathcal{S} = (\Sigma, x,\omega_{0,1}, \omega_{0,2})$ be a spectral curve and $p \in \Sigma$. We define the \emph{local parameters} $(r_p, s_p,\bar s_p,\tau_p)$  at $p$ as follows:
	\begin{itemize}
		\item $r_p$ is the ramification order of $x: \Sigma \rightarrow \mathbb{P}^1$ at $p$;
		\item $\bar s_p$ and $\tau_p$ give the minimal exponent and leading order coefficient of the expansion of $\omega_{0,1}$ near $p$ in a standard coordinate, as written in \eqref{eq:expomega};
		\item $s_p$ is the minimal exponent of the expansion of $\omega_{0,1}$, as written in \eqref{eq:expomega}, which is not divisible by $r_p$, if it exists, or is equal to $\infty$ otherwise.
	\end{itemize}
	The triple $(r_p,s_p,\bar s_p)$ is called the \emph{type} of $\mathcal{S}$ at $p$.
\end{definition}

\begin{remark}
\label{re:taupambiance} $\tau_p$ depends on the choice of a standard coordinate at $p$. Given a $r_p$-th root of unity $u$ we can replace $\zeta$ with $u\zeta$, and this replaces $\tau_p$ with $u^{\frac{\bar s_p}{r_p}} \tau_p$.
\end{remark}

We can now introduce the notion of a locally admissible spectral curve, which restricts the allowed behaviour of the spectral curve near its ramification points. 

\begin{definition}[Local admissibility]\label{de:localadm}
	Let $\mathcal{S} = (\Sigma, x, \omega_{0,1}, \omega_{0,2})$ be a spectral curve. We say that the spectral curve is \emph{locally admissible} at a point $p \in \Sigma$ if either $x$ is unramified at $p$, or $p$ is ramified and the local parameters satisfy three conditions:
	\begin{enumerate}[({lA}1)]
		\item $s_p$ and $r_p$ are coprime;
		\item either $s_p \leq -1$, or $s_p \in [r_p + 1]$ and $r_p = \pm 1\,\,{\rm mod}\,\,s_p$;
		\item $\bar s_p = s_p$ or $\bar s_p = s_p-1$.
	\end{enumerate}
	We say that the spectral curve is \emph{locally admissible} if  the set of ramification points $\mathsf{Ram} \subset \Sigma$ is finite and $\mathcal{S}$ is locally admissible at all $p \in \Sigma$.
\end{definition}

	This notion of local admissibility is tailored for \cref{th:ale} to hold.

If $p$ is a ramification point, (lA3) implies that $s_p$ is finite and if furthermore $\bar s_p \neq s_p$, we must have $s_p = 1\,\,{\rm mod}\,\,r_p$. We could drop the mention of  $s_p \in [r_p + 1]$ in (lA2) because for $s_p > 0$ it is implied by the congruence $r_p = \pm 1\,\,{\rm mod}\,\,s_p$; although it is redundant we have included it so that it remains evident to the reader. If $s_p > 0$, then the coprimality condition (lA1) is implied by (lA2) via the congruence condition, but this is not the case if $s_p \leq -1$. The only locally admissible type with $s_p = \infty$ corresponds to $p$ unramified.

	For $s_p > 0$ the necessary congruence condition in (lA2) was identified in \cite{BBCCN18}. The case $s_p \leq -1$ is usually not discussed in the literature, as those ramification points do not contribute to topological recursion (cf. \cref{th:ale}). We nevertheless include it here because such ramification points naturally arise in the context of compact spectral curves (cf. \cref{s:algcurves}) and they will have to be examined in the globalisation and limit procedures (Sections~\ref{S3} and \ref{S5}).

To complete our lexicon of spectral curves we introduce the notion of aspect ratio, which will become relevant when we discuss globalisation in \cref{S3}. Its relation to the notion of slope for algebraic curves is explained in \cref{Ssloperatio}.

	\begin{definition}[Aspect ratio]\label{de:aspectratio}
	The aspect ratio at $p \in \Sigma$ is $\nu_p = \frac{\bar{s}_p}{r_p}$.
	\end{definition}{}{}
	Note that if $p$ is ramified and the spectral curve is  locally admissible at $p$, then $\nu_p \in (-\infty,1]$.

\subsection{Correlators and topological recursion}

Topological recursion is a formalism that constructs symmetric $n$-differentials living on $\Sigma^n$, where $\Sigma$ is the complex curve underlying a spectral curve. These multi-differentials provide a solution to a system of equations known under the name of ``abstract loop equations''.

\subsubsection{Correlators}
\label{Sec:correldef}

\begin{definition}[System of correlators]\label{de:corr}
 Given a spectral curve $\mathcal{S} = (\Sigma,x,\omega_{0,1},\omega_{0,2})$, a \emph{system of correlators} is a family $(\omega_{g,n})_{g,n}$ of symmetric $n$-differentials on $\Sigma^n$ indexed by ${(g,n) \in \mathbb{Z}_{\geq 0} \times \mathbb{Z}_{> 0}}$, where $\omega_{0,1}$ and $\omega_{0,2}$ are already specified by $\mathcal{S}$, and for $2g - 2 + n > 0$, $\omega_{g,n}$ has poles only at ramification points of $x$ with vanishing residues. 
\end{definition}

To define topological recursion and abstract loop equations, we have to consider the following combinations of correlators:

\begin{definition}\label{de:calW}
	Given a system of correlators $(\omega_{g,n})_{(g,n) \in \mathbb{Z}_{\geq 0} \times \mathbb{Z}_{> 0}}$ on a spectral curve, we define
	\begin{equation}
		\label{Wprimes}\mathcal{W}_{g,i;n}(z_{[i]};w_{[n]}) = \sum_{\substack{\mathbf{L} \vdash [i] \\ \sqcup_{L \in \mathbf{L}} M_L = [n] \\ i + \sum_{L \in \mathbf{L}} (g_L - 1) = g}} \prod_{L \in \mathbf{L}} \omega_{g_L,|L| + |M_L|}(z_{L},w_{M_L})\,,
	\end{equation}
	and $\mathcal{W}'_{g,i;n}$ the similar quantity in which the terms containing a factor of $\omega_{0,1}$ are discarded.
\end{definition}

\subsubsection{Topological recursion}

We now define topological recursion.

\begin{definition}[Choice of local primitives]\label{de:recker}
For each $p \in \Sigma$, we choose $\alpha_{0,2}^{(p)}(w_0;z)$ which is a meromorphic $1$-form with respect to $w_0  \in \Sigma$ and a meromorphic function with respect to $z \in U_p$, such that $\dd_z \alpha_{0,2}^{(p)}(w_0;z) = \omega_{0,2}(w_0,z)$.  Such local primitive exist because $U_p$ is simply-connected and $\omega_{0,2}$ has no residues. The canonical local primitive is $\alpha_{0,2}^{(p)}(w_0;z) = \int_{p}^{z} \omega_{0,2}(w_0,\cdot)$.
\end{definition}

We observe that such local primitives admit a simple pole at $w_0 = z$ with residue $1$, but it could also admit other poles. For instance, the canonical local primitive admits $w_0 = p$ as  a simple pole as well with residue $-1$.
 
\begin{definition}[Recursion kernel]\label{de:kernel}
	Given a spectral curve $\mathcal{S} = (\Sigma,x,\omega_{0,1},\omega_{0,2})$, we introduce the meromorphic function $y = \frac{\omega_{0,1}}{\dd x}$ on $\Sigma$. If $Z$ is a finite tuple of points in $\Sigma$ (the order of the points in $Z$ will be irrelevant, so we use set notations), we define:
	\begin{equation}
		\label{UpsZZZ}\Upsilon_{|Z|}(Z;z) = \prod_{z' \in Z} (y(z') - y(z))\dd x(z)\, .
	\end{equation}
 The recursion kernel is then:
	\begin{equation}\label{eq:rkernel}
		K_{1+|Z|}^{(p)}(w_0; z,Z) = - \frac{\alpha_{0,2}^{(p)}(w_0;z)}{\Upsilon_{|Z|}(Z;z) } \,,
	\end{equation}
	where $\alpha_{0,2}^{(p)}$ is the choice of primitive made in \cref{de:recker}.
\end{definition}

These objects are well-defined in a union of simply-connected little neighbourhoods of ramification points, from which the ramification points are removed: $\Upsilon_{m}$ is a degree $m$ differential with respect to $z$ and a meromorphic function with respect to the points in $Z$, while $K_{1 + m}^{(p)}(w_0;z,Z)$ is a $1$-form with respect to $w_0$, a degree $(-m)$ differential with respect to $z$ and a meromorphic function with respect to the points in $Z$.
We will mostly be interested in $Z \subseteq \mathfrak{f}(z)$. In that case, we sometimes write the right-hand side of \eqref{UpsZZZ} as $\prod_{z' \in Z} (\omega_{0,1}(z') - \omega_{0,1}(z))$.

\begin{definition}[Topological recursion]\label{de:TR}
	We say that a system of correlators $(\omega_{g,n})_{(g,n) \in \mathbb{Z}_{\geq 0} \times \mathbb{Z}_{> 0}}$ on a spectral curve $\mathcal{S} = (\Sigma,x,\omega_{0,1},\omega_{0,2})$ satisfies \emph{topological recursion} if, for all $(g,n) \in \mathbb{Z}_{\geq 0}^2$ such that $2g-2+(1+n)>0$, 
	\begin{equation}\label{eq:TR}
		\omega_{g,1+n}(w_0, w_{[n]}) =  \sum_{p \in \Sigma} \Res_{z=p}\sum_{\substack{Z \subseteq  \mathfrak{f}_p'(z) \\ |Z| \geq 1}} K_{1+|Z|}^{(p)}(w_0; z, Z)\mathcal{W}'_{g,1+|Z|;n}(z, Z; w_{[n]}) \,,
	\end{equation}
	for some choice of local primitives. 
\end{definition}

	The sum is over all points $p \in \Sigma$ on the complex curve. However, whenever $p \not\in \mathsf{Ram}$, the set $\mathfrak{f}'_p(z)$ is empty, hence the integrand trivially vanishes. Therefore, only ramification points $p \in \mathsf{Ram}$ give nonzero contributions to the sum, which is why it is usually written as a sum over ramification points in the literature. We prefer to write it as a sum over all points $p \in \Sigma$ as it will be cleaner for globalisation later on. As observed in \cite[Remark~3.10]{BE17} and argued again in \cref{th:ale}, the system of correlators satisfy topological recursion for some choice of local primitives if and only if it satisfied it for any choice of local primitives.

If a system of correlators satisfies topological recursion, it is in fact uniquely specified by the spectral curve since \eqref{eq:TR} constructs $\omega_{g,n}$ by induction on $2g - 2 + n > 0$. We however stress that the notion of system of correlators (\cref{de:corr}) requires $\omega_{g,n}$ to be symmetric in their $n$ variables, while in \eqref{eq:TR} the first variable $w_0$ plays an asymmetric role compared to the others. In other words, the existence of a system of correlators satisfying topological recursion is equivalent to the statement that \eqref{eq:TR} produces symmetric differentials at each step of the recursion. Whether this statement holds or not depends on the local type of spectral curve (cf. \cref{de:localadm,th:ale}).

\subsubsection{Graphical intermezzo}

Although the above definition may seem involved at first (in particular the definition of $\mathcal{W}'_{g,i;n}$ in \eqref{Wprimes}), it can be tamed by resorting to a graphical description. This description is well known and explains the name ``topological recursion''. We review it to facilitate the orientation of the reader in the (sometimes heavy) algebra that we will have to handle.

We represent $\omega_{g,1+n}(w_0,\ldots,w_n)$ by a smooth compact oriented surface $\mathbf{S} = \mathbf{S}_{g,1 + n}$ of genus $g$ with $1 + n$ boundaries carrying the variables $w_0,\ldots,w_n$ as labels. We think of this up to label-preserving orientation-preserving diffeomorphisms. Terms in $\mathcal{W}'_{g;i,n}(z,Z;w_1,\ldots,w_n)$ are in one-to-one correspondence with the topological type of surfaces obtained by cutting off from $\mathbf{S}$ an embedded sphere $\mathbf{P}$ with $1 + i$ boundaries such that:
\begin{itemize}
\item $\mathbf{P}$ bounds the boundary of $\mathbf{S}$ labelled with $w_0$;
\item the other $i$ boundaries of $\mathbf{P}$ are essential (possibly boundary parallel) curves in the interior of $\mathbf{S}$ and are labelled by the elements of $Z \sqcup \{z\}$.
\end{itemize}
Contrarily to $\mathbf{S}$, the surface $\mathbf{S}_{{\rm cut}} = \overline{\mathbf{S} \setminus \mathbf{P}}$ may not be connected. The topological type of $\mathbf{S}_{{\rm cut}}$ is fully characterised by the data of the genus, number of boundaries and boundary labels of its connected components. The parts $L$ of the partition $\mathbf{L}$ in \eqref{Wprimes} collect the labels of $Z \sqcup \{z\}$ carried by boundaries belonging to the same connected component. For the connected component $\mathbf{S}_{{\rm cut},L}$ corresponding to $L$, the set $M_L$ collects the boundary labels (among $w_1,\ldots,w_n$) carried by the boundaries of $\mathbf{S}$ which are also boundaries of $\mathbf{S}_{{\rm cut},L}$. Contrarily to $L$, $M_L$ can be empty. Due to the ``essential'' condition, $\mathbf{S}_{{\rm cut},L}$ cannot be a disc.

This geometric picture allows for instance to identify the condition
$$
i - |\mathbf{L}| + \sum_{L \in \mathbf{L}} g_L = g\,
$$
with the relation between the genus $g$ of $\mathbf{S}$ and the genus $g_L$ of the connected components $\mathbf{S}_{{\rm cut},L}$ of $\mathbf{S}_{{\rm cut}}$. Observe that the number $|\mathbf{L}|$ of parts of $\mathbf{L}$ counts the connected components. The examples in \cref{fig:Wgraphic} should make clear how the formula \eqref{eq:TR} works in general.

\begin{figure}[!ht]
\begin{center}
\includegraphics[width=\textwidth]{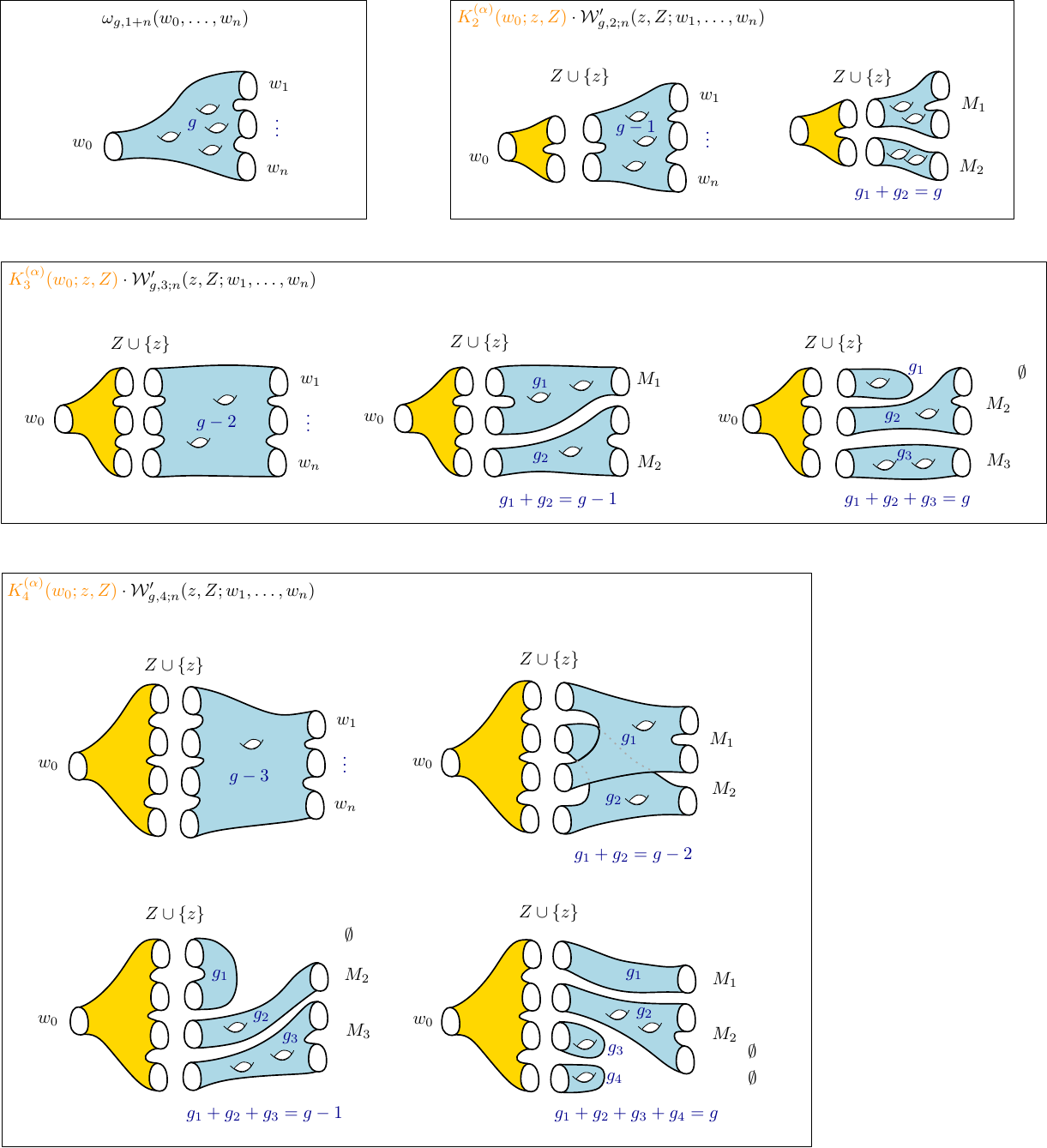}
\caption{\label{fig:Wgraphic} Examples of terms arising in the topological recursion formula \eqref{eq:TR} for $i = 2,3,4$, with all possible numbers of connected components for $\mathbf{S}_{{\rm cut}}$.}
\end{center}
\end{figure}

\subsubsection{Abstract loop equations}

Topological recursion is interesting because it constructs the unique solution (up to a projection property) of a system of equations known as abstract loop equations. This perspective was put forward in \cite{BEO13,BS17} for simple ramification points and generalised  to higher-order in \cite{BBCCN18,BKS20}.

Recall that given a point $q \in x(\Sigma) \subseteq\mathbb{P}^1$, we have an open neighbourhood $V$ of $q$ and disjoint open neighbourhoods $U_p$ for $p \in x^{-1}(q)$ such that $x^{-1}(V) = \bigsqcup_{p \in x^{-1}(q)} U_p$. We first define the following objects:

\begin{definition}\label{de:bigeps} 
	Let $(\Sigma,x,\omega_{0,1},\omega_{0,2})$ be a spectral curve, $(\omega_{g,n})_{(g,n) \in \mathbb{Z}_{\geq 0} \times \mathbb{Z}_{> 0}}$ a system of correlators, $q \in x(\Sigma)$ and $p \in x^{-1}(q)$. For $z \in x^{-1}(V)$ and $i \in [r_p]$, we define 
	\begin{equation}\label{eq:bigeps}
		\mathcal{E}_{g,i;n}^{(p)}(z; w_{[n]}) =  \sum_{\substack{Z \subseteq \mathfrak{f}_p(z) \\ |Z| = i}} \mathcal{W}_{g,i;n}(Z;w_{[n]}) \,.
	\end{equation}
\end{definition}

The object $\mathcal{E}_{g,i;n}^{(p)}(z; w_{[n]}) $ with respect to the variable $z \in x^{-1}(V)$ is the pullback by $x_{|U_p}$ of a degree $i$ differential on $V$ (the other variables $w_{[n]}$ are spectators), because the sum over the subsets $Z \subseteq \mathfrak{f}_p(z)$ with $|Z| = i$ is fully invariant under monodromy around the point $q = x(p) \in V$. In fact, the sum can be formulated as the pullback of the pushforward to $V$, see \cite[Section~5.2]{BKS20}.

\begin{remark}
	If $p \not\in \mathsf{Ram}$, then $r_p=1$ and the only nonzero object in \eqref{eq:bigeps} is
	\begin{equation}
		\mathcal{E}_{g,1;n}^{(p)}(z; w_{[n]}) = \mathcal{W}_{g,1;n}(\mathfrak{f}_p(z); w_{[n]}) = \omega_{g,1+n}(z,w_{[n]}) \,.
	\end{equation}
\end{remark}

The abstract loop equations are statements about the behaviour of these differentials on $x^{-1}(V)$.
\begin{definition}[Abstract loop equations]\label{de:localloop}
	A system of correlators $(\omega_{g,n})_{(g,n) \in \mathbb{Z}_{\geq 0} \times \mathbb{Z}_{> 0}}$ on a spectral curve $\mathcal{S} = (\Sigma,x,\omega_{0,1},\omega_{0,2})$ satisfies \emph{abstract loop equations} if, for all $(g,n) \in \mathbb{Z}_{\geq 0}^2$ such that $2g-2+(1+n)>0$, $q \in x(\Sigma)$, $p \in x^{-1}(q)$, and $i \in [r_p]$, the object $\mathcal{E}_{g,i;n}^{(p)}(z; w_{[n]}) $ defined for $z \in x^{-1}(V)$ is the pullback by $x$ of a meromorphic degree $i$ differential on $V$ \emph{that may have a pole of order at most $\mf{d}_p(i)$ at $q \in V$}, with\footnote{Our conventions for $\mathfrak{d}_p(i)$ and the abstract loop equations agree with \cite{BBCCN18}. A different convention for $\mathfrak{d}_p(i)$ is to leave the $(i - 1)$ out: it is used in \cite{BKS20} and is natural from the VOA perspective, as modes of fields are indexed including a shift by the conformal weight. Besides, in \cite{BKS20} $\delta_{i,1}$ is added to $\mathfrak{d}_p(i)$ to define the abstract loop equations. The abstract loop equations in \cite{BKS20} are equivalent to the present abstract loop equations after adding the condition (here moved to \cref{de:corr}) that correlators have no residues.}
	\begin{equation}
	\label{dpidef} \mf{d}_p(i) = i-1 - \left\lfloor \frac{s_p (i-1)}{r_p}  \right\rfloor \,.
	\end{equation} 
If $s_p = \infty$, we take as convention $\mathfrak{d}_p(1) = 0$ and $\mathfrak{d}_p(i) = -\infty$ if $i > 1$ (the latter means that $\mathcal{E}_{g,i;n}^{(p)}(z;w_{[n]}) = 0$ but this will never be needed for locally admissible spectral curves).
\end{definition}
The $i = 1$ and $i = 2$ are respectively called linear and quadratic loop equations.

In other words, for $\tilde p \in x^{-1}(q)$ distinct from $p$, we have

\begin{equation}
	\mathcal{E}_{g,i;n}^{(p)}(z; w_{[n]}) = \mc{O}\!\left( x^* \left( \frac{\dd t^i}{t^{\mf{d}_p(i)}} \right) \right) \quad {\rm as}\,\, z \rightarrow \tilde p\,,
\end{equation}
where $t$ is a coordinate on $V$ centered at $q \in V$. In particular, if we use the local normal form $\zeta \mapsto \zeta^{r_{\tilde p}}$ for $x$ on $U_{\tilde p}$,  we can write
\begin{equation}\label{eq:aleexp}
	\mathcal{E}_{g,i;n}^{(p)}(z; w_{[n]}) = \mc{O}\!\left( \frac{\dd \zeta^i}{\zeta^{r_{\tilde p} \mf{d}_p(i) - (r_{\tilde p}-1)i}} \right) \quad {\rm as}\,\, z \rightarrow \tilde p\,.
\end{equation}

\begin{remark}
It is important to note that even though we used a particular $p \in x^{-1}(q)$ to define the object $\mathcal{E}_{g,i;n}^{(p)}(z; w_{[n]})$ via the sum over subsets $Z \subseteq \mathfrak{f}_p(z)$, the behaviour \eqref{eq:aleexp} holds true as $z \rightarrow \tilde p$ for all $\tilde p \in x^{-1}(q)$, not just $\tilde p = p$. The dependence of $\mathfrak{d}_p(i)$ on $p$ then make the abstract loop equations subtle when there are several ramification points above the same branch point.
\end{remark}

To solve the abstract loop equations, we will require a specific property on the correlators.

\begin{definition}[Projection property]\label{de:projection}
	A system of correlators $(\omega_{g,n})_{(g,n) \in \mathbb{Z}_{\geq 0} \times \mathbb{Z}_{> 0}}$ on a spectral curve $\mathcal{S} = (\Sigma,x,\omega_{0,1},\omega_{0,2})$ satisfies the \emph{projection property}  if, for all $(g,n) \in \mathbb{Z}_{\geq 0}^2$ such that $2g - 2 + (1 + n) > 0$:
	\begin{equation}\label{eq:projection}
		\omega_{g,1+n}(w_0, w_{[n]}) = \sum_{p \in \Sigma} \Res_{z=p} \alpha_{0,2}^{(p)}(w_0;z)\omega_{g,1+n}(z,w_{[n]}) \,,
	\end{equation}
	where $\alpha_{0,2}^{(p)}$ is a choice of local primitive as in Definition~\ref{de:recker}. 
\end{definition}

Like in the topological recursion formula \eqref{eq:TR}, the sum is over all $p \in \Sigma$, but only ramification points $p \in \mathsf{Ram}$ give nonzero contributions as the correlators only have poles at the ramification points.
	
Since the $\omega_{g,n}$ have no residues at the ramification points, we can find a local primitive of $\omega_{g,n}(z,w_{[n]})$ with respect to $z$ and use it for an integration by parts in the residue at $p$. This shows that the projection property holds for a choice of local primitives $\alpha_{0,2}^{(p)}$  if and only if it holds for any choice of local primitives. 

\subsubsection{Topological recursion solves abstract loop equations}

The main reason behind the appearance and usefulness of topological recursion is that, for locally admissible spectral curves, it provides the unique solution to abstract loop equations that also satisfies the projection property. For spectral curves with simple regular ramification points (i.e. $(r_p, s_p) = (2,3)$ for all $ p\in\mathsf{Ram}$), this is a classical result of Eynard and Orantin in \cite{EO07}, which was generalised to positive $s_p = \bar s_p $ in \cite{BBCCN18} using the formalism of Airy structures of Kontsevich and Soibelman \cite{KS17}. Here we  extend the result to cover as well the cases $\bar s_p \leq -1$ and $\bar s_p = s_p - 1$, and allow arbitrary choice of local primitive of $\omega_{0,2}$ in the numerator of the recursion kernel.

\begin{theorem}[Topological recursion solves abstract loop equations]\label{th:ale}
	Let $\mathcal{S} = (\Sigma,x,\omega_{0,1},\omega_{0,2})$ be a locally admissible spectral curve. Then there exists a unique system of correlators on $\mathcal{S}$ that satisfies topological recursion for some choice of local primitives. It is also the unique system of correlators that satisfies the abstract loop equations and the projection property, and it satisfies topological recursion for any choice of local primitives. Moreover, in the topological recursion formula \eqref{eq:TR}, the residue at each point $p \in \Sigma$ such that $r_p = 1$ or $\bar s_p \leq -1$ vanishes.
\end{theorem}

Let us comment on the necessity of the local admissibility conditions of \cref{de:localadm}. If the coprime condition (lA1) did not hold, we would have $r_p > 1$ and $\Upsilon_{r_p - 1}(\mathfrak{f}_p'(z);z)$ would vanish identically: the recursion kernel \eqref{eq:rkernel} would be ill-defined. If (lA2) was violated, namely if we had $s_p > r_p + 1$ or $s_p \in [r_p + 1]$ with $r_p \neq \pm 1\,\,{\rm mod}\,\,s_p$, it was found in \cite{BBCCN18} that the topological recursion formula \eqref{eq:TR}  typically gives non-symmetric correlators. If (lA3) was violated, i.e. $s_p > \bar s_p + 1$, the proof of \cref{th:ale} (look around \eqref{sunsetsun}) reveals that the system of correlators may not satisfy the abstract loop equations.

\begin{proof}
	The idea of the proof is to first show that there exists a unique system of correlators that solves the abstract loop equations and that satisfies the projection property, and then show that it must be given by topological recursion.

	We first note that for any $p \not\in\mathsf{Ram}$, the abstract loop equations reduce to the statement that
	\begin{equation}
		\mathcal{E}^{(p)}_{g,1;n}(z; w_{[n]}) = \mathcal{W}_{g,1;n}(\mathfrak{f}_p(z); w_{[n]}) = \omega_{g,1+n}(z, w_{[n]}) 
	\end{equation}
	is the pullback by $x$ of a holomorphic $1$-form on $V$. Using the local normal form on $U_p$, which takes the form $\zeta \mapsto \zeta$ as $p$ is unramified, the statement becomes (this is \eqref{eq:aleexp} with $r_p=1$):
	\begin{equation}
\mathcal{E}^{(p)}_{g,1;n}(z; w_{[n]}) = \mathcal{W}_{g,1;n}(\mathfrak{f}_p(z); w_{[n]}) = \omega_{g,1+n}(z, w_{[n]})  = \mc{O}(\dd \zeta) \quad {\rm as}\,\,z \rightarrow p\,.
	\end{equation}
In other words, the abstract loop equations at unramified points are satisfied if and only if the correlators $\omega_{g,n}$ are holomorphic away from the ramification points. This is the reason why we already included this condition in the definition of a system of correlators, \cref{de:corr}.

	Let us now consider the abstract loop equations at ramification points. Let us first assume that all $p \in \mathsf{Ram}$ are such that $s_p \in [r_p + 1]$ with $r_p = \pm 1\,\,{\rm mod}\,\,s_p$, and $\bar s_p=s_p$. This is the setup considered in \cite{BBCCN18}. Using Airy structures constructed as modules for $\mathcal{W}$-algebras, it is shown in \cite{BBCCN18} that there  always exists a unique system of correlators that satisfies the abstract loop equations and the projection property; this is essentially \cite[Theorem 5.27]{BBCCN18} and we do not repeat this part of the proof here. The proof that this system of correlators must then satisfy topological recursion was provided in \cite[Appendix C]{BBCCN18}; we sketch the proof here, since our recursion kernel is slightly different from the setup there\footnote{The difference in the recursion kernel is that we allow arbitrary choices of local primitives of $\omega_{0,2}$ in the numerator of the recursion kernel, while in \cite{BBCCN18} we took the canonical local primitive $\alpha_{0,2}^{(p)}(w_0;z) = \int_{p}^{z} \omega_{0,2}(w_0,\cdot)$. However, the proof goes through in exactly the same way.}.

	Assume that there exists a solution to the abstract loop equations, and that it satisfies the projection property. Then we claim that, for all $p \in \mathsf{Ram}$, the differential (in $z$)
	\begin{equation}\label{eq:tshol}
		\frac{1}{\Upsilon_{r_p-1}(\mathfrak{f}_p'(z); z)} \sum_{i=1}^{r_p} (- \omega_{0,1}(z))^{r_p-i} \mathcal{E}^{(p)}_{g,i;n}(z; w_{[n]})
	\end{equation}
	is holomorphic as $z \rightarrow  p$. Let us use the local normal form $\zeta \mapsto \zeta^{r_p}$ for $z \in U_p$. By the abstract loop equations (cf.  \eqref{eq:aleexp}), we know that
	\begin{equation}
		\mathcal{E}_{g,i;n}^{(p)}(z; w_{[n]}) = \mc{O}\!\left( \frac{\dd \zeta^i}{\zeta^{r_p \mathfrak{d}_p(i) - (r_p-1)i}} \right) \,.
	\end{equation}
By definition of the local parameters (cf. \eqref{eq:expomega}), we know that
		\begin{equation}
			(- \omega_{0,1}(z))^{r_p-i} = \mc{O}\!\left( \zeta^{(r_p-i)(s_p-1)} \dd \zeta^{r_p-i} \right) \,,
	\end{equation}
since we assumed that $\bar s_p = s_p$. Finally,  using \eqref{eq:expomega} again, we see that
	\begin{equation}
		\Upsilon_{r_p-1}(\mathfrak{f}_p'(z); z) = \mc{O}\!\left( \zeta^{(r_p-1)(s_p-1)} \dd \zeta^{r_p-1} \right) \,.
	\end{equation}
Putting all this together, we see that each term in the sum over $i$ in \eqref{eq:tshol} behaves like
	\begin{equation}
		\frac{(- \omega_{0,1}(z))^{r_p-i} \mathcal{E}^{(p)}_{g,i;n}(z; w_{[n]})}{\Upsilon_{r_p-1}(\mathfrak{f}_p'(z); z)} = \mc{O}\!\left( \zeta^{(r_p-i)(s_p-1) - (r_p-1)(s_p-1) +(r_p-1)i - r_p \mathfrak{d}_p(i)} \dd \zeta \right) \,.
	\end{equation}
Simplifying the exponent, we find that these terms are holomorphic as $z \rightarrow p$ if and only if
	\begin{equation}
		r_p + s_p -  s_p i - 1 + r_p \left \lfloor \frac{s_p}{r_p}(i-1) \right \rfloor \geq 0 \,.
	\end{equation}
But
	\begin{equation}
		r_p + s_p - s_p i - 1 + r_p \left \lfloor \frac{s_p}{r_p}(i-1) \right \rfloor  > r_p + s_p - s_p i - 1 + r_p \left( \frac{s_p}{r_p}(i-1)  - 1 \right) = -1\,,
	\end{equation}
	and since this is a strict inequality and the left-hand side is an integer, this left-hand side must be nonnegative. Therefore, each term in the sum over $i$ in \eqref{eq:tshol} is holomorphic as $z \rightarrow p$, and so is the sum.

	Since \eqref{eq:tshol} is holomorphic as $z \rightarrow p$ for all $ p \in \mathsf{Ram}$, we can write
	\begin{equation}\label{eq:weq}
		\sum_{p \in \mathsf{Ram}} \Res_{z=p} \frac{\alpha_{0,2}^{(p)}(w_0;z)}{\Upsilon_{r_p-1}(\mathfrak{f}_p'(z); z)} \sum_{i=1}^{r_p} (- \omega_{0,1}(z))^{r_p-i} \mathcal{E}^{(p)}_{g,i;n}(z; w_{[n]}) = 0
	\end{equation}
	for an arbitrary choice of local primitive $\alpha_{0,2}^{(p)}$ of $\omega_{0,2}$ (cf. \cref{de:recker}). We then use the combinatorial identity 
	\begin{equation}\label{eq:comb}
		\sum_{i=1}^{r_p} (-\omega_{0,1}(z))^{r_p-i}  \mathcal{E}^{(p)}_{g,i;n}(z;w_{[n]}) = \sum_{Z \subseteq \mathfrak{f}'_p(z)}  \mathcal{W}'_{g,1+|Z|;n}(z,Z; w_{[n]}) \Upsilon_{r_p - 1 - |Z|}(\mathfrak{f}'_p(z) \setminus Z; z)
	\end{equation}
	(see \cref{le:comb2} later in the text) to rewrite \eqref{eq:weq} as
	\begin{equation} 
\label{theproj22}
\sum_{p \in \mathsf{Ram}} \Res_{z=p} \alpha_{0,2}^{(p)}(w_0;z)\omega_{g,1+n}(z, w_{[n]}) = \sum_{p \in \mathsf{Ram}} \Res_{z=p} \sum_{\substack{Z \subseteq  \mathfrak{f}_p'(z) \\ |Z| \geq 1}} K_{1 + |Z|}^{(p)}(w_0; z, Z)\mathcal{W}'_{g,1+|Z|;n}(z, Z; w_{[n]}) \,,
	\end{equation}
	where we introduced the recursion kernel defined in \eqref{eq:rkernel} for the subsets $|Z| \subseteq \mathfrak{f}'_p(z)$ with $|Z| \geq 1$, and moved those terms to the right-hand side, keeping the term with $Z = \emptyset$ on the left-hand side. Finally, since the system of correlators satisfy the projection property \eqref{eq:projection}, the left-hand side is equal to $\omega_{g,1+n}(w_0, w_{[n]})$, and we arrive to the topological recursion formula \eqref{eq:TR}. This concludes the proof if $s_p \in [r_p + 1]$ with  $r_p = \pm 1\,\,{\rm mod}\,\,s_p$, and $\bar s_p = s_p$ for all $p \in \mathsf{Ram}$.

	Next, we keep the assumption that $\bar s_p = s_p$ for all $p \in \mathsf{Ram}$ but suppose that $s_p \in [r_p + 1]$  with $r_{p} = \pm 1\,\,{\rm mod}\,\,s_{p}$ only for $p \in \mathsf{Ram}' \subset \mathsf{Ram}$. In other words, the remaining ramification points  $p \in \mathsf{Ram}'' = (\mathsf{Ram} \setminus \mathsf{Ram}')$ satisfy $s_p \leq -1$ and $s_p$ coprime with $r_p$. First, we notice that the ramification points in $\mathsf{Ram}'' $ do not contribute to topological recursion. Indeed,  using the local normal form $\zeta \mapsto \zeta^{r_p}$ on $U_p$, the recursion kernel behaves as
	\begin{equation}
		\label{KZ1s}K_{1 + |Z|}^{(p)}(w_0; z, Z) = \mc{O}\!\left(\frac{\zeta^{(1-s_p)|Z|}}{\dd \zeta^{|Z|}} \right)
	\end{equation}
	as $z \rightarrow p$ while $w_0$ remains away from $\mathsf{Ram}''$. Let us show by induction that it implies that all residues at $z=p$ in the topological recursion formula vanish.
	
	We initialise with $2g - 2 + n = 1$, that is:
	\begin{equation*}
	\begin{split}
	\omega_{0,3}(w_0,w_1,w_2) & = \sum_{p \in \Sigma} \Res_{z = p} \! \sum_{z' \in \mathfrak{f}_p'(z)} \!\! K_{2}^{(p)}(w_0;z,z') \big(\omega_{0,2}(z,w_1)\omega_{0,2}(z',w_2) + \omega_{0,2}(z,w_2) \omega_{0,2}(z',w_1)\big)\,, \\
	\omega_{1,1}(w_0) & = \sum_{p \in \Sigma} \Res_{z = p} \sum_{z' \in \mathfrak{f}_p'(z)} K_2^{(p)}(w_0;z,z') \omega_{0,2}(z,z')\,.
	\end{split}
	\end{equation*}
We look at the integrand in the residue at $p \in \mathsf{Ram}''$ in these formulae: it contains in both cases a double zero at $p$ from the recursion kernel $K_2^{(p)}$, times a factor which is holomorphic in $p$ (in $\omega_{0,3}$) or contains at most a double pole in $p$. Therefore, the integrand is in both cases holomorphic at $p$ and the corresponding residue vanishes. As a result, $\omega_{0,3}(w_0,w_1,w_2)$ and $\omega_{1,1}(w_0)$ do not have poles at $p \in \mathsf{Ram}''$.

Now take $(g,n) \in \mathbb{Z}_{\geq 0} \times \mathbb{Z}_{> 0}$ such that $2g - 2 + n > 0$. We assume that for $2g' - 2 + n' < 2g - 2 + n$ the residues at points in $\mathsf{Ram}''$ never contribute to $\omega_{g',n'}$ and that the latter do not have poles at $\mathsf{Ram}''$. The topological recursion \eqref{eq:TR} expresses $\omega_{g,n}$ as a sum over residues where integrands involve $\omega_{g',n'}$ with $0 < 2g' - 2 + n' < 2g - 2 + n$ only, and $\omega_{0,2}$. By the induction hypothesis, the only source of poles in the integrand of the residue at $p \in \mathsf{Ram}''$ are $\omega_{0,2}$s evaluated at two points in the fibre $\mathfrak{f}_p(z)$. More precisely, the poles of highest possible order with respect to $z$ come from terms of the form $K_{2m}^{(p)}(w_0;z,z_2,\ldots,z_{2m}) \prod_{k = 1}^{m} \omega_{0,2}(z_{2k - 1},z_{2k})$ where $z_1 = z$ and $\{z_2,\ldots,z_{2m}\} \subseteq \mathfrak{f}'_p(z)$ for some $m \leq \lfloor \frac{r_p}{2} \rfloor$. The product of $\omega_{0,2}$s creates a pole of order $2m$ when $z \rightarrow p$, which is compensated by a zero of order at least $2m$ from the recursion kernel $K_{2m}^{(p)}$ (use \eqref{KZ1s} with $|Z| = 2m$ and $s_p \leq -1$). So, the integrands are always holomorphic at $p \in \mathsf{Ram}''$, the corresponding residue vanishes, and $\omega_{g,n}$ can never develop a pole at points of $\mathsf{Ram}''$.

By induction, we conclude that  topological recursion constructs the exact same system of correlators as if we had only included the ramification points in $\mathsf{Ram}'$ and omitted those in $\mathsf{Ram}''$ in the topological recursion formula. So, the resulting correlators only have poles at the ramification points in $\mathsf{Ram}'$.

	We know that the system of correlators constructed in this way satisfies the projection property including only the ramification points in $\mathsf{Ram}'$. Clearly, it also satisfies the projection property including all ramification points in $\mathsf{Ram}$ in the sum over residues of the right-hand side of \eqref{eq:projection}, since the correlators are holomorphic at the ramification points in $\mathsf{Ram}''$.

	The system of correlators satisfies the abstract loop equations for the ramification points in $\mathsf{Ram}'$, but we need to check that it also satisfies the abstract loop equations for the remaining ramification points. Let $p \in \mathsf{Ram}''$. The abstract loop equation states that
	\begin{equation}
	\label{Ebound}	\mathcal{E}_{g,i;n}^{(p)}(z; w_{[n]}) = \mc{O}\!\left( \frac{\dd \zeta^i}{\zeta^{r_p \mathfrak{d}_p(i) - (r_p-1)i}} \right) \,.
	\end{equation}

Since the correlators $\omega_{g,n}$ do not have poles at $z=p$, the only poles on the left-hand side come from $\omega_{0,1}$. There can be at most $i-1$ factors of $\omega_{0,1}$s on the left-hand side, so we know that the left-hand side cannot have a pole at $z=p$ of order more than $(1-s_p)(i-1) = -1 - (i-1)s_p + i$. But we can be stricter; we know that the left-hand side is the pullback via $x$ of a differential on $V$, and thus it must behave like $ \mc{O}(\zeta^{-r_pk - i}\dd \zeta^i)$ for some integer $k$. A differential which is $\mc{O}(\zeta^{-\ell - i}\dd \zeta^i)$ and is the pullback via $x$ of a differential on $V$ therefore behaves automatically like $\mc{O}(\zeta^{-r_p \big\lceil \frac{\ell}{r_p} \big\rceil - i}\dd \zeta^i)$. Coming back to the definition \eqref{dpidef} of $\mathfrak{d}_p(i)$, \eqref{Ebound} then implies
$$
\mathcal{E}_{g,i;n}^{(p)}(z;w_{[n]}) = \mc{O}\!\left(  \frac{\dd \zeta^i}{\zeta^{-r_p \left \lceil \frac{1 + s_p (i-1)}{r_p} \right \rceil+ i }} \right)\,.
$$
We must compare the exponents
$$
\theta_p \coloneqq - r_p \left\lceil \frac{1 + s_p(i - 1)}{r_p} \right\rceil + i \quad {\rm and}\quad \tilde{\theta}_p \coloneqq r_p \mathfrak{d}_p(i) - (r_p - 1)i = i - r_p - r_p \left \lfloor \frac{s_p(i - 1)}{r_p} \right\rfloor\,.
$$
We first treat the case $\frac{1 + s_p (i-1)}{r_p}  \in \mathbb{Z}$. Then  $\theta_p = -1 - s_p(i-1) + i$; besides, we have $\left \lfloor \frac{s_p(i - 1)}{r_p} \right \rfloor =   \frac{s_p(i - 1)}{r_p}  + \frac{1-r_p}{r_p}$; this implies $\tilde{\theta}_p = i - r_p - s_p(i-1) - 1 + r_p = \theta_p$. In the other case $\frac{1 + s_p (i-1)}{r_p}  \not\in \mathbb{Z}$, we have $\left \lceil \frac{1 + s_p (i-1)}{r_p} \right \rceil = 1 +  \left \lfloor  \frac{1+s_p(i-1)}{r_p} \right \rfloor = 1 +  \left \lfloor  \frac{s_p(i-1)}{r_p} \right \rfloor$ and $\theta_p = i-r_p - r_p  \left \lfloor  \frac{s_p(i-1)}{r_p} \right \rfloor = \tilde{\theta}_p$.
  
 In both cases we have $\theta_p = \tilde{\theta}_p$ and conclude that the abstract loop equations are satisfied at the ramification points in $\mathsf{Ram}''$. This concludes the proof for the general case with $\bar s_p = s_p$.

	Finally, suppose that $\bar s_p = s_p-1$ for some $p \in \mathsf{Ram}$. This can only happen when $r_p | \bar s_p$, and thus $s_p = k r_p+ 1$ for some $k \in \mathbb{Z}$. Then the leading term in $\omega_{0,1}$ does not contribute to topological recursion, since $\omega_{0,1}$ only appears through $\Upsilon_{|Z|}(Z;z)$ in the denominator of the recursion kernel, and for any $z' \in \mathfrak{f}_p'(z)$, the leading term with exponent $\bar s_p$ in $\omega_{0,1}$ drops out of $(\omega_{0,1}(z') - \omega_{0,1}(z))$. We conclude that topological recursion constructs the exact same system of correlators as if we had removed the leading term from the expansion of $\omega_{0,1}$ at $z=p$ in the local coordinate $\zeta$. The correlators thus satisfy the projection property, and the abstract loop equations as well, but for the expansion of $\omega_{0,1}$ with the leading term removed.

	We claim that they also satisfy the abstract loop equations with the full $\omega_{0,1}$. This can be proved by induction on $i \in [r_p]$. First, it is clear that the correlators satisfy the linear loop equations for $\mathcal{E}_{g,1;n}^{(p)}(z; w_{[n]})$, since $\omega_{0,1}$ does not enter in those. Now, assume that for all $j < i$ the abstract loop equations are satisfied:
	\begin{equation}
		\mathcal{E}_{g,j;n}^{(p)}(z; w_{[n]}) = \mc{O}\!\left( \frac{\dd \zeta^j}{\zeta^{r_p \mathfrak{d}_p(j) - (r_p-1)j}} \right)\,,
	\end{equation}
	where the full $\omega_{0,1}$ is included on the left-hand side. We know that the abstract loop equations are satisfied for all terms in $\mathcal{E}_{g,i;n}^{(p)}(z; w_{[n]}) $ that do not involve the leading order term in $\omega_{0,1}$, namely $\zeta^{\bar s_p - 1} \dd \zeta$. Since $r_p|\bar s_p$, this term is invariant under permutations of $\mathfrak{f}_p(z)$. Thus the extra contributions that appear in $ \mathcal{E}_{g,i;n}^{(p)}(z; w_{[n]}) $ when we include these terms all take the form
	\begin{equation}
		\label{sunsetsun} \left( \zeta^{\bar s_p - 1} \dd \zeta \right)^{i-j} \mathcal{E}_{g,j;n}^{(p)}(z; w_{[n]}) 
	\end{equation}
	for some $j \in [i-1]$. On the one hand, by the induction assumption, we know that
	\begin{equation}
		\mathcal{E}_{g,j;n}^{(p)}(z; w_{[n]}) = \mc{O}\!\left( \frac{\dd \zeta^j}{\zeta^{r_p \mathfrak{d}_p(j) - (r_p-1)j}} \right)\,.
	\end{equation}
Then
	\begin{equation}
		\left( \zeta^{\bar s_p - 1} \dd \zeta \right)^{i-j} \mathcal{E}_{g,j;n}^{(p)}(z; w_{[n]}) = \mc{O}\!\left( \frac{\dd \zeta^i}{\zeta^{r_p \mathfrak{d}_p(j) - (r_p-1)j  + (1- \bar s_p) (i-j)}} \right) \,.
	\end{equation}
But, since $\bar s_p =  k r_p$ and $s_p = k r_p + 1$ for some $k \in \mathbb{Z}$, we have
$$
\mathfrak{d}_p(i) = i - 1 - \Big\lfloor \frac{(kr_p + 1)(i - 1)}{r_p}\Big\rfloor = (i -1)(1-k)\,,
$$
and one can check $r_p \mathfrak{d}_p(j) - (r_p-1)j  + (1- \bar s_p) (i-j) = r_p \mathfrak{d}_p(i) - (r_p - 1)i$. Therefore, all new terms appearing in $\mathcal{E}_{g,i;n}^{(p)}(z; w_{[n]}) $ satisfy the abstract loop equations. By induction, we conclude that the abstract loop equations are satisfied with the full $\omega_{0,1}$ included. This concludes the proof of the two first points of the theorem for all locally admissible spectral curves.

The first time a choice of local primitives $\alpha_{0,2}^{(p)}$ of $\omega_{0,2}$ appears is \eqref{eq:weq}, and we carry it until \eqref{theproj22}, where we use the projection property to conclude that this expression is equal to $\omega_{g,1+n}$, thus proving topological recursion formula \emph{for this choice of local primitives}. But, as these manipulations can be done for any choice of local primitives and the validity of the projection property does not depend on which local primitive is chosen, we deduce that a system of correlators satisfy topological recursion for a choice of local primitives $\alpha_{0,2}^{(p)}$ if and only if it satisfies topological recursion for any choice of local primitives.

For the last point, we recall that $p$ unramified does not contribute to the residue formula since $\mathfrak{f}'_p(z) = \emptyset$. When $r_p \geq 2$ but $\bar s_p \leq -1$, local admissibility requires $\bar s_p \leq s_p - 1$, hence $s_p \leq 0$. As $s_p$ is by definition not divisible by $r_p$, we must have $s_p \leq -1$, and in the course of this proof we have already checked that such points have vanishing contribution to \eqref{eq:TR}. 
\end{proof}

\section{Globalising topological recursion}
\label{S3}

\subsection{Basic principles}

Recall the local formula of \cref{de:TR} for topological recursion
\begin{equation}\label{eq:TRbp}
	\omega_{g,1+n}(w_0, w_{[n]}) = \sum_{q \in x(\Sigma)} \sum_{p \in x^{-1}(q)} \Res_{z=p}
\sum_{\substack{Z \subseteq  \mathfrak{f}_p'(z) \\ |Z| \geq 1}} K_{1+|Z|}^{(p)}(w_0; z, Z)\mathcal{W}'_{g,1+|Z|;n}(z, Z; w_{[n]}) \,.
\end{equation}
Here we rewrote the sum over $p \in \Sigma$ as a double sum over $q \in x(\Sigma) \subseteq\mathbb{P}^1$ and $p \in x^{-1}(q)$, using the holomorphic map $x: \Sigma \rightarrow \mathbb{P}^1$. The integrand depends on the point at which we take the residue in two ways: first, we are summing over subsets of points in $\mathfrak{f}_p'(z)$ that lie within $U_p$, i.e. remain near the point $p \in \Sigma$ when $z$ is near $p$; second, the recursion kernel depends on $p$ via the choice of local primitive of $\omega_{0,2}$. The idea of globalisation is simple: it aims to rewrite the recursion so that the integrand remains the same within clusters of points $p$ (in the same $x$-fibre or not), and eventually convert the sum of residues within each cluster into a single contour integral. Besides, we want to do it so that the contours that we use are pullbacks from the base. This procedure has two steps, that we call vertical and horizontal globalisation, and will be useful when we consider families of spectral curves and allow points in given $x$-fibres to collide within each cluster (see also the comment in \cref{S:comm}).

\begin{definition}[Vertical globalisation]\label{de:globalTR}
	Let $(\omega_{g,n})_{g,n}$ be a system of correlators on a locally admissible spectral curve $\mathcal{S} = (\Sigma,x,\omega_{0,1},\omega_{0,2})$ that satisfies topological recursion. Let $q \in x(\Sigma) \subseteq \mathbb{P}^1$, and $P \subseteq x^{-1}(q)$ a finite subset. We say that topological recursion can be \emph{globalised over $P$} if, for all  $(g,n) \in \mathbb{Z}_{\geq 0}^2$ such that $2g-2+(1+n)>0$, we have
	\begin{multline}\label{eq:globalized}
		\sum_{p \in P} \Res_{z=p}\sum_{\substack{Z \subseteq  \mathfrak{f}_p'(z) \\ |Z| \geq 1}}K_{1+|Z|}^{(p)}(w_0; z, Z)  \mathcal{W}'_{g,1+|Z|;n}(z, Z; w_{[n]}) \\= \sum_{p \in P} \Res_{z=p}\sum_{\substack{Z \subseteq  \mathfrak{f}_P'(z) \\ |Z| \geq 1}}K_{1+|Z|}^{(p)}(w_0; z, Z) \mathcal{W}'_{g,1+|Z|;n}(z, Z; w_{[n]})\,.
	\end{multline}
\end{definition} 
Note that the only difference between the two sides of the equation is that the second sum on the right-hand side uses the set $\mathfrak{f}'_P(z)$ as opposed to $\mathfrak{f}'_p(z)$ on the left-hand side. If all $p \in P$ are unramified, the left-hand side of \eqref{eq:globalized} trivially vanishes since $\mathfrak{f}_p'(z) = \emptyset$ for all $p \in P$, but the vanishing of the right-hand side is still a non-trivial condition. The key point in \eqref{eq:globalized} is that, in contrast to the left-hand side, the integrand on the right-hand side is the same for all $p \in P$, except perhaps for the recursion kernel. We sometimes call the rewriting of the topological recursion formula of \cref{de:TR} using the right-hand side of \eqref{eq:globalized} a (partial if $P$ is not the full fibre) \emph{vertical globalisation}.

\begin{definition} We also say that topological recursion  can be
	\begin{itemize}
	\item \emph{partially globalised over an open set $A \subseteq \Sigma$} if for any $p \in A$, $\mathfrak{f}(p) \cap A$ is finite and topological recursion can be globalised over it;
	\item \emph{globalised above $q \in x(\Sigma) $} if it can be globalised over $x^{-1}(q)$ (this requires $x^{-1}(q)$ to be finite);
	 \item \emph{globalised above an open set $V \subseteq x(\Sigma) $} if $x$ is a finite-degree covering and topological recursion can be globalised over $x^{-1}(V)$;
	 \item \emph{fully globalised} if it can be globalised above $x(\Sigma)$.
	\end{itemize}
\end{definition}

We now explain how to treat the $p$-dependence of the recursion kernel.

\begin{definition}[Disc collection]\label{de:adapcon}
Let $\mathcal{S} = (\Sigma,x,\omega_{0,1},\omega_{0,2})$ be a spectral curve. We define $\Pi_{0,2} : H_1(\Sigma,\mathbb{Z}) \rightarrow H^0(K_{\Sigma},\Sigma)$ by
\begin{equation}
\label{02periodes}\Pi_{0,2}(\gamma)(w_0) = \int_{\gamma} \omega_{0,2}(w_0,\cdot)\,.
\end{equation}
A \emph{disc collection adapted to $\mathcal{S}$} is a finite sequence of open subsets $(\mathsf{D}_i)_{i = 1}^{\mathsf{k}}$ of $\mathbb{P}^1$ such that
\begin{itemize}
\item[(DC1)] the $\overline{\mathsf{D}}_i$ are pairwise disjoint properly embedded discs in $x(\Sigma)$;
\item[(DC2)] each $\mathsf{D}_i$ contains at least one branch point, and each branch point belongs to some $\mathsf{D}_i$;
\item[(DC3)] for each $i$, the restriction of $x$ to each connected component of $x^{-1}(\mathsf{D}_i)$ is a finite-degree branched covering onto $\mathsf{D}_i$;
\item[(DC4)] for each $i$, we have $H_1(x^{-1}(\mathsf{D}_i),\mathbb{Z}) \subseteq {\rm Ker}\,\Pi_{0,2}$.
\end{itemize}
We then denote $\mathsf{c}_i = \pi_0(x^{-1}(\mathsf{D}_i))$. For $j \in \mathsf{c}_i$, we denote $\tilde{\mathsf{D}}_{i,j}$ the corresponding connected component and $\mathsf{d}_{i,j}$ the degree of the restriction of $x$ to $\tilde{\mathsf{D}}_{i,j}$. We also denote
$$
\mathsf{c}^+_i = \big\{j \in \mathsf{c}_i\,\,|\,\,d_{i,j} \geq 2\big\}
$$
the set of connected components containing at least a ramification point, and
$$
\tilde{\mathsf{D}}^+ = \bigcup_{i = 1}^{\mathsf{k}} \bigcup_{j \in \mathsf{c}_i^+} \tilde{\mathsf{D}}_{i,j}\,. 
$$
\end{definition}

\begin{proposition}[Horizontal globalisation]\label{pr:globalTRdomain} 
 Let $(\omega_{g,n})_{(g,n) \in \mathbb{Z}_{\geq 0} \times \mathbb{Z}_{>0}}$ be a system of correlators on a locally admissible spectral curve $\mathcal{S} = (\Sigma,x,\omega_{0,1},\omega_{0,2})$ that satisfies topological recursion. Assume that we are given a disc collection adapted to $\mathcal{S}$, such that the  topological recursion can be globalised over $\mathsf{D}_{i,j}$ for each $i \in [\mathsf{k}]$ and $j \in \mathsf{c}_i^+$.
 
 Then, for each $i \in [\mathsf{k}]$ and $j \in \mathsf{c}^+_i$ there exists a $\alpha_{0,2}^{(i,j)}(w_0;z)$ which is a meromorphic $1$-form with respect to $w_0$ in $\Sigma$ and a meromorphic function with respect to $z$ in  $\tilde{\mathsf{D}}_{i,j}$ such that $\dd_z \alpha_{0,2}^{(i,j)}(w_0;z) = \omega_{0,2}(w_0,z)$. Besides, for any $(g,n) \in \mathbb{Z}_{\geq 0}^2$ such that $2g - 2 + (1+n) > 0$ and $n$-tuple of points $w_{[n]}$ in the complement of $\tilde{\mathsf{D}}^+$, we have
\begin{equation}
\label{globalTR}
\omega_{g,1+n}(w_0, w_{[n]}) =  \frac{1}{2{\rm i}\pi} \sum_{i = 1}^{\mathsf{k}} \sum_{j \in \mathsf{c}^+_i} \oint_{z \in \gamma_{i,j}}  \sum_{\substack{Z \subseteq  \mathfrak{f}'(z) \,\cap \,\tilde{\mathsf{D}}_{i,j} \\ |Z| \geq 1}} K_{1+|Z|}^{(i,j)}(w_0; z, Z) \mathcal{W}'_{g,1+|Z|;n}(z, Z; w_{[n]})\,,
\end{equation}
where $K^{(i,j)}_{1 + |Z|}$ refers to the recursion kernel defined using the primitive $\alpha_{0,2}^{(i,j)}$ as above and the integration contour $\gamma_{i,j} \subset \tilde{\mathsf{D}}_{i,j}$ represents the homology class of (positively oriented) $\partial\tilde{\mathsf{D}}_{i,j}$ in the complement of $\mathsf{Ram}$.\end{proposition}

We call the rewriting \eqref{globalTR} a (partial, if there are more than one integration contour) \emph{horizontal globalisation}. In many cases of practical use it is sufficient to replace the most technical requirements (DC3)-(DC4) of \cref{de:adapcon} and the partial globalisability assumption of \cref{pr:globalTRdomain} by stronger but simpler ones. If $x$ has finite degree, (DC3) is automatic and instead of the partial globalisability assumption of \cref{pr:globalTRdomain}, one can ask topological recursion to be globalisable above an open set $V \subseteq x(\Sigma)$ containing all branch points, or to be fully globalisable. In the \cref{de:adapcon} of adapted contours, (DC4) includes the requirement that the homology class of the components of $x^{-1}(\partial\mathsf{D}_i)$ are in ${\rm Ker}\,\Pi_{0,2}$. The latter is automatically satisfied (and independent of $\omega_{0,2}$) if the $\tilde{\mathsf{D}}_{i,j}$ are topological discs or if $\tilde{\mathsf{D}}_{i,j}$ has genus $0$ and its boundary $\partial\tilde{\mathsf{D}}_{i,j}$ is homologous to zero in $\Sigma$. In particular, if $\Sigma$ is is a compact Riemann surface of genus $0$, (DC3) and (DC4) are automatic, and if furthermore topological recursion is fully globalisable, \eqref{globalTR} holds with a single integration contour realised as the $x$-preimage of a Jordan curve surrounding all branch points.

\begin{proof}

Pick a point $p_{i,j} \in \tilde{\mathsf{D}}_{i,j}$ and define
$$
\forall z \in \tilde{\mathsf{D}}_{i,j}\qquad \alpha^{(i,j)}_{0,2}(w_0;z) = \int_{p_{i,j}}^{z} \omega_{0,2}(w_0,\cdot)\,.
$$
where a choice of path from $z$ to $p_{i,j}$ in $\tilde{\mathsf{D}}_{i,j}$ which avoids $w_0$. The result is independent of this choice because
 $\omega_{0,2}(w_0,\cdot)$ has no residues and we have imposed (DC4).
  
Recall that the unramified points do not contribute to the local formula for topological recursion \eqref{eq:TRbp} and that we are free to use arbitrary local primitives. Subsequently, with the notation $\mathsf{Ram}_{i,j} = \mathsf{Ram}\, \cap\, \tilde{\mathsf{D}}_{i,j}$ we have
$$
 \omega_{g,1+n}(w_0,w_{[n]}) = \sum_{i = 1}^{\mathsf{k}} \sum_{j \in \mathsf{c}^+_i} \sum_{p \in \mathsf{Ram}_{i,j}} \Res_{z = p} \sum_{\substack{Z \subseteq  \mathfrak{f}_p'(z) \\ |Z| \geq 1}} K_{1+|Z|}^{(i,j)}(w_0; z, Z)  \mathcal{W}'_{g,1+|Z|;n}(z, Z; w_{[n]})\,.
 $$
In the definition of locally admissible spectral curve we required that $\mathsf{Ram}$ is finite, so there are finitely many $i,j,p$ giving nonzero contributions in these sums. We now use the assumption that topological recursion can be globalised over each $\tilde{\mathsf{D}}_{i,j}$ to write
\begin{equation}
\label{resununsnuns}\omega_{g,1+n}(w_0,w_{[n]}) = \sum_{i = 1}^{\mathsf{k}} \sum_{j \in \mathsf{c}^+_i} \sum_{p \in \mathsf{Ram}_{i,j}} \Res_{z = p} \sum_{\substack{Z \subseteq  \mathfrak{f}'(z)\, \cap\, \tilde{\mathsf{D}}_{i,j} \\ |Z| \geq 1}} K_{1+|Z|}^{(i,j)}(w_0; z, Z)  \mathcal{W}'_{g,1+|Z|;n}(z, Z; w_{[n]})\,.
\end{equation}
The integrand is a meromorphic $1$-form with respect to $z$ in $\tilde{\mathsf{D}}_{i,j}$ with poles at $\mathsf{Ram}_{i,j}$ only. We claim that it has no other poles for $z$ in $\tilde{\mathsf{D}}_{i,j}$. Indeed, partial globalisation over the whole $\tilde{\mathsf{D}}_{i,j}$ implies that for any unramified $p \in \tilde{\mathsf{D}}_{i,j}$ we have
\begin{equation}
\label{resununsss}\Res_{z = p} \sum_{\substack{Z \subseteq  \mathfrak{f}'(z)\, \cap\, \tilde{\mathsf{D}}_{i,j} \\ |Z| \geq 1}} K_{1+|Z|}^{(i,j)}(w_0; z, Z)  \mathcal{W}'_{g,1+|Z|;n}(z, Z; w_{[n]}) = 0
\end{equation}
identically for $w_0 \in \Sigma$. By the properties of fundamental bidifferentials (appearing in the numerator of the recursion kernel), this implies that the integrand in \eqref{resununsss} with respect to $z$ (which is also the integrand in \eqref{resununsnuns}) is holomorphic at $z = p$. Then, if $w_0,\ldots,w_n$ are in the complement of $\tilde{\mathsf{D}}$, there are no other poles in $\tilde{\mathsf{D}}_{i,j}$. Cauchy residue formula then yields
$$
\omega_{g,1+n}(w_0,w_{[n]}) = \sum_{i = 1}^{\mathsf{k}} \sum_{j \in \mathsf{c}^+_{i}}  \frac{1}{2{\rm i}\pi} \oint_{z \in \gamma_{i,j}} \bigg(\sum_{\substack{Z \subseteq  \mathfrak{f}'(z)\, \cap\, \tilde{\mathsf{D}}_{i,j} \\ |Z| \geq 1}} K_{1+|Z|}^{(i,j)}(w_0; z, Z)  \mathcal{W}'_{g,1+|Z|;n}(z, Z; w_{[n]})\bigg)\,.
$$  
where the (possibly disconnected) integration contour $\gamma_{i,j}$ is obtained by a slight push of $\partial\tilde{\mathsf{D}}_{i,j}$ into $\tilde{\mathsf{D}}_{i,j}$ not crossing any ramification points.
\end{proof}

\subsection{Comment on globalisation}
\label{S:comm}
As we have seen, the recursion kernel involves a choice of (a priori local) primitives of $\omega_{0,2}$. It is common to specify primitives by integrating from base points. Choosing the canonical local primitive has the drawback that it depends on the point $p$ and therefore unsuitable for horizontal globalisation. If $\Sigma$ is connected, one may think of choosing a single base point $o$ and define $\alpha_{0,2}^{(p)}(w_0;z) = \int_o^z\omega_{0,2}(w_0,\cdot)$ as primitive common to all $p$.  However, if $\Sigma$ is not simply-connected, this depends on the relative homology class of the chosen path between $o$ and $z$. This ambiguity can be waived by choosing a cut-locus $c$, such that $\Sigma \setminus c$ is a fundamental domain of $\Sigma$ avoiding the fibres of branch points: then the primitive is defined with the path from $o$ to $z$ in $\Sigma \setminus c$. An adapted disc collection gives a way to resolve this ambiguity by providing choices of global primitives in each $\tilde{\mathsf{D}}_{i,j}$. More precisely, one should think of disc collections as a way to cluster ramification points (those inside the same contour) and this is a preparation for the study of families of spectral curves where points in the same fibre and in the same cluster can collide.

In  \cite{BE13} the local primitive of $\omega_{0,2}$ used in the recursion kernel is defined by integration from a base point $o$, as we just said. If the spectral curve has non-trivial homology, it remains implicit in \cite{BE13} that a cut-locus has to be chosen to resolve the ambiguity. In the setting of compact spectral curves with regular ramification points and such that $x$-fibre contains at most one ramification point, \cite[Section 3]{BE13} shows  (in our language) that topological recursion can be fully globalised. What they call global topological recursion is the formula
\[
	\omega_{g,1+n}(w_0,w_{[n]}) = \sum_{q \in \mathbb{P}^1} \sum_{p \in x^{-1}(q)} \Res_{z = p} \bigg(\sum_{Z \subseteq \mathfrak{f}'(z)} K^{o}_{1 + |Z|}(w_0;z,Z)  \mathcal{W}'_{g,1+|Z|;n}(z, Z; w_{[n]})\bigg)\,,
\] with a choice of fixed base point $o \in \Sigma$ in the kernel $K^o_{1+|Z|}$ (and an implicit choice of a cut-locus as explained above).

This is what we called vertical globalisation and only the first step of our rewriting. Based on this,  \cite[Section 3.5]{BE13} sketched an argument to show that topological recursion commutes with collision of ramification points in families of spectral curves in the setting they consider. The main goal of the present article is to provide a complete argument for commutation with limits (this will be carried out in \cref{S5}).  We will see that the question is more complicated than envisioned in \cite[Section 3.5]{BE13} even in the case of compact spectral curves. Besides, we will be able to address a much more general setting for spectral curves. In particular our discussion covers the case of $x$-fibres containing several ramification points which was mentioned as an open question in \cite[Section 3.5]{BE13}. The main instrument for these arguments will be the second step of our rewriting, namely the partial horizontal globalisation given in \cref{pr:globalTRdomain}.

An example of a horizontal globalisation (combined with a prior vertical globalisation) is given in the context of simgularity theory near the caustic in~\cite[Section 4.3]{Milanov}.

\subsection{Criterion for vertical globalisation in terms of correlators}

We first determine sufficient conditions for vertical globalisation in the temporary form of \cref{pr:sufficient2}, where it is still formulated as properties of the system of correlators. In \cref{S33} we will transform these conditions into intrinsic conditions on the spectral curve.

\subsubsection{Auxiliary lemmata}

We start with a number of easy lemmata, that will prepare us before diving into the core of the argument. The first one is a straightforward combinatorial identity, generalising \cite[Lemma~1]{BE13}.

\begin{lemma}\label{le:comb}
	\begin{equation}
		\mathcal{W}'_{g,m+k;n}(z_{[m]}, t_{[k]}; w_{[n]}) = \!\!\!\!  \sum_{\substack{Z_1 \sqcup Z_2 = z_{[m]} \\ W_1 \sqcup W_2 = w_{[n]} \\ g_1+g_2 = g + |Z_2| - m}} \!\!\!\! \mathcal{W}'_{g_1,k;|W_1|+|Z_1|} (t_{[k]}; W_1,Z_1) \mathcal{W}'_{g_2,|Z_2|;|W_2|}(Z_2;W_2) \,,
	\end{equation}
	where we understand that $\mathcal{W}'_{g,0;n} (\emptyset; w_{[n]})= \delta_{g,0} \delta_{n,0}$.
\end{lemma}

\begin{proof}
	This is a rewriting of the sum over set partitions in \cref{de:calW}, where we single out subsets in the set partitions that are subsets of the variables $z_{[m]}$.
\end{proof}

The second lemma is an algebraic manipulation detailed in \cite[Lemma~3]{BE13} and independent of the type of the spectral curve.
\begin{lemma}\label{le:recursioncalW}
	Let $(\omega_{g,n})_{(g,n) \in \mathbb{Z}_{\geq 0} \times \mathbb{Z}_{> 0}}$ be a system of correlators on a locally admissible spectral curve $\mathcal{S} = (\Sigma,x,\omega_{0,1},\omega_{0,2})$ that satisfies topological recursion. Then, for all $(g,n) \in \mathbb{Z}_{\geq 0}^2$ such that $2g-2+(1+n) > 0$, and $m \geq 0$,
	\begin{multline}\label{eq:recursioncalW}
		\mathcal{W}'_{g,1+m;n}(w_0, z_{[m]}; w_{[n]}) \\=\sum_{q \in x(\Sigma)}\sum_{p \in x^{-1}(q)}  \Res_{z=p}  \sum_{\substack{Z \subseteq  \mathfrak{f}_p'(z) \\ |Z| \geq 1}} K_{1+|Z|}^{(p)}(w_0; z, Z)    \mathcal{W}'_{g,1+|Z|+m;n}(z, Z,z_{[m]}; w_{[n]}) \,.
	\end{multline}

\end{lemma}

The case $m=0$ is the usual topological recursion of \cref{de:TR}. We  then note that if topological recursion can be globalised over a subset of $x^{-1}(q)$, then the same is true for \eqref{eq:recursioncalW}. More precisely:

\begin{lemma}\label{le:calWglobal}
	Let $(\omega_{g,n})_{(g,n) \in \mathbb{Z}_{\geq 0} \times \mathbb{Z}_{> 0}}$ be a system of correlators on a locally admissible spectral curve $\mathcal{S} = (\Sigma,x,\omega_{0,1},\omega_{0,2})$ that satisfies topological recursion. Suppose that it can be globalised over a finite subset $P \subseteq x^{-1}(q)$ for some $q \in x(\Sigma)$. Then, for all $(g,n) \in \mathbb{Z}_{\geq 0}^2$ such that $2g-2+(1+n) > 0$,
	\begin{multline}
		\sum_{p \in P} \Res_{z=p}  \sum_{\substack{Z \subseteq  \mathfrak{f}_p'(z) \\ |Z| \geq 1}} K_{1+|Z|}^{(p)}(w_0; z, Z)  \mathcal{W}'_{g,1+|Z|+m;n}(z, Z,z_{[m]}; w_{[n]})\\
		=\sum_{p \in P} \Res_{z=p}  \sum_{\substack{Z \subseteq  \mathfrak{f}_{P}'(z) \\ |Z| \geq 1}}  K_{1+|Z|}^{(p)}(w_0; z, Z) \mathcal{W}'_{g,1+|Z|+m;n}(z, Z,z_{[m]}; w_{[n]}) \,.
	\end{multline}
\end{lemma} 
Note that the only difference between the two sides of the above equation is that the second sum uses the set $\mathfrak f '_P(z)$ on the right side as opposed to the set $\mathfrak f '_p(z)$ on the left.
\begin{proof}
	This follows directly from the proof of \cite[Lemma~3]{BE13}.
\end{proof}

The fourth lemma is another easy combinatorial identity, which was already used in the proof of \cref{th:ale}:

\begin{lemma}\label{le:comb2}
	\begin{equation}
		\sum_{i=1}^{r_p} (- \omega_{0,1}(z))^{r_p-i}  \mathcal{E}^{(p)}_{g,i;n}(z; w_{[n]}) = \sum_{\substack{Z \subseteq \mathfrak{f}_p(z) \\ |Z| \geq 1}}  \mathcal{W}'_{g,|Z|;n}(Z; w_{[n]}) \Upsilon_{r_p - |Z|}(\mathfrak{f}_p(z) \setminus Z; z) \,.
	\end{equation}
\end{lemma}

\begin{proof}
	It appears in various places, cf. for instance the proof of \cite[Theorem~3.26]{BE17}, \cite[Lemma~7.6.4]{Kra19}, or \cite[Appendix~C]{BBCCN18}. We note that if $z \in U_p$, all terms on the right-hand side with $z \notin Z$ vanish, and thus only the terms with $z \in  Z$ remain. This gives \eqref{eq:comb}, which is also how the identity was formulated for instance in \cite[Lemma~7.6.4]{Kra19}.
\end{proof}

We can also characterise when topological recursion can be globalised over a finite subset $P \subseteq x^{-1}(q)$ for some $q \in x(\Sigma)$.

\begin{lemma}\label{le:reformulation}
	Let $(\omega_{g,n})_{(g,n) \in \mathbb{Z}_{\geq 0} \times \mathbb{Z}_{> 0}}$ be a system of correlators on a locally admissible spectral curve  $\mathcal{S} = (\Sigma,x,\omega_{0,1},\omega_{0,2})$ that satisfies topological recursion. Let $P \subseteq x^{-1}(q)$ be a finite subset for some point $q \in x(\Sigma) $. Topological recursion can be globalised over $P$ if and only if, for all $(g,n) \in \mathbb{Z}_{\geq 0}^2$ such that  $2g-2+ (1 +n) > 0$ we have
	\begin{equation}\label{eq:toprove}
		\sum_{p \in P} \Res_{z=p}\sum_{\substack{Z \subseteq  \mathfrak{f}_P'(z) \\ Z \not\subset \mathfrak{f}_p'(z) \\ |Z| \geq 1}} K_{1+|Z|}^{(p)}(w_0; z, Z)  \mathcal{W}'_{g,1+|Z|;n}(z, Z; w_{[n]}) = 0 \,.
	\end{equation}
\end{lemma}

\begin{proof}
Direct comparison between the local topological recursion (involving the fibre $\mathfrak{f}_p(x)$) and the topological recursion globalised over $P$ (involving the fibre $\mathfrak{f}_P(z)$), i.e. between the left-hand side and the right-hand side of \eqref{eq:globalized}.
\end{proof}

\subsubsection{The criterion}

We are now ready to answer our question and provide sufficient conditions under which topological recursion can be globalised over a finite subset $P \in x^{-1}(q)$.

\begin{proposition}\label{pr:sufficient2}
	Let $(\omega_{g,n})_{(g,n) \in \mathbb{Z}_{\geq 0} \times \mathbb{Z}_{> 0}}$ be a system of correlators on a locally admissible spectral curve $\mathcal{S} = (\Sigma,x,\omega_{0,1},\omega_{0,2})$ that satisfies topological recursion. Let $q \in x(\Sigma)$ and $ \{p_1, p_2\}$ a pair of points in $x^{-1}(q)$. Suppose that, for all $(g,n) \in \mathbb{Z}_{\geq 0}^2$  such that $2g-2+(1+n) > 0$,  the function (in $z$, defined locally near $p_1$)
	\begin{equation}\label{eq:conditiongg}
 	\frac{1}{\Upsilon_{r_{p_2}}(\mathfrak{f}_{p_2}(z);z)} \sum_{k=1}^{r_{p_2}} (- \omega_{0,1}(z))^{r_{p_2}-k}  \mathcal{E}^{(p_2)}_{g,k;n}(z; w_{[n]})
	\end{equation}
	is holomorphic as $z \rightarrow p_1$, and vice-versa with $p_1 \leftrightarrow p_2$. Then the system of correlators can be globalised over $\{p_1,p_2\}$.
	
	If $P \subseteq x^{-1}(q)$ is a finite set and this condition holds for any pair $\{p_1,p_2\}$ in $P$, then the system of correlators can be globalised over $P$.
\end{proposition}
We note that $z$ in the expression \eqref{eq:conditiongg}  is a local coordinate in a neighbourhood of the point $p_1 \in \Sigma$. Thus, $\mathfrak f_{p_2}(z)$ does not contain  $z$, and as a result $\Upsilon_{r_{p_2}}(\mathfrak{f}_{p_2}(z);z)$ does not vanish identically. Consequently, \eqref{eq:conditiongg} is a well-defined function of $z$ locally near $p_1$.

\begin{proof} (\textit{Argument for a pair}) We start with $P = \{p_1,p_2\}$. In that case $\mathfrak{f}_P(z) = \mathfrak{f}_{p_1}(z) \sqcup \mathfrak{f}_{p_2}(z)$ and we want to prove \eqref{eq:toprove}.  Consider first the residue at  $z=p_1$. Its contribution to the left-hand side of \eqref{eq:toprove} can be decomposed as:
	\begin{multline}\label{eq:subsets}
		\Res_{z=p_1}   \sum_{\substack{Z \subseteq  \mathfrak{f}_{p_2}(z)\\ |Z| \geq 1}} K_{1+|Z|}^{(p_1)}(w_0; z, Z)  \mathcal{W}'_{g,1+|Z|;n}(z, Z; w_{[n]})  \\
  	+   \Res_{z=p_1}   \sum_{\substack{Z \subseteq  \mathfrak{f}'_P(z)\\ Z \not\subset \mathfrak{f}_{p_1}'(z) \\ Z \not\subset \mathfrak{f}_{p_2}(z)\\ |Z| \geq 2}}K_{1+|Z|}^{(p_1)}(w_0; z, Z)  \mathcal{W}'_{g,1+|Z|;n}(z, Z; w_{[n]}) \,.
  \end{multline}
Here the first sum ranges over subsets of $\mathfrak{f}_{p_2}(z)$, while the second sum ranges over ``mixed'' subsets that contain at least one variable in $\mathfrak{f}_{p_1}'(z)$ and one variable in $\mathfrak{f}_{p_2}(z)$. Consider the first sum in \eqref{eq:subsets}. We use \cref{le:recursioncalW} to write
	\begin{equation}
		\mathcal{W}'_{g,1+|Z|;n}(z, Z; w_{[n]}) \\= \sum_{q \in x(\Sigma) }\sum_{\tilde{p} \in x^{-1}(q)} \Res_{\lambda=\tilde{p}}\sum_{\substack{\Lambda \subseteq  \mathfrak{f}_{\tilde{p}}'(\lambda) \\ |\Lambda| \geq 1}}K_{1+|\Lambda|}^{(\tilde{p})}(z; \lambda, \Lambda)  \mathcal{W}'_{g,1+|\Lambda|+|Z|;n}(\lambda, \Lambda,Z; w_{[n]}) \,.
	\end{equation}
	Inserting this in the first term of \eqref{eq:subsets}, we get an expression:
	\begin{equation}\label{eq:subsets2}
		\sum_{q \in x(\Sigma) }\sum_{\tilde{p} \in x^{-1}(q)} \Res_{z=p_1}\Res_{\lambda=\tilde{p}} \sum_{\substack{Z \subseteq  \mathfrak{f}_{p_2}(z)\\ |Z| \geq 1}} \sum_{\substack{\Lambda \subseteq  \mathfrak{f}_{\tilde{p}}'(\lambda) \\ |\Lambda| \geq 1}}K_{1+|Z|}^{(p_1)}(w_0; z, Z)  K_{1+|\Lambda|}^{(\tilde{p})}(z; \lambda, \Lambda) 
		\mathcal{W}'_{g,1+|\Lambda|+|Z|;n}(\lambda, \Lambda,Z; w_{[n]}) \,.
	\end{equation}  
	Next, we want to exchange the order of the residues. We can think of the residues as contour integrals along small circles centred at the points $z=p_1$ and $\lambda=\tilde{p}$. We can exchange the order of the contour integrals without problem if $p_1 \neq \tilde{p}$, but if $p_1 = \tilde{p}$, when we exchange the order we must pick up residues at $z = \lambda'$ for all $\lambda' \in \mathfrak{f}_{p_1}(\lambda)$. Looking at the integrand, we see that the only extra residue comes from the simple pole at $z=\lambda$ of $\alpha_{0,2}^{(p_1)}(z;\lambda)$ in the numerator of $K_{1+|\Lambda|}^{(p_1)}(z; \lambda, \Lambda) $ (cf. \cref{de:kernel}).   Therefore, we can rewrite  \eqref{eq:subsets2} as
	\begin{multline}\label{eq:subsets3}
		\!\!\!\! \sum_{q \in x(\Sigma) }\sum_{\tilde{p} \in x^{-1}(q)}  \Res_{\lambda=\tilde{p}}  \Res_{z=p_1} \sum_{\substack{Z \subseteq  \mathfrak{f}_{p_2}(z)\\ |Z| \geq 1}} \sum_{\substack{\Lambda \subseteq  \mathfrak{f}_{\tilde{p}}'(\lambda) \\ |\Lambda| \geq 1}}K_{1+|Z|}^{(p_1)}(w_0; z, Z)  K_{1+|\Lambda|}^{(\tilde{p})}(z; \lambda, \Lambda) 
 		\mathcal{W}'_{g,1+|\Lambda|+|Z|;n}(\lambda, \Lambda,Z; w_{[n]}) 
 		\\
		+\Res_{\lambda=p_1}\Res_{z=\lambda} \sum_{\substack{Z \subseteq  \mathfrak{f}_{p_2}(z)\\ |Z| \geq 1}} \sum_{\substack{\Lambda \subseteq  \mathfrak{f}_{p_1}'(\lambda) \\ |\Lambda| \geq 1}}K_{1+|Z|}^{(p_1)}(w_0; z, Z)  K_{1+|\Lambda|}^{(p_1)}(z; \lambda, \Lambda)  \mathcal{W}'_{g,1+|\Lambda|+|Z|;n}(\lambda, \Lambda,Z; w_{[n]}) \,.
	\end{multline}
	To evaluate the first sum in \eqref{eq:subsets3}, we use \cref{le:comb} to write:
	\begin{equation}
		\mathcal{W}'_{g,1+|\Lambda|+|Z|;n}(\lambda, \Lambda,Z; w_{[n]}) = \sum_{\substack{\Lambda_1 \sqcup \Lambda_2 = \{\lambda,\Lambda\} \\ W_1 \sqcup W_2 = w_{[n]} \\ g_1+g_2 = g + |\Lambda_2| - j-1}} \mathcal{W}'_{g_1,|Z|;|W_1|+|\Lambda_1|} (Z; W_1,\Lambda_1) \mathcal{W}'_{g_2,|\Lambda_2|;|W_2|}(\Lambda_2;W_2) \,.
	\end{equation}
Inserting this in the first sum of \eqref{eq:subsets3} and taking out the $z$-independent factors, the residue in $z$ amounts to
\begin{equation}
\label{eq:res}
\Res_{z=p_1} \sum_{\substack{Z \subseteq  \mathfrak{f}_{p_2}(z)\\ |Z| \geq 1}} \alpha_{0,2}^{(p_1)}(z;\lambda) K_{1+|Z|}^{(p_1)}(w_0; z, Z)   \mathcal{W}'_{g_1,|Z|;|W_1|+|\Lambda_1|} (Z; W_1,\Lambda_1) \,.
\end{equation}
By \cref{le:comb2}, the assumption \eqref{eq:conditiongg} can be rewritten as the statement that the function
	\begin{equation}
	\label{the:as}	\sum_{\substack{ Z \subseteq \mathfrak{f}_{p_2}(z) \\ |Z| \geq 1}}  \frac{1}{\Upsilon_{|Z|}(Z;z)}   \mathcal{W}'_{g,|Z|,n}(Z; w_{[n]}) 
	\end{equation}
	is holomorphic as $z \rightarrow p_1$. Recalling that:
	\begin{equation}
	\label{reckern}
	K_{1+|Z|}^{(p_1)}(w_0;z,Z) = -\frac{\alpha_{0,2}^{(p_1)}(w_0;z)}{\Upsilon_{|Z|}(Z;z)}\,,
	\end{equation}
	where $\alpha_{0,2}^{(p_1)}(w_0;z)$ is a local primitive of $\omega_{0,2}(w_0,z)$ with respect to $z$, we deduce from \eqref{the:as} that the residue in \eqref{eq:res} vanishes.
	
	We are thus left with the second sum in \eqref{eq:subsets3}. Since the only pole of the expression at $z=\lambda$ is the simple pole of $-\alpha_{0,2}^{(p_1)}(z;\lambda)$ in the numerator of $ K_{1+|\Lambda|}^{(p_1)}(z; \lambda, \Lambda)$, which has residue $-1$, taking the residue at $z=\lambda$ simply amounts to replacing $z$ by $\lambda$ everywhere in the expression and change the sign. Using again the expression \eqref{reckern} of the recursion kernel, \eqref{eq:subsets3} then becomes
	\begin{equation}\label{eq:subsets4}
		\Res_{\lambda=p_1}\sum_{\substack{\Lambda' \subseteq  \mathfrak{f}_{p_2}(\lambda)\\ |\Lambda'| \geq 1}} \sum_{\substack{\Lambda \subseteq  \mathfrak{f}_{p_1}'(\lambda) \\ |\Lambda| \geq 1}}  \frac{ \alpha_{0,2}^{(p_1)}(w_0;\lambda)}{\Upsilon_{|\Lambda'|}(\Lambda';\lambda)\Upsilon_{|\Lambda|}(\Lambda;\lambda)} \mathcal{W}'_{g,1+|\Lambda|+|\Lambda'|;n}(\lambda, \Lambda,\Lambda'; w_{[n]}) \,.
	\end{equation}
We observe that $\Upsilon_{|\Lambda|}(\Lambda;\lambda)\Upsilon_{|\Lambda'|}(\Lambda';\lambda) = \Upsilon_{|\tilde{\Lambda}|}(\tilde{\Lambda};\lambda)$ where $\tilde{\Lambda} = \Lambda \sqcup \Lambda'$, and thus recognise the recursion kernel $K_{1 + |\tilde{\Lambda}|}^{(p_1)}(w_0;\lambda,\tilde{\Lambda})$ from \eqref{eq:rkernel}. The sum over $\Lambda,\Lambda'$ can then be rewritten as a sum over subsets $\tilde{\Lambda} \subseteq \mathfrak{f}'_P(\lambda)$ containing at least one variable in $\mathfrak{f}_{p_1}'(\lambda)$ and one variable in $\mathfrak{f}_{p_2}(\lambda)$. This transforms \eqref{eq:subsets4} into
	\begin{equation}
		- \Res_{\lambda=p_1}\sum_{\substack{\tilde{\Lambda} \subseteq  \mathfrak{f}'_P(\lambda)\\ \tilde{\Lambda} \not\subset \mathfrak{f}'_{p_1}(\lambda) \\ \tilde{\Lambda}\not\subset \mathfrak{f}_{p_2}(\lambda)\\ |\tilde{\Lambda}| \geq 2}}  K^{(p_1)}_{1+|\tilde{\Lambda}|}(w_0; \lambda,\tilde{\Lambda})\mathcal{W}'_{g,1+|\tilde{\Lambda}|;n}(\lambda, \tilde{\Lambda}; w_{[n]}) \,,
	\end{equation}
	which precisely cancels out the second term in \eqref{eq:subsets}. Therefore, the residue at $z=p_1$ in \eqref{eq:toprove} vanishes.

	We can do the same calculation for the contribution of the residue at $z=p_2$, with the role of $p_1$ and $p_2$ swapped everywhere. The result is that \eqref{eq:toprove} is satisfied for $\{p_1, p_2\}$, and topological recursion can be globalised over $\{p_1,p_2\}$.
	
	\medskip
	
\noindent (\textit{General argument}) Take a finite set $P \subseteq x^{-1}(q)$ and assume that topological recursion can be globalised over each pair of points in $P$.  We want to argue that it can be globalised over the whole $P$.

Pick any three distinct points $p_i, p_j, p_k \in P$, and write $P_{ijk} = \{p_i, p_j, p_k\} \subseteq P$. We know that topological recursion can be globalised over $P_{ij}$, over $P_{ik}$, and over $P_{jk}$. We want to show that it can be globalised over $P_{ijk}$, and for this we need to prove \eqref{eq:toprove} for $P_{ijk} = \{p_i, p_j, p_k\} \subseteq P$. We proceed as in the argument for pairs. Consider the residue at $z=p_i$ in \eqref{eq:toprove}. We rewrite the left-hand side as:
\begin{multline}\label{eq:subsetsd}
		  \quad \Res_{z=p_i}   \sum_{\substack{Z \subseteq  \mathfrak{f}_{p_k}(z)\\ |Z| \geq 1}}  K_{1+|Z|}^{(p_i)}(w_0; z, Z) \mathcal{W}'_{g,1+|Z|;n}(z, Z; w_{[n]})  \\
  	  +  \Res_{z=p_i}   \sum_{\substack{Z \subseteq  \mathfrak{f}'_{P_{ijk}}(z)\\ Z \not\subset \mathfrak{f}_{P_{ij}}'(z) \\ Z \not\subset \mathfrak{f}_{p_k}(z)\\ |Z| \geq 2}} K_{1+|Z|}^{(p_i)}(w_0; z, Z)  \mathcal{W}'_{g,1+|Z|;n}(z, Z; w_{[n]}) \,.
\end{multline}

	Here, we singled out the subsets of $\mathfrak{f}_{p_k}(z)$ in the first term, and in the second sum we used the fact that the system of correlators can be globalised over $P_{ij}$, and hence the sum over subsets of sheets in $P_{ij}$ vanishes. 

	Consider the first sum in \eqref{eq:subsetsd}. This time, we use \cref{le:calWglobal} to rewrite it, using the fact that we can globalise the system of correlators over $P_{ij}$:
\begin{equation*}
\begin{split}
		\mathcal{W}'_{g,1+|Z|;n}(z, Z; w_{[n]})  & =  \sum_{q \in x(\Sigma) }\sum_{\substack{\tilde{p} \in x^{-1}(q) \\ \tilde{p} \notin P_{ij}}} \Res_{\lambda=\tilde{p}}\sum_{\substack{\Lambda \subseteq  \mathfrak{f}_{\tilde{p}}'(\lambda) \\ |\Lambda| \geq 1}}K_{1+|\Lambda|}^{(\tilde{p})}(z; \lambda, \Lambda)  \mathcal{W}'_{g,1+|\Lambda|+|Z|;n}(\lambda, \Lambda,Z; w_{[n]})
		\\
		& \quad + \Res_{\lambda=p_i} \sum_{\substack{\Lambda \subseteq  \mathfrak{f}_{P_{ij}}'(\lambda) \\ |\Lambda| \geq 1}}K_{1+|\Lambda|}^{(p_i)}(z; \lambda, \Lambda) 
		\mathcal{W}'_{g,1+|\Lambda|+|Z|;n}(\lambda, \Lambda,Z; w_{[n]}) \, \\
		& \quad  + \Res_{\lambda = p_j}  \sum_{\substack{\Lambda \subseteq  \mathfrak{f}_{P_{ij}}'(\lambda) \\ |\Lambda| \geq 1}}K_{1+|\Lambda|}^{(p_j)}(z; \lambda, \Lambda) 
		\mathcal{W}'_{g,1+|\Lambda|+|Z|;n}(\lambda, \Lambda,Z; w_{[n]})\,.
		\end{split}
\end{equation*}
Inserting this in the first sum of \eqref{eq:subsetsd}, we get an expression:
	\begin{multline*}
 		\!\!\!\!\!\!\sum_{q \in x(\Sigma) }\sum_{\substack{\tilde{p} \in x^{-1}(q) \\ \tilde{p} \notin P_{ij}}} \Res_{z=p_i}\Res_{\lambda =\tilde{p}} \sum_{\substack{Z \subseteq  \mathfrak{f}_{p_k}(z)\\ |Z| \geq 1}} \sum_{\substack{\Lambda \subseteq  \mathfrak{f}_{\tilde{p}}'(\lambda) \\ |\Lambda| \geq 1}}  
		K_{1+|Z|}^{(p_i)}(w_0; z, Z)K_{1+|\Lambda|}^{(\tilde{p})}(z; \lambda, \Lambda)
		\mathcal{W}'_{g,1+|\Lambda|+|Z|;n}(\lambda, \Lambda,Z; w_{[n]}) 
		\\
		\!\!\!\!\!\!\!\!\!\!\!\!\!\!\!\!\!\!+ \Res_{z=p_i} \Res_{\lambda=p_i} \sum_{\substack{Z \subseteq  \mathfrak{f}_{p_k}(z)\\ |Z| \geq 1}} \sum_{\substack{\Lambda \subseteq  \mathfrak{f}_{P_{ij}}'(\lambda) \\ |\Lambda| \geq 1}}  K_{1+|Z|}^{(p_i)}(w_0; z, Z)K_{1+|\Lambda|}^{(p_i)}(z; \lambda, \Lambda)
		\mathcal{W}'_{g,1+|\Lambda|+|Z|;n}(\lambda, \Lambda,Z; w_{[n]}) \\
		\!\!\!\!\!\!\!\!\!\! + \Res_{z = p_i}  \Res_{\lambda=p_j} \sum_{\substack{Z \subseteq  \mathfrak{f}_{p_k}(z)\\ |Z| \geq 1}} \sum_{\substack{\Lambda \subseteq  \mathfrak{f}_{P_{ij}}'(\lambda) \\ |\Lambda| \geq 1}}  K_{1+|Z|}^{(p_i)}(w_0; z, Z)K_{1+|\Lambda|}^{(p_j)}(z; \lambda, \Lambda)
		\mathcal{W}'_{g,1+|\Lambda|+|Z|;n}(\lambda, \Lambda,Z; w_{[n]})\,.
	\end{multline*}
As before, we exchange the order of the residues. Using assumption \eqref{eq:conditiongg} for the pair $P_{ij}$, after this exchange  we see that the residue at $z=p_i$ vanishes. What remains is the residue at $z=\lambda$ for the case where both residues in $z$ and $\lambda$ are at $p_i$. We evaluate this residue, and the result is a sum over mixed subsets of sheets in $\mathfrak{f}'_{P_{ijk}}(\lambda)$ with at least one sheet in $\mathfrak{f}_{p_k}(\lambda)$ and one sheet in $\mathfrak{f}_{P_{ij}}'(\lambda)$, which precisely cancels out with the second term in \eqref{eq:subsetsd}. Therefore, the residue at $z=p_i$ in \eqref{eq:toprove} vanishes.

	We can do the same thing for the residue at $z=p_j$, and we conclude that the condition \eqref{eq:toprove} is satisfied on $P_{ijk}$ for the residues at $z=p_i$ and $z=p_j$.

	It remains to show that \eqref{eq:toprove} is also satisfied for the residue at $z=p_k$. To do this, we proceed as above, but instead of starting with the pair $P_{ij}$, we start with the pair $P_{ki}$. We conclude that \eqref{eq:toprove} is also satisfied for the residue at $z=p_k$, and hence topological recursion can be globalised over $P_{ijk}$.

	To get the whole $P$ (and not just $P_{ijk}$), we iterate the process, adding one more point in $P$ at each step and checking that globalisation holds as we did.
	\end{proof}

\subsection{Criterion for vertical globalisation in terms of the spectral curve}

\label{S33}

\Cref{pr:sufficient2} provides sufficient conditions for topological recursion to be globalisable over a finite subset $P \subseteq x^{-1}(q)$ for some $q \in x(\Sigma) \subseteq\mathbb{P}^1$. These are conditions bearing on all pairs of points $p_1,p_2$ in $P$: a certain function constructed from the local behaviour of correlators at $p_2$ (cf. \eqref{eq:conditiongg}) must be holomorphic as $z \rightarrow p_1$. Our next task is to find an intrinsic criterion for this to happen, depending on the pair and on the spectral curve but not on the knowledge of the correlators.

\subsubsection{Primitive form}

Recall the definition of the aspect ratio $\nu_p = \frac{\bar s_p}{r_p}$. In the following, we will also denote $\tau_i := \tau_{p_i}$, where $\tau_p$ is the leading order coefficient of the expansion of $\omega_{0,1} $ near $p$ in a standard coordinate (cf. \cref{de:localparameters}). As the following property will appear often, we give it a name.

\begin{definition}[Non-resonance]\label{de:nonres}
A pair of points $p_1,p_2$ in $x^{-1}(q)$ is \emph{non-resonant} if either $\nu_1 \neq \nu_2$, or  $\nu_1 = \nu_2$ and $(\frac{\tau_1/r_1}{\tau_2/r_2}\big)^{r'} \neq 1$, where we have set $r' = \frac{r_1}{{\rm gcd}(r_1,\bar s_1)} = \frac{r_2}{{\rm gcd}(r_2,\bar s_2)}$. 
\end{definition}

Denoting $\delta_k = {\rm gcd}(r_k,\bar s_k)$, we have $r_k = \delta_k r_k'$ and $\bar s_k = \delta_k \bar s_k'$. Therefore, if $\nu_1 = \nu_2$ we get $\frac{\bar s_1'}{r_1'} = \frac{\bar s_2'}{r_2'}$ which is an equality of irreducible fractions, hence $r_1' = r_2'$ and we denoted $r'$ this common value. The non-resonance condition is unaffected by the choice of standard coordinates. Indeed, due to \cref{re:taupambiance}, different choices of standard coordinates at $p_k$ would replace $\tau_k$ with $u_k\tau_k$ for some $r_k$-th root of unity $u_k$, but $(u_k^{\bar s_k})^{r'} = (u_k^{\bar s_k'})^{r} = 1$.

We now prove a primitive form of an intrinsic criterion for globalisation.

\begin{lemma}\label{le:pairs}
	Let $(\omega_{g,n})_{(g,n) \in \mathbb{Z}_{\geq 0} \times \mathbb{Z}_{> 0}}$ be a system of correlators on a locally admissible spectral curve $\mathcal{S} = (\Sigma,x,\omega_{0,1},\omega_{0,2})$ that satisfies topological recursion. Let $q \in x(\Sigma)$ and  $p_1, p_2$ a pair of points $x^{-1}(q)$. Let $(r_k,s_k,\bar s_k,\tau_k)$ be the local parameters at $p_k$ (cf. \cref{de:localparameters}). If the following conditions are simultaneously satisfied
	\begin{enumerate}[({C}1)]
		\item $\forall i \in [r_2] \qquad \bar s_1 r_2 - \bar s_1 i +r_1 +  r_{1} \left\lfloor \frac{s_{2} (i-1)}{r_{2}} \right\rfloor - \min(r_1 \bar s_2, r_2 \bar s_1) \geq 0$;
		\item $\forall i \in [r_1] \qquad \bar s_2 r_1 - \bar s_2 i +r_2 +  r_{2} \left\lfloor \frac{s_{1} (i-1)}{r_{1}} \right\rfloor - \min(r_1 \bar s_2, r_2 \bar s_1) \geq 0$;
		\item $p_1,p_2$ is non-resonant;
	\end{enumerate}
	then topological recursion can be globalised over $\{p_1,p_2\}$. If $P \subseteq x^{-1}(q)$ is a finite set such that the above conditions are satisfied for any pair of distinct points in $P$, then topological recursion can be globalised over $P$.
	\end{lemma}

\begin{proof}
	We will show that these conditions are sufficient to apply \cref{pr:sufficient2}. Consider the function in \eqref{eq:conditiongg} as $z \rightarrow p_1$ and use a standard coordinate in $U_{p_1}$. First, we know that 
	\begin{equation}
		(-\omega_{0,1}(z))^{r_2-i} = \mc{O}\!\left(\zeta^{(\bar s_1-1)(r_2-i)} \dd \zeta^{r_2-i} \right) \quad {\rm as}\,\, z \rightarrow p_1\,.
	\end{equation}
	Second, by the abstract loop equations \eqref{eq:aleexp}, we know that
	\begin{equation}
		\mathcal{E}_{g,i;n}^{(p_2)}(z; w_{[n]}) = \mc{O}\!\left( \frac{\dd \zeta^i}{\zeta^{r_1 \mathfrak{d}_{p_2}(i) - (r_1-1)i}} \right)  = \mc{O}\!\left( \zeta^{r_1 \left \lfloor \frac{s_2}{r_2}(i-1) \right \rfloor + r_1-i} \dd \zeta^i\right) \quad {\rm as}\,\,z \rightarrow p_1\,.
	\end{equation}

	Consider now a standard coordinate $\zeta_2$ in $U_{p_2}$. For any $z' \in \mathfrak{f}_{p_2}(z)$, since $z' \rightarrow p_2$ as $z \rightarrow p_1$ and by comparison of the two local normal forms and definition of the fibres of $x$, the set of $\zeta_2$-coordinates of the points in $\mathfrak{f}_{p_2}(z)$ is
	$$
	\Big\{\zeta^{\frac{r_1}{r_2}}e^{\frac{2{\rm i}\pi m}{r_2}}\,\,|\,\,m \in [r_2]\}\Big\}\,.
	$$
As $\omega_{0,1}(z) \sim \tau_2 \zeta_2^{\bar s_2 - 1}\dd \zeta_2$ when $z \rightarrow p_2$, we have when $z \rightarrow p_1$ with $z' \in \mathfrak{f}_{p_2}(z)$ in the $m$-th branch:
$$
\omega_{0,1} (z') \sim \tau_2\,\frac{r_1}{r_2}\,e^{\frac{2{\rm i}\pi \bar s_2 m}{r_2}} \zeta^{\frac{r_1 \bar s_2}{r_2} - 1}\, \dd \zeta\,.
$$
Therefore, assuming that $\frac{\bar s_1}{r_1} \neq \frac{\bar s_2}{r_2}$,  and since $\omega_{0,1}(z) = \mc{O}(\zeta^{\bar s_1 - 1} \dd \zeta)$ as $z \rightarrow p_1$, we get
	\begin{equation}
		\Upsilon_{r_{2}}(\mathfrak{f}_{p_2}(z);z) = \mc{O}\!\left( \zeta^{\min(r_1 \bar s_2, r_2 \bar s_1) - r_2} \dd \zeta^{r_2} \right) \quad {\rm as}\,\, z \rightarrow p_1\,.
	\end{equation}
We want to make sure that this estimate remains the same if $\frac{\bar s_1}{r_1} = \frac{\bar s_2}{r_2}$. In this case, we find that the coefficient of $\zeta^{\min(r_1\bar s_2,r_2 \bar s_1) - r_2} \dd \zeta^{r_1}$ in $\Upsilon_{r_2}(\mathfrak{f}_{p_2}(z);z)$ when $z \rightarrow p_1$ is
	$$
	c = \prod_{m = 1}^{r_2} \left( \tau_2\, \frac{r_1}{r_2}\,e^{\frac{2{\rm i}\pi \bar{s}_2 m}{r_2}} - \tau _1 \right) = \prod_{m = 1}^{r'} \left(\tau_2\,\frac{r_1}{r_2}\,e^{\frac{2{\rm i}\pi \bar{s}_2' m}{r'}} - \tau_1 \right)^{\delta_2}\,,
	$$  
where $\delta_2 = {\rm gcd}(r_2,\bar s_2)$, so that $r_2 = \delta_2 r_2'$ and $\bar s_2 = \delta_2 \bar s_2'$. Thus:
$$
c = \left((-1)^{r_2'}\Big[\big(\tau_2\,\tfrac{r_1}{r_2}\big)^{r_2'} - \tau_1^{r_2'}\Big]\right)^{\delta_2} = (-\tau_1)^{r_2} \left[1 - \Big(\frac{\tau_2 r_1}{\tau_1 r_2}\Big)^{r_2'}\right]^{\delta_2}\,.
$$
To ensure that $c \neq 0$, we must impose $\big(\frac{\tau_1/r_1}{\tau_2/r_2}\big)^{r'} \neq 1$.

	This shows that the $i$-th term (with $i \in [r_2]$) in \eqref{eq:conditiongg} behaves like
	\begin{equation*}
	\begin{split}
		\frac{ (-  \omega_{0,1} (z))^{r_{2}-i} }{\Upsilon_{r_{2}}(\mathfrak{f}_{p_2}(z);z)} \mathcal{E}^{(p_2)}_{g,i;n}(z; w_{[n]}) 
		&= \mc{O}\!\left(\zeta^{(\bar s_{1}-1) (r_{2}-i) +r_{1} - i +  r_{1} \left\lfloor \frac{s_{2} (i-1)}{r_{2}} \right\rfloor - \min(r_1 \bar s_2, r_2 \bar s_1) + r_2} \right)
		\\
		&= \mc{O}\!\left( \zeta^{\bar s_1 r_2 -  \bar s_1 i +r_1 +  r_{1} \left\lfloor \frac{s_{2} (i-1)}{r_{2}} \right\rfloor - \min(r_1 \bar s_2, r_2 \bar s_1) } \right)
	\end{split}
	\end{equation*}
	as $z \rightarrow p_1$. Each of these terms will be holomorphic as $z \rightarrow p_1$ (and hence the sum over $i$ will also be holomorphic) provided that
$$
\forall i \in [r_2]\qquad 		\bar s_1 r_2 -  \bar s_1 i +r_1 +  r_{1} \left\lfloor \frac{s_{2} (i-1)}{r_{2}} \right\rfloor - \min(r_1 \bar s_2, r_2 \bar s_1) \geq 0\,,
$$
	with the extra requirement that $\big(\frac{\tau_1/r_1}{\tau_2/r_2}\big)^{r'} \neq 1$ in case $\frac{\bar s_1}{r_1} = \frac{\bar s_2}{r_2}$.

	The same calculation goes through with $p_1 \leftrightarrow p_2$, and we obtain the second condition
$$
\forall i \in [r_1]\qquad 		\bar s_2 r_1 - \bar s_2 i +r_2 +  r_{2} \left\lfloor \frac{s_{1} (i-1)}{r_{1}} \right\rfloor - \min(r_1 \bar s_2, r_2 \bar s_1) \geq 0\,,
$$
with the extra requirement that $\big(\frac{\tau_1/r_1}{\tau_2/r_2}\big)^{r'} \neq 1$ in case $\frac{\bar s_1}{r_1} = \frac{\bar s_2}{r_2}$, i.e. the pair is non-resonant.
\qedhere

\end{proof}

\subsubsection{Definitive form}

\Cref{le:pairs} is a starting point, but the pairwise conditions are still fairly complicated to check. They can be further simplified as follows.

\begin{definition}\label{de:rij}
If $r_i,r_j \geq 1$ we denote $r_{ij} = \max(r_i,r_j)$.
\end{definition}

\begin{theorem}[Conditions for vertical globalisation]\label{th:rewriting}
	Let $(\omega_{g,n})_{(g,n) \in \mathbb{Z}_{\geq 0} \times \mathbb{Z}_{> 0}}$ be a system of correlators on a locally admissible spectral curve $\mathcal{S} = (\Sigma,x,\omega_{0,1},\omega_{0,2})$ that satisfies topological recursion. Let $q \in x(\Sigma)$ and $\{p_1,p_2\}$ a pair of points in $x^{-1}(q)$. Let $(r_k,s_k,\bar s_k, \tau_k)$ be the local parameters at $p_k$ and $\nu_k = \frac{\bar s_k}{r_k}$ the aspect ratio. Assume without loss of generality that $ \nu_1 \leq \nu_2 $. Then conditions (C1) and (C2) in \cref{le:pairs} are satisfied if and only if \textit{one of} the following conditions holds:
	\begin{enumerate}[({C}-i)]
		\item $\nu_1 \leq \frac{1}{m} \leq \nu_2$ for some integer $m \in [r_{12} - 1]$;
		\item $\nu_1 \leq \frac{1}{r_{12}}$.
	\end{enumerate}
Topological recursion can be globalised over a finite set $P \subseteq x^{-1}(q)$ provided that for any pair of points in $P$, ``(C-i) or (C-ii)'' holds and the pair is non-resonant.
\end{theorem}

\begin{proof}
	Let us recall the conditions (C1) and (C2) of \cref{le:pairs}, using that $\min(r_1 \bar s_2, r_2 \bar s_1) = r_2 \bar s_1$ because we choose $\nu_1 \leq \nu_2$:
	\begin{enumerate}[({C}1)]
		\item $\forall i \in [r_2] \qquad 1 +  \left\lfloor \frac{s_{2} (i-1)}{r_{2}} \right\rfloor -  \nu_1 i \geq 0$ ,
		\item $\forall i \in [r_1] \qquad 1- \bar s_1+\left\lfloor \frac{s_{1} (i-1)}{r_{1}} \right\rfloor  + \nu_2 (r_1 - i) \geq 0$ .
	\end{enumerate}
	First, we argue that we can replace $s_k$ by $\bar s_k$ everywhere in these conditions. This is of course true if $s_k = \bar s_k$. Now assume that $r_1 \neq 1$ and $s_1 = \bar s_1 + 1$. Then $r_1|\bar s_1$ and:	\begin{equation}
		\left\lfloor \frac{s_{1} (i-1)}{r_{1}} \right\rfloor  = \left\lfloor \frac{\bar s_{1} (i-1)}{r_{1}}  + \frac{i-1}{r_1}\right\rfloor  = \left\lfloor \frac{\bar s_{1} (i-1)}{r_{1}} \right\rfloor \,,
	\end{equation}
	since $\frac{\bar s_{1} (i-1)}{r_{1}} \in \mathbb{Z}$ and $i \in [r_1]$. In the special case $r_1=1$, the only value of $i$ is $i=1$ and in \cref{de:localloop} we had taken the convention that $\left \lfloor \frac{s_1(i-1)}{r_1} \right \rfloor$ is equal to $0$ (even if formally $s_1 = \infty$), which is of course equal to $\left \lfloor \frac{\bar s_1 (i-1)}{r_1} \right \rfloor$. A similar argument justifies the replacement of $s_2$ with $\bar{s}_2$.
	
Eventually, we can rewrite the conditions as
		\begin{enumerate}[({C}1)]
		\item $\forall i \in [r_2] \qquad 1 +  \left\lfloor \nu_2(i-1)  \right\rfloor-  \nu_1 i \geq 0,$
		\item $\forall i \in [r_1] \qquad 1- \bar s_1+\left\lfloor \nu_1 (i-1)  \right\rfloor + \nu_2 (r_1 - i) \geq 0.$
	\end{enumerate}

	Now let us rewrite condition (C2) with $j =r_1 + 1 -i$. The conditions become
	\begin{enumerate}[({C}1)]
		\item $\forall i \in [r_2] \qquad 1 +  \left\lfloor \nu_2(i-1)  \right\rfloor-  \nu_1 i \geq 0,$
		\item $\forall j \in [r_1] \qquad 1+\left\lfloor - \nu_1 j   \right\rfloor + \nu_2 (j-1) \geq 0.$
	\end{enumerate}
	Since in this inequality two terms out of three are integers, using the fact that for any  $y \in \mathbb{R}$ and $A \in \mathbb{Z}$ the inequality $y \geq A$ is equivalent to $\lfloor y \rfloor \geq A$, we obtain the two equivalent conditions	\begin{enumerate}[({C}1)]
		\item $\forall i \in [r_2] \qquad 1 +  \left\lfloor \nu_2(i-1)  \right\rfloor+ \left \lfloor- \nu_1 i \right \rfloor \geq 0,$
		\item $\forall j \in [r_1] \qquad 1+\left\lfloor -  \nu_1j   \right\rfloor + \left \lfloor  \nu_2 (j-1)  \right \rfloor\geq 0.$
	\end{enumerate}
	In other words, the two conditions are the same, except for the range of $i$ and $j$. So let us consider condition (C1) only, with the index $i$ running from $1$ to $r_{12}$ and call it (C*).
	
	For $i=1$, we get $1 +  \left \lfloor-  \nu_1 \right \rfloor \geq 0$, which is satisfied if and only if $\nu_1 \leq 1$. If $r_1 = r_2 = 1$, we stop here. Otherwise, we have to examine the remaining conditions $1 +  \left\lfloor \nu_2(i - 1)  \right\rfloor+ \left \lfloor- \nu_1 i \right \rfloor \geq 0$ for $i \in (r_{12}]$. Since $\nu_2 \geq \nu_1$ and $\nu_1 \leq 1$ (from the case $i = 1$), we observe that
\begin{equation}
\label{3ineq}
		1 +  \left\lfloor \nu_2 (i - 1)  \right\rfloor+ \left \lfloor- i \nu_1 \right \rfloor \geq 1 +  \left\lfloor (i-1)\nu_1  \right\rfloor+ \left \lfloor- i \nu_1 \right \rfloor \geq \left \lfloor (i-1)\nu_1 - i\nu_1 \right \rfloor  \geq -1 \,.
\end{equation}
Thus, (C*) is satisfied if and only if for each $i \in (r_{12}]$ one of the inequalities in \eqref{3ineq} is strict. Introducing $f_i(x) = 2 + \lfloor x(i - 1) \rfloor + \lfloor -x i \rfloor$, condition (C*) is equivalent to having simultaneously $\nu_1 \leq 1$ and
\begin{equation}
\label{alternt0}
\forall i \in (r_{12}],\qquad f_i(\nu_1) > 0 \quad {\rm or} \quad \big(f_i(\nu_1) = 0 \quad {\rm and}\quad \lfloor \nu_2(i - 1) \rfloor > \lfloor\nu_1 (i - 1) \rfloor\big)\,.
\end{equation}
If $\nu_1 < 0$, we have
$$
f_i(\nu_1) = 2 + \lfloor -|\nu_1|(i-1) \rfloor + \lfloor |\nu_1|i \rfloor \geq 2  + \big(-|\nu_1|(i - 1) - 1\big) + (|\nu_1|i - 1) = |\nu_1| > 0\,,
$$
so the inequality $f_i(\nu_1) > 0$ is always realised. In the limit case $\nu_1 = 1$, we have $f_i(\nu_1) = 1 > 0$. To analyse the case $\nu_1 \in [0,1)$ we decompose
$$
[0,1) = \bigsqcup_{\ell = 1}^{i - 1} \big(T_{\ell}^{(i)} \sqcup \hat{T}_{\ell}^{(i)}\big),\qquad T_{\ell}^{(i)} = \left[\frac{\ell - 1}{i - 1},\frac{\ell}{i}\right],\quad \hat{T}_\ell^{(i)} = \left(\frac{\ell}{i},\frac{\ell}{i - 1}\right)\,.
$$
If $\nu_1 \in T_\ell^{(i)}$ we find that $\lfloor \nu_1(i - 1) \rfloor = \ell -1$ and $\lfloor -\nu_1 i \rfloor = - \ell$, therefore $f_i(\nu_1) = 1 > 0$. If $\nu_1 \in \hat{T}_\ell^{(i)}$ we find that $\lfloor \nu_1 (i - 1) \rfloor = \ell - 1$ while $\lfloor -\nu_1 i \rfloor = -\ell-1$, therefore $f_i(\nu_1) = 0$. In the latter case, the alternative condition $\lfloor \nu_2(i - 1) \rfloor > \lfloor \nu_1(i - 1)\rfloor = \ell - 1$ from \eqref{alternt0} reads $\nu_2 \geq \frac{\ell}{i - 1}$.

Introducing
$$
\hat{T} = \bigcup_{i = 2}^{r_{12}} \bigcup_{\ell = 1}^{i - 1} \hat{T}^{(i)}_\ell,\qquad T = [0,1] \setminus \hat{T}\,,
$$
what we have established so far, is that Condition (C*) is satisfied if and only if one of the two following conditions is satisfied
\begin{enumerate}[(i)]
\item $\nu_1 \in (-\infty,0) \sqcup T$
\item $\forall i \in (r_{12}] \quad \forall \ell \in [i-1] \qquad \nu_1 \in \hat{T}_\ell^{(i)} \,\,\, \Rightarrow\,\,\, \nu_2 \geq \frac{\ell}{i - 1}$.
\end{enumerate}

We claim that
\begin{equation}
\label{TThar} T = \left[0,\frac{1}{r_{12}}\right] \cup \left\{\frac{1}{r_{12} - 1},\frac{1}{r_{12} - 2},\ldots,\frac{1}{2},1\right\},\quad \hat{T} = \bigsqcup_{m = 1}^{r_{12} - 1} \left(\frac{1}{m + 1},\frac{1}{m}\right)\,.
\end{equation}
This comes from three basic observations. First, we have $T_1^{(r_{12})} = [0,\frac{1}{r_{12}}]$, therefore $\hat{T} \subseteq (\frac{1}{r_{12}},1)$ and thus its complement $T$ contains $[0,\frac{1}{r_{12}}]$. Second, any $x \in (\frac{1}{r_{12}},1)$ which is not of the form $\frac{1}{m}$ for some $m \in (r_{12}]$ must belong to $\hat{T}_1^{(i)} = (\frac{1}{i},\frac{1}{i - 1})$ for some $i \in (r_{12}]$, therefore does not belong to $T$. Third, for any given $m \in (r_{12}]$, $\frac{1}{m}$ belongs to $\bigcap_{i = 2}^{r_{12}} T_{\ell_i}^{(i)}$, where $\ell_i = 1 + \lfloor \frac{i}{m} \rfloor$, therefore $\frac{1}{m} \in T$. It results that condition (i) is equivalent to $\nu_1 \leq \frac{1}{r_{12}}$ or $\nu_1 = \frac{1}{m}$ for some $m \in [r_{12}]$.

We now claim that, under the assumption $\nu_1 \in \hat{T}$, condition (ii) is equivalent to the simpler condition
\begin{enumerate}[(ii')]
\item $\exists m \in [r_{12}-1] \quad \nu_1 <\frac{1}{m} \leq \nu_2$.
\end{enumerate}

First assume condition (ii) is satisfied. Since $\nu_1 \in \hat{T}$, due to \eqref{TThar} there exists $m \in [r_{12}-1]$ such that $\nu_1 \in (\frac{1}{m + 1},\frac{1}{m})$. Using (ii) with $i = m + 1$ and $\ell = 1$ shows that $\nu_2 \geq \frac{1}{m}$, therefore (ii') is satisfied.

Conversely, assume condition (ii') is satisfied. We then have $m \in [r_{12}-1]$ such that
\begin{equation}
\label{mu1j}\nu_1 < \frac{1}{m} \leq \nu_2\,,
\end{equation}
and we want to check that condition (ii) is satisfied. For $i \in [m]$, we have $\nu_1 < \frac{1}{m} \leq \frac{1}{i}$ therefore $\nu_1 \in T_1^{(i)}$, so (ii) is automatically satisfied. For $i \geq m + 1$ and $\ell \in [i - 1]$, assume that $\nu_1 \in \hat{T}_\ell^{(i)}$, i.e.
\begin{equation}
\label{mu2j} \frac{\ell}{i} < \nu_1 < \frac{\ell}{i - 1}\,.
\end{equation}
Combining \eqref{mu1j} and \eqref{mu2j} we see that $\frac{\ell}{i} < \frac{1}{m}$. This gives $\ell m < i$, and as this inequality involves only integers, we have actually $\ell m \leq i - 1$, which gives $\frac{\ell}{i - 1} \leq \frac{1}{m}$. The second inequality in \eqref{mu1j} yields $\frac{\ell}{i - 1} \leq \frac{1}{m} \leq \nu_2$, which is what we needed to check.

All in all, we have proved that, under the assumption $\nu_1 \in \hat{T}$, (ii) is equivalent to (ii'). Consequently, ``(i) or (ii)'' is equivalent to ``(i) or (ii')''. Eventually, condition (i) where we remove  the cases $\nu_1 = \frac{1}{m}$ for some $m \in [r_{12} - 1]$ is exactly condition (C-ii), and condition (ii') where we include those cases is exactly condition (C-i).
\end{proof}

\subsubsection{Applying the local criterion}

We illustrate how the globalisation criterion of \cref{th:rewriting} applies in some typical situations. Suppose that we have a non-resonant pair of points $\{p_1,p_2\} \subseteq x^{-1}(q)$. Then, globalisation holds provided (C-i) or (C-ii) holds.
\begin{itemize}
\item If $\omega_{0,1}$ has a pole at $p_1$ (i.e. $\bar s_1 \leq 0$), (C-ii) holds. 
\item If $\omega_{0,1}$ is holomorphic but nonzero at $p_1$ (i.e. $\bar s_1 =s_1= 1$) and $r_2 > r_1$, then (C-i) holds with $m=r_1 < r_{12} = r_2$. Likewise if $\omega_{0,1}$ is holomorphic but nonzero at $p_2$, with $m=r_2$.
\item If $\nu_2 \geq 1$ and we want (C-i) or (C-ii), then we must have $\nu_1 \leq 1$ (this realises (C-ii) if $r_1=r_2=1$ or (C-i) with $m=1$ if $r_{12} > 1$). In particular, if $p_1$ and $p_2$ are unramified ($r_1=r_2=1$), this means that $\omega_{0,1}$ cannot have a zero both at $p_1$ and $p_2$. 
\item Suppose that $\bar s_1, \bar s_2 \geq 2$, $r_1, r_2 \geq 2$,  and $0 < \nu_1 \leq \nu_2 < 1$. Then (C-ii) is never satisfied. We have $\bar s_1 = s_1$ and $\bar s_2 = s_2$ (because $\nu_k \in (0,1)$ amounts to $\bar{s}_k \in (0,r_k)$, and thus $\bar{s}_k$ cannot be divisible by $r_k$). By local admissibility, we can write $r_k = a_k s_k + b_k$, with $0 \leq a_k < r_k$ and $b_k \in \{1, s_k-1\}$. We claim that (C-i) is satisfied if and only if $a_1 > a_2$. 
\end{itemize}

To justify the last claim, we rewrite (C-i) as the existence of $m \in [r_{12}-1]$ such that $\frac{r_2}{s_2} \leq m \leq \frac{r_1}{s_1}$, which is equivalent to $a_2 + \frac{b_2}{s_2} \leq m \leq a_1 + \frac{b_1}{s_1}$.  This is clearly satisfied if $a_1 > a_2$, by taking $m = a_1 < r_1 \leq r_{12} - 1$. This is clearly not satisfied for $a_1 < a_2$. And, it were satisfied for $a_1 = a_2$, we would have an integer $m' = m - a_1$ such that $\frac{b_2}{s_2} \leq m' \leq \frac{b_1}{s_1}$, which is impossible since $\frac{b_k}{s_k} \in (0,1)$ for $k = 1,2$. 

Here is a noteworthy consequence of the two first points. Assume that $P \subseteq P' \subseteq x^{-1}(q)$ are given, that we already know that globalisation over $P$ holds, and that $P' \setminus P$ consists only of points where $\omega_{0,1}$ does not have a zero. Then, if any pair $\{p_1,p_2\}$ with $p_1 \in P' \setminus P$ and $p_2 \in P'$ is non-resonant, we can conclude that globalisation over $P'$ holds.

\subsection{Globalisation above an open set}

We go back to \cref{pr:globalTRdomain}. Given a spectral curve, to globalise topological recursion above an open set $V \subseteq x(\Sigma) \subseteq \mathbb{P}^1$ (or above the whole $x(\Sigma)$), we need to verify the conditions of \cref{th:rewriting} for all pairs of points in $x^{-1}(q)$, and this for each $q \in V$ (or $q \in x(\Sigma)$). This looks like a daunting task. Fortunately, for all points $q \in x(\Sigma) $ that are not branch points of $x$, the conditions that guarantee vertical globalisation are rather simple and easy to check. This is the content of the next lemma.

\begin{lemma}[Conditions for unramified vertical globalisation]\label{th:unramified}
	Let $(\omega_{g,n})_{(g,n) \in \mathbb{Z}_{\geq 0} \times \mathbb{Z}_{>0}}$ be a system of correlators on a locally admissible spectral curve $\mathcal{S} = (\Sigma,x,\omega_{0,1},\omega_{0,2})$ that satisfies topological recursion. Suppose that $q \in x(\Sigma) $ is not a branch point  and let $x^{-1}(q) = \{ p_1, \ldots, p_d\}$. Then topological recursion can be globalised above $q$ if the following two conditions are satisfied:
	\begin{itemize}
		\item At most one of the $p_i$ has $\bar s_i \geq 2$ (i.e. the $1$-form $\omega_{0,1}$ cannot vanish at two distinct points in the same fibre of $x$).
		\item For any $i,j$ such that $\bar s_i = \bar s_j$, we have $\tau_i \neq \tau_j$ (i.e. if the $1$-form $\omega_{0,1}$ has the same order at $p_i$ and at $p_j$, it must have distinct leading order coefficients when read in the standard coordinate).
	\end{itemize}
	When these two conditions are met we say that $\omega_{0,1}$ \emph{separates fibres at $q$}.
\end{lemma}

\begin{proof}
	Since $q \not\in \mathsf{Br}$, we have $r_i=1$ for all $i \in [d]$. In particular the standard coordinate is unambiguously defined. Therefore, by \cref{th:rewriting},  for all pairs $\{p_i, p_j\} \in x^{-1}(q)$, with the ordering $\bar s_i \leq \bar s_j$,  we would like to verify $\bar s_i \leq 1$ (condition (C-i) or (C-ii)). Since this has to hold for all pairs of points in $x^{-1}(q)$, we conclude that at most one $p_i$ can have $\bar s_i \geq 2$. Finally, condition (C3) of \cref{le:pairs} imposes that $\tau_i \neq \tau_j$ whenever $\bar s_i = \bar s_j$.
\end{proof}

Daunting it is no more! To show that topological recursion can be globalised above an open set $V \subseteq x(\Sigma)$, or fully globalised, all we have to check is that topological recursion can be globalised vertically above all branch points (by verifying the conditions in \cref{th:rewriting}, this is a finite number of tests), and that $\omega_{0,1}$ separates fibres at other points, in the sense of \cref{th:unramified}.

\subsection{Digression: local vs. global topological recursion}
\label{Sec:intermz}
In order to handle topological recursion on the normalisation of possibly singular spectral curves, \cite{BKS20} identified sufficient conditions as well as necessary conditions for the global topological recursion to be well defined, i.e. for the residue formula  with a sum over subsets $Z \subseteq \mathfrak{f}'(z)$ of the whole fibre  (instead of just subsets $Z \subseteq \mathfrak{f}'_p(z)$ of the local fibre at $p$ in \eqref{eq:TR}) to produce symmetric multi-differentials.  This is a weaker property than globalisation: the latter requires the correlators produced by global topological recursion to coincide with those produced by local topological recursion \eqref{eq:TR}.

It is worth pausing for a moment to compare our results with theirs. While the sufficient conditions found in \cite{BKS20} could potentially be weaker than the sufficient conditions for globalisation given in \cref{th:rewriting} (since global topological recursion could potentially produce symmetric differentials that differ from those produced by local topological recursion), the necessary conditions found in \cite{BKS20} should be implied by the sufficient conditions of \cref{th:rewriting} (since if topological recursion can be globalised, then certainly global topological recursion should produce symmetric differentials). Yet, there are three small differences between the setting of spectral curves adopted in the present article and in \cite{BKS20}. First, they allow $\bar s_p = \infty$, i.e. $\omega_{0,1}$ can be identically zero on some connected component, but we do not. Second, we allow $\bar s_p \neq s_p$, while they impose $\bar s_p = s_p$ (i.e. $r_p$ does not divide $\bar s_p$). Third, we allow $s_p \leq -1$, while they impose $s_p > 0$. For this reason, we quote results from \cite{BKS20} only after specialisation to the intersection of the two setting.

Let us start by checking that the necessary conditions for global topological recursion to be well-defined are implied by the sufficient conditions for globalisation above any branch point found in \cref{th:rewriting}.

\begin{theorem}\cite[Theorem 2.13 and Proposition 2.16]{BKS20} \label{th:BKSnecessary}
  Let $\mathcal{S}$ be a finite spectral curve such that, for any $p \in \mathsf{Ram}$, we have $s_p = \bar s_p \in \mathbb{Z}_{> 0}$ and any pair of points in the fiber $x^{-1}(x(p))$ is non-resonant. Then the global topological recursion produces symmetric $n$-differentials $ \omega_{0,n}$ for any $n \geq 3$, if and only if for any branch point $q$ the three following conditions are simultaneously satisfied
  \begin{itemize}
    \item[(C-a)] for any $p \in x^{-1}(q)$, we have $ r_p = \pm 1\,\,{\rm mod}\,\,s_p$;
    \item[(C-b)] for all pairs $\{p_1,p_2\} \subseteq x^{-1}(q)$ such that $\min(s_1,s_2) \geq 3$ and for which there exists $\epsilon \in \{\pm 1\}$ with $r_k = \epsilon\,\,{\rm mod}\,\,s_k$ for $k = 1,2$, we have $\lfloor \frac{1}{\nu_1} \rfloor \neq \lfloor \frac{1}{\nu_2} \rfloor$;
    \item[(C-c)] for all triplets $\{p_1,p_2,p_3\} \subseteq x^{-1}(q)$ such that $ \lfloor \frac{1}{\nu_1} \rfloor = \lfloor \frac{1}{\nu_2} \rfloor = \lfloor \frac{1}{\nu_3} \rfloor$, then $s_{k}= 1$ for some $k \in \{1,2,3\}$.
  \end{itemize}
\end{theorem} 

\begin{lemma}{}{}
  Let $\mathcal{S}$ be a locally admissible finite spectral curve such that $\bar s_p = s_p \in \mathbb{Z}_{> 0}$ for any $p \in x^{-1}(\mathsf{Br})$. The condition ``(C-i) or (C-ii)'' of \cref{th:rewriting} imply the conditions ``(C-a) and (C-b) and (C-c)'' of \cref{th:BKSnecessary}.
\end{lemma}
\begin{proof}
Since $s_p > 0$, local admissibility imposes $r_p = \pm 1 \,\,{\rm mod}\,\,s_p$ , so (C-a) is part of the assumptions.

If a pair $\{p_1,p_2\} \subseteq x^{-1}(q)$ has $\min(s_1,s_2) \geq 3$ and satisfy $r_k = \pm 1\,\,{\rm mod}\,\,s_k$ for $k = 1,2$, then $\frac{1}{\nu_k} = \frac{r_k}{s_k}$ is an irreducible fraction which cannot be an integer. If (C-i) is satisfied this implies that there is an integer in $(\frac{1}{\nu_1},\frac{1}{\nu_2})$ therefore $\lfloor \frac{1}{\nu_1} \rfloor = \lfloor \frac{1}{\nu_2} \rfloor$ which is the conclusion of (C-b). If (C-ii) holds, then the premise of (C-b) does not and there is nothing to check.

Assume we have a triplet $\{p_1,p_2,p_3\} \subseteq x^{-1}(q)$ such that $\lfloor \frac{1}{\nu_1} \rfloor = \lfloor \frac{1}{\nu_2} \rfloor = \lfloor \frac{1}{\nu_3} \rfloor$. If (C-i) is satisfied for the pair $\{p_1,p_2\}$, then $\nu_1 = \nu_2$ is an integer, in particular $s_1 = s_2 = 1$. If (C-ii) is satisfied for the pair $\{p_1,p_2\}$, then $\nu_1 > 0$ imposes $r_1 \geq r_2$ and $\nu_1 = \frac{1}{r_1}$, hence $s_1 = 1$. So, the conclusion of (C-c) holds.
\end{proof}

While sufficient conditions for globalisation should imply necessary conditions for global topological recursion to be well defined, there need not be a logical relation between sufficient conditions for globalisation and sufficient conditions for  global topological recursion to be well-defined. In fact, the next lemma shows that the sufficient conditions for global topological to be well defined from \cite{BKS20} are weaker than the sufficient conditions for globalisation from \cref{th:rewriting}: there are exceptional cases that do not satisfy the conditions of \cref{th:rewriting} for which global topological recursion produces symmetric differentials.

\begin{theorem}\cite[Theorem 2.11]{BKS20} \label{th:BKSsufficient}
  Let $\mathcal{S}$ be a locally admissible finite spectral curve such that, for any $p \in x^{-1}(\mathsf{Br})$, we have $s_p = \bar s_p \in \mathbb{Z}_{> 0}$ and any pair of points in the fiber $x^{-1}(x(p))$ is non-resonant. Above any branch point $q$, order the points in $x^{-1}(q)$ as  $p_1,\ldots,p_d$ such that $\nu_1 \leq \nu_2 \leq \cdots \leq \nu_d$. Assume that above any branch point $q$ we have
$$
r_1 = -1 \,\,{\rm mod}\,\,s_1\qquad {\rm and}\quad \big(\forall k \in \{2,\ldots,d-1\}\quad\,\,\, s_k = 1\big)\qquad {\rm and} \qquad r_d = 1\,\,{\rm mod}\,\,s_d\,.
$$
Then global topological recursion (with sum over subsets $Z \subseteq \mathfrak{f}'(z)$ of the full fibre) is well defined.
\end{theorem}

\begin{lemma}\label{pr:notthesame}
Let $\mathcal{S}$ be a locally admissible finite spectral curve such that, for any $p \in x^{-1}(\mathsf{Br})$, we have $s_p = \bar s_p \in \mathbb{Z}_{> 0}$ and any pair of points in the fiber $x^{-1}(x(p))$ is non-resonant.  Let $q \in \mathsf{Br}$ and order the points in $x^{-1}(q)$ as above.

If $|x^{-1}(q)| > 2$, then the sufficient conditions from \cref{th:BKSsufficient} imply the sufficient conditions ``(C-i) or (C-ii)'' for any pair of points in $x^{-1}(q)$ from \cref{th:rewriting}.

The same holds for $|x^{-1}(q)| = 2$, except when $ (r_1, s_1, r_2, s_2)$ are such that $\min(s_1,s_2) \geq 2$  and $\lfloor \frac{1}{\nu_1} \rfloor = \lfloor \frac{1}{\nu_2} \rfloor$, in which case the pair $x^{-1}(q)$ does not satisfy ``(C-i) or (C-ii)''.
\end{lemma}

\begin{proof}
Assume the premises of \cref{th:BKSsufficient}. First consider $|x^{-1}(q)| > 2$. For any pair of points $\{p_i,p_j\} \in x^{-1}(q)$ there are two possibilities:
\begin{itemize}
\item either one of them is equal to $1$. Say it is $s_i$. Then $\nu_i = \frac{1}{r_i}$, which implies that the pair satisfies``(C-i) or (C-ii)'';
\item or the pair is $\{p_1,p_d\}$. Then we have $\nu_1 \leq \nu_2 = \frac{1}{r_2} \leq \nu_d$, hence (C-i) is satisfied with $m = r_2$.
\end{itemize}
We now turn to $|x^{-1}(q)| = 2$. If $s_1 = 1$ , then  for $r_1 \geq r_2$ we have $\nu_1 = \frac{1}{r_{12}}$ and (C-ii) is satisfied, and for $r_1 < r_2$ we have $\nu_1 = \frac{1}{r_{1}} \leq \nu_2$ so (C-i) is satisfied with $m = r_1$. If $\min(s_1,s_2) \geq 2$, we may write $r_1 = r_1's_1 + s_1 - 1$ and $r_2 = r_2's_2 + 1$ for $r_1',r_2' \in \mathbb{Z}_{\geq 0}$. Then $r_i' = \lfloor \frac{r_i}{s_i} \rfloor$ and
$$
r_1' + \frac{s_1 - 1}{s_1} = \frac{1}{\nu_1} \geq  \frac{1}{\nu_2} = r_2' + \frac{1}{s_2}\,.
$$
If $r_1' \neq r_2'$, then (C-i) is satisfied with $m = r_1'$. However, if $r_1' = r_2'$, then (C-i) is not satisfied, and $\frac{s_1}{r_1} > \frac{1}{r_2} \geq \frac{1}{r_{12}}$ prevents (C-ii) to be satisfied. 
\end{proof}

In the exceptional cases of \cref{pr:notthesame}, global topological recursion produces well-defined symmetric multi-differentials, but we do not know if these are the same as the multi-differentials produced by local topological recursion. We do not study this further here but it would be interesting to understand what happens for such spectral curves. The exception is realised for instance with $(r_1,s_2) = (3,2)$ and $(r_2,s_2) = (4,3)$.

\section{Spectral curves, algebraic curves and normalisations}

\label{s:algcurves}
\label{S4}

So far we talked about spectral curves abstractly, according to \cref{de:sc}. In practice, many spectral curves come from algebraic curves. In the literature on topological recursion, spectral curves are often obtained by first considering an affine curve $P(x,y)=0$ and then ``parametrising it'' by thinking of $x$ and $y$ as two meromorphic functions on a compact Riemann surface. There are subtleties with this statement, in particular if the affine curve (or its projectivisation) is singular, which is usually the case. In this section we construct spectral curves from algebraic curves carefully, and determine which kind of algebraic curves give rise to spectral curves: (1) which are locally admissible; (2) on which topological recursion can be fully globalised. Along the way we review classical aspects of the resolution of singularities of algebraic curves which are relevant to the construction.

\subsection{Non-compact spectral curves from affine curves}
\label{S4noncom}

\subsubsection{Setting}

Consider an affine curve
\begin{equation}
	\mathfrak{C}= \{(x,y) \in \mathbb{C}^2 \,\,|\,\, P(x,y) = 0 \} \,,
\end{equation}
where $P(x,y)$ is a polynomial of positive degrees $d_x$ in $x$ and $d_y$ in $y$. We assume that $P$ is reduced, which means that it can be decomposed into pairwise distinct irreducible factors $P = P_1 \cdots P_N$ for some positive integer $N$. We assume that $P$ has no factors in $\mathbb{C}[x]$, and that $\{y = 0\} \not\subseteq \mathfrak{C}$. Let us write $\mathfrak{C}_i = \{p \in \mathbb{C}^2 \,\,|\,\, P_i(p) = 0\}$ for the irreducible affine curves corresponding to each irreducible component. We denote by $\pi_x \colon \mathfrak{C} \rightarrow \mathbb{C}$ and $\pi_y \colon \mathfrak{C} \rightarrow \mathbb{C}$ the two projections $\pi_x(x,y) = x$ and $\pi_y(x,y) = y$. 

To define a spectral curve from this data, we need to obtain a Riemann surface (or a disjoint union thereof). However, the affine curve $\mathfrak{C}$ is a Riemann surface (or a disjoint union of Riemann surfaces) if and only if it is smooth. Yet, we often need to deal with singular affine curves. How can we proceed?

The singular locus
\begin{equation}
	\mathfrak{s} = \left\{ p \in \mathbb{C}^2 \quad | \quad P(p)= \partial_{x} P(p) = \partial_{y} P(p) = 0\right\} \subset \mathfrak{C}\,.
\end{equation}
 is a finite subset. It contains the points at which some irreducible component $\mathfrak{C}_i$ is singular, and the intersection points between the $N$ components (if the curve is reducible). If $\mathfrak{s} \neq \emptyset$, $\mathfrak{C}$ is not a Riemann surface. However, $\mathfrak{C}^{\circ} = \mathfrak{C} \setminus \pi_x^{-1}(\pi_x(\mathfrak{s}))$, obtained by removing the whole $\pi_x$-fibres of singular points, is a disjoint union of $N$ non-compact Riemann surfaces. As $\mathfrak{C}$ has no factors in $\mathbb{C}[x]$, $\pi_x$ is non-constant on any of the components of $\mathfrak{C}^{\circ}$. In fact, it defines a finite-degree branched covering\footnote{If we had only removed the singular points from $\mathfrak{C}$, instead of their whole $\pi_x$-fibre, we would have obtained a disjoint union of Riemann surfaces, but $\pi_x$ may not have been a finite-degree branched covering.}. The curve $\mathfrak{C}^{\circ}$ is connected if and only if  $\mathfrak{C}$ is irreducible. Let
\begin{equation} 
	\mathfrak{r} =  \left\{ p \in \mathbb{C}^2 \quad | \quad P(p)= \partial_{y} P(p) = 0 \right\},\quad \mathsf{Ram} \coloneqq \big(\mathfrak{r} \setminus \pi_x^{-1}(\pi_x(\mathfrak{s}))\big) \subset \mathfrak{C}^{\circ}\,.
\end{equation} 
$\mathsf{Ram}$ is the set of ramification points of $\pi_x \colon \mathfrak{C}^{\circ} \rightarrow \mathbb{C}$.
 
\begin{definition}[Non-compact spectral curve associated to a reduced affine curve]\label{de:redac}
	Let $\mathfrak{C} = \{P(x,y) = 0 \} \subset \mathbb{C}^2$ be a reduced affine curve that has no factors in $\mathbb{C}[x]$ and such that $\{ y=0 \} \not\subseteq \mathfrak{C}$, $\pi_x$ and $\pi_y$ the two canonical projections, $\mathfrak{s}$ the set of singular points and $\mathfrak{C}^{\circ} = \mathfrak{C} \setminus \pi_x^{-1}(\pi_x(\mathfrak{s}))$. Given a fundamental bidifferential $\omega_{0,2}$ on $\mathfrak{C}^{\circ}$, we define the \emph{non-compact spectral curve associated to $\mathfrak{C}$ and $\omega_{0,2}$} to be the quadruple $(\mathfrak{C}^{\circ},\pi_x, \pi_y \dd \pi_x, \omega_{0,2})$. 
 \end{definition}

\subsubsection{Local admissibility and globalisation}

Spectral curves obtained by removing singularities from affine curves are nice: the next lemma says that they are always locally admissible, and topological recursion can always be fully globalised on them,  regardless of the type and location of the ramification points of $\pi_x$ (even if some fibres contain several ramification points).

\begin{lemma}\label{le:redacla}
	Let $\mathcal{S}=(\mathfrak{C}^{\circ},\pi_x, \pi_y \dd \pi_x,\omega_{0,2})$ be a non-compact spectral curve associated to a reduced affine curve $\mathfrak{C} = \{P(x,y) = 0 \} \subset \mathbb{C}^2$ as in \cref{de:redac}. Then $\mathcal{S}$ is locally admissible and topological recursion can be fully globalised on $\mathcal{S}$.
\end{lemma}

\begin{proof}
	We first show that $\mathcal{S}$ is locally admissible.
	The ramification set $\mathsf{Ram} \subset \mathfrak{C}^{\circ}$ is finite because $P$ is a polynomial. Let $p \in \mathsf{Ram} \subset \mathfrak{C}^{\circ}$ and $r$ the order of $\pi_x$ at $p$. Let $b = \pi_y(p)$. Then, near $p$, we have $y = b + c \zeta + o(\zeta)$ with some $c \neq 0$ since $p$ is a smooth point. It implies
	$$
	\pi_y \dd \pi_x = \big(rb \zeta^{r - 1} + rc_1 \zeta^{r} + o(\zeta^r)\big) \dd \zeta \quad {\rm near}\,\,p\,.
	$$
	If $b \neq 0$, the local parameters (\cref{de:localparameters}) at $p$ are $(r_p,s_p,\bar s_p) = (r,r+1,r)$. If $b=0$, the local parameters are $r_p = r$ and $\bar s_p = s_p = r+1$. We conclude that the spectral curve is locally admissible.

	We move on to globalisation. Let $q \in \pi_x(\mathfrak{C}^{\circ}) \subseteq \mathbb{C}$. We need to show that topological recursion can be globalised at $q$. Let $\{p_1, \ldots, p_d\} = \pi_x^{-1}(q) \subset \mathfrak{C}^{\circ}$ and $b_i = \pi_y(p_i)$. Then $b_i \neq b_j$ for $i \neq j$ as these correspond to distinct points on the curve with same $\pi_x$-coordinate $q$. If $r_i$ is the order of $\pi_x$ at $p_i$, we have in a standard coordinate $y = b_i + c_i \zeta + o(\zeta)$ near $p_i$ for some $c_i \neq 0$. It implies
	$$
	\pi_y \dd \pi_x = \big(r_ib_i \zeta^{r_i - 1} + r_ic_i\zeta^{r_i} + o(\zeta^{r_i})\big)\dd \zeta \quad {\rm near}\,\,p\,.
	$$
	As the $b_i$ are pairwise distinct, at most one of them can vanish. If this happens, for this particular point $p_{i_0}$ we have $\bar s_{i_0} > r_{i_0}$, and hence $\nu_{i_0} = \frac{\bar s_{i_0}}{r_{i_0}} > 1$, and for $j \neq i_0$  we have $\bar s_j = r_j$ and $\tau_j = r_j b_j$, and hence $\nu_j = \frac{\bar s_j}{r_j } = 1$ and $\frac{\tau_j}{r_j} = b_j$. As a result, for any pairs of points $\{p_j,p_k\} \subseteq \pi_x^{-1}(q)$, either $\nu_j  = \nu_k=1$ and $\frac{\tau_j}{r_j} =b_j\neq b_k= \frac{\tau_k}{r_k}$, or $\nu_j = 1$ and $\nu_k> 1$ (or vice versa). Those satisfy the conditions of \cref{th:rewriting} ((C-i) with $m=1$), and topological recursion can be globalised above $q$.
	
	As $\pi_x$ is a finite-degree branched covering, this shows that topological recursion can be fully globalised.
\end{proof}

\subsection{Compact spectral curves after projectivisation and normalisation}
\label{S4com}

Whenever possible it is preferable to work with compact spectral curves, which is not the case for the spectral curves described in \cref{S4noncom}. This is the setting in which the Hitchin integrable system on the Higgs moduli space is usually studied, and we refer to \cite{DM14,BH19,CNST20} for the relevance of topological recursion in this context. It is also interesting to have compact spectral curves when studying deformation of curves and limits in families, as we do in \cref{S5}.

As a prototypical example, consider the affine curve $\mathfrak{C}_{r,s} \subset \mathbb{C}^2$ defined as the vanishing locus of 
\begin{equation}
	P_{r,s}(x,y) = x^{r-s} y^r - 1\quad {\rm with} \quad r \geq 2\,\,{\rm and}\,\, s \in [r-1]\,.
\end{equation}
We see that neither $\partial_x P$ nor $\partial_y P$ vanish on $\mathfrak{C}_{r,s}$, since the lines $\{x=0\}$ and $\{y=0\}$ do not intersect in $\mathfrak{C}_{r,s}$. Thus $\mathfrak{C}_{r,s}$ is smooth as an affine curve: $\mathfrak{C}^{\circ}_{r,s} = \mathfrak{C}_{r,s}$. It is connected if and only if $r$ and $s$ are coprime. However, since $\partial_yP$ never vanishes on $\mathfrak{C}_{r,s}$, $\pi_x$ is unramified, hence topological recursion on the non-compact spectral curve associated to $\mathfrak{C}_{r,s}$ as in \cref{de:redac} produces correlators $\omega_{g,n}$ that vanish identically for $2g - 2 + n > 0$. This simply comes from the fact that ramification points at $\infty$ were not taken into account. To get a non-trivial output from topological recursion, a solution is to projectivise the affine curve and resolve its singularities via normalisation. We note that taking into account ramification points at $x = \infty$ in the topological recursion and working on the normalisation first appeared in the work of Dumitrescu and Mulase \cite{DM14}.

In this section we would like to understand systematically how this procedure affects local admissibility (\cref{S423}) and the possibility of globalisation (\cref{S424}). On the negative side, we find  in \cref{le:singularities} an obstruction to obey the criterion of globalisation, which pertains to the location of the singularities of $C$. On the positive side, we introduce a notion of global admissibility (\cref{de:globaladm}), tailored to show in \cref{th:ga} that topological recursion is fully globalisable on globally admissible spectral curves.

\subsubsection{Setting}
 
Start with a reduced affine curve $\mathfrak{C} = \{(x,y)  \in \mathbb{C}^2 \,\, | \,\, P(x,y) = 0 \}$, where $P(x,y)$ is a polynomial of positive degrees $d_x$ in the variable $x$ and $d_y$ in  the variable $y$, with no factors in $\mathbb{C}[x]$, and such that $\{y = 0\} \not\subseteq \mathfrak{C}$.

The first step is to compactify the curve by projectivising it. It is standard to projectivise in $\mathbb{P}^2$, but this is not what we want to do. The reason is that we want to keep the structure of the map given by projection on the $x$-axis. It is therefore more natural to projectivise in $\mathbb{P}^1 \times \mathbb{P}^1$, so that the $x$-projection becomes projection on the first $\mathbb{P}^1$ factor.

Let $([X_0:X_1], [Y_0:Y_1])$ be homogeneous coordinates on $\mathbb{P}^1 \times \mathbb{P}^1$.
We use notations $\mathbf{0} = [0:1]$ and $\boldsymbol{\infty} = [1:0]$ for the special points in $\mathbb{P}^1$ and $p_{\epsilon\epsilon'} = (\boldsymbol{\epsilon},\boldsymbol{\epsilon}')$ with $\epsilon,\epsilon' \in \{0,\infty\}$ for the special points in $\mathbb{P}^1 \times \mathbb{P}^1$. The bihomogenisation of $P$ is
\begin{equation}
	F(X_0,X_1,Y_0,Y_1) = X_1^{d_x} Y_1^{d_y} P\left( \frac{X_0}{X_1}, \frac{Y_0}{Y_1} \right)\,.
\end{equation}
This is a bihomogeneous polynomial of bidegree $(d_x,d_y)$: all terms in $F$ have total degree $d_x$ in $(X_0,X_1)$ and total degree $d_y$ in $(Y_0,Y_1)$. We write
$$
C = \{p \in \mathbb{P}^1 \times \mathbb{P}^1 \quad |\quad F(p) =0\}\,.
$$
The polynomial $F$ is reduced; we write $F = F_1 \cdots F_N$ for a decomposition in irreducible factors, and $C_i = \{F_i =0 \} \subset \mathbb{P}^1 \times \mathbb{P}^1$ for its irreducible components. Let $(d_x^i, d_y^i)$ be the bidegree of the bihomogeneous irreducible polynomial $F_i$. From the assumptions on $P$ we know that $F$ has no factor in $\mathbb{C}[X_0,X_1]$, $\{ Y_0 = 0\} \not\subset C$, and $\{ Y_1 = 0\} \not\subset C$.

The natural projection $\pi_X \colon C \rightarrow \mathbb{P}^1$, defined by $\pi_X([X_0:X_1],[Y_0:Y_1])  =  [X_0:X_1]$ is non-constant on each component of $C$ since $F$ has no factor in $\mathbb{C}[X_0,X_1]$. The bihomogenisation of the 1-form $\pi_y \dd \pi_x$ is
\begin{equation}
\alpha = \frac{Y_0}{Y_1} \dd \left (\frac{X_0}{X_1} \right) \, .
\end{equation}
 Since $\{ Y_0 = 0\} \not\subset C$ and $\{ Y_1 = 0\} \not\subset C$, $\alpha$ is a well-defined meromorphic $1$-form on $C$, whose restriction to each component is not identically zero.

In general, the curve $C$ may be singular. Let
\begin{equation}\label{eq:singular}
	\mathsf{Sing} = \left\{p\in \mathbb{P}^1 \times \mathbb{P}^1 \quad | \quad F(p) = \partial_{X_0} F(p) = \partial_{X_1} F(p) =  \partial_{Y_0} F(p) = \partial_{Y_1} F(p) = 0\right\} \subset C
\end{equation}
be its singular locus. It is finite, and contains as before the points at which some irreducible component $C_i$ is singular, and the intersection points between the components.

If $\mathsf{Sing} \neq \emptyset$, we could proceed as in the previous section and remove the singularities from $C$, but the result would not be compact. Instead, we resolve the singularities by normalising the curve.

Given a projective curve $C$ as above, there exists a $\tilde C$, which is a disjoint union of $N$ compact Riemann surfaces, and a holomorphic map $\eta : \tilde C \rightarrow C$, such that $\eta^{-1}(\mathsf{Sing})$ is finite and $\eta : \tilde C \setminus \eta^{-1}(\mathsf{Sing}) \rightarrow C \setminus \mathsf{Sing}$ is biholomorphic. $(\tilde C, \eta)$ is called \emph{normalisation} of $C$ --- it is unique up to unique isomorphism. Roughly speaking, each singularity $p \in C$ is replaced with $\mathsf{b}_p$ points in $\tilde C$, where $\mathsf{b}_p$ is the number of branches meeting at the singularity. The $\delta$-invariant of $p$ is by definition $\delta_p = \dim_{\mathbb{C}}\big(\mathcal{O}_{\tilde{C},p}/\mathcal{O}_{C,p}\big)$. It can be calculated using Milnor's formula $2 \delta_p = \mathsf{j}_p + \mathsf{b}_p - 1$, where  $\mathsf{j}_p$ is the dimension of the Jacobi algebra at $p$. In a local affine patch $f(u,v) = 0$ where $p$ has coordinates $(u,v) = (0,0)$, this is $\mathsf{j}_p = \dim_{\mathbb{C} }\big(\mathbb{C}[\![u,v]\!]/\langle \partial_u f,\partial_v f\rangle \big)$. 
 
The normalisation procedure splits the components of $C$ and normalises each component separately; we have $\tilde C = \tilde C_1 \sqcup \cdots \sqcup \tilde C_{N}$, where $\tilde C_i$ is the normalisation of the irreducible curve $C_i$. 
The \emph{geometric genus} of $C_i$, which we denote by $\mathsf{g}(C_i)$, is the genus of its normalisation $\tilde C_i$. The \emph{arithmetic genus} of $C_i$, which we denote by $\mathsf{p}_{{\rm a}}(C_i)$, is equal to $\mathsf{p}_{{\rm a}}(C_i) = (d_x^i-1)(d_y^i-1)$, where $(d_x^i,d_y^i)$ is the bidegree of the $i$-th component. By the adjunction formula, the difference between the arithmetic and the geometric genus is the sum of the $\delta$-invariants at the singularities. 
 \begin{equation}\label{eq:gg}
	\mathsf{g}(C_i) = \mathsf{p}_{{\rm a}}(C_i) - \sum_{p \in \mathsf{Sing} \cap C_i} \delta_p\,.
\end{equation}
On $\tilde{C}_i$, the set of fundamental bidifferentials is an affine space of dimension $\frac{1}{2}\mathsf{g}(C_i)(\mathsf{g}(C_i) + 1)$: its underlying vector space is the space of holomorphic bidifferentials on $\tilde{C}_i$, symmetric in their two variables. In particular if $\mathsf{g}(C_i) = 0$ we have a unique fundamental bidifferential $\omega_{0,2}(z,z') = \frac{\dd z\, \dd z'}{(z - z')^2}$ in terms of the uniformising coordinate $z$. This allows us to construct a compact spectral curve.

\begin{definition}[Compact spectral curve associated to a reduced affine curve]\label{de:irredpc}
	Let $\mathfrak{C} \subset \mathbb{C}^2$ be a reduced affine curve cut out by a polynomial $P(x,y)$ that has no factors in $\mathbb{C}[x]$ and such that $\{ y=0 \} \not\subseteq \mathfrak{C}$, and
	\begin{equation}
		C = \Big\{ X_1^{d_x} Y_1^{d_y} P\left( \frac{X_0}{X_1}, \frac{Y_0}{Y_1} \right) = 0 \Big\} \subset \mathbb{P}^1 \times \mathbb{P}^1
	\end{equation}
	its bihomogenisation. Let $\eta \colon \tilde C \rightarrow C$ be the normalisation of $C$, $\pi_X \colon C \rightarrow \mathbb{P}^1$ be the natural projection. Given a fundamental bidifferential $\omega_{0,2}$ on $\tilde C$, we define the \emph{compact spectral curve associated to $C$ and $\omega_{0,2}$} to be the quadruple $(\tilde C,  x, \omega_{0,1}, \omega_{0,2})$, where
	$$
	x = \pi_X \circ \eta \qquad \text{and} \qquad \omega_{0,1} = \eta^* \left(\frac{Y_0}{Y_1} \dd\left(\frac{X_0}{X_1}\right)\right)\,.
	$$
	Note that $\tilde C$ is a disjoint union of $N$ compact Riemann surfaces and $x \colon \tilde C \rightarrow \mathbb{P}^1$ is a finite-degree branched covering. The spectral curve is connected if and only if $\mathfrak{C}$ is irreducible.
\end{definition}

This construction can be reversed.

\begin{proposition}\label{pr:everycurve}
	 Every compact spectral curve in which $\omega_{0,1}$ separates fibres (cf. \cref{th:unramified}) can be obtained (up to the choice of a fundamental bidifferential) from a unique reduced affine curve as in \cref{de:irredpc}.
	\end{proposition}
\begin{proof} 
	Given a compact spectral curve $ \mc{S} = ( \Sigma, x, \omega_{0,1}, \omega_{0,2})$, set $y = \frac{\omega_{0,1}}{\dd x}$.  The image  $ (x, y) (\Sigma ) \subseteq \P^1 \times \P^1 $ is a closed curve, and therefore it is cut out by a reduced bihomogeneous polynomial $F$, unique up to rescaling. As $ x $ is not allowed to be constant on any component of $ \Sigma$, $F$ cannot have factors in $ \C [X_0, X_1]$. In particular it has no factor $ X_1 $, and therefore comes from the homogenisation of a polynomial in $x$. 
	
	As $ \omega_{0,1} $ is a meromorphic $1$-form, not identically zero on any component, it is also not identically $\infty$ on any component, and again as $x$ is not allowed to be constant on a component, $\dd x $ is not identically zero, so $ y = \frac{\omega_{0,1}}{\dd x}$ is a non-zero meromorphic function on any component. By the argument already used for $x$, this shows that $F$ comes from homogenisation of a polynomial in $y$.
	
	Therefore, the zero locus of $F$ is a reduced curve $C \subset \mathbb P^1 \times \mathbb P^1$ obtained from the procedure of \cref{de:irredpc}. By the universal property of its normalisation $ \eta \colon \tilde{C} \rightarrow C$, the map $ (x, y) \colon \Sigma \rightarrow C$ factors through $ \eta$. The resulting map $\Sigma \to \tilde{C}$ is now a surjective map between smooth compact curves. It is generically injective, as $ \omega_{0,1}$ separates fibres. Therefore, it must be injective, and hence an isomorphism.
\end{proof}

\begin{example}[The $(r,s)$-spectral curve]
\label{s:rscurve}
We illustrate the construction for the $(r,s)$-curve $\mathfrak{C}^{(r,s)} = \{x^{r - s}y^r - 1 = 0\} \subset \mathbb{C}^2$, where $r \geq 2$ and $s \in [r-1]$ are coprime. Its bihomogenisation is
	\begin{equation}
		C^{(r,s)} = \{X_0^{r-s} Y_0^r - X_1^{r-s} Y_1^r = 0 \} \subset \mathbb{P}^1 \times \mathbb{P}^1 \,.
	\end{equation}
If $s=r-1$,  the curve is smooth and has genus $0$. If $s < r-1$, the curve has arithmetic genus $(r - 1)(r - s - 1)$ but is singular at the two points $p_{0\infty}$ and $p_{\infty 0}$. These singularities are unibranched as $r$ and $s$ are coprime: each of them is replaced by a single point in the normalisation $\tilde{C}^{(r,s)}$. A local affine equation centered at any of these singularities is $f(u,v) = u^{r - s} - v^{r}$, so the local $\mathsf{j}$ is $(r - 1)(r - s - 1)$. This leads to a total $\delta = (r - 1)(r - s - 1) = \mathsf{p}_{{\rm a}}(C^{(r,s)})$ and geometric genus is $\mathsf{g}(C^{(r,s)}) = 0$. In other words $\tilde{C}^{(r,s)} \simeq \mathbb{P}^1$. We can write the normalisation map explicitly by parametrising the equation
	\begin{equation}
		\begin{array}{llcll} \eta & \colon & \tilde{C}^{(r,s)} & \longrightarrow & C^{(r,s)}  \\ & &  [w_0: w_1] & \longmapsto &  ([w_0^r: w_1^{r}], [w_1^{r-s}:w_0^{r-s}]) \end{array}\,.
	\end{equation}
	Then
	\begin{equation}
	\begin{array}{llcll} x = \pi_X \circ \eta &  \colon & \tilde{C}^{(r,s)} & \longrightarrow & \mathbb{P}^1 \\
	& & [w_0:w_1] & \longmapsto & [w_0^r:w_1^{r}] \end{array}
	\end{equation}
	is a branched covering of degree $r$, with ramification points at $\mathbf{0}$ and $\boldsymbol{\infty}$, both of order $r$. The data of the spectral curve is completed by giving the $1$-form 
$$
\omega_{0,1} = \eta^* \alpha = \frac{w_1^{r-s}}{w_0^{r-s}} \dd \left ( \frac{w_0^r}{w_1^r} \right) \,,
$$
and the unique fundamental bidifferential existing on $\tilde{C}^{(r,s)} = \mathbb{P}^1$.
$$
\omega_{0,2}(p,p') = \frac{\dd(w_0/w_1) \dd (w_0'/w_1')}{(w_0/w_1 - w_0'/w_1')^2},\qquad p = [w_0:w_1],\quad p' = [w_0':w_1']\,.
$$
In the global coordinate $w = \frac{w_0}{w_1}$, the curve can be parametrised as
\begin{equation}
\label{paramrs}
x(w) = w^r, \qquad y(w) = w^{s - r},\qquad \omega_{0,2}(w,w') = \frac{\dd w\, \dd w'}{(w - w')^2}\,.
\end{equation}
At the ramification point $\mathbf{0}$ a standard coordinate is $\zeta = \frac{w_0}{w_1} = w$. As $\omega_{0,1} = r\zeta^{s - 1} \dd \zeta$ the local parameters are
$$
r_{\mathbf{0}} = r,\qquad s_{\mathbf{0}} = \bar s_{\mathbf{0}} = s, \qquad \tau_{\mathbf{0}} = r\,.
$$
At the ramification point $\boldsymbol{\infty}$, we rather have $\zeta = \frac{w_1}{w_0} = \frac{1}{w}$ as standard coordinate. Then $x = \zeta^{-r}$ and $\omega_{0,1} = \zeta^{r - s} \dd(\zeta^{-r}) = -r \zeta^{-s - 1} \dd \zeta$. From there we read the local parameters:
$$
r_{\boldsymbol{\infty}} = r,\qquad s_{\boldsymbol{\infty}} = \bar s_{\boldsymbol{\infty}} = -s,\qquad\tau_{\boldsymbol{\infty}} = -r.
$$
Comparing with \cref{de:localadm}, we see that the spectral curve $\tilde{C}^{(r,s)}$ is locally admissible if and only if $r = \pm 1\,\,{\rm mod}\,\,s$. Due to \cref{th:ale}, the ramification point at $\boldsymbol{\infty}$ does not contribute to the topological recursion.

This spectral curve in parametrised form \eqref{paramrs} was first considered in \cite{BBCCN18}, where it was proved that topological recursion is well-defined if and only if $r = \pm 1\,\,{\rm mod}\,\,s$. This explains the origin of this condition in \cref{de:localadm}.

\end{example}

\subsubsection{Local admissibility and globalisation}
\label{S423}

The next question we address is: when are the compact spectral curves obtained in \cref{de:irredpc} locally admissible, and when can topological recursion on them be globalised? For this purpose, we need to analyse in detail the local behaviour of $\omega_{0,1}$.

\begin{definition}\label{de:linesinf}
Let $L_Y = \{\boldsymbol{\infty}\} \times \mathbb{P}^1$ and $L_X = \mathbb{P}^1 \times \{\boldsymbol{\infty}\}$ be the two lines at infinity in $\mathbb{P}^1 \times \mathbb{P}^1$. We introduce the finite sets
$$
E = (L_X \cup L_Y) \cap C,\qquad E' = \big(L_X \cup \{p_{\infty 0}\}\big)\cap C\,.
$$
\end{definition}

The set $E$ contains the points that are added when we go from the affine curve $\mathfrak{C}$ to the projective curve $C$. That is: $C \setminus E = \mathfrak{C}$. The set $E' \subseteq E$ contains only those additional points which are either poles of $y$, or poles of $x$ that are zeros of $y$. The set $E'$ is illustrated in \cref{fig:Eprime}. It turns out that, if we want $\tilde{C}$ to give a locally admissible spectral curve and be suitable for globalisation, the singularities of $C$ are restricted to appear in $E'$, as we show in the next lemma. In particular, $\mathfrak{C}$ must be smooth.

\begin{figure}[!ht]
\begin{center}
\includegraphics[width=0.3\textwidth]{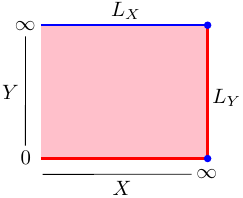}
\caption{The locus $E'$ is shown in blue. According to \cref{le:singularities}, the possible singularities of $C$ can only be in $E'$.}
\label{fig:Eprime}
\end{center}
\end{figure}

\begin{lemma}\label{le:singularities}
	Let $\mathcal{S}=(\tilde C,  x, \omega_{0,1}, \omega_{0,2})$ be the compact spectral curve associated to a reduced affine curve $\mathfrak{C}$ as in \cref{de:irredpc}. Suppose that $p \in C \setminus E'$ is a singular point.
	\begin{itemize}
		\item If $p$ is unibranched, then the spectral curve is not locally admissible.
		\item If $p$ is multibranched, then either the spectral curve is not locally admissible or the conditions of \cref{th:rewriting} for globalisation at $q = \pi_X(p) \in \mathbb{P}^1$ are not satisfied.
	\end{itemize}
\end{lemma}

\begin{proof}
	First, let $p$ be a singular point in $\mathfrak{C} = C \setminus E$. We write $(\pi_X(p),\pi_Y(p)) = (q,b) \in \mathbb{C}^2$.
	
	Suppose that $p$ is unibranched. This means that $\eta^{-1}(p)$ contains only one point $\tilde p$ in the normalisation $\tilde C$. Moreover, by the definition of normalisation, in a standard coordinate near $\tilde{p}$ we have $x = q + \zeta^r$  for some $r \geq 2$ and $y = b + c \zeta^{k} + o(\zeta^{k})$ for some $k \geq 2$ (otherwise $p$ would not be singular) with $c \neq 0$. This implies:
	$$
	\omega_{0,1} = y \dd x = r \zeta^{r-1} \big( b+ c \zeta^k + o(\zeta^k)\big)\dd \zeta\quad {\rm near}\,\,\tilde{p}\,.
	$$
	Regardless of whether $b=0$ or $b \neq 0$, the local parameter $s$ at $\tilde p$ must satisfy $s > r+1$, since $k \geq 2$. Therefore the condition (lA3) in \cref{de:localadm} is violated and the spectral curve is not locally admissible at $\tilde p$.

	Now suppose that $p$ is multibranched. Then there are at least two distinct points $\tilde p_1, \tilde p_2 \in \eta^{-1}(p)$ in the normalisation $\tilde C$. In a standard coordinate near $\tilde{p}_i$ we have $x = q + \zeta^{r_i}$ for some $r_i \geq 1$ and $y = b + c_i \zeta^{k_i} + o(\zeta^{k_i})$ for some $k_i \geq 1$ with $c_i \neq 0$. This implies
	$$
	\omega_{0,1} = r_i \zeta^{r_i-1} \big( b + c_i \zeta^{k_i} + o(\zeta^{k_i})\big)\dd \zeta \quad {\rm near} \,\,\tilde p_i\,.
	$$
	The important fact is that $b$ is the same for both $\tilde p_1$ and $\tilde p_2$, as they project to the same point $p \in C$. Thus, if $b \neq 0$, we have two points $\tilde p_1$ and $\tilde p_2$ in the fibre $x^{-1}(q)$ with same aspect ratio $\nu_1=\frac{\bar s_1}{r_1} = 1 = \frac{\bar s_2}{r_2}=\nu_2$ and $\frac{\tau_1}{r_1} = b = \frac{\tau_2}{r_2}$, which contradicts the sufficient conditions given in \cref{th:rewriting} for vertical globalisation above $q=\pi_X(p)$.  If $b = 0$, either the curve is not locally admissible (if $k_1 \geq 2$ or $k_2 \geq 2$, like in the unibranched case), or it does not satisfy the conditions for globalisation of \cref{th:rewriting} (if $k_1 = k_2 = 1$, then $\bar s_1 = r_1+1$ and $\bar s_2 =r_2+1$, leading to $\nu_1>1$ and $\nu_2>1$).
		
 Alternatively, let $p$ be a singular point in $E \setminus E'$. This means that $\pi_X(p) = \boldsymbol{\infty}$ and $\pi_Y(p) = b \in \mathbb{C}^*$.

Suppose that $p$ is unibranched. $\eta^{-1}(p)$ contains only one point $\tilde p$. In a standard coordinate near $\tilde{p}$ we have $x = \zeta^{-r}$  for some $r \geq 2$ and $y = b + c \zeta^k + o(\zeta^k)$ for some $k \geq 2$ with $c \neq 0$. This implies
$$
\omega_{0,1}= -r \zeta^{-r-1} \big( b + c \zeta^k + o(\zeta^k)\big)\dd \zeta \quad {\rm near}\,\,\tilde p\,.
$$
Since $b \neq 0$, the local parameters at $\tilde p$ are $\bar s = -r$, but $s > -r+1$, since $k \geq 2$. Therefore the spectral curve is not locally admissible at $\tilde p$: it violates condition (lA3) in \cref{de:localadm}.

	Now suppose that $p$ is multibranched. Then there are at least two distinct points $\tilde p_1, \tilde p_2 \in \eta^{-1}(p)$. In a standard coordinate near $\tilde{p}_i$ we have $x = \zeta^{-r_i}$ for some $r_i \geq 1$ and $y = b + c_i \zeta^{k_i} + o(\zeta^{k_i})$ for some $k_i \geq 1$ with $c_{i} \neq 0$. This implies
	$$
	\omega_{0,1} = -r_i \zeta^{-r_i-1} \big( b + c_i\zeta^{k_i} + o(\zeta^{k_i})\big) \dd \zeta \quad {\rm near}\,\,\tilde p_i\,.
	$$
	The constant $b$ is the same for both $\tilde p_1$ and $\tilde p_2$. Since $b \neq 0$, then we have two points $\tilde p_1$ and $\tilde p_2$ in the fibre $x^{-1}(q)$ with same aspect ratio $\nu_1=\frac{\bar s_1}{r_1} = -1 = \frac{\bar s_2}{r_2}=\nu_2$ and $\frac{\tau_1}{r_1} = b = \frac{\tau_2}{r_2}$, which contradicts the conditions of \cref{th:rewriting} for vertical globalisation above $\pi_X(p) = \boldsymbol{\infty}$.  
\end{proof}

We summarize the local parameters and whether the curve is locally admissible at points $p\in C$ that are either smooth ramification points or unibranched singularities in \cref{Fig:sum}.

\begin{table}[!ht]
\begin{center}
\begin{tabular}{|c||c|c||c|c||c||c|}
\hline
$p \in C$ & $x(p)$ & $y(p)$ & $\bar s$ & $s$ & $\nu$ & {\rm loc. adm.} \\
\hline
\multirow[c]{6}{2cm}{smooth ramified} & $\in \mathbb{C}$ & $0$ & $r + 1$ & $r + 1$ & $1+\frac{1}{r}$ & yes \\
\cline{2-7}
& $\in \mathbb{C}$  & $ \in \mathbb{C}^*$  & $r$ & $r + 1$ & $1$ & yes \\
\cline{2-7}
& $\in \mathbb{C}$  & $\infty$ &   $r - 1$ & $r - 1$ & $1-\frac{1}{r}$ & yes \\
\cline{2-7}
& $\infty$ & $0$ & $-r + 1$ & $-r + 1$ & $-1 + \frac{1}{r}$ & yes \\
\cline{2-7}
& $\infty$ & $ \in \mathbb{C}^*$  & $-r$ & $-r + 1$ & $-1$ & yes \\
\cline{2-7}
& $\infty$ & $\infty$ &   $-r -1$ & $-r - 1$ & $-1-\frac{1}{r}$ & yes \\
\hline\hline
\multirow[c]{5}{2cm}{unibranched singular} & $\in \mathbb{C}$  & $0$ & $> r + 1$ & $\textcolor{red}{>r + 1}$ & $> 1$ & no \\
\cline{2-7}
& $\in \mathbb{C}$  & $ \in \mathbb{C}^*$ & $r$ & $\textcolor{red}{> r + 1}$ & $1$ & no  \\
\cline{2-7}
& $\in \mathbb{C}$  & $\infty$ & $r - k$ & & $< 1$ & maybe \\ 
\cline{2-7}
& $\infty$ & $0$  & $-r+k$ &  & $>-1$ & maybe  \\
\cline{2-7}
& $\infty$ & $\in \mathbb{C}^*$  & $-r$ & $\textcolor{red}{> -r + 1}$ & $-1$ & no  \\
\cline{2-7}
& $\infty$ & $\infty$ & $-r - k$ & & $<-1$ & maybe \\
\hline
\end{tabular}
\vspace{1em}
\end{center}
\caption{\label{Fig:sum} Local parameters for smooth ramification points and unibranched singularities of projectivised affine curves. In red we emphasise the reason for not being locally admissible. In the table, $k$ is an integer $\geq 2$.}
\end{table}

\begin{remark}
We \emph{did not prove} that topological recursion can be globalised above $q \in \mathbb{P}^1$ \emph{if and only if} the conditions of \cref{th:rewriting} hold; the conditions of \cref{th:rewriting} are only sufficient conditions. Thus we cannot conclude from the previous lemma that topological recursion cannot be globalised if the affine curve $\mathfrak{C}$ contains multibranched singularities. However, in this case we believe that the implication holds in both directions, and that vertical globalisation does not work if the affine curve has multibranched singularities. Henceforth we will require that $\mathfrak{C}$ is smooth.
\end{remark}

\subsubsection{Global admissibility}
\label{S424}
\Cref{le:singularities} leads us to propose the following notion of`global admissibility.

\begin{definition}[Global admissibility]\label{de:globaladm}
	Let $\mathcal{S}=(\tilde C,  x, \omega_{0,1}, \omega_{0,2})$ be the compact spectral curve associated to a reduced affine curve $\mathfrak{C}$ as in \cref{de:irredpc}. We say that the spectral curve is \emph{globally admissible} if, for any singular point $p$ of $C$:
\begin{enumerate}[({gA}1)]
		\item  $p \in E'$.
		\item  $\mathcal{S}$ is locally admissible at all points in $\eta^{-1}(p)$.
		\item Every pair of points in $\eta^{-1}(p)$ is non-resonant and satisfies ``(C-i) or (C-ii)'' from \cref{th:rewriting}.
		\end{enumerate}
\end{definition}

Global admissibility is all about the singular points of the projective curve $C$. This definition is tailored for the following theorem.

\begin{theorem}[Global admissibility implies local admissibility and globalisation]\label{th:ga}
	Let $\mathcal{S} = (\tilde C,  x, \omega_{0,1}, \omega_{0,2})$ be the compact spectral curve associated to a reduced affine curve $\mathfrak{C}$ as in \cref{de:irredpc}. If $\mathcal{S}$ is globally admissible, then $\mathcal{S}$ is locally admissible and topological recursion can be fully globalised on $\mathcal{S}$.
\end{theorem}

\begin{proof}
	We start with local admissibility. First, since the affine curve $\mathfrak{C}$ is smooth, by \cref{le:redacla}, we know that the spectral curve is locally admissible at all ramification points in  $\eta^{-1}(\mathfrak{C})$. By global admissibility, we also know that it is locally admissible at all points in $\eta^{-1}(p)$ for singular points $p \in \mathsf{Sing} \subset C$. It remains to check local admissibility at the smooth points in $E$ which are ramified for $x$.  For such points $p$, $\eta^{-1}(p)$ contains a single point $\tilde p$. Let $r \geq 2$ be the order of ramification.
	
\noindent (i) \textit{Case $\pi_X(p) = \boldsymbol{\infty}$ and $\pi_Y(p) = b \in \mathbb{C}$}. In a standard coordinate near $\tilde{p}$, we have $x = \zeta^{-r}$ and $y = b + c \zeta + \mathcal{O}(\zeta^2)$ for some $c \neq 0$. This implies
\[\
\omega_{0,1} = - r \zeta^{-r-1}\big( b + c \zeta + \mathcal{O}(\zeta^2)\big)\dd \zeta\quad {\rm near}\,\,\tilde p\,.
\]
If $b\neq 0$, the local parameters at $\tilde p$ are $\bar s = -r$ and $s = -r+1$, since $c \neq 0$. If $b=0$, the local parameters are $\bar s = s = -r+1$. In both cases the curve is locally admissible at $\tilde p$.

\noindent (ii) \textit{Case $\pi_X(p) = q \in \mathbb{C}$ and $\pi_Y(p) = \boldsymbol{\infty}$}. In a standard coordinate near $\tilde{p}$, we have $x = q + \zeta^r$ and $y \sim \frac{c}{\zeta}$ for some $c \neq 0$. This implies
\[
\omega_{0,1} \sim c r \zeta^{r-2}\dd \zeta\quad {\rm near}\,\,\tilde p\,.
\]
The local parameters are $\bar s = s = r-1$, showing locally admissibility at $p$.
	
\noindent (iii) \textit{Case $\pi_X(p) = \pi_Y(p) = \boldsymbol{\infty}$}. In a standard coordinate near $\tilde{p}$, we have $x = \zeta^{-r}$ and $y \sim \frac{c}{\zeta}$ for some $c \neq 0$. This implies
\[
\omega_{0,1} \sim - c r \zeta^{-r - 1} \dd \zeta\quad {\rm near}\,\,\tilde p\,.
\]
The local parameters are $\bar s = s = -r-1$, and the curve is locally admissible at $\tilde p$.

This concludes the proof of local admissibility, and we move on to globalisation. Let $q \in \mathbb{P}^1$. We want to show that topological recursion can be globalised above $q$. If $q \neq \boldsymbol{\infty}$ and $\boldsymbol{\infty} \notin \pi_Y(\pi_X^{-1}(q))$, then  all the preimages in $\pi_X^{-1}(q)$ are in fact on the affine curve $\mathfrak{C}$ and by \cref{le:redacla}, we know that topological recursion can be globalised above $q$.

	Suppose next that $q \neq \boldsymbol{\infty}$ but that there is one special point $p_0 \in \pi_X^{-1}(q)$ such that $\pi_Y(p_0) = \boldsymbol{\infty}$. All other points $p \in \pi_X^{-1}(q)$ are on $\mathfrak{C}$, and thus are smooth and have a unique preimage $\tilde p$ in the normalisation. From the proof of \cref{le:redacla} we know that all such $\tilde p$ except potentially one have aspect ratio $\nu = \frac{\bar s}{r} = 1$ and $\omega_{0,1}$ separates fibres;  there can be one point $\tilde p$ with $\nu  > 1 $ (this happens when there is a zero of $\pi_Y$ in $\pi_X^{-1}(q)$). 
	If $p_0$ is a smooth point in $C$, then it has a unique preimage $\tilde p_0$ in $\tilde{C}$ and we have learned in (ii)  that $\bar s = s = r-1$ so $\nu < 1$. If $p_0$ is singular, all preimages in $\eta^{-1}(p_0)$ have $\bar s < r - 1$, so we still have $\nu < 1$, and global admissibility tells us that we can globalise over $\eta^{-1}(p_0)$.
	Putting all this together, if we pick a pair of points $\{\tilde p_i, \tilde p_j\} \subseteq x^{-1}(q)$, then:
	\begin{itemize}
	\item either $\nu_i \leq 1$ and $\nu_j \geq 1$ (or vice-versa) and if $\nu_i = \nu_j = 1$ then $\omega_{0,1}$ separates fibres;
	\item or $\nu_i < 1$ and $\nu_j <1$ but both points are in $\eta^{-1}(p_0)$ where $p_0$ is a singular point, and by global admissibility the points satisfy the conditions of \cref{th:rewriting}.
	\end{itemize}
Therefore topological recursion can be globalised above $q \in \mathbb{C}$.

	It remains to consider the special point $q = \boldsymbol{\infty}$. If $p \in \pi_X^{-1}(q)$ is such that $\pi_Y(p) \in \mathbb{C}^*$, then $p$ must be smooth (due to condition (gA1)) and have a unique preimage $\tilde p$ in the normalisation. By (i) it has $\bar s = -r$ and hence $\nu = -1$. Moreover, such points $\tilde p$ map to different points $p$ in $C$, so $\omega_{0,1}$ separates fibres for all such points. If $p_{\infty 0} \in \pi_X^{-1}(q)$, then there are two situations. If it is smooth, then it has a unique preimage $\tilde p_{\infty 0}$, and by (i) the local parameter is $\bar{s} = -r+1$ and hence $\nu = -1 + \frac{1}{r} > -1$. If it is singular, we find $\bar s > -r + 1$ and thus $\nu > -1$ for all preimages in $\eta^{-1}(p_{\infty 0})$, and the assumption of global admissibility guarantees that topological recursion can be globalised over these preimages. Finally, if $p_{\infty\infty} \in \pi_X^{-1}(q)$, then there are two situations. If it is smooth, by (iii) the corresponding point in $\tilde{C}$ has $\bar s = -r -1$ and hence $\nu = -1 - \frac{1}{r} < -1$. If it is singular, we find $\bar s < -r-1$ and thus $\nu< -1$ at all preimages in $\eta^{-1}(p_{\infty\infty})$, and by global admissibility topological recursion can be globalised over these preimages. Putting all this together, we conclude that if we pick a pair of points $\{\tilde p_i, \tilde p_j\} \subseteq x^{-1}(\boldsymbol{\infty})$, then
\begin{itemize}
\item either $\nu_i \leq -1$ and $\nu_j \geq -1$ (or vice-versa) and if $\nu_i=\nu_j=-1$ then $\omega_{0,1}$ separates fibres;
\item or $\nu_i < -1$ and $\nu_j < -1$ with both points in $\eta^{-1}(p_{\infty\infty})$; by global admissibility the pair satisfy the conditions of \cref{th:rewriting};
\item or $\nu_i >-1$ and $\nu_j > -1$ with both points in $\eta^{-1}(p_{\infty 0})$; by global admissibility the pair satisfy the conditions of \cref{th:rewriting}.
\end{itemize}
Therefore topological recursion can be globalised as well above $q= \boldsymbol{\infty}$.
\end{proof}

\subsection{Global admissibility and Newton polygons}

In this section we study how global admissibility of spectral curves can be understood from the point of view of the Newton polygon of the original affine curve. Recall from \cref{de:globaladm} that for a compact spectral curve associated to an affine curve to be globally admissible, the affine curve must be smooth, and, in particular, irreducible. Thus from now we assume that we start with an irreducible affine curve.

\subsubsection{Generalities on Newton polygons}
\label{sec:conrern}
We first review classical facts about the resolution of singularities seen from the point of view of the Newton polygon, see for instance \cite{K78,BP00}. Let $\mathfrak{C} = \{P(x,y) = 0 \} \subset \mathbb{C}^2$ be an irreducible affine curve of positive degree $d_x$ in $x$ and $d_y$ in $y$.  Write
\begin{equation}
P(x,y) = \sum_{(i,j) \in \mathbb{Z}_{\geq 0}^2} c_{ij} x^i y^j\,,
\end{equation}
and denote by ${\rm supp}(P) = \{(i,j) \in  \mathbb{Z}_{\geq 0}^2~|~c_{ij} \neq 0 \}$ the support of $P$.

\begin{definition}[Newton polygon and diagram]\label{de:Np}
The \emph{Newton polygon} $\Delta(P)$ is the convex hull of ${\rm supp}(P)$ in $\mathbb{R}^2_{\geq 0}$. Let $\Gamma^+(P) \subseteq \mathbb{R}_{\geq 0}^2$ be the convex hull of the union of the quadrants $(i,j) + \mathbb{R}_{\geq 0}^2$ over $(i,j) \in {\rm supp}(P)$. The \emph{Newton diagram} $\Gamma(P)$ associated to $P$ is the closure of the complement of $\Gamma^+(P)$ in $\mathbb{R}_{\geq 0}^2$. We denote by $\hat{\Gamma}(P) = \Gamma(P) \cap \mathbb{R}_{> 0}^2$ the set of points in $\Gamma(P)$ that are not on the $x$- or $y$-axes.
\end{definition}

It will be useful for us to distinguish between edges and maximal edges of $\Delta(P)$, vertices and extremal vertices.

\begin{definition}[Discrete geometry conventions]\label{de:edges} 
An \emph{integral point} is a point in $\mathbb{Z}^2$.  If $A \subseteq \mathbb{R}^2$ is bounded, we denote $|A|_{\mathbb{Z}} = | A \cap \mathbb{Z}^2|$ its number of integral points.

An \emph{edge} of a Newton polygon $\Delta(P)$ is a closed segment in the boundary of $\Delta(P)$ whose only integral points are located at the endpoints. A \emph{maximal edge} is a closed segment obtained by a maximal union of edges (cf. \cref{fig:edgem}). Maximal edges are edges if and only if they do not contain any integral interior points. An \emph{extremal vertex} is the endpoint of a maximal edge (in particular, $\Delta(P)$ is the convex hull of its extremal vertices). A \emph{vertex} is any integral point in the boundary of $\Delta(P)$ (it is not required to be in the support of $P$). 
\end{definition}

The Newton diagram of $P$ characterises a potential singularity of $\mathfrak{C}$ at the origin $(0,0)$. More precisely, the $\delta$-invariant of the origin satisfies $\delta_{(0,0)} \geq |\hat{\Gamma}(P)|_{\mathbb{Z}}$,  see e.g. \cite[Section 3]{BP00}.

\begin{figure}[!ht]
\begin{center}
\includegraphics[width=0.7\textwidth]{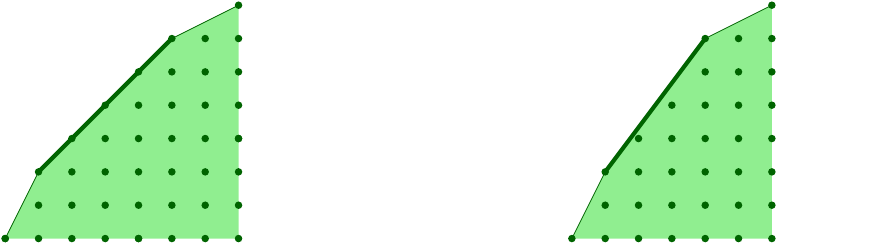}
\caption{\label{fig:edgem} Left panel: the maximal edge with slope $1$ consists of 4 edges. Right panel: the maximal edge of slope $\frac{4}{3}$ consists of a single edge.}
\end{center} 
\end{figure}

Let us return to the projectivised curve
\begin{equation}
C = \bigg\{F(X_0,X_1,Y_0,Y_1) = X_1^{d_x} Y_1^{d_y} P\left( \frac{X_0}{X_1}, \frac{Y_0}{Y_1} \right) = 0 \bigg\} \subset \mathbb{P}^1 \times \mathbb{P}^1\,.
\end{equation}
Let $\square(d_x,d_y)$ be the rectangle in $\mathbb{R}^2$ with vertices $(0,0)$, $(d_x,0)$, $(0, d_y)$ and $(d_x, d_y)$. Clearly, $\Delta(P) \subseteq \square(d_x,d_y)$. Since $P$ is irreducible, we know that $\Delta(P)$ intersects each of the four maximal edges of $\square(d_x,d_y)$. The rectangle $\square(d_x,d_y)$ naturally splits into five regions: the interior $\mathring{\Delta}(P)$ of the Newton polygon and the four corners (cf. \cref{Newpolyfig}):
\begin{itemize}
\item $\Gamma_{00}(P)$ is the Newton diagram of $F(x,1,y,1)$;
\item $\Gamma_{0\infty}(P)$ is obtained from the Newton diagram of $F(x,1,1,y)$ by a horizontal reflection;
\item $\Gamma_{\infty 0}(P)$ is obtained from the Newton diagram of $F(1,x,y,1)$ by a vertical reflection;
\item $\Gamma_{\infty\infty}(P)$ is obtained from the Newton diagram of $F(1,x,1,y)$ by a rotation.
\end{itemize}
By the discussion above, the $\delta$-invariants of the singularities of $C$ at these four special points must satisfy
 \begin{equation}
\label{dpj}
\forall \epsilon,\epsilon' \in \{0,\infty\}\qquad \delta_{p_{\epsilon\epsilon'}} \geq |\hat{\Gamma}_{\epsilon\epsilon'}(P)|_{\mathbb{Z}}\,,
\end{equation}
where by $\hat{\Gamma}_{\epsilon\epsilon'}(P)$ we mean the points in $\Gamma_{\epsilon\epsilon'}(P)$ which are not on the boundary of $\square(d_x,d_y)$.

\begin{figure}[!ht]
\begin{center}
\includegraphics[width=0.34\textwidth]{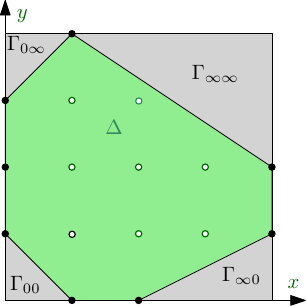}
 \end{center}
\caption{\label{Newpolyfig} Example of a Newton polygon and its four corners. We have $|\hat{\Gamma}_{\infty\infty}|_{\mathbb Z} = 1$ and  $|\hat{\Gamma}_{00}|_{\mathbb{Z}} = |\hat{\Gamma}_{0\infty}|_{\mathbb{Z}} = |\hat{\Gamma}_{\infty0}|_{\mathbb{Z}} = 0$.}
 \end{figure} 
 
These five regions do not overlap in the interior $\mathring{\square}(d_x,d_y)$ of the rectangle:
$$
\mathring{\square}(d_x,d_y) = \mathring{\Delta}(P) \sqcup \left( \bigsqcup_{\epsilon,\epsilon' \in \{0,\infty\}} \hat{\Gamma}_{\epsilon\epsilon'}(P) \right)\,.
$$
This implies for the number of integral points
\begin{equation}\label{eq:ggcd}
(d_x-1)(d_y-1) = |\mathring{\Delta}(P)|_{\mathbb{Z}} + \sum_{\epsilon,\epsilon' \in \{0,\infty\} } |\hat{\Gamma}_{\epsilon\epsilon'}(P)|_{\mathbb{Z}}\,.
\end{equation}
Since the geometric genus of $C$ is $\mathsf{g}(C) = (d_x-1) (d_y-1) - \sum_{p \in \mathsf{Sing}} \delta_p$ where $\mathsf{Sing}$ is the singular locus of $C$, \eqref{dpj} implies that $\mathsf{g}(C) \leq |\mathring{\Delta}(P)|_{\mathbb Z}$ (this is known as Baker's formula, see e.g. \cite[Theorem 4.2]{BP00}).

\subsubsection{Nondegenerate curves and properties of Puiseux series}

For generic curves, the above bounds are equalities, and the geometric genus and the type of singularities of the generic curve  can be directly read off its Newton polygon. More precisely, this is true for curves that are nondegenerate, which is defined as follows.

Let $\Delta(P) \subset \mathbb{R}_{\geq 0}^2$ be the Newton polygon of an irreducible affine curve $\mathfrak{C} = \{P(x,y)= \sum_{i,j} c_{ij}\, x^iy^j = 0 \}$ as in the previous section. Denoting by $\tau$ a maximal edge of $\Delta(P)$ or $\Delta(P)$ itself, we write
\begin{equation}\label{Ptau}
P_\tau(x,y) = \sum_{(i,j) \in \tau \cap \mathbb{Z}_{\geq 0}^2} c_{ij} \,x^i y^j\,,
\end{equation}

\begin{definition}[Nondegenerate curves]\label{de:nd}
$\mathfrak{C} = \{P(x,y) = 0\} \subset \mathbb{C}^2$ is \emph{nondegenerate} if for any $\tau$ which is a maximal edge of $\Delta(P)$ or $\Delta(P)$ itself, the system $P_\tau(x,y) = \partial_x P_{\tau}(x,y) = \partial_y  P_{\tau}(x,y) = 0$
has no solution in $(\mathbb{C}^*)^2$.
\end{definition}

In the literature this is usually called ``nondegenerate with respect to its Newton polygon''. Some references, for instance  \cite[Definition 4.1]{BP00}, do not include a condition on $\tau = \Delta(P)$ in the definition, which then refers to a weaker notion. For a fixed Newton polygon, nondegenerate curves are generic. Nondegenerate curves are nice as many properties of a nondegenerate curve can be read off directly from its Newton polygon.

\begin{remark}
If $\mathfrak{C} = \{(x,y) \in \mathbb{C}^2 \,\,|\,\,P(x,y) = 0\} \subset \mathbb{C}^2$ is an irreducible and nondegenerate affine curve, then the polynomial $P$ has positive degrees in $x$ and in $y$, does not have factors in $\mathbb{C}[x]$ and is such that $\{y = 0\} \not\subseteq \mathfrak{C}$. Therefore, by \cref{de:irredpc} we obtain from it a connected and compact spectral curve.
\end{remark}

\begin{lemma}\label{le:ndnp}
Let $\mathfrak{C} = \{P(x,y) = 0\} \subset \mathbb{C}^2$ be an irreducible and nondegenerate affine curve. Let $C \subset \mathbb{P}^1 \times \mathbb{P}^1$ the projectivisation and $\mathsf{Sing}$ its singular locus. Then:
\begin{itemize} 
\item $\mathsf{g}(C) = |\mathring{\Delta}(P)|_{\mathbb{Z}}$;
\item $\mathsf{Sing}  \subseteq \{p_{00},p_{\infty 0},p_{0\infty},p_{\infty\infty}\}$;
\item $\delta_{p_{\epsilon\epsilon'}} = |\hat{\Gamma}_{\epsilon\epsilon'}(P)|_{\mathbb{Z}}$ for $\epsilon,\epsilon' \in \{0,\infty\}$.
\end{itemize}
\end{lemma}

\begin{proof}
First, take $\tau = \Delta(P)$. Then nondegeneracy implies that $\mathfrak{C} \cap (\mathbb{C}^*)^2$ is smooth. Now suppose that $\mathfrak{C}$ is singular at $(x,y) = (0,b)$ for some $b \in \mathbb{C}^*$. Then $\Delta(P)$ contains a vertical maximal edge $e$ issuing from the origin. If we write $P(x,y) = \sum_{i=0}^{d_x} c_i(y) x^i$, then we must have $c_0(b) = c_0'(b) = 0$. But $P_{e}(x,y) = c_0(y) = 0$. Thus, $P_e(0,b) = c_0(b) = 0$, $\partial_x P_e (0,b) = 0$ and $\partial_y P_{e}(0,b) =  c'_0(b) = 0$. This contradicts the nondegeneracy condition for this particular $e$. A similar argument applies to $(x,y) = (a,0)$ with $a \in \mathbb{C}^*$, the role of $e$ being now played by the horizontal maximal edge issuing from the origin. We conclude that $\mathfrak{C}$ can only have a singularity at the origin.

The bihomogenisation $F$ of $P$ has the property that, if $P(x,y) = F(x,1,y,1)$ is nondegenerate, then $F(x,1,1,y)$, $F(1,x,y,1)$ and $F(1,x,1,y)$ are also nondegenerate. Therefore, the argument of the previous paragraph goes through in the other affine patches, and we conclude that singularities can only be found among the points $ \{p_{00}, p_{0\infty}, p_{\infty0}, p_{\infty \infty} \}$.

Finally, from \cite[Proposition 3.17]{BP00} we know that $\delta_{p_{\epsilon\epsilon'}} = |\hat{\Gamma}_{\epsilon\epsilon'}(P)|_{\mathbb{Z}}$ for all $\epsilon,\epsilon' \in \{0,\infty\}$. Since there are no other singularities, from \eqref{eq:gg} and \eqref{eq:ggcd} it follows that $\mathsf{g}(C) = |\mathring{\Delta}(P)|_{\mathbb{Z}}$.
\end{proof}

Our goal is to extract from the Newton polygon information about the local parameters for compact spectral curves obtained of \cref{de:irredpc}, for which the affine curve is nondegenerate. To this end, we recall some basic facts about the Newton--Puiseux algorithm (see e.g. \cite{Tsing}). The algorithm determines the Puiseux series $y(x)$ satisfying $P(x,y(x)) = 0$. They take the form
\begin{equation}\label{k0r}
	y(x) = \sum_{m \in \mathbb{Z}_{\geq 0}} c_m x^{\frac{m}{r}}\,,\quad k = \min \{m\,\,|\,\,c_m \neq 0\} \in \mathbb{Z}_{\geq 0}
	\end{equation}
	for some minimal $r \in \mathbb{Z}_{> 0}$, where there is at least one $m$ not divisible by $r$ such that $c_l \neq 0$.  We will refer to $ (r,k) $ as the \textit{parameters of the Puiseux series}. In general, $k$ and $r$ may not be coprime. The ratio $\frac{k}{r}$ is often called the valuation, $-\frac{r}{k}$ coincides with a slope of some maximal edge $e$ in the corner $\Gamma_{00}(P)$, and $ c_{k} $ is a nonzero root of the \emph{slope polynomial} $Q_e(w) = w^{-l_e}P_e(1,w)$, where $l_e$ is the degree of the monomial of lowest degree in $P_e(1,w)$. There are $ (r-1) $ other Galois-conjugate Puiseux series, obtained by multiplying $x^{\frac{1}{r}} $ by a $r$-th root of unity. A similar discussion applies in the affine patches at the other points $p_{\epsilon\epsilon'}$ with $ \epsilon,\epsilon' \in \{0,\infty\}$, by considering maximal edges in the corners $\Gamma_{\epsilon\epsilon'}(P)$ and the corresponding slope polynomials. Although we always take $k \in \mathbb{Z}_{\geq 0}$ and $r \in \mathbb{Z}_{> 0}$ in the definition, the sign of the slope and its expression in terms of $k$ and $r$ depends on the corner.

\begin{remark}
	We shall soon see that the parameter of the Puiseux series named $ r $ coincides with the local parameters of the spectral curve named $r$ in \cref{de:localparameters}.
\end{remark}

\begin{remark}
Nondegeneracy is equivalent to the property that all slope polynomials have only simple roots \cite[Remark 3.16]{BP00}, and that the system $ P(x,y) = \partial_x P(x,y) = \partial_y P(x,y) = 0$ has no roots in $(\mathbb{C}^*)^2  $. Again, note  that our definition of nondegeneracy has the additional condition that $ P(x,y) = \partial_x P(x,y) = \partial_y P(x,y) = 0$ has no roots in $(\mathbb{C}^*)^2  $,
as compared to \cite{BP00}.
\end{remark}

We summarise some basic properties below.

\begin{lemma}\label{le:coprime}
Let $C$ be the projectivisation of a nondegenerate irreducible affine curve, $\mathcal{S}$ the spectral curve and let $\epsilon,\epsilon' \in \{0,\infty\}$.
	\begin{itemize}
	\item[(a)] Every branch above $p_{\epsilon\epsilon'}$ is associated to a maximal edge in $\Delta(P) \cap \Gamma_{\epsilon\epsilon'}(P)$ and a Puiseux series with parameters $(r,k)$ as in \eqref{k0r}. There are as many branches associated to a maximal edge $e$ as there are edges in $e$.
	\item[(b)] At a branch $p \in \eta^{-1}(p_{\epsilon\epsilon'})$, the local parameters $(r_p,\bar s_p)$  in the sense of \cref{de:localparameters} can be expressed in terms of the parameters $(r,k)$ of the corresponding Puiseux series. Namely, $r_p = r$ and
\begin{equation}
\label{allvaluesbarp}
\bar s_p = \left\{\begin{array}{lll} r+k && {\rm if} \,\,(\epsilon,\epsilon') = (0,0) \\ r - k && {\rm if} \,\,(\epsilon,\epsilon') = (0,\infty) \\ -r + k &&  {\rm if} \,\,(\epsilon,\epsilon') = (\infty,0) \\  -r - k && {\rm if} \,\,(\epsilon,\epsilon') = (\infty,\infty) \end{array}\right.
\end{equation}
	In all cases, $r_p$ and $\bar s_p$ are coprime and if $r_p \geq 2$, we have $\bar s_p = s_p$.
\item[(c)] Any pair of branches associated to the same maximal edge is non-resonant in the sense of \cref{de:nonres}.
\end{itemize}
\end{lemma}

\begin{proof}
	Consider a maximal edge $e$ in the corner $ \Gamma_{00}(P)$, of slope $ -\frac{n_e}{m_e}  $, where $m_e,n_e \in \mathbb{Z}_{> 0}$  are coprime. It is the union of $\ell_e$ edges. By the Newton--Puiseux algorithm, Puiseux series associated to this maximal edge must have as leading term $c_e x^{\frac{m_e}{n_e}}$  and $c_e$ is a nonzero root of $Q_e(w) = w^{-l_e}P_e(1,w)$. The latter is a polynomial in the variable $ w^{n_e}$ of degree $\ell_e$. Due to nondegeneracy, roots of $Q_e$ come in $\ell_e$ disjoint groups. Within a group, the roots differ by multiplication by a $n_e$-th root of unity. With the Newton--Puiseux algorithm we find at least $\ell_en_e$ distinct Puiseux series (associated to $e$) that are roots of $P(x,y)$, and for degree reasons there cannot be more.

By \cite[Remark 3.18]{BP00}, the sets of conjugate Puiseux series we have described are in bijection with the branches of the singularity at $p_{00}$. For each branch $\tilde{p}_e \in \eta^{-1}(p_{00})$, call $(r_e,k_e,c_e)$ the parameters of the corresponding the Puiseux series. By comparison we must have $r_e = n_e $ representing the number of conjugate Puiseux series within the corresponding group. As $\frac{m_e}{n_e} = \frac{k_e}{r_e}$, we deduce that $ k_e = m_e$. In particular this means that $k_e,r_e$ are coprime, and by the above description if $e'$ is an edge in the same maximal edge we have
\begin{equation}
\label{ccprime} c_{e'}^{r_e} \neq c_e^{r_e}\,.
\end{equation}

Clearly, $x$ has order $r_e$ at $\tilde{p}_e$. And in a standard coordinate near $p$ we have $x = x(p) + \zeta^{r_e}$ and $y \sim c_e \zeta^{\frac{k_e}{r_e}}$. This implies
$$
\omega_{0,1} \sim r_ec_e \zeta^{r_e + k_e - 1} \dd \zeta\qquad {\rm near}\,\,p
$$
Therefore $\bar s_e = r_e + k_e$ and it is coprime with $r_e$. This implies that $s_e = \bar s_e$ as soon as $r_e \geq 2$. Besides $\frac{\tau_e}{r_e} = c_e$ and the fact that $r_e$ and $\bar s_e$ are coprime together with \eqref{ccprime} implies that all pairs of points above the same branch of $p_{00}$ (corresponding to all edges in the same maximal edge) are non-resonant.

For branches corresponding to edges at other corners, the argument is similar. We only indicate how to compute the local parameters. For a branch $\tilde{p}_e \in \eta^{-1}(p_{0\infty})$ corresponding to an edge $e$, the Puiseux series gives the series expansion of $y^{-1}$ near $\tilde{p}_e$ in the variable $x$. Then, in a standard coordinate near $p$ we have $x = x(p) + \zeta^{r_e}$ and $y^{-1} \sim c_e \zeta^{\frac{k_e}{r_e}}$, therefore
$$
\omega_{0,1} \sim r_ec_e \zeta^{r_e - k_e - 1} \dd \zeta\quad {\rm near}\,\,p\,.
$$
and we read $\bar s_e = r_e - k_e$. For a branch $p \in \eta^{-1}(p_{\infty0})$, the Puiseux series expresses the series expansion of $y$ in the variable $x^{-1}$. Then, in a standard coordinate near $p$ we have $x = \zeta^{-r_e}$ and $y \sim c_e \zeta^{\frac{k_e}{r_e}}$, therefore
$$
\omega_{0,1} \sim -r_ec_e \zeta^{-r_e + k_e - 1} \dd \zeta \quad {\rm near}\,\,p\,.
$$
 For a branch $p \in \eta^{-1}(p_{\infty\infty})$, the Puiseux series expresses the series expansion of $y^{-1}$ in the variable $x^{-1}$. Then, in standard coordinate near $p$ we have $x = \zeta^{-r_e}$ and $y^{-1} \sim c_e \zeta^{\frac{k_e}{r_e}}$, therefore
 \[
	 \omega_{0,1} \sim -r_e c_e \zeta^{-r_e - k_e - 1} \dd \zeta \quad {\rm near}\,\,p\,. \qedhere
 \]
\end{proof}

\subsubsection{Global admissibility for nondegenerate curves}

We can now formulate a characterisation of global admissibility for compact spectral curves coming from nondegenerate curves.

\begin{proposition}[Global admissibility for nondegenerate spectral curves]\label{pr:ndga}
Let $\mathfrak{C} = \{P(x,y) = 0\} \subset \mathbb{C}^2$ be an irreducible and nondegenerate affine curve, and $\mathcal{S}=(\tilde C,  x, \omega_{0,1}, \omega_{0,2})$ be a compact spectral curve as in \cref{de:irredpc}. Then $\mathcal{S}$ is globally admissible if and only if the following conditions are satisfied:
\begin{enumerate}[({GA}1)]
\item $|\hat{\Gamma}_{00}(P)|_{\mathbb{Z}} = 0$, i.e. $\mathfrak{C}$ is smooth at the origin;
\item $\mathcal{S}$ is locally admissible at all points in $\eta^{-1}(p_{0\infty})$ and any pair of points in $\eta^{-1}(p_{0\infty})$ satisfies ``(C-i) or (C-ii)'' from \cref{th:rewriting};
\item $\mathcal{S}$ is locally admissible at all points in $\eta^{-1}(p_{\infty0})$ and any pair of points in $\eta^{-1}(p_{\infty 0})$ satisfies ``(C-i) or (C-ii)'' from \cref{th:rewriting}.
\end{enumerate}
\end{proposition}

\begin{proof}
	The direct implication is clear. Conversely, in view of \cref{th:ga} and assuming (GA1)-(GA2)-(GA3), we need to check that if $p_{\infty\infty}$ is a singular point of $C$, $\mathcal S $ is locally admissible at all points $ \eta^{-1}(p_{\infty\infty})$, and  every pair of points in $\eta^{-1}(p_{\infty \infty})$ is non-resonant and satisfies ``(C-i) or (C-ii)''. We also need to check that all pairs of points in $\eta^{-1}(p_{0\infty})$ (resp. in $\eta^{-1}(p_{\infty 0})$) are non-resonant (cf. \cref{de:nonres}).
	
	The absence of resonance is given by \cref{le:coprime}. Besides, if $p_{\infty\infty}$ is a singular point of $C$, take $\tilde{p} \in \eta^{-1}(p_{\infty\infty})$. If $p$ is unramified, the curve is locally admissible at $\tilde{p}$. Otherwise, $r_{\tilde{p}} \geq 2$ and in \cref{le:coprime} we have learned that $s_{\tilde{p}} = \bar s_{\tilde{p}} = -r_{\tilde{p}} - k_{\tilde{p}}$ in terms of the parameters of a corresponding Puiseux series. Since $k_{\tilde{p}} > 0$ by definition, this gives $\nu_p = - 1 - \frac{k_p}{r_p} < -1$ and thus the curve is locally admissible at $p$. So the property of global admissibility concerning $p_{0\infty}$ (resp. $p_{\infty 0}$) reduces to checking ``(C-i) or (C-ii)'' of \cref{th:rewriting}. 
	 \end{proof}

\subsubsection{Global admissibility in terms of slopes}
\label{Ssloperatio}

Consider a nondegenerate spectral curve. We now reformulate (GA2) and (GA3) in \cref{pr:ndga} in terms of slopes of the edges in
\begin{equation}
E_{0\infty} = \Gamma_{0\infty}(P) \cap \Delta(P) \qquad {\rm and}\qquad E_{\infty 0} = \Gamma_{\infty 0}(P)  \cap \Delta(P)\,.
\end{equation}
For this we make a choice of bijection between $\eta^{-1}(p_{0\infty})$ (set of branches) and the set of edges in $E_{0\infty}$, so that each branch is mapped to an edge contained in the maximal edge associated with this branch. We do the same for $\eta^{-1}(p_{\infty 0})$ and the set of edges in $E_{\infty 0}$. 

$E_{0\infty}$ is a point or a union of edges. In the latter situation, let $0 < \mu_1 \leq \cdots \leq \mu_{\ell} < \infty$ be the slopes of the edges, which appear in weakly increasing order from right to left in $E_{0\infty}$. When a slope appear several times in this list, it means that we have a maximal edge made of several edges. By our choice of bijection, to an edge $e$ we have associated a point $\tilde{p}_e \in \eta^{-1}(p_{0\infty})$. We have local parameters  $(r_e,\bar s_e,s_e)$ for the spectral curve at $\tilde{p}_e$, and local parameters $(k_e,r_e)$ of the group of Puiseux series associated with $\tilde{p}_e$. From \cref{le:coprime} we can read them from the slope of $e$. More precisely, we have $\bar s_e = r_e - k_e$ with $\bar s_e < r_e$ and $\bar s_e,r_e$ coprime, and the relation between the aspect ratio and the slope reads
\begin{equation}
\label{aspslo} \mu_e = \frac{r_e}{k_e} = \frac{r_e}{r_e - \bar{s}_e},\qquad \nu_e = \frac{\bar{s}_e}{r_e} = 1 - \frac{k_e}{r_e} = 1 - \frac{1}{\mu_e}\,.
\end{equation}
Likewise, for $E_{\infty 0}$ the slopes $0 < \mu_1 \leq \ldots \leq \mu_{\ell} < \infty$ appear in weakly increasing order from left to right. For each edge $e$ we have a point $\tilde p$ in $\eta^{-1}(p_{\infty 0})$ and the local parameters satisfy $\bar s_e = k_e - r_e$ with $\bar s_e > -r_e$ and $\bar s_e,r_e$ coprime, and
\begin{equation}
\label{aspslo2}
 \mu_e = \frac{r_e}{k_e} = \frac{r_e}{\bar s_e + r_e},\qquad \nu_e = - 1 + \frac{k_e}{r_e} = -1 + \frac{1}{\mu_e}\,.
\end{equation}
In both corners, the slope faithfully encodes the local parameters, we have $s_e = \infty$ whenever $r_e = 1$, and $\bar s_e = s_e$ whenever $r_e \geq 2$. For an edge in $E_{0\infty}$ we note that  $\mu_e > 1$ implies $r_e \geq 2$.

\begin{definition}[Local and global admissibility for edges in $E_{0 \infty}$]\label{de:edgeadm}
	Consider an edge of slope $ \mu $ in $ E_{0\infty}$ with local parameters $(r,s)$ as above. We say that this edge is \emph{locally admissible} if $\mu \leq 1$ or $r = \pm 1\,\,{\rm mod}\,\,s$.

A pair of edges in $E_{0\infty}$ with slopes $ \mu_i \leq \mu_j $ is \emph{globalisable} if			\begin{enumerate}[({G}1)]
			\item$  1 < \mu_i \leq \mu_j$ and there exists $m \in (r_{ij} - 1]$ such that $\frac{1}{\mu_j} \leq \frac{m-1}{m} \leq \frac{1}{\mu_i}$; or 
			\item $ 1 < \mu_i \leq \mu_j$ and $ \frac{1}{\mu_i} \geq  \frac{r_{ij}-1}{r_{ij}}$; or 
			\item $ \mu_i \leq 1 $.
		\end{enumerate}
		A sequence of edges in $E_{0\infty}$ is \emph{globally admissible} if any edge in the sequence if locally admissible and any pair of edges in the sequence is globalisable.
		\end{definition}
		
\begin{lemma}\label{le:coincm}
A pair of edges in $E_{0\infty}$ having the same slope $\mu$ is globalisable if and only if $\mu \leq 1$ or $\mu = \frac{r}{r - 1}$ for some $r \geq 2$.
\end{lemma}
\begin{proof} Assume we have a pair of edges in $E_{0\infty}$ have the same slope $\mu_i = \mu_j = \mu$. (G3) is the condition $\mu \leq 1$. If $\mu > 1$, we have $r_i,k_i$ coprime and likewise for $r_j,k_j$. Then $\mu = \frac{r_i}{k_i} = \frac{r_j}{k_j}$ implies $(k_i,r_i) = (k_j,r_j)$ and we simply denote it $(k,r)$. (G1) is then equivalent to the existence of $m \in (r-1]$ such that $\frac{1}{\mu} = \frac{m - 1}{m}$, but this would force $m = r$ and is thus impossible. (G2) states that $\frac{1}{\mu} \geq \frac{r - 1}{r}$, but as $\frac{1}{\mu} < 1$ has denominator $r$ the only possibility is $\frac{1}{\mu} = \frac{r - 1}{r}$.
\end{proof} 
		
\begin{definition}[Local and global admissibility for edges in $E_{\infty 0}$]\label{de:edgeadm2}
Consider an edge of slope $ \mu $ in $ E_{\infty 0}$ with local parameters $(r,s)$ as above. We say that this edge is \emph{locally admissible} if $\mu \geq 1$ or $r = \pm 1\,\,{\rm mod}\,\,s$.
A pair of edges in $E_{\infty 0}$ with slopes $ \mu_i \leq \mu_j $ is \emph{globalisable} if
			\begin{enumerate}[({G}1')]
			\item$  0 < \mu_i \leq \mu_j < 1$ and there exists $m \in [r_{ij} - 1]$ such that $\frac{1}{\mu_j} \leq \frac{m+1}{m} \leq \frac{1}{\mu_i}$, or 
			\item $ 0 < \mu_i \leq \mu_j < 1$ and $ \frac{1}{\mu_j} \leq  \frac{r_{ij}+1}{r_{ij}}$, or 
			\item $ \mu_j \geq 1 $.
		\end{enumerate}
An edge $e$ in $E_{0\infty}$ is \emph{globally admissible} if it is locally admissible and for any other edge $e'$ in $E_{\infty 0}$, the pair $e,e'$ is globalisable.
\end{definition}

\begin{lemma}\label{le:coincm2}
A pair of edges in $E_{\infty 0}$ having the same slope $\mu$ is globalisable if and only if $\mu \geq 1$ or $\mu = \frac{r }{r+1}$ for some $r \geq 1$.
\end{lemma}

\begin{proof} Similar to \cref{le:coincm} and thus omitted. \end{proof}

\begin{theorem}[Conditions for global admissibility in terms of Newton polygon]\label{th:globsimp}
	Let $\mathfrak{C} = \{P(x,y) = 0\} \subset \mathbb{C}^2$ be an irreducible and nondegenerate affine curve, and $\mathcal{S}=(\tilde C,  x, \omega_{0,1}, \omega_{0,2})$ be a compact spectral curve as in \cref{de:irredpc}. Then $\mathcal{S}$ is globally admissible if and only if the following three conditions are satisfied:
	\begin{enumerate}[({$\Gamma$A}1)]
		\item $|\hat{\Gamma}_{00}(P)|_{\mathbb{Z}} = 0$, i.e. $\mathfrak{C}$ is smooth at the origin.
		\item The edges of $\Delta(P)$ in $E_{0\infty}$  are globally admissible.
		\item The edges of $\Delta(P)$ in $ E_{\infty0}$ are globally admissible.
	\end{enumerate}
\end{theorem}

\begin{proof}
This is merely a reformulation of \cref{pr:ndga}, taking into account the coding of points in $\eta^{-1}(p_{\epsilon\epsilon'})$ by edges and the discussion preceding \cref{de:edgeadm}, especially formulae \eqref{aspslo}-\eqref{aspslo2} relating aspect ratios and slopes. Condition ($\Gamma$A1) is (GA1) verbatim. We recall that condition (GA2) from \cref{pr:ndga} requires local admissibility of all points $\eta^{-1}(p_{0\infty})$ and that every pair in $\eta^{-1}(p_{0\infty})$ satisfies ``(C-i) or (C-ii)''.

We recall that edges with  $\mu > 1$ have $r \geq 2$ coprime to $\bar s = s$. So, in the \cref{de:localadm} of local admissibility, (lA1) and (lA3) are automatic and in (lA2) the condition $s \in [r + 1]$ is equivalent to $\mu > 1$ while the congruence condition is repeated in \cref{de:edgeadm}. Edges with $\mu = 1$ have $r = 1$ (unramified point) and they automatically correspond to locally admissible points. Edges with $\mu < 1$ have either $r = 1$, or $\bar s = s \leq -1$ coprime to $r$, so automatically correspond to locally admissible points. This identifies local admissibility of edges with local admissibility of the corresponding point. Condition (C-i) for a pair of points matches (G1) for the corresponding pair of edges,  once we  notice that slopes of edges in $E_{0\infty}$ are finite, so $m = 1$ can be excluded. Condition (C-ii) can be split in two cases:  $\nu_1 \in (0,\frac{1}{r_{12}})$, which amounts to $\frac{1}{\mu_1} \geq \frac{r_{12} - 1}{r_{12}}$ and matches (G2); and $\nu_1 \leq 0$, which corresponds to $\mu_1 \leq 1$ and matches (G3).

Checking that condition (GA3) is equivalent to ($\Gamma$A3) can be carried out similarly. One should only be careful that $\nu$ is now a decreasing function of $\mu$ by \eqref{aspslo2}: the choice $\mu_1 \leq \mu_2$ in \cref{th:rewriting} leads to $\nu_2 \leq \nu_1$, hence the conditions (G2')-(G3') translating (C-ii) involve $\mu_2$ instead of $\mu_1$.
\end{proof}

\subsection{Global admissibility and deformations of the Newton polygon}

It turns out that the property of global admissibility is preserved under certain types of deformations, which will be useful in the next section.

Given a Newton polygon $\Delta(P)$, we call  \emph{interior point} an integral point in the interior of  $\Delta(P)$, and \emph{edge interior point} an integral point belonging to a maximal edge of the polygon $\Delta(P)$ and which is not an endpoint of this maximal edge. The context will not allow any confusion between this denomination and points in the interior of a set.

\subsubsection{Farey sequences and triangles}
 \label{s:Farey}
 
 Let us recall the notion of a Farey sequence and some of its properties. Given an integer $ r > 1 $, the \textit{Farey sequence of order $ r $} is the sequence $\mathcal{F}_r$ of irreducible nonnegative fractions with denominator $\leq r$, considered in increasing order. Unlike some references, we allow elements larger than $1$ and thus our $\mathcal{F}_r$ is infinite to the right. For example:
\begin{equation}\label{eq:F5}
\mathcal{F}_5 = 0, \frac{1}{5},\frac{1}{4},\frac{1}{3},\frac{2}{5},\frac{1}{2},\frac{3}{5},\frac{2}{3},\frac{3}{4},\frac{4}{5},1,\frac{6}{5},\frac{5}{4},\frac{4}{3},\frac{7}{5},\cdots
\end{equation}
\textit{Farey neighbours} of a fraction $ \frac{c}{d} > 0 $ in the Farey sequence $\mathcal{F}_r$ are elements $ \frac{a}{b} , \frac{e}{f} $ in $\mathcal{F}_r $ that are adjacent to $ \frac{c}{d} $ and such that  $\frac{a}{b} < \frac{c}{d} <  \frac{e}{f}$. They are known to satisfy the  relation 
\begin{equation}\label{eq:fsum}
	\frac{c}{d} = \frac{a+e}{b+f}\,.
\end{equation}  In addition, two neighbouring elements $ \frac{a}{b} < \frac{c}{d} $  in $\mathcal{F}_r $ must satisfy \textit{B\'ezout's identity}:
\begin{equation}
\label{Bezout}	bc - ad = 1\,.
\end{equation}
We will also need to consider the trivial Farey sequence
$$
\mathcal{F}_1 = 0,1,2,3,\ldots, \infty .
$$
By convention, we define Farey neighbours in this sequence to be triples $(p-1,p,\infty)$ for $p \geq 1$. Viewing $\infty = \frac{1}{0}$ and $p = \frac{p}{1}$, we see that \eqref{eq:fsum} and \eqref{Bezout} remain valid. This convention may seem strange, but it will make some of the results below easier to state.

In this section Newton polygons with the shape of a triangle will play a crucial role. Given rational numbers $\alpha,\mu,\beta$ (with $\beta$ possibly infinite) such that $0 \leq \alpha < \mu < \beta \leq \infty$, we consider a triangle $T \subseteq \mathbb{R}_{\geq 0}^2$ having integral extremal vertices, and slopes $\alpha,\mu,\beta$. We write
\begin{equation}
\label{slopelocal}
\frac{1}{\mu} = \frac{r -\bar s}{r},\quad \frac{1}{\alpha} = \frac{r_\alpha -\bar s_\alpha}{r_\alpha},\quad \frac{1}{\beta} = \frac{r_\beta - \bar s_\beta}{r_{\beta}}\,.
\end{equation}
Here $r \in \mathbb{Z}_{> 0}$ and $\bar s \in \mathbb{Z}$ are coprime such that $\bar s \leq r$; in particular if $\bar s = 0$ we must have $r = 1$, representing the inverse slope $1$. Analogous conditions hold for $(r_{\alpha},\bar s_{\alpha})$ and $(r_{\beta},\bar s_{\beta})$. In particular, if $r_\alpha = 0$, we must have $\bar s_\alpha = -1$, representing the slope $\alpha=0$, and if $r_\beta = \bar s_\beta$, we must have $r_\beta = \bar s_\beta = 1$, representing the slope $\beta = \infty$.

 \begin{definition}[Elementary triangle]\label{de:EleT}
 $T$ is an \emph{elementary triangle} if its only integral points are the extremal vertices. Equivalently, if it does not have edge interior points nor interior points.
 \end{definition}
 
 Elementary triangles can be characterised in terms of the slopes of the edges.
 
\begin{lemma}\label{le:FT}
	Let $T$ be a triangle with no edge interior points. $ T $ is elementary if and only if $ \frac{1}{\beta} < \frac{1}{\mu} < \frac{1}{\alpha} $ are Farey neighbours in $\mathcal{F}_r$, where $r$ is the numerator of $\mu$.  In this case, $\frac{1}{\mu}$ uniquely determines $\frac{1}{\beta}$ and $\frac{1}{\alpha}$ (but $\frac{1}{\alpha}$  or $\frac{1}{\beta}$ alone does not uniquely determine the two other ratios).
\end{lemma}

\begin{proof}
Up to global translation, the endpoints of the edge of slope $ \mu $ must be of the form $ (0,0)$ and $ (r-\bar s,r) $ as $ r,\bar s $ are coprime, and the edge has no interior points. Without loss of generality we can assume that the edge with slope $ \beta $ has $ (0,0) $ as endpoint. Since $r_\beta,\bar s_{\beta}$ are coprime, the third vertex of the triangle must have coordinates $ (r_\beta-\bar s_\beta,r_\beta) $ and  the edge with slope $ \beta  $ has no interior points. Calculating the slope of the remaining edge yields
\begin{equation}
	\frac{r_\alpha}{r_\alpha - \bar s_\alpha} = \alpha = \frac{r-r_\beta}{r-r_\beta -\bar s +  \bar s_\beta}\,.
\end{equation}
However, the edge of slope $ \alpha $ is also assumed to have no interior point, and hence we must have $ r_\alpha = r - r_\beta $ and $ \bar s_\alpha = \bar s - \bar s_\beta $. In particular $r_{\alpha\beta} \leq r$ so $\frac{1}{\beta},\frac{1}{\mu},\frac{1}{\alpha}$ belong to $\mathcal{F}_r$. We shall first treat the case $r > 1$, which implies $\frac{1}{\beta} < \frac{1}{\mu} < \frac{1}{\alpha}$, and deal with the case $ r = 1 $ later.

If $T$ had an interior point, there would exist edges of slopes $ \beta' $ and $ \alpha' $ satisfying $ \alpha < \alpha' <\mu < \beta' <\beta $, such that $\frac{1}{\beta'} < \frac{1}{\mu} < \frac{1}{\alpha'}$ belong to $\mathcal{F}_r$. Then $\frac{1}{\beta}$ and $\frac{1}{\alpha}$ would not be Farey neighbours of $\frac{1}{\mu}$ in $\mathcal{F}_r$. Therefore, if $\frac{1}{\beta}$ and $\frac{1}{\alpha}$ are Farey neighbours of $\frac{1}{\mu}$ in $\mathcal{F}_r$, then $T$ has no interior points.

For the converse statement, assume that $T$ has no interior point. Suppose that $\frac{1}{\beta}$ is not a left Farey neighbour of $\frac{1}{\mu}$ in $\mathcal{F}_r$, and denote $\frac{1}{\beta'} > \frac{1}{\beta}$ be this left Farey neighbour. We write as before $\frac{1}{\beta'} = \frac{r_{\beta'} - \bar s_{\beta'}}{r_{\beta'}}$.  Defining $r_{\alpha'} = r - r_{\beta'}$ and $\bar s_{\alpha'} = \bar s - \bar s_{\beta'}$, we see from \eqref{eq:fsum} that $\frac{1}{\alpha'} = \frac{r_{\alpha'} - \bar s_{\alpha'}}{r_{\alpha'}}$ is the right  Farey neighbour of $\frac{1}{\mu}$ in $\mathcal{F}_r$. In particular,
\begin{equation}
\label{slopineq} \frac{1}{\beta} < \frac{1}{\beta'} < \frac{1}{\mu} < \frac{1}{\alpha'} \leq \frac{1}{\alpha}\,.
\end{equation}
Following the first paragraph of the proof, we obtain a triangle $T'$ with integral vertices: two of them are $(0,0)$ and $(r- \bar s,r)$ joined by the edge of slope $\mu$, the second edge issuing from $(0,0)$ has slope $\beta'$, the third edge has slope $\alpha'$, and none of the three edges has interior points. Besides, \eqref{slopineq} implies that $T' \subset T$, in particular the third vertex of $T'$ must be interior to $T$  (cf. \cref{fig:triang}), which is a contradiction. The same argument goes through if $\frac{1}{\alpha}$ is not a right Farey neighbour of $\frac{1}{\mu}$ in $\mathcal{F}_r$. We conclude that $\frac{1}{\beta}$ and $\frac{1}{\alpha}$ must be Farey neighbours of $\frac{1}{\mu}$ in $\mathcal{F}_r$, which completes the proof for $r>1$.

\begin{figure}[!ht]
\begin{center}
\includegraphics[width=0.2\textwidth]{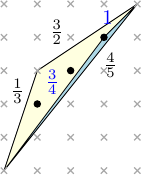}
\caption{\label{fig:triang} In yellow: a triangle $T$ with inverse slopes $\frac{1}{3}$, $\frac{4}{5}$ and $\frac{3}{2}$ in $\mathcal{F}_5$, in which edges have no interior points. Note that $\frac{1}{3}<\frac{4}{5}<\frac{3}{2}$ are not Farey neighbours in $\mathcal{F}_5$, and $T$ has interior points. In blue: a triangle $T' \subset T$ with inverse slopes $\frac{3}{4}$, $\frac{4}{5}$ and $1$ in $\mathcal{F}_5$. Note that $\frac{3}{4}<\frac{4}{5}<1$ are Farey neighbours in $\mathcal{F}_5$, and $T'$ has no interior point.}
\end{center}
\end{figure}

If $ r = 1 $, since $r = r_{\alpha} + r_{\beta}$ we must have $ r_\beta = 1$ and $r_{\alpha} = 0$ so that the edge with slope $ \alpha $ has no interior point. Hence $\bar s_\beta = \bar s - \bar s_\alpha = \bar s + 1 $ and $\bar s_\alpha = -1$. We obtain the sequence $\frac{-\bar s}{1}  < \frac{1-\bar s}{1} < \infty$, where $ \bar s < r = 1$. We recognise a triple of Farey neighbours in $\mathcal{F}_1$ according to our convention. In this case, the triangle never has interior points. Thus, the equivalence holds as well for $ r = 1 $.
\end{proof} 

\subsubsection{Deformations of Newton polygons}
\label{Sec45}
The main result of this section is that the global admissibility of spectral curves is invariant under enlarging (or conversely, shrinking) Newton polygons while preserving its set of interior points. We start by studying general properties of these deformations.

\begin{definition}[Polygons and equivalence]\label{de:prxpry}
For us, a \emph{polygon} is a non-empty subset of $\mathbb{R}^2$ which is the convex hull of finitely many points in $\mathbb{Z}_{\geq 0}^2$. It is \emph{inscribed} if there exists $(d_x,d_y) \in \mathbb{Z}_{> 0}^2$ (its \emph{bidegree}) such that $\Delta \subseteq \square(d_x,d_y)$ and $\Delta$ touches each four closed segments forming the boundary of $\square(d_x,d_y)$. We say that $\Delta$ is a \emph{long diagonal} if $d_x,d_y$ are not coprime and $\Delta = [(0,0),(d_x,d_y)]$ or $\Delta = [(d_x,0),(0,d_y)]$. Corners of $\Delta$ can be defined as in \cref{sec:conrern} and are denoted $\Gamma_{\epsilon\epsilon'}(\Delta)$. 
We say that two inscribed polygons are \emph{equivalent} if they have the same bidegree and the same interior points.

\end{definition}

We can move within an equivalence class of inscribed polygons by adding or removing certain vertices.

\begin{definition}[Adding and removing vertices]
Let $\Delta$ be an inscribed polygon in $\square(d_x,d_y)$ and $v \in \square(d_x,d_y) \cap \mathbb{Z}^2$. We say that $v \notin \Delta$ is \emph{addible} to $\Delta$ if the convex hull of $\Delta \cup \{v\}$, denoted by $\Delta_+[v]$, is an inscribed polygon equivalent to $\Delta$. We say that $v$ is \emph{removable} from $\Delta$ if there exists an inscribed polygon $\Delta'$ equivalent to $\Delta$ and such that $\Delta = \Delta'_+[v]$ (we then write $\Delta' = \Delta_-[v]$). We say that two inscribed polygons in $\square(d_x,d_y)$ are \emph{strongly equivalent} if one can transform one into the other by a finite sequence of addition/removal of addible/removable vertices.
\end{definition}
We note that equivalent inscribed polygons may not be strongly equivalent, for instance $\Delta = [(0,0), (d_x, d_y)]$ and $\Delta' = [(d_x,0), (0,d_y)]$.

An interesting question is whether (strong) equivalence classes of Newton polygons have minimal or maximal polygons under inclusion. It is clear that a given equivalence class of Newton polygons may not have a minimal polygon (cf. \cref{fig:nomin}) nor a maximum (cf. \cref{fig:nomax}). It turns out that strong equivalence classes may still not have minimal polygons; cf. \cref{fig:nomin}. However, up to a pathological case, they do have a maximum, as we show next.

\begin{figure}[!ht]
\begin{center}
\includegraphics[width=0.23\textwidth]{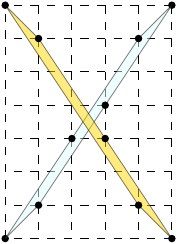}
\caption{\label{fig:nomax} The two Newton polygons depicted in the figure are inscribed in $\square(5,7)$ and have no interior points, and thus are equivalent. Furthermore, they are both maximal under inclusion. However, they are not strongly equivalent: one cannot transform one into the other by addition/removal of vertices without changing the bidegree, keeping them inscribed or creating an interior point in an intermediate step.}
\end{center} 
\end{figure}

\begin{figure}[!ht]
\begin{center}
\includegraphics[width=0.7\textwidth]{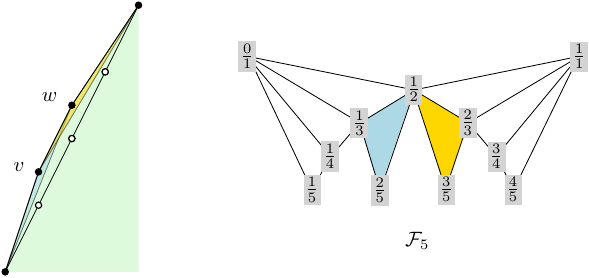}
\caption{\label{fig:nomin}  A Newton polygon $\Delta$ is depicted in green. An elementary triangle $T_v$ is shown in blue (with slopes $2 < \frac{5}{2} < 3$), and another elementary triangle $T_w$ is shown in orange (with slopes $\frac{3}{2}< \frac{5}{3} < 2$). $v$ and $w$ are both separately removable from $\Delta$, aligning with the fact that the inverse slopes of $T_v$ and $T_w$ are separately Farey neighbours in $\mathcal{F}_5$ (the numerator of the middle slope is $5$). However, $w$ is not removable from $\Delta_-[v]$ (the $\circ$ which were interior points become boundary points), aligning with the fact that $\frac{2}{5}$ and $\frac{2}{3}$ are not Farey neighbours in $\mathcal{F}_{5}$. Likewise, $v$ is not removable from $\Delta_-[w]$. This shows that equivalence classes, and even strong equivalence classes, may not have minimal polygons under inclusion.}
\end{center} 
\end{figure}

\begin{proposition}\label{pr:Qdelt}
Let $\Delta$ be an inscribed polygon which is not a long diagonal. Then the following properties hold:
\begin{enumerate}
\item If $v$ is addible to $\Delta$, there exists a sequence of $v^1,\ldots,v^k \in \mathbb{Z}^2_{\geq 0}$ and polygons $\Delta^0\ldots,\Delta^k$ equivalent to $\Delta$ such that $\Delta^0 = \Delta$, $\Delta^k = \Delta_+[v]$, and for each $i \in [k]$, $v_i$ is addible to $\Delta^{i - 1}$ and $\Delta^i = \Delta^{i -1}_+[v_i]$ differs from $\Delta^{i-1}$ by the addition of an elementary triangle.
\item If $v,\tilde{v}$ are addible to $\Delta$, then $\tilde v$ is addible to $\Delta_+[v]$.
\end{enumerate}
Each irreducible strong equivalence class which does not contain a long diagonal admits a maximal polygon.
\end{proposition}
\begin{proof}
We start with claim (1). Let $v \notin \Delta$ be addible to $\Delta$. As $\Delta_+[v]$ has the same interior points as $\Delta$, $\Delta_+[v]$ differs from $\Delta$ by the addition of  a triangle $T_v$ with vertices $v$ and $u,w \in \Delta$, with no interior points. The edge $(u,w)$ cannot have edge interior points, unless $\Delta$ consists of the single maximal edge $(u,w)$; otherwise, the edge interior points of $(u,w)$ would become interior points in $\Delta_+[v]$, and $\Delta$ and $
\Delta_+[v]$ would not be equivalent. However, if we assume that $\Delta$ is not a long diagonal, it cannot consist of a single edge $(u,w)$ with edge interior points. We conclude that $(u,w)$ has no edge interior point. 

As for the other two edges of the triangle $T_v$, if there were edge interior points in $(u,v) \cup (v,w)$, we could add them to $\Delta$ one by one, so that at each step one adds only an elementary triangle (cf. \cref{fig:finetrig}). This justifies claim (1).

\begin{figure}[!ht]
\begin{center}
\includegraphics[width=0.18\textwidth]{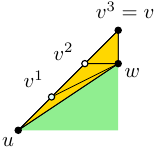}
\caption{\label{fig:finetrig} The triangle $T_v$ (in orange) can be added to the polygon $\Delta$ (in green) by a succession of addition of elementary triangles.}
\end{center}
\end{figure}

We now proceed to claim (2). Claim (1), which we have just proved, shows that we can add $v$ (respectively $\tilde v$) to $\Delta$ by adding a sequence of elementary triangles. Thus, it suffices to prove claim (2) when $v,\tilde{v}$ are addible to a polygon $\Delta$ which is not a long diagonal and such that $T_v$ and $T_{\tilde{v}}$ are elementary triangles, which we now assume. We denote $\Delta_+[v,\tilde{v}] = \Delta_v \cup T_v \cup T_{\tilde{v}}$.

First assume that $v$ and $\tilde{v}$ lie in different corners of $\Delta$. If $\Delta$ consists of a single maximal edge, then among the two vertices $v$ and $\tilde{v}$, one lies below $\Delta$ and the other above. As $\Delta$ is not a long diagonal, the maximal edge in $\Delta$ does not contain an edge interior point, and $\tilde{v}$ is clearly addible to $\Delta_+[v]$. If $\Delta$ does not consist of a single maximal edge, then it is also clear that $\tilde{v}$ is addible to $\Delta_+[v]$.

We now assume that $v$ and $\tilde{v}$ belong to the same corner $\Gamma_{0\infty}(\Delta)$ (the discussion at the other corners would be similar).  Let $u,w \in \Delta$ be the other two vertices of the elementary triangle $T_v$ as in the beginning of the proof.
As $\Delta$ is inscribed, it must touch both the closure of the left vertical side and the upper horizontal side of $\square(d_x,d_y)$. This forces $u$ and $w$ to belong to the same corner $\Gamma_{0\infty}(\Delta)$. Then, denoting $\mu > 0$ the slope of the base $[u,w] \subseteq \Delta$, and $\beta,\alpha > 0$ the slopes of $[u,v]$ and $[v,w]$ respectively, we have (up to renaming $u$ and $w$ so that $u$ is below $w$)  the inequality $0 \leq \frac{1}{\alpha} < \frac{1}{\mu} < \frac{1}{\beta}$. By \cref{le:FT}, this is a triple of neighbours in $\mathcal{F}_r$, where $r$ is the numerator of $\mu$. The same construction applied to $\tilde{v}$ yields an elementary triangle $T_{\tilde{v}}$ with vertices $\tilde{v}$ and $\tilde{u},\tilde{w} \in \Delta$ in $\Gamma_{0\infty}(\Delta)$ such that $\tilde{u}$ is below $\tilde{v}$. Up to exchanging the role of $v$ and $\tilde{v}$, we can assume that $v$ is below $\tilde{v}$.

If $[u,w] \cap [\tilde{u},\tilde{w}] = \emptyset$, then $\Delta_+[v,\tilde{v}]$ is convex and has the same interior points as $\Delta$, so $\tilde{v}$ is addible to $\Delta_+[v]$.

If $w = \tilde{u}$, $\Delta_+[v,\tilde{v}]$ contains no other interior points than those already present in $\Delta$. Then, $\tilde{v}$ is addible to $\Delta_+[v]$ if and only if $\Delta_+[v,\tilde{v}]$ is convex. As $T_v$ and $T_{\tilde{v}}$ are elementary triangles, we have two triples of Farey neighbours $0 \leq \frac{1}{\alpha} < \frac{1}{\mu} < \frac{1}{\beta}$ in $\mathcal{F}_r$ and $0 \leq \frac{1}{\tilde{\alpha}} < \frac{1}{\tilde{\mu}} < \frac{1}{\tilde{\beta}}$ in $\mathcal{F}_{\tilde{r}}$. Since $v$ and $\tilde{v}$ are addible to $\Delta$, we know that $\Delta_+[v]$ and $\Delta_+[\tilde{v}]$ are convex. This implies $\frac{1}{\beta} \leq \frac{1}{\tilde{\mu}}$ and $\frac{1}{\mu} \leq \frac{1}{\tilde{\alpha}}$. The aforementioned Farey neighbour property forces $\frac{1}{\beta} \leq \frac{1}{\tilde{\alpha}}$, which implies that $\Delta_+[v,\tilde{v}]$ is convex.

The last case is when the bases of the two triangles coincide $[u,w] = [\tilde{u},\tilde{w}]$. In that case $T_v$ and $T_{\tilde{v}}$ have same middle slope $\mu$, and this datum determines uniquely the two other slopes of an elementary triangle (Lemma~\ref{le:FT}) with that middle slope, forcing $T_v = T_{\tilde{v}}$. Thus,  (2)  holds trivially in this case. This concludes the proof of (2).

The last claim then follows by adding simultaneously all addible vertices to $\Delta$. This gives an inscribed polygon strongly equivalent to $\Delta$, which by construction is the maximum (for the inclusion) among all inscribed polygons obtained from $\Delta$ by addition of vertices.
\end{proof}

We now return to the main subject of this section -- is global admissibility of spectral curves  preserved within strong equivalence classes? Our main result is the following theorem.

\begin{theorem}[Global admissibility is preserved in strong equivalence classes]\label{th:gapairs}
Let $\mathfrak{C} = \{P(x,y) = 0\} \subset \mathbb{C}^2$ and $\mathfrak{C}' = \{ P'(x,y) = 0\} \subset \mathbb{C}^2$ be two irreducible nondegenerate affine curves whose Newton polygon belong to the same strong equivalence class, and we assume that the latter does not contain a long diagonal. Let $\mathcal{S}$ and $\mathcal{S}'$ be the corresponding compact spectral curves as in  \cref{de:irredpc}.
Then $\mathcal{S}$ is globally admissible if and only if $\mathcal{S}'$ is globally admissible.
\end{theorem}

\begin{remark}
If $\mathfrak{C}$ is irreducible, $\Delta(P)$ must be inscribed in $\square(d_x,d_y)$ where $(d_x,d_y)$ is the bidegree. If $\Delta(P)$ and $\Delta(P')$ are equivalent, $P$ and $P'$ must have the same bidegree. Although the Newton polygon of an irreducible affine curve cannot be a long diagonal, it could still be strongly equivalent to a long diagonal, and such a case is excluded in \cref{th:gapairs}. 
\end{remark}

Given $P(x,y) = 0$ defining a globally admissible spectral curve $\mathcal{S}$, these results tell us that we can perform two types of operations preserving global admissibility:
\begin{itemize}
\item \emph{Enlarging the Newton polygon/Deforming the curve.} We can add to $P$ linear combinations of monomials $x^iy^j$ for which $(i,j)$ is addible to $\Delta(P)$. Due to \cref{pr:Qdelt}, we can in fact add all addible vertices at the same time, and include $\mathcal{S}$ in a largest possible family of globally admissible spectral curves.
\item \emph{Squeezing the Newton polygon/Taking limits in a family of curves.} We can send to $0$ the coefficient in $P$ of one monomial $x^iy^j$ where $(i,j)$ is removable from $\Delta(P)$. Contrarily to enlargements, we can kill monomials one by one, but not necessarily simultaneously (as it may remove interior points). The possible existence of several minimal elements in the strong equivalence class of $\Delta(P)$ means that different types of spectral curves may be obtained by taking limits from $\mathcal{S}$. 
\end{itemize}
In the two operations, for generic choice of coefficients, the resulting affine curve will remain nondegenerate, and we can also deform the coefficients of monomials whose exponents $(i,j)$ are already in $\Delta(P)$.

\begin{proof}[Proof of \cref{th:gapairs}]
	
	As we can transform both $\Delta(P)$ and $\Delta(P')$ to the maximal polygon in their common strong equivalence classes, it is enough to consider the situation where $\Delta(P')$ is obtained from $\Delta(P)$ by addition of an addible vertex, and prove that global admissibility of $\Delta(P)$ is then equivalent to global admissibility of $\Delta(P')$. 
		Let us analyse what happens corner by corner. Global admissibility does not impose any condition on $\Gamma_{\infty\infty}$.  Condition ($\Gamma$A1) for global admissibility in \cref{th:globsimp} requires that $|\hat{\Gamma}_{00}|_{\mathbb{Z}} = 0$. This is preserved by addition/removal of vertices which does not create/remove interior points --- in particular these vertices must be on the horizontal or vertical axes.  We need to take a closer look at $\Gamma_{0\infty}$ where global admissibility requires ($\Gamma$A2), and $\Gamma_{\infty0}$ where global admissibility requires ($\Gamma$A3). 
	
	We proceed to the top left corner. The edges in $E_{0\infty}$ have slopes $ 0 < \mu_1 < \mu_2 < \ldots < \mu_{\ell} $, appearing from right to left. Let us add $\mu_0 = 0$ and $\mu_{\ell + 1} = \infty$ to this sequence, although they do not represent edges in $E_{0\infty}$. We claim that a vertex $ v = (p,q) $ added in the corner $ \Gamma_{0\infty}(P)$ to produce $\Delta(P')$ must create exactly two new edges in $ E_{0\infty} $ with slopes $ \alpha, \beta $ such that $ \mu_{i-1} \leq  \alpha < \mu_i < \beta \leq  \mu_{i+1} $ for some $i \in [\ell]$. Indeed, we must keep the endpoints of the edges in $E_{0\infty}$ (except perhaps the leftmost and the rightmost one) as vertices, otherwise it would result in new interior points in $\Delta(P')$ compared to $\Delta(P)$. Adding the vertex $v$ can then only result in  adding to the convex hull of a triangle with slopes $\alpha < \mu_i < \beta$. If we add such a $v$, the edge of slope $\mu_i$ cannot have an interior point, unless $\Delta(P)$ consists of a single edge from $(0,0)$ to $(d_x,d_y)$. In the latter case, any interior point of that edge would remain on the boundary  of $\Delta(P')$. However, in such a case, the assumption that $\Delta(P)$ is not a long diagonal forces $d_x,d_y$ to be coprime, implying the absence of interior point on this edge.

	\begin{figure}[!ht]
	\begin{center}
	\includegraphics[width=0.75\textwidth]{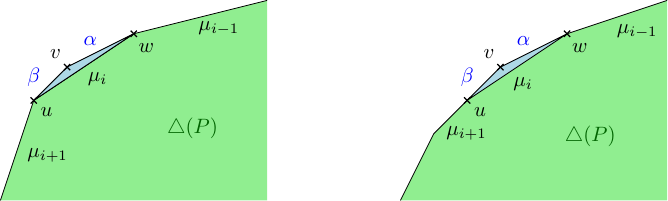}
	\caption{\label{fig:slopepoly} Allowed addition of a vertex $v$ to $\Delta(P)$.}
	\end{center}
	\end{figure} 
	Denote by $u$  (resp. $w$) the left (resp. right) endpoint of $\mu_i$, and consider the triangle $ T $ formed by the vertices $ u , v $ and $ w $ (cf. \cref{fig:slopepoly}). Assume that the edge with slope $\alpha$ has interior points $v_1,\ldots,v_q$ when going from $u$ to $v$. The segment joining $v_i$ to $w$ cannot have an interior point as $T$ does not have interior points. Then, instead of adding $ v $ as a vertex, we could have added the $v_j$ one by one, with increasing $j$. The $j$-th step amounts to adding to the so-far obtained Newton polygon a triangle $T_j$ whose edges do not have interior points. Therefore, it is enough to study these elementary steps where the added triangle is an elementary triangle.

For such a triangle, still denoting by $\alpha,\mu_i,\beta$ the slopes of its edges, as in \cref{fig:slopepoly}:
\begin{itemize}
\item We claim that the sequence $\mu_1, \dots, \mu_{i-1}, \alpha, \beta, \mu_{i+1}, \ldots, \mu_\ell$ is globally admissible provided that the sequence $\mu_1, \ldots, \mu_{i-1}, \mu_i, \mu_{i+1}, \ldots, \mu_\ell$ is globally admissible. If $\mu_i > 1$ this is proved in \cref{le:globad}; if $\mu_i \geq 1$ this is proved in \cref{le:localad2}, showing that if $\alpha$ and/or $\beta$ are in $E_{0\infty}$, they are $\leq 1$ and thus are automatically globally admissible.
\item Conversely, we claim that the sequence $\mu_1, \ldots, \mu_{i-1}, \mu_i, \mu_{i+1}, \ldots, \mu_\ell$ is globally admissible  provided that the sequence $\mu_1, \dots, \mu_{i-1}, \alpha, \beta, \mu_{i+1}, \ldots, \mu_\ell$ is globally admissible. This is proved in \cref{le:muiglob}.
\end{itemize}
	The proof for the corner $  \Gamma_{\infty 0} $ and condition ($\Gamma$A3) can be obtained by reflection of these arguments over the axis $y = x$ and is thus omitted.
\end{proof}

 \subsection{Proof of the lemmata}

We are going to complete the proof of \cref{th:gapairs} in several steps. We consider as in this proof a situation where an elementary triangle with edges of slopes $\alpha < \mu_i < \beta$ is added to $\Delta(P)$ (or removed from $\Delta(P')$). Recall the encoding of slopes by parameters $(r_{\alpha},\bar s_{\alpha})$, $(r_i,\bar s_i)$ and $(r_{\beta},\bar s_{\beta})$ from \eqref{slopelocal}.  From \cref{le:FT} we know that $\frac{1}{\beta},\frac{1}{\mu_i},\frac{1}{\alpha}$ are Farey neighbours in $\mathcal{F}_{r_i}$. We decide to call edges by their slopes, when it does not lead to confusion.

\begin{lemma}\label{le:localad}
	Assume that $ \mu_i $ is  locally admissible and $\mu_i  > 1$.  Then $r_i \geq 2$,  $\bar s_i = s_i > 0$ and we have the following alternative:
		\begin{itemize}
		\item If $ r_i = +1\,\,{\rm mod}\,\,s_i $, then $\bar s_\beta = 1  $ and $\bar s_\alpha = s_i -1 $;
		\item If $ r_i = -1\,\,{\rm mod}\,\,s_i $ and $s_i \geq 2$, then $\bar s_\alpha = 1 $ and $\bar s_\beta = s_i - 1 $.
	\end{itemize}
In both cases, $ \alpha $ and $ \beta $ are locally admissible.
\end{lemma}

	\begin{proof}
The assumption $\mu_i > 1$ implies that $r_i \geq 2$ and $\bar s_i \geq 1$. It also implies $\alpha \geq 1$ as $\frac{1}{\alpha}$ is a right Farey neighbour of $\frac{1}{\mu_i} < 1$.
		
If $\bar s_i = 1$, we have $\alpha = 1$ and must be in the situation $(\frac{1}{\beta},\frac{1}{\mu_i},\frac{1}{\alpha}) = (\frac{r_i - 2}{r_i - 1},\frac{r_i - 1}{r_i},1)$, in particular $\bar s_i = \bar s_{\beta} = 1$ and $\bar s_\alpha = 0$. In this case the first alternative is satisfied, $\alpha = 1$ is locally admissible because $r_{\alpha} = 1$, and $\beta$ is locally admissible because either $r_{\beta} = 1$ (if $r_i = 2$) or $r_{\beta} \geq 2$ is congruent to $1$ modulo $s_{\beta} = \bar s_{\beta} = r_{\beta} - 1$.
		
Now assume $\bar s_i \geq 2$. This forces $\beta > \alpha > 1$, in particular $r_{\alpha},r_{\beta} \geq 2$ and $(s_i,s_{\alpha},s_{\beta}) = (\bar s_i,\bar s_{\alpha},\bar s_{\beta})$ is a triple of positive numbers, implying that $s_{\alpha} + s_{\beta} =  s_i \geq 2$ and $\max(s_{\alpha},s_{\beta}) <  s_i$. Furthermore, B\'ezout's identity \eqref{Bezout} for the two pairs of neighbours yields
\begin{equation}
	\label{riki}	r_i s_{\alpha} - r_{\alpha} s_i = 1,\qquad r_{\beta} s_i - r_i s_{\beta} = 1\,.
\end{equation}
Taking the reduction modulo $s_i$ yields:
\begin{equation}
\label{rikimod}
r_i  s_{\alpha} = 1\,\,{\rm mod}\,\,s_i,\qquad r_i s_{\beta} = -1\,\,{\rm mod}\,\,s_i\,.
\end{equation}

Recall from \eqref{aspslo} and the discussion above it that the edge $\mu_i = \frac{r_i}{r_i - s_i} > 1$ has local parameter $ s_i \in [r_i-1]$. As we assumed it is locally admissible, we have $r_i = \pm 1\,\,{\rm mod}\,\,s_i$. Together with \eqref{rikimod} we deduce that $s_{\alpha} = \pm 1 \,\,{\rm mod}\,\,s_i$ and $s_{\beta} = \mp 1\,\,{\rm mod}\,\,s_i$. Recalling that $1 \leq s_{\alpha} < s_i$ and $1 \leq s_{\beta} < s_i$ we conclude that $(s_{\alpha},s_{\beta}) = (1,s_i - 1)$ in the case $r_i = -1\,\,{\rm mod}\,\,s_i$ and $(s_{\alpha},s_{\beta}) = (s_i - 1,1)$ in the case $r_i = 1\,\,{\rm mod}\,\,s_i$, as announced.

We now turn to the local admissibility of $\alpha$ and $\beta$. When $s_{\alpha} = 1$, the edge $\alpha$ is clearly  locally admissible. We are also in the situation $s_{\beta} = s_i - 1$ due to the previous discussion, while the reduction of \eqref{riki} modulo $s_{\beta}$ yields $r_{\beta} s_i = 1\,\,{\rm mod}\,\,s_{\beta}$. This implies $r_{\beta} = 1\,\,{\rm mod}\,\,s_{\beta}$, thus $\beta$ is locally admissible as well. When $ s_{\beta} = 1$, the edge $\beta$ is clearly admissible, while $ s_{\alpha} = s_i - 1$ and the reduction of \eqref{riki} modulo $r_{\alpha}$ leads to $r_{\alpha} = -1\,\,{\rm mod}\,\,s_{\alpha}$, thus $\alpha$ is locally admissible as well.
\end{proof}
		
\begin{lemma}\label{le:globad}
	Assume that $ \mu_1,\ldots,\mu_i,\ldots,\mu_{\ell}$ is a sequence of globally admissible edges, and that $ \mu_i  > 1$. Then, $\mu_1,\ldots,\mu_{i - 1},\alpha,\beta,\mu_{i + 1},\ldots,\mu_{\ell} $ is a sequence of globally admissible edges.
	\end{lemma}
\begin{proof}
	As $ \mu_i $ is assumed locally admissible, \cref{le:localad} tells us that $r_i \geq 2$ and $\bar s_i = s_i$, that $r_i = \pm 1\,\,{\rm mod}\,\,\bar s_i$, and that $\alpha,\beta$ are locally admissible. According to \cref{de:edgeadm2}, we only need to check that the pairs $ \alpha,\mu_j $, the pairs $ \beta,\mu_j $ (for each $j  \neq i$) and the pair $\alpha,\beta$ are globalisable. We shall examine the conditions (G1)-(G2)-(G3) of \cref{de:edgeadm2}. Since $\frac{1}{\alpha}$ is the right Farey neighbour of $\frac{1}{\mu_i}$ in  $\mathcal{F}_{r_i}$, we have $\alpha \geq 1$.
		
We treat in detail the case $r_i = 1\,\,{\rm mod}\,\,s_i$, meaning $\bar s_{\beta} = 1$ from \cref{le:localad} and thus $\frac{1}{\beta} = \frac{r_{\beta} - 1}{r_{\beta}}$, with $r_{\beta} \geq 2$ since $\beta > \mu_i > 1$.

\medskip 

\noindent \underline{\textit{Case $j > i$}.} We have $1 < \mu_i < \mu_j $ and thus
\begin{equation}
\label{thuenu}1 \leq \alpha < \mu_i < \beta \leq \mu_j\,.
\end{equation}  First, we check that the pair $\beta,\mu_j $ is globalisable. When $\beta = \mu_j$, the pair $\beta,\mu_j$ is globalisable due to \cref{le:coincm} as $\frac{1}{\beta} = \frac{r_{\beta} - 1}{r_{\beta}}$. We now assume $\beta < \mu_j$, and observe that $\frac{1}{\mu_j} \leq  \frac{r_\beta-1}{r_\beta}   < \frac{1}{\mu_i}$ with $0 < r_\beta < r_i  $. If $ r_\beta < r_j $, the pair $\beta, \mu_j  $ is globalisable by choosing $m = r_{\beta} \leq r_j - 1 = r_{\beta j} - 1$ in (G1). If we rather have $r_{\beta} \geq r_j$, then $\frac{r_\beta-1}{r_\beta} = \frac{ r_{\beta j} -1}{r_{\beta j}}$  and thus $ \beta, \mu_j $ is globalisable by (G2). 

	Now, we check that $\alpha,\mu_i$ is globalisable. When $\alpha = 1$ the pair $\alpha,\mu_i$ satisfies (G3), so we move on to the case $\alpha > 1$.  As $\mu_i,\mu_j$ is globalisable by assumption and $1 < \mu_i < \mu_j$, it satisfies either (G1) or (G2).
	\begin{itemize}
			\item If $ \mu_i, \mu_j $ satisfies (G1), then there exists $m \in (r_{ij} - 1]$ such that $\frac{1}{\mu_j} \leq \frac{m - 1}{m} \leq \frac{1}{\mu_i}$. When $m \in (r_{\alpha j} - 1]$, since $\frac{1}{\mu_i} < \frac{1}{\alpha}$ we recognise that $ \alpha , \mu_j $ is globalisable by (G1). When $m \geq r_{\alpha j}$, we write $\frac{r_{\alpha j} -1}{r_{\alpha j}} \leq \frac{m - 1}{m} < \frac{1}{\alpha}$ and deduce that $\alpha , \mu_j $ is globalisable by (G2).
		\item If $\mu_i,\mu_j$ satisfies (G2), we have $\frac{1}{\mu_i} \geq \frac{r_{ij} - 1}{r_{ij}}$, and since $r_i > r_{\alpha}$ we get $\frac{1}{\alpha} > \frac{1}{\mu_i} \geq \frac{r_{\alpha i} - 1}{r_{\alpha i}} \geq \frac{r_{\alpha j} - 1}{r_{\alpha j}}$, thus $\alpha,\mu_j$ is globalisable by (G2).  
\end{itemize}

\medskip
 
\noindent \underline{\textit{Case $j < i$}.} We have $1 < \mu_j \leq \alpha < \mu_i < \beta$. First, we check that the pair $\mu_j,\beta$ is globalisable. We have $\frac{r_\beta-1}{r_\beta}	 < \frac{1}{\mu_i}	< \frac{1}{\alpha}   \leq \frac{1}{\mu_j}$ with $0 < r_\beta < r_i $. When $ r_\beta < r_j $, the pair  $\mu_j,\beta$ is globalisable by choosing $m = r_\beta \leq r_j - 1 = r_{\beta j}-1$ in (G1). When $ r_\beta \geq r_j $, we have $\frac{r_\beta-1}{r_\beta} =  \frac{r_{\beta j}-1}{r_{\beta j}}$, hence $\mu_j,\beta$ is globalisable by (G2).

Second we check that $\mu_j,\alpha$ is globalisable. We use that $\mu_j,\mu_i$ is globalisable with $\mu_j > 1$, so satisfies either (G1) or (G2).
\begin{itemize}
			\item If $ \mu_j, \mu_i $ satisfies (G1), then there exists $m \in (r_{ij} -1]$ satisfying $ \frac{1}{\mu_i}  \leq  \frac{m-1}{m} \leq \frac{1}{\mu_j}$. When $m < r_\alpha $, since $r_{\alpha} < r_i$ the first inequality must be strict and $m \in (r_{\alpha j} - 1]$. As $\frac{1}{\alpha}$ is the right Farey neighbour of $\frac{1}{\mu_i}$ in $\mathcal{F}_{r_i}$, we deduce that $\frac{1}{\mu_i} < \frac{1}{\alpha}  \leq  \frac{m -1}{m} \leq \frac{1}{\mu_j}$, thus $ \mu_j,\alpha$ is globalisable by (G1). When $m \geq r_{\alpha j} $, we have $\frac{r_{\alpha j} -1}{r_{\alpha j}} \leq \frac{m - 1}{m} \leq \frac{1}{\mu_j}$ and thus $\mu_j,\alpha$ is globalisable by (G2). When $r_\alpha \leq m <r_j $, recalling that $\alpha > 1$ implies $\bar s_{\alpha} \geq 1$, we have $\frac{1}{\alpha} = \frac{r_\alpha - \bar s_\alpha}{r_\alpha}  \leq \frac{r_\alpha-1}{r_\alpha} \leq \frac{m-1}{m} \leq \frac{1}{\mu_j}$ with $m \in (r_{\alpha j} - 1]$, hence $\mu_j,\alpha$ is globalisable by (G1).
			\item If $ \mu_j,\mu_i$ satisfies (G2),  we have $\frac{1}{\mu_j}  \geq \frac{r_{ij} -1}{r_{ij}} $. Since $r_i > r_{\alpha}$ we deduce $\frac{1}{\mu_j} \geq \frac{r_{\alpha j} - 1}{r_{\alpha j}}$ and thus $\mu_j,\alpha$ is globalisable by (G2).
		\end{itemize}
This establishes that the sequence $\mu_1,\ldots,\mu_{i - 1},\alpha,\beta,\mu_{i +1},\ldots,\mu_{l}$ is globally admissible in the case $ r_i = 1\,\,{\rm mod}\,\,\bar s_i $. The case $r_i = -1\,\,{\rm mod}\,\,\bar s_i$ can be handled similarly.
\end{proof} 

\begin{lemma}\label{le:localad2}
	Assume $ \mu_i  \leq 1 $. Then:
	\begin{itemize}
	\item either $\alpha = 0$ (the corresponding edge does not belong to $E_{0\infty}$) or $\alpha \in (0,1)$.
	\item either $\beta = \infty$ (the corresponding edge does not belong to $E_{0\infty}$) or $\beta \leq 1$.
	\end{itemize}
\end{lemma}

\begin{proof}
	As $\mu_i \leq 1$, the corresponding edge is locally admissible.
	
	When $r_i = 1$, we have $\mu_i = \frac{1}{m}$ for some $m \in \mathbb{Z}_{> 0}$. Due to our convention, the right Farey neighbour of $\frac{1}{\mu_i}$ in $\mathcal{F}_1$ is $\infty$, leading to $\alpha = 0$. This edge is horizontal, so is not in $E_{0\infty}$, does not correspond to a branch of the singularity at $ p_{0\infty} $ and can be ignored. Considering the left Farey neighbour of $\frac{1}{\mu_i} = m$ in $\mathcal{F}_1$ we obtain $\beta = \frac{1}{m - 1}$. If $m \geq 2$ we have $\beta \leq 1$.  If $m = 1$, we rather have $\beta = \infty$: the corresponding edge is vertical, so not in $E_{0\infty}$ and can be ignored.
	
 	Since $\mu_i = 1$ has $(r_i,\bar s_i) = (1,0)$, if $r_i \geq 2$ we must have $\mu_i < 1$. Then, the right Farey neighbour $\frac{1}{\alpha}$ of $\frac{1}{\mu_i}$ in $\mathcal{F}_{r_i}$ is finite and $> 1$, while the left Farey neighbour $\frac{1}{\beta}$ is $\geq 1$.
\end{proof}

\begin{lemma}\label{le:muiloc}
Assume that $\alpha,\beta$ is globalisable. Then $\mu_i$ is locally admissible.
\end{lemma}

\begin{proof} We will rely on the properties of $\alpha,\mu_i,\beta$ seen in the course of the proof of \cref{le:localad}. By assumption, $ \alpha,\beta$ satisfies one of the conditions (G1), (G2) or (G3).
\begin{itemize}
\item If $\alpha,\beta$ satisfies (G1), then there exists $m \in (r_{\alpha \beta} - 1]$ such that $\frac{1}{\beta} = \frac{r_\beta - \bar s_\beta}{r_\beta} \leq \frac{m-1}{m} \leq \frac{1}{\alpha} = \frac{r_\alpha - \bar s_\alpha}{r_\alpha}$. As $r_{\alpha \beta} < r_i$, all the elements in the inequality are in the Farey sequence $\mathcal{F}_{r_i}$. Since $\frac{1}{\beta},\frac{1}{\mu_i},\frac{1}{\alpha}$ are Farey neighbours in $\mathcal{F}_{r_i}$, we must have $\frac{r_i - \bar s_i}{r_i} = \frac{1}{\mu_i} = \frac{m - 1}{m}$. As $r_i,\bar s_i$ are coprime, we deduce $r_i = m$ and $\bar s_i = 1$. Thus $\mu_i$ is locally admissible.
\item If $\alpha,\beta$ satisfies (G2), we have $1 > \frac{1}{\alpha} = \frac{r_{\alpha} - \bar s_{\alpha}}{r_{\alpha}} \geq \frac{r_{\alpha \beta} - 1}{r_{\alpha \beta}}$. From this we infer $\bar s_{\alpha} = 1$, hence $\bar s_{\beta} = \bar s_i - 1$. As $\mu_i > \alpha > 1$, we have $r_i \geq 2$ and $\bar s_i = s_i$. The second equation of \eqref{rikimod} then yields $r_i = -1\,\,{\rm mod}\,\,s_i$, and we deduce that $\mu_i$ is locally admissible.
\item If $\alpha,\beta$ satisfies (G3), we have $\alpha \leq 1$. Let us first consider the case $r_i \geq 2$. Then $\frac{1}{\beta},\frac{1}{\mu_i},\frac{1}{\alpha}$ are Farey neighbours in $\mathcal{F}_{r_i}$. When $\alpha < 1$, this implies $\mu_i \leq 1$, hence $\mu_i$ is locally admissible. When $\alpha = 1$, we have $(r_{\alpha},r_j)  = (1,r_i - 1)$ and $\frac{1}{\mu_i} = \frac{r_i - 1}{r_i}$, leading to $\bar s_i = s_i = 1$, so $\mu_i$ is locally admissible. In the case $r_i = 1$, we have a triple $\frac{1}{\beta},\frac{1}{\mu_i},\frac{1}{\alpha}$ of Farey neighbours in $\mathcal{F}_1$, implying that  $\frac{1}{\mu_i} \in \mathbb{Z}_{> 0}$, hence $\mu_i \leq 1$ is again locally admissible.
\end{itemize}
\end{proof}
  
\begin{lemma}\label{le:muiglob}
	Assume that the sequence $ \mu_1,\ldots,\mu_{i-1}, \alpha,\beta, \mu_{i+1},\ldots, \mu_\ell $ is globally admissible. Then  the sequence $\mu_1,\ldots,\mu_i,\ldots,\mu_\ell$ is also globally admissible.
\end{lemma}

	\begin{proof}
			The assumption implies that the pair $\alpha,\beta$ is globalisable, so \cref{le:muiloc} already says that $\mu_i$ is locally admissible. It remains to show that for any $j \neq i$, the pair $\mu_j,\mu_i$ is globalisable.
			
\medskip
			
\noindent \underline{\textit{Case $\mu_j < \alpha < \mu_i < \beta$}.} By assumption $\mu_j,\alpha$ is globalisable.
\begin{itemize}
\item If $\mu_j,\alpha$ satisfies (G1), we have $1 < \mu_j < \alpha$ and there exists $m \in (r_{j \alpha} - 1]$ such that $\frac{1}{\alpha} \leq \frac{m - 1}{m} \leq \frac{1}{\mu_j}$. Taking into account $\alpha < \mu_i$ and $r_{\alpha} < r_i$ we deduce $\frac{1}{\mu_i} \leq \frac{m - 1}{m} \leq \frac{1}{\mu_j}$ with $m \in (r_{ij} - 1]$. Since $1 < \mu_j < \mu_i$, we conclude that $\mu_i,\mu_j$ is globalisable by (G1).
\item If $\mu_j,\alpha$ satisfies (G2), we have $1 < \mu_j < \alpha$ and $\frac{1}{\mu_j} \geq \frac{r_{j \alpha} - 1}{r_{j \alpha}}$. When $r_j \geq r_i$, since $r_i > r_{\alpha}$ we have $\frac{1}{\mu_i} \geq \frac{r_{ji} - 1}{r_{ij}}$ so $\mu_j,\mu_i$ is globalisable by (G2). When $r_j < r_i$, recalling that $\alpha > 1$ guarantees $\bar s_{\alpha} \geq 1$, we write $\frac{1}{\mu_i} < \frac{1}{\alpha} = \frac{r_{\alpha} - \bar s_{\alpha}}{r_{\alpha}} \leq \frac{r_{\alpha} - 1}{r_{\alpha}} \leq \frac{r_{j\alpha} - 1}{r_{j \alpha}} \leq \frac{1}{\mu_j}$. Taking $m = r_{\alpha} \leq r_j - 1$ we see that the pair $\mu_j,\mu_i$ is globalisable by (G2).
\item If $\mu_j,\alpha$ satisfies (G3), then $\mu_j \leq 1$ and thus $\mu_j,\mu_i$ is globalisable by (G3).
\end{itemize} 

\noindent \underline{\textit{Limit case $\mu_j = \alpha < \mu_i < \beta$}.} By assumption $\mu_j,\alpha$ is globalisable, and \cref{le:coincm} imposes $\mu_j \leq 1$ (only relevant for (G3)) or $s_{j} = 1$. 
\begin{itemize}
\item If $\mu_j,\beta$ satisfies (G1), we have $1 < \mu_j$ and there exists $m \in (r_{j \beta} - 1]$ such that $\frac{1}{\beta} \leq \frac{m - 1}{m} \leq \frac{1}{\mu_j}$. As $r_j = r_{\alpha} < r_i$ and $r_{\beta} < r_i$, the fraction $\frac{m - 1}{m}$ is in the Farey sequence $\mathcal{F}_{r_i}$. We recall that $\frac{1}{\beta},\frac{1}{\mu_i},\frac{1}{\alpha} = \frac{1}{\mu_j}$ are Farey neighbours of $\mathcal{F}_{r_i}$. If $\frac{1}{\mu_i} \leq \frac{m - 1}{m} \leq \frac{1}{\mu_j}$ with $m \in (r_{ij} - 1]$, then $\mu_j,\mu_i$ is globalisable by (G1). If $\frac{1}{\mu_i} > \frac{m - 1}{m}$, we must have $\frac{1}{\beta} = \frac{m - 1}{m}$ and then $r_{\beta} = m$ and $\bar s_{\beta} = 1$. Recalling that $s_j = 1$, we obtain $\bar s_i = 2$. B\'ezout's identities for the triple of Farey neighbours $\frac{m - 1}{m},\frac{r_i - 2}{r_i},\frac{r_j - 1}{r_j}$ then yield $r_i = 2m + 1$ and $r_j = m + 1 < r_{i} = r_{ij}$, that is $(\frac{1}{\beta},\frac{1}{\mu_i},\frac{1}{\mu_j}) = (\frac{m - 1}{m},\frac{2m - 1}{2m + 1},\frac{m}{m + 1})$, and we conclude that $\mu_j,\mu_i$ still satisfies (G1) using $m + 1$.
\item If $\mu_j,\beta$ satisfies (G2), we have $1 < \mu_j$ and $\frac{1}{\mu_j} \geq \frac{r_{j \beta} - 1}{r_{j \beta}}$, meaning in fact $\frac{1}{\mu_j} = \frac{r_j - 1}{r_j}$ and in particular $\bar s_j = 1$. When $r_i \leq r_j$, the latter is also $\frac{r_{ij} - 1}{r_{ij}}$, hence the pair $\mu_j,\mu_i$ is globalisable by (G2). When $r_i > r_j$, we have clearly $\frac{1}{\mu_i} \leq \frac{r_j - 1}{r_j} \leq \frac{1}{\mu_j}$ hence the pair $\mu_j,\mu_i$ is globalisable by (G1) with $m = r_j \in (r_i - 1]$.
 \item If $\mu_j,\beta$ satisfies (G3), we have $\alpha \leq 1$. So the pair consisting of $\alpha = \mu_j$ and $\mu_i$ also satisfies (G3).
\end{itemize}

\medskip

\noindent \underline{\textit{Case $\alpha < \mu_i < \beta \leq \mu_j$}.} The analysis is similar to the two previous cases and is omitted.
\end{proof}

\section{Analyticity of topological recursion in families}
\label{S5}

In this section, we investigate how topological recursion behaves in families. We will consider families of spectral curves over a base $T$, and analyse the behavior of the topological recursion correlators as we move in the family. In particular, we are interested in knowing when we can take limits, i.e., when the correlators $ \omega_{g,n}$ at a certain point in the family can be recovered by taking an appropriate limit of the correlators at nearby points.\par

In the local topological recursion, the integrand in \cref{de:TR} of the recursion is only locally defined on the curve: it is different for each ramification point $ p  \in \Sigma $ that we take the residue at, and the structure of the neighbourhoods where they are locally defined as well as the behavior of the branched covering $x$ there could change when we deform the spectral curve. For instance, if we deform a locally admissible spectral curve we could end up with a spectral curve that is not locally admissible or vice versa, and in such cases the local topological recursion typically cannot hold or does not make sense for all deformation parameters.

The global topological recursion is a tool to bypass these difficulties and it enables studying the behaviour of topological recursion under deformations and/or limits of spectral curves. More precisely, we take the following approach. Given a spectral curve, we first use the results of  \cref{S3,S4} to globalise the topological recursion as defined in \cref{de:globalTR}. Recall that horizontal globalisation is a rewriting of the recursion formula as a sum of contour integrals containing clusters of ramification points. Now, we can deform the spectral curve while keeping the $x$-projection of the contours fixed so that the ramification points remain inside the contour (and perhaps collide there), and finally transform back from the global to the local topological recursion on the deformed spectral curve of interest. In summary, the strategy is
$$
\text{local to global} \quad \longrightarrow \quad \text{take limit or deform} \quad \longrightarrow \qquad \text{global to local}\,.
$$

This procedure can be made to work under various conditions: the local admissibility conditions, the sufficient conditions we found for globalisation, and conditions on the type of deformations of spectral curves that are allowed. In this section, we will formulate conditions on families of spectral curves guaranteeing that the correlators behave nicely in families. One important request is that one can find adapted disc collections so that the branch points do not escape the discs as we vary the spectral curve. Our main result is then \cref{th:TRLimits} which says that the topological recursion correlators are analytic over the base of the family.

This nevertheless raises an interesting question: why not start directly with global topological recursion and avoid the discussion of whether local topological recursion is equivalent to global topological recursion? That is, why not take any family of spectral curves and simply study global topological recursion directly for this family? The  difficulty with this approach is that one would need to show that global topological recursion is well-defined in the first place, i.e. that it produces symmetric differentials, and this is far from obvious a priori (see the discussion in \cref{Sec:intermz}). By starting with local topological recursion, we avoid this issue, as we know that the correlators produced by local topological recursion are well-defined symmetric differentials, assuming that the spectral curve is locally admissible (\cref{th:ale}). The price to pay is that, to study limits and deformations, we need to assume that local topological recursion is equivalent to global topological recursion, for instance by requiring that the sufficient conditions of \cref{th:rewriting} are satisfied.

\subsection{Admissible families of spectral curves}

\label{Sec51}

\begin{definition}[Family of spectral curves]\label{de:FamSC}
A \emph{family of spectral curves} over a connected complex manifold $ T$ is a quadruple \hbox{ $ \mc{S}_T = (f \colon \Sigma_T \rightarrow T, x_T, (\omega_{0,1})_T,(\omega_{0,2})_T)$ } such that
\begin{enumerate}
	\item $ f$ is a surjective submersion of complex manifolds with $1$-dimensional fibres $\Sigma_t = f^{-1}(\{t\})$ having finitely many connected components. We denote by $\iota_t \colon \Sigma_t \hookrightarrow \Sigma_T$ the natural inclusion.
	\item $x_T \colon \Sigma_T \rightarrow \P^1 $ is a holomorphic map whose restriction to each connected component of the fibres of $f$ is not constant:
		\item $(\omega_{0,1})_T $ is a meromorphic section of the relative cotangent bundle $ \Omega_f$, whose divisor of poles and zeros intersects each fibre on a discrete set;
	\item $(\omega_{0,2})_T \in H^0 ( \Sigma_T\times_T\Sigma_T ; \Omega_f^{\boxtimes 2} ( 2 \Delta ))$ with biresidue $1$ along the diagonal $ \Delta_T \subseteq \Sigma_T \times_T \Sigma_T $.
\end{enumerate}
We say that the family of spectral curves is \textit{proper} if the map $ f : \Sigma_T \rightarrow T $ is proper. We call \emph{partial spectral curve} the data of $(f : \Sigma_T \rightarrow T,x_T)$ satisfying in (1) and (2).
\end{definition} 
 
 In practice, the base  $ T $ of the family can be taken to be a open subset of $\mathbb{C}^D$. We do not necessarily restrict to proper families, as there are various situations in topological recursion (most notably, in applications to Hurwitz theory) where the family happens not to be proper. In lack of properness there will be additional subtleties that we should treat carefully. For instance, we could have situations where ramification points move off the curve, and we will want to exclude such pathologies. Nevertheless, proper families form an important class of families of spectral curves that we treat in detail based on \cref{s:algcurves}. In this case, $T$ is the parameter space for the coefficients of the bihomogeneous polynomial equation cutting out the spectral curve.

Given a family of spectral curves as above, the \textit{fibre} over $ t \in T $ is the spectral curve
\[
	\mc{S}_t \coloneqq \big( \Sigma_t, x_t \coloneqq x_T \circ \iota_t, (\omega_{0,1})_t \coloneqq \iota_t^*(\omega_{0,1})_T, (\omega_{0,2})_t \coloneqq (\iota_t, \iota_t)^* (\omega_{0,2})_T \big)\,.
\]
Note that the fibre of a family of proper spectral curves is a compact spectral curve as defined in \cref{de:sc}. We define the map 
\begin{equation}\label{eq:Xdef}
	X \coloneqq (x_T,f) : \Sigma_T \rightarrow \mathbb P^1 \times T\,,
\end{equation}
and denote $\mathsf{Ram}_T$ the zero locus of $\dd X$, which is a closed complex submanifold of $ \Sigma_T $ intersecting each fibre in a discrete set, and $\mathsf{Br} = x_T(\mathsf{Ram}_T) \subset \mathbb{P}^1$. For later use, we introduce $f_n \colon \Sigma^n_T \rightarrow T$, defined as the cartesian product over $T$ of $n$ copies of $f: \Sigma_T \rightarrow T$. We set $\mathsf{Ram}^n_T = \bigsqcup_{i = 1}^n {\rm pr}_i^*(\mathsf{Ram}_T)$, where ${\rm pr}_i \colon \Sigma_T^n \rightarrow \Sigma_T$ is the projection onto the $i$-th factor.

We shall add conditions on our families of spectral curves in order to establish an analytic dependence of the correlators on points $ t \in T $. The most basic ones are finiteness conditions on $x_t$ (useful for globalisation) and conditions preventing ramification points from ``escaping the curve'' as we vary in family.  If a ramification point were escaping the curve, we would lose its contribution to the correlators, and unless this contribution is identically zero for obvious reasons (i.e. $s_p \leq -1$)  an analytic dependence of the correlators on $ t $ cannot in general be expected (however, in \cref{sec:flyoff} we give an example with this pathology where compatibility with limits is independently known to hold). Then, we also want to impose the sufficient conditions for globalisation found in \cref{S3}. In order to take advantage of horizontal globalisation \cref{pr:globalTRdomain}, we need a notion of disc collections upgrading \cref{de:adapcon} to work in families.

\begin{definition}[Disc collection in families]\label{de:dc}
Let $\mathcal{S}_T$ be a family of spectral curves and $\mathcal{D} = x_t(\Sigma_t)$ be independent of $t \in T$. Call $(\Pi_{0,2})_t$ the period map \eqref{02periodes} corresponding to $(\omega_{0,2})_t$ on $\mathcal{S}_t$. A \emph{disc collection adapted} to $\mathcal{S}_T$ is a (independent of $t \in T$) finite sequence of open subsets $(\mathsf{D}_i)_{i = 1}^{\mathsf{k}}$ of $\mathbb{P}^1$ such that $\overline{\mathsf{Br}} \subseteq \bigcup_{i = 1}^{\mathsf{k}} \mathsf{D}_i$ and
\begin{itemize}
\item[(DC1)] the $\overline{\mathsf{D}_i}$ are pairwise disjoint, properly embedded discs in $\mathcal{D}$;
\item[(DC2)] we have for any $t \in T$, each $\mathsf{D}_i$ contains at least one branch point of $x_t$, and each branch point of $x_t$ belongs to some $\mathsf{D}_i$;
\item[(DC3)] for any $t \in T$ the restriction of $x_t$ to each connected component of $x_t^{-1}(\mathsf{D}_i)$ is a finite-degree branched covering onto $\mathsf{D}_i$, the list of their degree (taken in weakly decreasing order) is independent of $t$ and has only finitely many entries different from $1$;
\item[(DC4)] for any $t \in T$ and $i \in [\mathsf{k}]$, we have $H_1(x_t^{-1}(\mathsf{D}_i),\mathbb{Z}) \subseteq {\rm Ker}\,(\Pi_{0,2})_t$.
\end{itemize}
\end{definition}
\begin{remark}
The requirement that $x_t(\Sigma_t)$ is independent of $t$ is not a big restriction: one can restrict to the region of the $x$-plane one is interested in and simply take $\Sigma_t = x_t^{-1}(\mathcal{D})$. What is important is the existence of $t$-independent region of the $x$-plane delimiting where the contours can move in the step of horizontal globalisation.
\end{remark}

Given such a disc collection and choosing $t_0 \in T$ and an unramified point $p_{i,j}$ in each connected component of $x_{t_0}^{-1}(\mathsf{D}_i)$, and assuming that $T$ is simply-connected, we get unique analytic sections $\mathsf{p}_{i,j} : T \rightarrow \Sigma_T$ such that $\mathsf{p}_{i,j}(t_0) = p_{i,j}$. By following to which component $\mathsf{p}_{i,j}(t)$ belongs for $t \in T$, this allows to label the connected components of $x_t^{-1}(\mathsf{D}_i)$ as $\tilde{\mathsf{D}}_{i,j,t}$ by an index $j$ in the $t$-independent set $\mathsf{c}_i = x_{t_0}^{-1}(\mathsf{D}_i)$.  Due to (DC3), the degree $\mathsf{d}_{i,j}$ of the restriction of $x_t$ to $\tilde{\mathsf{D}}_{i,j,t}$ is independent of $t$. We denote again
$$
\mathsf{c}_i^+ = \{j \in \mathsf{c}_i\,\,|\,\,\mathsf{d}_{i,j} \geq 2\}
$$
labelling those connected components containing at least one ramification point (this property is independent of $t$), and
$$
\tilde{\mathsf{D}}^+ = \bigcup_{i = 1}^{\mathsf{k}} \bigcup_{j \in \mathsf{c}_i^+} \tilde{\mathsf{D}}_{i,j}\,. 
$$
  
\begin{definition}[Globally admissible families]\label{de:admfamily}
A family of spectral curves $ \mc{S}_T$  is \textit{globally admissible} if the ramification locus $\mathsf{Ram}_T $ is proper over $ T $ (i.e. the composition $\mathsf{Ram}_T \hookrightarrow \Sigma_T \twoheadrightarrow T $ is proper) and if for every $t \in T$:
\begin{itemize}
\item the fibre $ \mc{S}_t $ is a locally admissible spectral curve in the sense of \cref{de:localadm};
\item $(\omega_{0,1})_t $  separates fibres at every point $ q  \in  \mathcal{D}$ which is not a branch point of $x_t$;
\item the image $\mathcal{D} = x_t(\Sigma_t)$ is connected and independent of $t$;
\item there exists a simply-connected open neighborhood $T' \subseteq T$ of $t$ and a disc collection adapted to $\mathcal{S}_{T'}$ such that for any $t' \in T$, topological recursion is globalisable (for instance, through the sufficient conditions of \cref{th:rewriting}) over $x_{t'}^{-1}(q) \cap \tilde{\mathsf{D}}_{i,j,t'}$ for each branch point $q \in \mathsf{D}_i$ of $x_{t'}$ and each $j \in \mathsf{c}_i^+$.
\end{itemize}
\end{definition}

If $x_t$ is finite, we can also work with stronger but simpler assumptions.

\begin{definition}[Globally admissible families in finite-degree setting]\label{de:admfamilyfinite} A family of spectral curves $\mathcal{S}_T$ is \emph{globally admissible in the finite-degree setting} if the ramification locus $\mathsf{Ram}_{T}$ is proper over $T$ and if for every $t \in T$
\begin{itemize}
\item $x_t$ is a finite-degree branched covering onto a $t$-independent image $\mathcal{D} = x_t(\Sigma_t)$ whose degree is independent of $t$;
\item the fibre $ \mc{S}_t $ is a locally admissible spectral curve in the sense of \cref{de:localadm};
\item  $(\omega_{0,1})_t $  separates fibres at every point $ q  \in  \mathcal{D}$ which is not a branch point of $x_t$;
\item topological recursion is globalisable over $x^{-1}_t(q)$ (for instance, through the sufficient conditions of \cref{th:rewriting}) for each branch point $q$ of $x_t$.
\end{itemize}
\end{definition}
\begin{lemma}\label{le:finset}
If a family of spectral curves $\mathcal{S}_T$ is globally admissible in the finite-degree setting (in the sense of \cref{de:admfamilyfinite}), then it is globally admissible (in the sense of \cref{de:admfamily}).
\end{lemma}
\begin{proof} 
This is all about checking the existence of disc collection locally adapted to the family. 
Note that in the setup of \cref{de:admfamilyfinite} the associated partial spectral curve forms a family of branched coverings over $\mathcal{D}$ which is automatically continuous in the standard topology of the space of smooth branched coverings. 

Let $t \in T$, and denote $d$ the degree of $x_t$, which is by assumption finite and $t$-independent. We can take small enough properly embedded smooth discs $\mathsf{D}_i \subseteq \mathcal{D}$, each of them containing a single branch point of $x_t$, and such that $x_t^{-1}(\mathsf{D}_i)$ is a union of less than $d$ properly embedded smooth discs. Then, for $t'$ in a small enough neighborhood $T' \subseteq T$ of $t$, all branch points of $x_{t'}$ for $t' \in T$ remain in $\mathsf{D}_i$ (for $t' \neq t$ there can be more than one in a given $\mathsf{D}_i$) and the connected components $x_{t'}^{-1}(\mathsf{D}_i)$ remain contractible (this latter property refers to the standard description of the topology of the space of smooth branched coverings). Then the four conditions of \cref{de:admfamily} are matched.
\end{proof}

In the compact setting, the conditions can be simplified further using our analysis in \cref{s:algcurves}.

\begin{lemma}\label{le:admproper}
	Assume that  $ \mc{S}_T$ is a proper family of spectral curves such that for every $t \in T$, $\mathcal{S}_t$ is a compact spectral curve obtained from a reduced affine curve as in \cref{de:irredpc} and it is a globally admissible spectral curve in the sense of  \cref{de:globaladm}.  Assume that the genus of $\mathcal{S}_t$ is independent of $t$. Then, the family $ \mc{S}_T $ is globally admissible in the finite-degree setting.
	\end{lemma} 
	
	It may seem that \cref{le:admproper} does not cover all proper families of spectral curves, but it does cover all  proper families that are potentially globally admissible. Indeed, the requirement in \cref{de:admfamilyfinite} that $ (\omega_{0,1})_t $  separates fibres at every point $ q  \in \mathbb P^1 $ which is not a branch point implies after \cref{pr:everycurve}  that $ \mc{S}_t $ is obtained from a reduced curve in $ \mathbb P^1 \times \mathbb P^1 $ in the sense of \cref{de:irredpc}.

\begin{proof}[Proof of \cref{le:admproper}]

	The composition $\mathsf{Ram}_T \rightarrow \Sigma_T \rightarrow T $ is proper as it is the composition of two proper maps.
		
	The fibre $ x_t : \Sigma_t \rightarrow \mathbb P^1 $ is a non-constant morphism of compact Riemann surfaces, hence a finite-degree branched covering onto its image $\mathcal{D} = \mathbb{P}^1$. Thus, the morphism $ X : \Sigma_T \rightarrow \mathbb P^1 \times T $ (viewed as a morphism of algebraic varieties) is quasi-finite, i.e., has finite fibres. The projection $ \pi_2 \colon \P^1 \times T \rightarrow T $ and the map $ f = \pi_2 \circ X$ are both proper, hence so is $ X$. By \cite[02OG]{stacks-project}, a proper map with finite  fibres is finite. Notice furthermore that as the morphism $ X  : \Sigma_T \rightarrow \mathbb P^1 \times T$ is finite and flat and $T$ is connected,  the degree of the branched covering $ x_t : \Sigma_t \rightarrow \mathbb{P}^1$ is independent of $t \in T$ \cite[Lemma 37.53.8]{stacks-project}. This checks the first condition in \cref{de:admfamilyfinite}.
	
 By \cref{th:ga}, the assumptions imply that $ \mc{S}_t $ is locally admissible, $ (\omega_{0,1})_t $  separates fibres at every point $ q  \in \mathbb P^1 $ which is not a branch point of $x_t$ and the conditions for globalisation in \cref{th:rewriting} are  satisfied for the sets $ x_t^{-1}(q) $ for every branch point $q \in \mathbb{P}^1$ of $x_t$.  This proves the last three conditions in \cref{de:admfamilyfinite}.

\end{proof}

In \cref{de:irredpc}, starting from a bivariate polynomial $P$ defining a reduced affine curve $\mathfrak{C} = \{P(x,y) = 0\}$,  we have constructed a compact curve by taking a normalisation $\eta : \tilde{C} \rightarrow C$. When doing this in family, some care is needed to guarantee the existence of a simultaneous normalisation. In \cref{Sec45},  we studied enlargements $\Delta$ of the Newton polygon of $P$ preserving the set of interior points. We thus get a family of bivariate polynomials $(P_t)_{t \in \mathbb{C}^D}$ such that $P_t$ has Newton polygon $\Delta$ for generic $t$ and $P = P_0$. Assuming that $P$ is nondegenerate and up to restricting the base to an open dense $T \subseteq \mathbb{C}^D$ containing $0$, we can assume that for any $t \in T$, $P_t$ is nondegenerate and its Newton polygon contains the one of $P$ and is contained in $\Delta$. In particular, its set of interior points is independent of $t \in T$. Let us denote this family of curves by $ C_T $.

Nondegeneracy implies (\cref{le:ndnp}) that the $\delta$-invariants of the projectivisation $C_t \subseteq \mathbb{P}^1 \times \mathbb{P}^1$ of $\mathfrak{C}_t = \{P_t(x,y) = 0\}$ (cf. \cref{de:irredpc}) are independent of $t \in T$. In such a  situation where the $ \delta $-invariant is  constant in the family, a simultaneous normalisation $\eta : \tilde{C}_T \rightarrow C_T$  is known to exist \cite{Tei76, CL06} and we obtain a proper family of  curves $f \colon \tilde{C}_T \rightarrow T$ together with $x_T$ and $(\omega_{0,1})_T = y_T \dd x_T$ as in \cref{de:FamSC}. 

To upgrade it to a proper family of spectral curves, it should be equipped with a choice of bidifferential $(\omega_{0,2})_T$. This can always be done (not in a unique way) in a connected, simply-connected open subset $T' \subseteq T$. Indeed, the relative homology of $f \colon \tilde{C}_T \rightarrow T$ carries the Gau\ss{}--Manin connection. Taking $t_0 \in T'$ and choosing a Lagrangian sublattice $L_{t_0} \subseteq H_1(\tilde{C}_{t_0},\mathbb{Z})$ (with the intersection pairing determining the symplectic structure), the Gau\ss{}--Manin connection transports it to Lagrangian sublattices $L_t \subseteq H_1(\tilde{C}_t,\mathbb{Z})$. Then, for each $t \in T'$, there exists a unique meromorphic bidifferential $(\omega_{0,2})_t$ with biresidue $1$ on the diagonal of $\tilde{C}_{t} \times \tilde{C}_t$ and such that
$$
\forall \gamma \in L_t,\qquad \int_{\gamma} (\omega_{0,2})_t(w_0,\cdot) = 0\,.
$$
Using local trivialisations for $f$, one can check that the so-determined $(\omega_{0,2})_T$ fulfills the condition (3) in \cref{de:irredpc}.

The constant genus assumption was used both for the existence of the simultaneous normalisation and in the construction of $(\omega_{0,2})_T$.

\subsection{The main result}

If we have a family $\mathcal{S}_T$ of spectral curves for which each fibre $\mathcal{S}_t$ for $t \in T$ is locally admissible, topological recursion yields a system of correlators $(\omega_{g,n})_t$ indexed by $(g,n) \in \mathbb{Z}_{\geq 0} \times \mathbb{Z}_{> 0}$ and symmetric under action of the symmetric group $\mathfrak{S}_n$. They define meromorphic sections $(\omega_{g,n})_T$ of $\Omega_{f_n}(*\mathsf{Ram}^n_T)$ on $\Sigma_T^n$.  In this section we investigate continuity (in fact, analyticity) of $(\omega_{g,n})_t$ with respect to $t \in T$, based on the following simple observation.

\begin{lemma}\label{le:LimitsOfContours}
	Let $ \mc{S}_{T}$ be a family of partial spectral curves. Let $ \Gamma $ be a contour in $ \P^1$. If $ \phi$ is a meromorphic section of $\Omega_f$ whose divisor of poles intersects each fibre on a discrete set and avoids $ x_{T}^{-1}(\Gamma)$, then $ \int_{x_t^{-1}(\Gamma)} \iota_t^{*} \phi $ is an analytic function of $t \in T$. 
\end{lemma}

\begin{proof}
We write:
$$
\int_{x_t^{-1} (\Gamma)} \iota_t^* \phi = \int_{X^{-1} ( \Gamma \times \{ t\} )} \phi \,.
$$
By assumption, $\iota_t^*\phi$ is a meromorphic $1$-form on $\Sigma_t$ and the set of poles of $ \phi $ does not intersect $ x_T^{-1}(\Gamma)$. This immediately implies that the divisor of poles of $ \phi $ does not intersect $ X^{-1} (\Gamma \times \{t\}) $, where $X: \Sigma_T \to \mathbb P^1 \times T$ was defined in \cref{eq:Xdef}. Hence, the integral is an analytic function of $t \in T $.
\end{proof}

\begin{theorem}[Topological recursion is analytic in globally admissible families]\label{th:TRLimits}
	Let $ \mc{S}_T = ( f \colon \Sigma_T, x_T, (\omega_{0,1})_T, (\omega_{0,2})_T) $ be a globally admissible family of spectral curves over a manifold $ T$. Then, for any $(g,n) \in \mathbb{Z}_{\geq 0}^2$, the correlator $ (\omega_{g,1+n})_t$ obtained by applying the topological recursion to a fibre $\mathcal{S}_t$, depends analytically on $t \in T$, in the sense that $(\omega_{g,1+n})_T$ is a meromorphic section over $ \Sigma_T \times_T \dotsb \times_T \Sigma_T $ of $ \Omega_f^{\boxtimes (1+n)}$. If $2g - 2 + (1 + n) > 0$, it has poles on $\mathsf{Ram}_T^n$.	
\end{theorem}

\begin{proof}
	
	 The base cases $ \omega_{0,1} $ and $ \omega_{0,2} $ are covered by  \cref{de:FamSC}. Now take $(g,n) \in \mathbb{Z}_{\geq 0}^2$ and assume the theorem is proved for all $\omega_{g',1+n'}$ with $(g',n') \in \mathbb{Z}_{\geq 0}^2$ such that $2g' - 2 + (1+n') < 2g - 2 + (1 + n)$. We would like to prove that $(\omega_{g,1+n})_t$ is a meromorphic section of $\Omega_{f_{n + 1}}$ over $\Sigma_T^{n + 1}$ with poles at $\mathsf{Ram}_{T}^{n + 1}$. As this is a local property, up to restricting to small neighborhood of arbitrary points in the parameter space we can assume that  $T$ is simply-connected and that we are given a disc collection $(\mathsf{D}_i)_{i = 1}^{\mathsf{k}}$ uniformly adapted to $\mathcal{S}_{T}$.
	 
	 Then, for each $t \in T$ the assumptions in \cref{de:admfamily} allow us a horizontal globalisation as in \cref{pr:globalTRdomain}, i.e. we can rewrite
	 \begin{equation}\label{eq:integrandTR}
	 		(\omega_{g,1+n})_t(w_0, w_{[n]}) = \frac{1}{2{\rm i}\pi} \sum_{\substack{1 \leq i \leq \mathsf{k} \\ j \in \mathsf{c}_i^+}} \oint_{\gamma_{i,j,t}} \sum_{\substack{Z \subseteq  \mathfrak{f}'(z) \,\cap\,\tilde{\mathsf{D}}_{i,j,t} \\ |Z| \geq 1}} K_{1+|Z|}^{(i,j,t)}(w_0; z, Z) (\mathcal{W}'_{g,1+|Z|;n})_t(z, Z; w_{[n]}) \,.
	 \end{equation}
	 for $w_0,\ldots,w_n$ kept outside of $\tilde{\mathsf{D}}^+$. The various elements of this formula were described in \cref{Sec:correldef}. The details are not important for the present argument, except for four things that we should discuss: the choice of local primitive for the recursion kernel, the contour of integration $\gamma_{i,j,t}$, and the set $\mathfrak{f}'(z) \,\cap\,\tilde{\mathsf{D}}_{i,j,t}$.
	 	 
For the first thing: to define the recursion kernel $K^{(i,j,t)}$, we choose the local primitive $\alpha_{0,2}^{(i,j,t)}(w_0;z) = \int_{\mathsf{p}_{i,j}(t)}^{z} \omega_{0,2}(\cdot,w_0)$ involving base points $\mathsf{p}_{i,j}(t) \in \tilde{\mathsf{D}}_{i,j,t}$ introduced below \cref{de:dc} and (most importantly) depending analytically on $t$. For the second thing: the contour $\gamma_{i,j,t}$ is a union of the connected components of $x_t^{-1}(\gamma_i)$ that are in $\tilde{\mathsf{D}}_{i,j}$, where $\gamma_i$ is a obtained by pushing the (positively oriented) $\partial\mathsf{D}_i$ into $\mathsf{D}_i$, not crossing $\overline{\mathsf{Br}_T}$. Up to take a smaller $T$, we can also arrange that $\gamma_i$ surrounds $\overline{x_T \circ \mathsf{p}_{i,j}(T)}$ with index $+1$.
For the third thing: the sets $\mathsf{c}_i^+$ indexing the sums are finite and independent of $t$. It then suffices to examine the $t$-analyticity of each term $(i,j)$ separately. Since the restriction of $x_t$ to $\tilde{\mathsf{D}}_{i,j,t}$ is a finite-degree branched covering, we can take as partial spectral curve
	 $$
	 (\Sigma_{i,j})_T = \{z \in \Sigma_T\,\,|\,\,f(z) = t\,\,{\rm and}\,\,z \in \tilde{\mathsf{D}}_{i,j,t} \subseteq \Sigma_t\}\,,
	 $$
	 equipped with the restriction of $x_T$ to $(\Sigma_{i,j})_T$ that we denote $(x_{i,j})_T$. Then the set $\mathfrak{f}'(z) \cap \tilde{\mathsf{D}}_{i,j,t}$ appearing in \eqref{eq:integrandTR} is simply $\mathfrak{f}'(z)$ with respect to this partial spectral curve. So, we can write the $(i,j)$ term of \eqref{eq:integrandTR} as
	 \begin{equation}
	\label{phihphi}\bigg(\oint_{(x_{i,j})_t^{-1}(\gamma_i)} \iota_t^*\phi_{i,j}(\cdot;w_0,\ldots,w_n)\bigg) \prod_{i = 0}^{n} \dd x_t(w_i)\,,
	 \end{equation}
and in this expression it is sufficient to keep $w_0,\ldots,w_n$ outside $\tilde{\mathsf{D}}_{i,j}$ (instead of $\tilde{\mathsf{D}}^+$).

To handle the parametric dependence in $w_0,\ldots,w_n$ of the right-hand side, we change base to
$$
\tilde{T} = W_{i,j}^{(\emptyset)} := \big\{(w_0,\ldots,w_n) \in \Sigma_T^{n + 1}\,\,|\,\,w_0,\ldots,w_n \in \Sigma_T \setminus \tilde{\mathsf{D}}_{i,j,t}\,\,{\rm with}\,\,t = f(w_0) = \cdots = f(w_n)\big\}\,.
$$
Pulling back $\mathcal{S}_T$ via $f_{n + 1} : W_{\emptyset} \rightarrow T$ yields a new family of partial spectral curves over $\tilde{T}$, with defining morphisms $\tilde{f} : \Sigma_T \times_T W_{\emptyset}^{(i,j)}$ and $x_{\tilde{T}} = x_T \circ {\rm pr}_{\Sigma_T}$. Then, from the expression of $\phi_{i,j}$ is in terms of the $\omega_{g',n'}$ with $2g' - 2 + n' < 2g - 2 + (1 + n)$ one deduce that it is a meromorphic section of $\Omega_{\tilde{f}}$. Thanks to \cref{le:LimitsOfContours} we conclude that
\begin{equation}
\label{thephiform}
\frac{1}{2{\rm i}\pi} \oint_{(x_{i,j})_t^{-1}(\gamma_i)} \iota_t^*\phi(\cdot;w_0,\ldots,w_n)
\end{equation}
is an analytic function of $(w_0,\ldots,w_n) \in W_{i,j}^{(\emptyset)}$. From our definition of families, $\prod_{i = 0}^{n} \dd x_t(w_i)$ is a meromorphic section of $\Omega_f^{\boxtimes (n + 1)}$ over $\Sigma_T^{n + 1}$, therefore \eqref{phihphi} is a meromorphic section of $\Omega_f{\boxtimes(n + 1)}$ over $W_{i,j}^{(\emptyset)}$. 

We would like to extend this argument to $w_0,\ldots,w_n$ in $\Sigma_T$. For this purpose we choose nested contours $(\gamma_{i}^{(\ell)})_{\ell = 0}^{n + 1}$, which all represent the homology class of (positively oriented) $\partial\mathsf{D}_i$ in the complement of $\overline{\mathsf{Br}}$ in  $\mathcal{D}$, such that $\gamma_i^{(0)} := \gamma_i$ and for each $\ell \in [n + 1]$, $\gamma_i^{(\ell)}$ is obtained by slightly pushing off $\gamma_i^{(\ell - 1)}$ towards the inside of $\mathsf{D}_i$. We denote $\mathsf{D}_i^{(\ell)} \subset \mathsf{D}_i$ the disc bounded by $\gamma_i^{(\ell)}$ and $\mathsf{D}_{i,j,t}^{(\ell)} = x_t^{-1}(\mathsf{D}_i^{(\ell)}) \cap \tilde{\mathsf{D}}_{i,j,t}$. Since $(n + 1)$ points cannot meet $(n + 2)$ pairwise disjoint sets, the open sets
$$
W_{i,j}^{(\ell)} = \Big\{(w_0,\ldots,w_n) \in \Sigma_T^{n + 1}\,\,\Big|\,\,\forall a \in \{0,\ldots,n\}\quad x_T(w_a) \notin \gamma_i^{(\ell)}\Big\}
$$
indexed by $\ell \in \{0,\ldots,n + 1\}$ cover $\Sigma_T^{n + 1}$. For each $A \subseteq \{0,\ldots,n\}$, we have a connected component $W_{i,j}^{(A,\ell)} \subseteq W_{i,j}^{(\ell)}$ consisting of the $(n + 1)$-tuples $(w_0,\ldots,w_n)$ such that $w_a \in \tilde{D}_{i,j,t}^{(\ell)}$ if and only if $a \in A$, where $t = f(w_0) = \cdots = f(w_n)$. 

Up to replacing $\gamma_i$ with $\gamma_i^{(\ell)}$, the previous argument showed analyticity of $\omega_{g,1+n}(w_0,\ldots,w_n)$ over $W^{(\emptyset,\ell)}_{i,j}$. Now take a nonempty subset $A \subseteq \{0,\ldots,n\}$ and consider $(w_0,\ldots,w_n) \in W_{i,j}^{(A,\ell)}$. There is one more thing we should remember about the integrand $\iota_t^*\phi_{i,j}(z;w_0,\ldots,w_n)$ which originates from \eqref{eq:integrandTR}: with respect to $z \in \tilde{\mathsf{D}}_{i,j,t}^{(\ell)} \setminus \mathsf{Ram}_t$:
\begin{itemize}
\item if $0 \in A$ it has a singularity at $z = w_0$ due to the simple pole in the recursion kernel;
\item  it may have singularities at $z \in \bigcup_{a \in A}\mathfrak{f}(w_a)$  due to the presence of $\omega_{0,2}(z',w_a)$ with $z' \in \mathfrak{f}(z)$ in the topological recursion formula and the fact that $\omega_{0,2}(z',w_a)$ has a double pole at $z' = w_a$. Here, $\mathfrak{f}$ should be interpreted as the $x_t$-fibre restricted to $\tilde{\mathsf{D}}_{i,j,t}^{(\ell)}$.
\end{itemize}
Since the horizontal globalisation of \cref{pr:globalTRdomain} came from moving contours from small circles around the ramification points in $\tilde{\mathsf{D}}_{i,j}$ to $\gamma_{i,j}^{(\ell)}$ close to $\partial\tilde{\mathsf{D}}_{i,j}^{(\ell)}$, in this case the formula \eqref{thephiform} is replaced with
$$
\bigg(\frac{1}{2{\rm i}\pi} \oint_{(x_{i,j})_t^{-1}(\gamma_i)} \iota_t^*\phi(\cdot;w_0,\ldots,w_n)\Bigg) - \Big(\Res_{z = w_0} + \sum_{\substack{a \in A \\ p \in \mathfrak{f}(w_a)}} \Res_{z = p} \Big) \iota_t^*\phi_{i,j}(z;w_0,\ldots,w_n)
$$
The residues can be evaluated in terms of $\omega_{g',n'}$ with $2g' - 2 + n' < 2g - 2 + (1 + n)$ and by the induction assumption, the result is a meromorphic function of $(w_0,\ldots,w_n) \in W_{i,j}^{(A,\ell)}$.

Having covered $\Sigma_{T}^{n + 1}$ with the opens $W_{i,j}^{(A,\ell)}$, we conclude that $\omega_{g,1+n}(w_0,\ldots,w_n)$ is a meromorphic section of $\Omega_{f}^{\boxtimes (n + 1)}$. The location of its poles is clear from the fact that its restriction to each fibre $\Sigma_t^{n +1}$ is meromorphic with poles at $\mathsf{Ram}_t$, which is a discrete subset of $\Sigma_t$  due to (2) in \cref{de:FamSC}.
\end{proof}

This answers the question posed in the introduction: given a family of spectral curves, does the topological recursion procedure commutes with taking limits in the family? For globally admissible families, the answer is yes. As the correlators $(\omega_{g,n})_t$ are analytic in $t$, the $t \to t_0$ limit of the $t$-dependent correlators $(\omega_{g,n})_t$ is equal to the correlators constructed from the spectral curve at the point $t=t_0$ in the family, in the following sense: for any continuous local sections $W_0,\ldots,W_n$ of $\Sigma_T \rightarrow T$ defined near $t_0$, we have
\begin{equation}
	\lim_{t \rightarrow t_0}\, \frac{(\omega_{g,n})_t(W_0(t),\ldots,W_n(t))}{\dd x_t(W_0(t)) \cdots \dd x_t(W_n(t))} = \frac{(\omega_{g,n})_{t_0}(W_0(t_0),\ldots,W_n(t_0))}{\dd x_{t_0}(W_0(t_0)) \cdots \dd x_{t_0}(W_n(t_0))}\,. 
\end{equation}

 If we want to apply \cref{th:TRLimits} in a particular situation to  check that  the topological recursion commutes with taking limits, the main task is to prove that the family of spectral curves in  question is globally admissible (as given in \cref{de:admfamily} or \cref{de:admfamilyfinite}). Often, this boils down to checking the sufficient conditions for global admissibility as given in \cref{th:rewriting}. We illustrate this strategy in the next section by applying it to various interesting examples.

\begin{remark}
	\Cref{th:TRLimits} does not cover the free energies $ F_g = \omega_{g,0}$ of topological recursion. In fact, these free energies need not be analytic under these assumptions. This phenomenon was observed in \cite[Remark 3.21]{CGG22} even at the level of cohomological field theories and should be related to the conifold gap~\cite{HK07}. We do not investigate this question here.
\end{remark}

\section{Examples} \label{sec:EX}

In this section, we apply our main result on limits proved in  \cref{S5} to various cases of interest including deformations of $ (r,s) $-spectral curves and the spectral curves that appear in Hurwitz theory.  Finally, we also discuss cases where the commutation of topological recursion with limits fails. We will describe curves with $x$ and $y$ only, and it is implicit that we turned into a spectral curve by taking as $1$-form $\omega_{0,1} = y \dd x$ and (for curves of genus $0$) as bidifferential $\omega_{0,2}^{{\rm std}}(w_1,w_2) = \frac{\dd w_1\,\dd w_2}{(w_1 - w_2)^2}$.

\subsection{Deformations of \texorpdfstring{$(r,s)$}{(r,s)}-spectral curves}
\label{sec:rsexamples}

In this section, we revisit the $(r,s)$-spectral curves $\mathcal{S}^{(r,s)} = (\tilde{C}^{(r,s)},x,\omega_{0,1},\omega_{0,2})$, which are compact genus $0$ curves constructed in \cref{s:rscurve} from the equation $x^{r - s}y^r = 1$ with $r \geq 2$ and $s \in [r - 1]$ coprime. They can be presented as
\begin{equation}
\label{rscurveparams} x(w) = w^r,\qquad y(w) = w^{s - r}\,,
\end{equation}
with $w \in \mathbb{P}^1$ is a uniformising variable.  Such spectral curves are locally admissible if and only if $r = \pm 1 \,\,{\rm mod}\,\,s$. The extra case $s = r + 1$  (associated to the polynomial equation $x = y^r$) will only be briefly mentioned, as it corresponds to the well-studied $A_{r - 1}$-singularity.

We shall study the existence of maximal globally admissible families $\mathcal{S}^{(r,s)}_{T}$ deforming $\mathcal{S}^{(r,s)}$ as an illustration of the results of \cref{S4,S5}.  The importance of this example is twofold. First, as is manifest in \cref{de:localadm}, all  locally admissible spectral curves are modelled locally on the locally admissible $ (r,s) $-curves. In fact, we prove in \cref{le:class0} that $\mathcal{S}^{(r,s)}_{T}$ with $r = \pm 1\,\,{\rm mod}\,\,s$ and $s \in [r + 1]$ already provide a large class among the admissible families of genus $0$ spectral curves. Second, it is expected that one can associate to each locally admissible $(r,s)$-curve a sequence of classes of CohFT-like classes on $\prod_{g,n} \overline{\mathcal{M}}_{g,n}$, thus giving the correlators of topological recursion applied to $\mathcal{S}^{(r,s)}$ an intersection-theoretic representation \cite{BBCCN18}, \cite[Conjecture 7.23]{BKS20}. This would imply an intersection-theoretic representation for the correlators of the local topological recursion of any locally admissible spectral curve \cite[Theorem 7.26]{BKS20}. The $(r,s)$-class is known for $s = r + 1$ (it is the Witten $r$-spin class) and $s = r -1$ (the $ \Theta^r $-class). Deformations and limits of spectral curves already play an important role in the study of these two special cases \cite{DNOPS19, CCGG22, CGG22}. The results of this section should offer a basis for the construction of $(r,s)$-classes in general.

\subsubsection{Deformations of the Newton polygon}

We want to construct a maximal deformation family $\mathcal{S}^{(r,s)}_{T}$ of $\mathcal{S}^{(r,s)}$, defined by projectivisation and normalisation of a family of affine curves, such that the Newton polygon of the polynomial defining the affine curve is inscribed in the rectangle $\square(r-s,r)$ and has no interior points. In particular the bidegree and the (empty) set of interior points remain constant in the family. Keeping only the fibres associated to nondegenerate polynomials (this is an open dense condition on the base), we obtain a maximal family of smooth compact genus $0$ spectral curves.

\begin{definition}\label{de:maxhalf}
	Let $\Delta_{\max}^{(r,s)}$ be the maximal (under inclusion) inscribed polygon in $\square(r-s,r)$ having no interior points and obtained from $\Delta^{(r,s)} = [(0,0),(r-s,r)]$ by adding vertices (that is, it is in the same strong equivalence class). The existence of  $\Delta_{\max}^{(r,s)}$  is proved in \cref{pr:Qdelt}. Denote by $\frac{r' - s'}{r'}$ the left neighbour of $\frac{r - s}{r}$ in the Farey sequence $\mathcal{F}_r$, and introduce the positive integers
	\begin{equation}
	b \coloneqq \Big\lfloor \frac{r}{r'} \Big\rfloor,\quad 
	c \coloneqq \Big\lfloor \frac{r - s}{(r - s) - (r' - s')} \Big\rfloor\,.
	\end{equation}
\end{definition}
\begin{lemma}\label{le:rsmax}
	Let $ r \geq 2$ and $s \in [r-1]$ coprime. Then, $\Delta_{\max}^{(r,s)}$ is the quadrilateral with vertices $(0,0)$, $(r-s,r)$, $(p - q,p)$, $(p' - q',p')$ with $p' = r - p$, $q' = s - q$ given by one of the following two situations
	\begin{itemize}
		\item[(F${}_+$)]  $c = 1$ and $p = b r' $ and $ q = b s' $;
		\item[(F${}_-$)]  $b = 1$ and $p = r -  c (r-r') $ and $q = s -  c (s-s')$;
\end{itemize}
In the overlapping case $b = c = 1$, we have $p = r'$ and $q = s'$. Note that $(p-q,p)$ lies above $[(0,0),(r-s,r)]$. In particular, for the locally admissible values of $(r,s)$ the possible alternatives are the following:
\begin{itemize}
\item if $(r,s) = (2,1)$, then (F${}_{+}$) is realised with $(r',s') = (1,1)$ while $b = 2$;
\item if $r \geq 3$ and $s = 1$, then (F${}_-$) is realised with $(r',s') = (r - 1,1)$ and $c = r - 1$;
\item if $s \in \{2,\ldots,r-1\}$ and $r = 1\,\,{\rm mod}\,\,s$, then (F${}_+$) is realised with $r = r's + 1$ and $s' = 1$, while $b = r$ (for $s = r - 1$) or $b = s$ (for $s < r - 1$);
\item if $s \in \{3,\ldots,r-2\}$ and $r = -1\,\,{\rm mod}\,\,s$, then (F${}_-$) is realised with $r = (r - r')s - 1$, while $c = s' = s - 1$.
\end{itemize}
\end{lemma}

\begin{figure}[!ht]
\begin{center}
\includegraphics[width=0.9\textwidth]{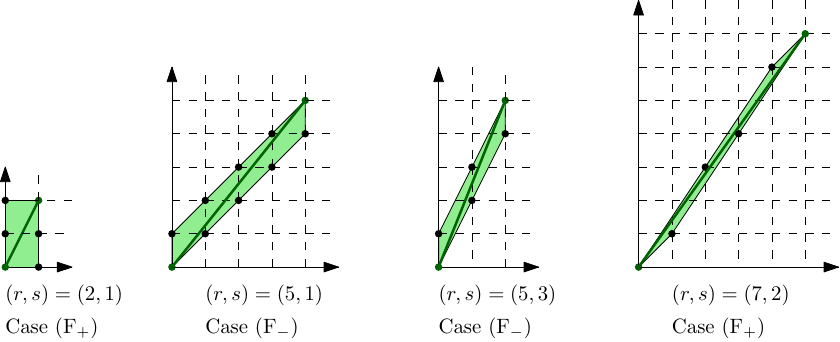}
\caption{\label{fig:examrs} $\Delta^{(r,s)}_{\max}$ for examples of the locally admissible types cited in \cref{le:rsmax}.}
\end{center}
\end{figure}

\begin{proof}
Recall B\'ezout's identity $rs' - r's = 1$. This implies in particular that $r' \leq r - 1$ and $s' \leq s$, that $r',s'$ are coprime unless $s = r- 1$, and that if $s = r - 1$ we in fact have $s' = r' = 1$. Recall as well that the right neighbour of $\frac{r - s}{r}$ in $\mathcal{F}_r$ must be $\frac{(r - s) - (r' - s')}{r - r'}$.

Let $\mathcal{P}_+$ (resp.  $\mathcal{P}_-$) be the set of polygons inscribed in $\square(r-s,r)$, containing no interior points, and obtained from $\Delta^{(r,s)} = [(0,0),(r-s,r)]$ by adding vertices above it (resp. below it). Denoting $\Delta_{\pm}^{(r,s)} = \max \mathcal{P}_{\pm}$, it is clear that
\begin{equation}
\label{Deltatwo}\Delta_{\max}^{(r,s)} = \Delta_+^{(r,s)} \cup \Delta_-^{(r,s)}\,.
\end{equation}
We can define an inclusion-preserving bijection $\sigma \colon \mathcal{P}_+ \rightarrow \mathcal{P}_-$: if $(x_i,y_i)_{i = 1}^{r}$ are the steps between consecutive vertices in $\Delta \in \mathcal{P}_+$ listed from $(0,0)$ to $(r-s,r)$, we define $\sigma(\Delta)$ to be the polygon where $(-x_i,-y_i)_{i = 1}^{r}$ is the steps between consecutive vertices listed from $(r - s,r)$ to $(0,0)$. Subsequently
\begin{equation}
\label{isodelta}
\Delta_-^{(r,s)} = \sigma(\Delta_+^{(r,s)})\,.
\end{equation}

Next, we show that $\Delta_+^{(r,s)}$ is a triangle. In the following, the set of edges $E_{0\infty}$ and the corner $\Gamma_{0\infty}$ will refer to those of $\Delta_+^{(r,s)}$ and we recall that ``point" means ``integral point'' unless precised otherwise. For any two vertices $(p_1 - q_1,p_1),(p_2 - q_2,p_2)$ of $\Delta_+^{(r,s)}$, the fact that a point $(p_0 - q_0,p_0) \in \Delta_+^{(r,s)}$ cannot be an interior point (as this polygon does not contain any) is expressed by the alternative 
	$$
	\frac{p_1}{p_1 - q_1} \leq \frac{p_0}{p_0 - q_0} \qquad {\rm or}\qquad \frac{r - p_2}{(r - s) - (p_2 - q_2)} \geq \frac{r - p_0}{(r - s) - (p_0 - q_0)}\,.
	$$
Taking $(p_1,q_1) = (p_2,q_2) = (p,q) \in \Gamma_{0\infty}$ and $(p_0 - q_0,p_0) = (r' - s',r')$, this condition yields
$$
\frac{r' - s'}{r'} \leq \frac{p - q}{p}\qquad {\rm or}\qquad \frac{(r - s) - (r' - s')}{r - r'} \geq \frac{(r - s) - (p - q)}{r - p}\,.
$$
Since $\max(p,r') \leq r$, the ratios $\frac{r' - s'}{r'}$ and $\frac{p - q}{p}$ belong to the Farey sequence $\mathcal{F}_r$. Since $\frac{r' - s'}{r'}$ is the left neighbour of $\frac{r - s}{r}$ in $\mathcal{F}_r$, we must in fact have
\begin{equation}
\label{eq:alternat012}\frac{p - q}{p} = \frac{r' - s'}{r'}\qquad {\rm or}\qquad \frac{(r - s) - (p - q)}{r - p} = \frac{(r - s) - (r' - s')}{r - r'}\,.
\end{equation}
As this is true for any vertex of $\Delta_+^{(r,s)}$, this means that $\Delta_+^{(r,s)}$ must be a triangle with vertices $(0,0),(r-s,r),(p-q,p)$ with $p,q$ satisfying the alternative \eqref{eq:alternat012}. We deduce from \eqref{Deltatwo}-\eqref{isodelta} that $\Delta_{\max}^{(r,s)}$ is the quadrilateral with vertices  $(0,0),(r-s,r),(p-q,p),(p' - q',p')$ where $p' = r - p$ and $q' = s - q$.
	
The first equality $\frac{p - q}{p} = \frac{r' - s'}{r'}$ is realised if and only if $p = br'$ and $q = bs'$ for some $b \in \mathbb{Z}_{> 0}$. We claim that the open segment $((p-q,p),(r-s,s))$ has no point. Indeed, such a point $(p' - q',p')$ would have to satisfy $\frac{(r - s) - (r' - s')}{r - r'} = \frac{(r - s) - (p' - q')}{r - p'} = \frac{(r - s) - b(r' - s')}{b(r - r')}$, which forces $b = 1$ and $(p' - q',p') = (r' - s',r')$, and therefore the latter would not be a point in $((p-q,p),(r-s,s))$. Then, the maximality of $\Delta_+^{(r,s)}$ for the inclusion implies that $b = \lfloor \frac{r}{r'} \rfloor$. If furthermore $b \geq 2$, from the Farey sequence $ \mathcal F_r $, we have $ 2 \leq \frac{r}{r'} < \frac{r-s}{r'-s'} $, which implies that $ c < 2 $ and thus $ c = 1 $. This fits in the case (F$_{+}$).

The second equality $\frac{(r - s) - (p-q)}{r  - r'} = \frac{(r  - s) - (r' - s')}{r - r'}$ is realised if and only if $r - p = c(r - r')$ and $s - q = c(s - s')$ for some $c \in \mathbb{Z}_{> 0}$. The previous argument can be adapted to show that the open segment $((0,0),(p-q,p))$ has no point, and then, the maximality of $\Delta_+^{(r,s)}$ implies that $c = \big\lfloor \frac{r - s}{(r - s) - (r' - s')} \big\rfloor$. If furthermore $c \geq 2$, from the Farey sequence $ \mathcal F_r $, we have $ 2 \leq \frac{r}{r'} < \frac{r-s}{r'-s'} $, which implies that $ b < 2 $ and thus $ b = 1 $. This fits in the case (F$_{-}$).

The only remaining case is $b = c = 1$, which means that $p = r'$ and $q = s'$, hence fits in both (F$_{\pm}$).

The locally admissible values of $(r,s)$ are those satisfying $r = \pm 1\,\,{\rm mod}\,\,s$. The value of $(r',s')$ is uniquely characterised by B\'ezout's identity $rs' - r's = 1$ with $r' \in [r - 1]$ and $s' \in [s]$. Writing $r = ks \pm 1$ (and treating $s = 1$ separately), it is then easy arithmetics to understand which of (F$_{\pm}$) is realised and find the corresponding $(r',s',b,c)$ depending on the congruence and $k$. This leads to the announced result. As a matter of fact, admissible $(r,s)$ always give $\max(b,c) \geq 2$, so are covered either by (F$_{+}$) or (F$_{-}$) but not by both.
\end{proof}

\subsubsection{Deformation of spectral curves}
\label{Sec532}
We now describe maximal proper families of spectral curves $\mathcal{S}^{(r,s)}_{T}$ fitting \cref{de:FamSC} and for which the Newton polygon for the generic fibre is $\Delta_{\max}^{(r,s)}$. Their expression is slightly different in the two situations (F${}_{\pm}$) described in \cref{le:rsmax}, and we will see that the extra case $s = r + 1$ is in fact covered by (F${}_+$). We present them in two ways: by the equation of the underlying affine curve $P(x,y) = 0$; and by a parametrisation of the normalised compact (genus $0$) spectral curve $\mathcal{S}^{(r,s)}_t$ constructed as in \cref{de:irredpc} (see also the discussion about simultaneous normalisations at the end of \cref{Sec51}). In all cases we indicate a parametrisation $(x(w),y(w))$ in terms of a uniformising coordinate $w \in \mathbb{P}^1$ and this completes the description of the spectral curve by equipping it with $\omega_{0,1} = y \dd x$ and $\omega_{0,2}^{{\rm std}}$. Checking that the given formulae provide a parametrisation is a direct computation with help of B\'ezout's identity $rs' - r's = 1$.

%

\medskip

\noindent \underline{\textit{Case (F$_{+}$) with $s < r -1$}.} The defining equation is
$$
P(x,y) = - L\big(x^{r' - s'}y^{r'}\big) + x^{r - s}y^{r}R\big(x^{-(r' - s')}y^{-r'}\big)\,,
$$
where $L,R$ are polynomials of degree at most $b$ with $L(0)R(0) \neq 0$.  This non-vanishing condition is needed to have a Newton polygon inscribed in $\square(r-s,r)$. We can take $w= x^{-(r' - s')}y^{-r'}$ as uniformising variable. Introducing $\check{L}(w) = w^{b}L(1/w)$, we obtain the parametrisation
$$
x(w) = w^{r - br'}\bigg(\frac{\check{L}(w)}{R(w)}\bigg)^{r'}\,,\qquad y(w) = \frac{1}{w^{r - s - b(r' - s')}} \bigg(\frac{R(w)}{\check{L}(w)}\bigg)^{r' - s'}\,.
$$

\medskip

\noindent \underline{\textit{Case (F$_{+}$) with $s = r -1$}.} The previous formulae are valid, but we only impose that $L(0)R(0) \neq 0$ or $\check{L}(0)\check{R}(0) \neq 0$, as this suffices to have a Newton polygon inscribed in $\square(1,r)$. Here we denoted as before $\check{L}(w) = w^rL(1/w)$ and $\check{R}(w) = w^{r}R(1/w)$. As $r' = s' = 1$ and $b = r$ the expressions simplify: the defining equation is $P(x,y) = -L(y) + x y^{r}\,R(y^{-1})$ and the parametrisation is
\begin{equation}
\label{theonesp}x(w) = \frac{\check{L}(w)}{R(w)}\,,\qquad y(w) = \frac{1}{w}\,.
\end{equation}
For $L(w) = R(w) = w^r$ we retrieve the curve $\mathcal{S}^{(r,r+1)}$ --- in its usual parametrisation if we use $1/w$ as uniformising variable --- while $L(w) = R(w) = 1$ corresponds to $\mathcal{S}^{(r,r - 1)}$.
 
\medskip

\noindent \underline{\textit{Case (F$_{-}$)}.} The defining equation is
$$
P(x,y) = - R\big(x^{r - r' - (s - s')}y^{r - r'}\big) + x^{r - s}y^{r} L\big(x^{-(r - r' - (s - s'))}y^{-(r - r')}\big)\,,
$$
where $L$ and $R$ are polynomials of degree at most $c$ and such that $L(0)R(0) \neq 0$. We can take $w = x^{r - r' - (s - s')}y^{s - s'}$ as uniformising variable. Introducing $\check{L}(w) = w^{c}L(1/w)$, we obtain the parametrisation
$$
x(w) = w^{r - c(r - r')} \bigg(\frac{\check{L}(w)}{R(w)}\bigg)^{r - r'}\,,\qquad y(w) = \frac{1}{w^{r - s - c(r - r' - (s - s'))}}\bigg(\frac{R(w)}{\check{L}(w)}\bigg)^{r - r' - (s - s')}\,.
$$

\medskip

In the overlapping case $b = c = 1$, our description of the two families coincide up to a linear transformation on the parameters of the polynomials $L,R$. In all three cases the base $T$ parametrises the space of polynomials $(L,R)$ as above that have simple roots (which amounts to requiring that $P$ is nondegenerate). It is the complement of a divisor in some $\mathbb{C}^D$, hence is connected. The central fibre --- namely the one identified with the $(r,s)$-curve --- corresponds to $L(w) = R(w) = 1$. When $t \in T$ approaches the central fibre via generic fibres, half of the finite ramification points (and the finite branch points) merge to form a single ramification point of order $r$ above the branch point $0$, while the other half merge and form a single ramification point of order $r$ above the branch point $\infty$. For a precise description of the branching structure of a generic fibre for all three cases, with the order $r$ and associated local parameter $\bar s$, see \cref{assidesec}.
 
All fibres in these families are irreducible. This is essentially a complete list of globally admissible spectral curves of genus $0$ which cannot be deformed (in a globally admissible family) to a reducible spectral curve, due to the following lemma.

\begin{lemma}\label{le:class0}
Let $\mathcal{S}$ be an irreducible globally admissible compact spectral curve of genus $0$ coming from a nondegenerate affine curve (as in \cref{de:irredpc}) of coprime bidegree $(d_x,d_y)$. We have the following alternative:
\begin{itemize}
\item either it has no ramification points;
\item or its ramification points above the branch points $0,\infty \in \mathbb{P}^1$ do not contribute to topological recursion;
\item or it is a fibre in the family $\mathcal{S}^{(r,s)}_{T}$ for $s \in [r + 1]$ and $r = \pm 1\,\,{\rm mod}\,\,s$.
\end{itemize}
\end{lemma}

\begin{proof}
Let $(d_x,d_y)$ be the bidegree and $\Delta$ be the Newton polygon of such an affine curve. If $d_y = 1$, $x$ has no ramification point as it is a global coordinate on the curve. If $d_x = 1$, the curve has an equation of the form $R_0(y) + xR_1(y) = 0$ and is a fibre in the family (F$_{+}$). We now assume $d_x,d_y \geq 2$. Since the curve is nondegenerate and has genus $0$, $\Delta$ has no interior points. Since it is irreducible and globally admissible, it satisfies (GA1) of \cref{pr:ndga}, i.e.  $|\hat{\Gamma}_{00}(\Delta)|_{\mathbb{Z}} = 0$. For $d_x,d_y \geq 2$ this implies $(0,0) \in \Delta$. Irreducibility also implies that $\Delta$ is inscribed in $\square(d_x,d_y)$. Assume that $(d_x,d_y) \notin \Delta$. Then, we have
\begin{equation*}
\begin{split}
& h_x := \max\{h \in \mathbb{Z}_{\geq 0}\,\,|\,\,(h,d_y) \in \Delta\}  < d_x\,, \\
& h_y := \max\{h\in \mathbb{Z}_{\geq 0}\,\,|\,\,(d_x,h) \in \Delta\} < d_y\,.
\end{split}
\end{equation*}
 By convexity, the triangle $T$ with vertices $(0,0),(h_x,d_y),(d_x,h_y)$ is contained in $\Delta$, but $T$ always contains an interior point:
\begin{itemize}
\item if $\min(h_x,h_y) > 0$, then $(h_x,h_y) \in \mathring{T}$;
\item If  $h_x = 0$ and $h_y > 0$, then $(d_x - 1,h_y) \in \mathring{T}$;
\item if $h_y = 0$ and $h_x > 0$, then $(h_x,d_y - 1) \in \mathring{T}$;
\item if $h_x = h_y = 0$, then $(1,1) \in \mathring{T}$.
\end{itemize} 
This contradicts the original assumption, hence $(d_x,d_y) \in \Delta$ and $\ell\coloneqq [(0,0),(d_x,d_y)] \subseteq \Delta$.

We now take into account the assumption that $d_x,d_y$ are coprime. If $d_x < d_y$, $\ell = \Delta^{(r,s)}$ with $r = d_x$ and $s = d_x - d_y$ coprime, so the curve appears as a fibre in one of the families (F$_{\pm}$).  As we will see in \cref{le:analrs}, the global admissibility assumption forces $r = \pm 1\,\,{\rm mod}\,\,s$, so we are in the third case.

It remains to treat the case $d_x > d_y$. Let $\rho$ be the plane reflection along the diagonal axis and set $r = d_x$, $s = d_x - d_y$. Then $\rho(\Delta)$ is inscribed in $\square(r-s,r)$ with $s \in [r - 1]$, has no interior points and contains $\rho(\ell) = \Delta^{(r,s)}$, so $\rho(\Delta)$ is contained in $\Delta^{(r,s)}_{\max}$ described in \cref{le:rsmax}. In particular, $\Delta$ only has edges in $E_{0\infty}$ and $E_{\infty0}$ and all their slopes must be between (or equal to) the left and right neighbours of $\frac{r - s}{r}$ in the Farey sequence $\mathcal{F}_r$, and they are smaller or equal to $1$. The edges $E_{0\infty}$ (resp. $E_{\infty0}$) correspond to points $p$ in the projectivised and normalised curve $\tilde{C}$ projecting to $x=0$ (resp. $x = \infty$). Recalling the correspondence \eqref{slopelocal} between slopes in $E_{0\infty}$ or $E_{\infty 0}$ and local parameters at $p$ we see that an edge of slope $1$ corresponds to an unramified point while an edge of slope $< 1$ corresponds to $\bar s_p \leq -1$. In both situations such points do not contribute to the topological recursion, cf. \cref{th:ale}.
\end{proof}

\subsubsection{Global admissibility and limits}

For the proper families of spectral curves $ \mathcal S^{(r,s)}_T $ described in \cref{Sec532}, we would like to understand how local admissibility of the central fibre relates to global admissibility of the family. To this end, we apply the deformation result of \cref{th:gapairs} regarding the invariance of the notion of global admissibility under deformations of Newton polygons that preserve the set of interior points.

 \begin{lemma}\label{le:analrs}
	Let $s \in [r + 1]$ coprime to $r \geq 2$. The family of spectral curves $\mathcal S^{(r,s)}_T$  is globally admissible if and only if $ r = \pm 1\,\,{\rm mod}\,\,s $.
\end{lemma}

\begin{proof}
We rely on the notations and constructions of \cref{S4} and in particular \cref{s:rscurve}. We first claim that the central fibre $\mathcal{S}^{(r,s)}$ is globally admissible if and only if it is locally admissible, i.e. if and only if $r = \pm 1\,\,{\rm mod}\,\,s$. Since the underlying projectivised curve $ C^{(r,s)} \subset \mathbb{P}^1 \times \mathbb{P}^1$ is nondegenerate, we rely on \cref{th:globsimp} providing equivalent conditions of global admissibility, named ($\Gamma$A1)--($\Gamma$A3). The Newton polygon $\Delta^{(r,s)}$ of $x^{r -  s}y^r = 1$ is a diagonal in $\square(|r - s|,r)$. Its corner $\Gamma_{00}$ does not contain integral points off the axes, so ($\Gamma$A1) is satisfied. ($\Gamma$A2)-($\Gamma$A3) say that the edges in $ E_{0\infty} $ and $ E_{\infty 0} $ should be globally admissible in the sense of \cref{de:edgeadm,de:edgeadm2}. If $ s = r + 1 $, this is automatic as $ E_{0\infty} $ and $ E_{\infty 0} $  are both empty.  If  $ s \in [r-1]$, both $ E_{0\infty} $ and $ E_{\infty 0} $ consist of the same edge that connects $ (0,0) $ to $ (r-s,r) $. The global admissibility condition for this edge is  $ r = \pm 1\,\,{\rm mod}\,\,s $ since the slope is $ \frac{r}{r-s} > 1 $. This justifies the claim.

Next, for every $t \in T$ the affine curves $\mathfrak{C}_t^{(r,s)}$ underlying $\mathcal{S}_{t}^{(r,s)}$ are by design irreducible, nondegenerate and have constant bidegree $(|r - s|,r)$. By construction, the Newton polygons of $\mathcal{S}^{(r,s)}_{t}$ for $t \in T$ are obtained from the one of the central fibre $\mathcal{S}^{(r,s)}$ by adding vertices in $\square(|r - s|,r)$ in such a way that no interior point is created. By  \cref{th:gapairs}, we deduce that $\mathcal{S}^{(r,s)}_t$ is globally admissible if and only if $r = \pm 1\,\,{\rm mod}\,\,s $.	 
\end{proof}

We note that the family is trivialisable, i.e. $\Sigma_T^{(r,s)} \simeq \mathbb{P}^1 \times T$ as the global coordinate $w$ provides a global section of $T \rightarrow \Sigma_T^{(r,s)}$. The system of correlators $(\omega_{g,n})_t$ constructed by topological recursion on $\Sigma_{t}^{(r,s)}$ can therefore be considered as a $t$-dependent $n$-differentials on $(\mathbb{P}^1)^n$. Applying  \cref{th:TRLimits}, we arrive at the following result. 

	\begin{theorem}\label{th:lim}
		Let $s \in [r + 1]$ coprime to $r \geq 2$.  If $r = \pm 1 \,\,{\rm mod}\,\,s$, then $(\omega_{g,n})_t(w_1,\ldots,w_n)$ depends analytically on $t \in T$ and in particular
$$
\forall w_1,\ldots,w_n \in (\mathbb{P}^1 \setminus \{0\})^n,\qquad \lim_{t\rightarrow 0} (\omega_{g,n})_t(w_1,\ldots,w_n) = (\omega_{g,n})_0(w_1,\ldots,w_n)\,.
$$
with uniform convergence for $(w_1,\ldots,w_n)$ in any compact of $(\mathbb{P}^1 \setminus \{0\})^n$. If $r \neq \pm 1\,\,{\rm mod}\,\,s$, then topological recursion does not commute with the limit $t \rightarrow 0$, in particular $(\omega_{0,3})_t(w_1,w_2,w_3)$ diverges when $t \rightarrow 0$.
\end{theorem}
\begin{proof} The first part is the consequence of \cref{th:TRLimits}. The second part comes from the fact that, if the congruence condition is not satisfied,  $(\omega_{0,3})_0(w_1,w_2,w_3)$ is non-symmetric \cite[Appendix B]{BBCCN18} while the limit of a symmetric $3$-differential (if it exists) must be symmetric.
\end{proof}

The first cases where the congruence does not hold are $(r,s) = (7,5), (8,5), (9,7)$ and by explicitly computing $\omega_{g,n}$ for low values of $g$ and $n$ we observed that $(\omega_{g,n})_{t}(w_1,\ldots,w_n)$ diverges when $t \rightarrow 0$. For instance, the family of compact spectral curves associated to
$$
-1 + x^2 y^7 - t^2 y = 0
$$
is a deformation of the $(7,5)$-curve. The left neighbour of $\frac{2}{7}$ in the Farey sequence of order $7$ is $\frac{1}{4}$, so $(r',s') = (4,3)$, and $(b,c) = (1,2)$: this is covered by a (F$_-$) case in \cref{le:rsmax}, and it can be uniformised as
$$
x_t(w) = w(w^2 - t^2)^{3}\,,\qquad y_t(w) = \frac{1}{w^2 - t^2}\,.
$$
For $t \in \mathbb{C}^*$ this is a locally admissible curve, with one regular ramification point at $w = \frac{t^2}{7}$ and two ramification points of type $(3,2)$ at $w = \pm t$. The topological recursion formula \eqref{eq:TR} yields 
\begin{equation*}
\begin{split}
(\omega_{0,3})_t(w_0,w_1,w_2)  & = \frac{ 343\big(343w_0^2w_1^2w_2^2 + 49t^2 A_1 + 7t^4A_2+ t^6\big)}{2t^2(7w_0^2 - t^2)^2(7w_1^2 - t^2)^2(7w_2^2 - t^2)^2}\, \dd w_0 \dd w_1 \dd w_2,
\\
(\omega_{1,1})_t(w_0) & = \frac{A_3}{16t^4(w_0^2 - t^2)^3(7w_0^2 - t^2)^4}\, \dd w_0,  \\
\end{split}
\end{equation*}
where
\begin{equation*}
\begin{split}
A_1 & = w_0^2w_1^2 + w_0^2w_2^2 + w_1^2w_2^2 + 4w_0w_1w_2(w_0 + w_1 + w_2)\,, \\
A_2 & =  w_0^2 + w_1^2 + w_2^2 + 4w_0w_1 + 4w_0w_2 + 4w_1w_2\,, \\
A_3 & = 2401w_0^{12} - 6174w_0^{10}t^2 - 5537w_0^8t^4 + 4284w_0^6t^6 + 2255w_0^4t^8 - 1790w_0^2t^{10} - 47t^{12}\,.
\end{split}
 \end{equation*}
We explicitly see the divergence of $(\omega_{0,3})_t$ and $(\omega_{1,1})_t$ when $t \rightarrow 0$.

\subsubsection{Example: Chebyshev family}

\label{sec:Chebychev} 
As mentioned in the introduction, the CohFT associated to the spectral curve  $ \mathcal S^{(r,r+1)}$ is the  Witten $ r $-spin class, and deformations of this spectral curve that are parametrised as 
$$
	x(w) =  w^r +  a_{r-1} w^{r-1} + \cdots + a_0 , \qquad y(w) = w\,,
$$
correspond to shifts of the Witten $ r $-spin class. For a precise statement we refer to \cite{DNOPS19} and \cite[Remark 4.8]{CCGG22}). The spectral curve for one of the two specific shifts studied in \cite{PPZ19} corresponds to the family
 \begin{equation}
\label{xtyz} x_t(w) = 2 t^{\frac{r}{2}} \mathcal{T}_r \left(\frac{w}{2\sqrt{t}}\right),\qquad y(w) = w\,,
\end{equation}
where $ \mathcal{T}_r $ is the Chebyshev polynomial of the second kind \cite[Theorem A (1)]{CCGG22} and $t \in \mathbb{C}$. A notable feature of this example is that there are multiple ramification points lying over the same branch point, so commutation with the $t \rightarrow 0$ was left as an open problem in \cite[Remark 4.8 and 5.5]{CCGG22}. Since the roots of $x_t'(w)$ are simple for $t \in \mathbb{C}$, the curve \eqref{xtyz} is nondegenerate and appears as a subfamily of $\mathcal{S}^{(r,r+1)}_{T}$. Therefore, Theorem~\ref{th:lim} establishes that topological recursion commutes with the $t \rightarrow 0$ for the family \eqref{xtyz}.

\subsection{More (non-)examples of deformations of \texorpdfstring{$(r,s)$}{(r,s)}-curves}
\label{sec:rsexamples2}

\subsubsection{Non-example: singular deformations}
\label{sec:singular}
One of the conditions in the construction of the family of spectral curves $ \mathcal S_T^{(r,s)}  $ is that  the underlying affine curve $ \mathfrak{C}_t $ remains smooth for any $ t \in T $. If the affine curve had a singular point, we proved in \cref{le:singularities} that the spectral curve cannot be globally admissible. As global admissibility is only a sufficient condition for vertical globalisation (cf. \cref{le:pairs} and \cref{th:rewriting}), it could happen that the contributions from the singular points vanish in the global topological recursion formula. Then, topological recursion would still commute with the $ t \to 0 $ limit even for some deformations of the $ (r,s) $-curves where the underlying affine curve $  \mathfrak{C}_t $  is singular.  However, we strongly believe that this is not the case  and that in the case of deformations of the $ (r,s) $-curves, the conditions of  \cref{th:rewriting} are both necessary and sufficient.

For various locally admissible deformation families where the generic underlying affine curve is singular, we did extensive computations of the correlators symbolically using \textsc{Mathematica}~\cite{Mathematica}. These computations suggest  that indeed the contributions from the singular points do not vanish in the global topological recursion formula, and thus $ \lim_{t \to 0} (\omega_{g,n})_t  $ is finite, but $ \lim_{t \to 0} (\omega_{g,n})_t  \neq (\omega_{g,n})_0$. Hence, topological recursion does not seem to commute with limits in these cases. 

An example is the family of affine curves of equation
$$
y^3x^2 - 1 + 6t^2y - 9t^4y^2 = 0\,.
$$
It has a singular point at $(x,y) = (0,\frac{1}{t})$. The corresponding family of smooth compact curves has genus $0$ (despite the presence of one interior point in the Newton polygon) and can be uniformised by a global coordinate  $w \in \mathbb{P}^1$ as
$$
x_t(w) = w^3 - 3t^2w\,,\qquad y_t(w) = w^{-2}\,.
$$
We compute from the topological recursion formula \eqref{eq:TR} for $2g - 2 + n = 1$:
\begin{equation}
\begin{split}
(\omega_{0,3})_t(w_0,w_1,w_2) & = \frac{t^2 \dd w_0 \dd w_1 \dd w_2}{12} \bigg(\sum_{\varepsilon = \pm 1} \frac{1}{(w_0 +\epsilon t)^2(w_1 + \epsilon t)^2(w_2 + \epsilon t)^2}\bigg)\,, \\
(\omega_{1,1})_t(w_0) & = -\frac{\dd w_0}{288}\bigg(\frac{3t^2 - 13tw_0 + 7w_0^2}{(w_0 - t)^4} + \frac{3t^2 + 13tw_0 + 7w_0^2}{(w_0 + t)^4}\bigg)\,.
\end{split}
\end{equation}
We observe that these correlators have a limit for $w_0,w_1,w_2$ kept away from $0$:
$$
\lim_{t \rightarrow 0}\,(\omega_{0,3})_t(w_0,w_1,w_2) = 0\,,\qquad \lim_{t \rightarrow 0}\,(\omega_{1,1})_t(w_0) = -\frac{7 \dd w_0}{144w_0^2}\,,
$$
but it does not in general agree with the correlators of the $(3,1)$-spectral curve\footnote{The latter can be computed directly with \eqref{eq:TR}, or extracted from  \cite[Appendix B and Equation 5.10]{BBCCN18} where $y$ should be be multiplied by $-3$ to match the present convention for the $(3,1)$-curve.}:
$$
\omega_{0,3}^{(3,1)}(w_0,w_1,w_2) = 0\,,\qquad \omega_{1,1}^{(3,1)}(w_0) = \frac{\dd w_0}{9w_0^2}\,.
$$

\subsubsection{Example: components at horizontal \texorpdfstring{$\infty$}{infinity}}
\label{sec:PaulsExample}
 This example was suggested by Paul Norbury. For any $ k \geq 3$, consider the family of affine curve of equation
  \begin{equation}\label{PaulsCurve}
 x y^2 - t y^k - 1 = 0\,.
  \end{equation}
  with parameter $t \in \mathbb{C}$. If $ t \in \mathbb{C}^*$, this is a smooth affine curve of degree $k$, while if $ t = 0$, its degree is $ 2$. This is a deformation of the $(2,1)$-curve (also called Bessel curve) which falls out of the scope of \cref{sec:rsexamples} since it does not respect the bidegree. The compact spectral curve associated to this family has genus $0$ and can be uniformised as
   \begin{equation}
    x_t(w) = w^2 + t w^{2-k}\,, \quad y_t(w) = w^{-1}\,.
  \end{equation}
Its ramification points are at the $k$-th roots of $ (1-\frac{k}{2}) t$. Let $(\omega_{g,n}^{{\rm N}})_t$ be the correlators of the topological recursion. Interestingly, in this case the limit $ t \to 0 $ does in fact commute with topological recursion, but the reason is subtle.

\begin{proposition}\label{pr:Norbury}
For $2g - 2 + n > 0$ and $w_1,\ldots,w_n$ kept away from $0$, we have
\begin{equation}
\label{naivlm}\lim_{t \rightarrow 0} (\omega_{g,n}^{{\rm N}})_t(w_1,\ldots,w_n) = \omega_{g,n}^{(2,1)}(w_1,\ldots,w_n)\,.
\end{equation}
\end{proposition}
\begin{proof}
For $ t \in \mathbb{C}^*$,  \eqref{PaulsCurve} is a well-behaved affine spectral curve: it is reduced, smooth, and has only simple ramifications lying above pairwise distinct branchpoints. To construct the associated the compact spectral curve according to \cref{de:irredpc}, we should look at the bihomogenisation
  \begin{equation}\label{SCHorizontalInfinity}
    X_0 Y_0^2 Y_1^{k-2} - t^k X_1 Y_0^k - X_1 Y_1^k = 0\,.
  \end{equation}
  At $ t= 0$, we find that there is a \emph{horizontal} component given by $ Y_1^{k-2} = 0$, which was invisible in the affine picture. On this component, topological recursion is problematic, as $y = \frac{Y_0}{Y_1}$ is identically $\infty$ and $\omega_{0,1} = y \dd x$. However, we can still make it work, as follows. The map $x$ on this component is unramified, so at first we look at the other component, which is the $(2,1)$-curve $X_0Y_0^2 - X_1Y_1^2 = 0$: we know that topological recursion is well-defined and fully globalisable on this $(2,1)$-curve. Now, let us try to apply \cref{de:globalTR} with the recursion kernel from \cref{de:kernel}. As we have formally $ \omega_{0,1} \equiv \infty$ on the horizontal component, it is natural to \emph{take as convention} $ \frac{1}{\Upsilon_{|Z|}(Z;z)} = 0$ as soon as one of the elements of $Z \sqcup \{z\}$ lies on this component. This means that none of these terms contribute, and hence topological recursion can be fully globalised on the union of the two components. Moreover, the $ \omega_{g,n} $ for $ 2g-2+n > 0$ are identically zero as soon as any argument is on the extra component, cf.   \cite[Proposition~5.27]{BKS20}.
 
  We can now see that the conditions in \cref{de:admfamily} are mostly satisfied, if we understand the globalisation argument as above. More precisely, in the proof of \cref{th:TRLimits}, we require the integrand to be meromorphic. This is achieved as $ t \to 0$ since $ K^{(p)}_{1+|Z|}(w_0;z,Z)$ does tend to $0$ as one of the arguments goes to the horizontal sheet and, by \cite[Corollary~3.2]{Fay73}, $\omega_{0,2}$ is also continuous, with limit as $ t \to 0 $ the individual $ \omega_{0,2}$ on the two sheets. Therefore, the naive limit \eqref{naivlm} that discards the horizontal component actually works.
 \end{proof}

The mechanism in this example is of general nature: if we have a family of spectral curves in which certain fibres develop horizontal components (i.e. irreducible components where $y$ is identically $\infty$), these components can be discarded and checking that topological recursion commutes with limits can be done by examining the globalisation argument on the other components.

\subsubsection{Euler class of the Chiodo bundle}
\label{sec:Chiodo}

We now discuss an example of deformation of $(r,s)$-curves, which contains a logarithmic singularity for $t \neq 0$, where the existence of the $t \rightarrow 0$ limit can be proved thanks to a detour via intersection theory on $\overline{\mathcal{M}}_{g,n}$ but is valid independently whether the congruence relation $r = \pm 1\,\,{\rm mod}\,\,s$ is satisfied or not and does not seem in general to satisfy topological recursion even when the congruence condition is satisfied.

Given $r \in \mathbb{Z}_{> 0}$ and $s \in \mathbb{Z}$, and given a cut $L \subseteq \mathbb{C}$ from $0$ to $\infty$ for the logarithm, we call the following family of spectral curves the Chiodo family:
\begin{equation}
\label{chiodocurve}
\Sigma_t = \mathbb{C} \setminus L,\qquad x_t(w) = w^r - r t^r \ln w , \qquad  y_t(w) = w^{s-r}\,,
\end{equation}
This is a family parametrised by $t$ in
$$
T =  \mathbb{C} \setminus \bigg(\bigcup_{j = 0}^{r - 1} e^{\frac{2{\rm i}\pi j}{r}}L \bigg)\,,
$$
while we obtain the $(r,s)$-curve at $t = 0$. Let $(\omega_{g,n}^{{\rm C},(r,s)})_t$ be the correlators of the topological recursion for this spectral curve for $t \in \mathbb{C}^*$, and $\omega_{g,n}^{(r,s)}$ the ones for the $(r,s)$-curve. 

As we now review, the $(\omega_{g,n}^{{\rm C},(r,s)})_t$ have an intersection-theoretic expression stemming from \cite{LPSZ17} and involving classes first studied by Chiodo in  \cite{Chiodo}. For $\mathbf{a} \in \mathbb{Z}^r$ such that
\begin{equation}
\label{divisi}\frac{s(2g - 2 + n) - \sum_{i = 1}^n a_i}{r} \in \mathbb{Z}\,,
\end{equation}
we let $\overline{\mathcal{M}}_{g,n}^{\frac{s}{r} - 1}(\mathbf{a})$ be the compactified moduli space parametrising line bundles $L$ over a complex curve $C$ with $n$ marked points $q_1,\ldots,q_n$ together with an isomorphism
$$
L^{\otimes r} \simeq (K_{{\rm log},C})^{\otimes (s - r)}\Big(-\sum_{i = 1}^n a_i q_i\Big)\,.
$$
We denote $\mathcal{L}$ the universal line bundle over the universal curve
$$
\mathcal{C} \mathop{\longrightarrow}^{\pi} \overline{\mathcal{M}}_{g,n}^{\frac{s}{r} - 1} (\mathbf{a}) \mathop{\longrightarrow}^{p} \overline{\mathcal{M}}_{g,n}\,.
$$
The total Chern class
$$
{\rm C}_{g,n}^{\frac{s}{r} - 1}(\mathbf{a}) = p_*c(- R^\bullet\pi_*\mathcal{L}) = p_*\exp\bigg(\sum_{m \geq 1} (-1)^m(m-1)!\,{\rm ch}_m(-R^\bullet \pi_*\mathcal{L})\bigg)
$$
is called the Chiodo class. If $\mathbf{a}$ does not satisfies the divisibility condition \eqref{divisi}, we set ${\rm C}_{g,n}^{\frac{s}{r} - 1}(\mathbf{a}) = 0$.

The universal line bundle has degree
$$
\deg \mathcal{L} = \frac{(s - r)(2g - 2 + n) - \sum_{i = 1}^n  a_i}{r}\,,
$$
In the regime $s \in [r - 1]$ and $a_1,\ldots,a_n \in [r]$, we deduce $\deg \mathcal{L} < 0$. Hence $h^0(\mathcal{L}) = 0$ and $H^1(\pi_*\mathcal{L})$ forms a bundle over $\overline{\mathcal{M}}_{g,n}^{\frac{s}{r} - 1}(\mathbf{a})$. By Riemann--Roch formula, its rank is
\begin{equation}
\label{Dgnra} D_{g,n}(\mathbf{a}) = \frac{(3r - 2s)(g - 1) + \sum_{i = 1}^n (a_i + r - s)}{r}\,.
\end{equation}
\begin{definition}
For $s \in [r - 1]$, we denote ${\rm e}_{g,n}^{\frac{s}{r} - 1}(\mathbf{a}) \in H^{2D_{g,n}(\mathbf{a})}(\overline{\mathcal{M}}_{g,n})$ the pushforward of the Euler class of this bundle by the forgetful map $p$ to $\overline{\mathcal{M}}_{g,n}$.
\end{definition}

According to \cite{LPSZ17}, the topological recursion for the spectral curve
\begin{equation}
\label{tildecurve}
\tilde{\Sigma} = \mathbb{C} \setminus L,\qquad \tilde{x}(w) = -w^r + \ln w,\qquad \tilde{y}(w) = w^{s - r}
\end{equation}
yields correlators storing intersection indices of Chiodo classes in the following way
$$
\tilde{\omega}_{g,n}(w_1,\ldots,w_n) = \frac{r^{(2g - 2 + n)\frac{s}{r}}}{(s - r)^{2g - 2 + n}} \sum_{\substack{1 \leq a_1,\ldots,a_n \leq r \\ m_1,\ldots,m_n \geq 0}} \bigg( \int_{\overline{\mathcal{M}}_{g,n}} {\rm C}^{\frac{s}{r} - 1}_{g,n}(\mathbf{a}) \prod_{i = 1}^{n} \psi_i^{m_i}\bigg) \prod_{i = 1}^n \dd \tilde{\Xi}_{a_i,m_i}(w_i)\,,
$$
where
$$
\tilde{\Xi}_{a,m}(w) = \frac{1}{r^{\frac{a}{r} + m}} \partial_{x}^{m} \bigg(\frac{w^{r - a}}{1 - rw^{r}}\bigg)\,.
$$

As we shall see, introducing the parameter $t$ as in \eqref{chiodocurve} amounts to considering the Chern polynomial
$$
{\rm C}_{g,n}^{\frac{s}{r} - 1}(u;\mathbf{a}) = p_* \exp\bigg(\sum_{m \geq 1} (-1)^m\,u^m (m - 1)!\,{\rm ch}_m(-R^{\bullet}\pi_*\mathcal{L})\bigg)
$$
instead of the total Chern class. Taking into account \eqref{Dgnra} and by definition of the Euler class, we observe that for $s \in [r - 1]$
\begin{equation}
\label{CgnEul} {\rm C}_{g,n}^{\frac{s}{r} - 1}(u;\mathbf{a}) \mathop{=}_{u \rightarrow \infty} u^{D_{g,n}(\mathbf{a})}\big({\rm e}_{g,n}^{\frac{s}{r} - 1}(\mathbf{a}) + O(\tfrac{1}{u})\big)\,.
\end{equation}

\begin{proposition}\label{pr:Chiodo}
For $2g - 2 +n > 0$ and $t \in \mathbb{C}^*$, we have
\begin{equation}
\begin{split}
\label{omgntch}
& \quad (\omega_{g,n}^{{\rm C},(r,s)})_t(w_1,\ldots,w_n) \\
& = (r - s)^{2 - 2g - n} \sum_{\substack{1 \leq a_1,\ldots, a_n \leq r \\ m_1,\ldots,m_n \geq 0}} \bigg(\int_{\overline{\mathcal{M}}_{g,n}} t^{rD_{g,n}(\mathbf{a})}\,{\rm C}_{g,n}^{\frac{s}{r} - 1}(t^{-r};\mathbf{a}) \prod_{i = 1}^n \psi_i^{m_i}\bigg) \prod_{i = 1}^m (\dd \Xi_{a_i,m_i})_t(w_i)  \,,
\end{split}
\end{equation}
when decomposed on the basis of $1$-forms generated by the differential of the rational functions
$$
(\Xi_{a,m})_t(w) = \frac{(-1)^m}{r} \partial_{x_t}^m\bigg(\frac{w^{r - a}}{t^r - w^r}\bigg)\,.
$$
If furthermore $s \in [r - 1]$,  the following limit exists for $2g - 2 + n > 0$ and is uniform for $w_1,\ldots,w_n$ in any compact of $\mathbb{C}$:
\begin{equation}
\label{rsmoinsEUl}
\begin{split}
& \quad \lim_{t \rightarrow 0}\, (\omega_{g,n}^{{\rm C},(r,s)})_t(w_1,\ldots,w_n) = \omega_{g,n}^{{\rm E}}(w_1,\ldots,w_n) \\
& \coloneqq (r - s)^{2 - 2g - n} \sum_{\substack{1 \leq a_1,\ldots,a_n \leq r \\ m_1,\ldots,m_n \geq 0}}  \bigg(\int_{\overline{\mathcal{M}}_{g,n}} e_{g,n}^{\frac{s}{r} - 1}(\mathbf{a}) \prod_{i = 1}^n \psi_i^{m_i}\bigg) \prod_{i = 1}^n \frac{(m_ir + a_i)!^{(r)}\,\dd w_i}{w_i^{m_ir + a_i + 1}}\,.
\end{split}
\end{equation}
In this formula, we have the $r$-fold factorial $(rm + a)!^{(r)} = \prod_{j = 0}^{m} (jr + a)$.
\end{proposition}
\begin{proof}
This is folklore knowledge. Here is a short derivation. For any $t \in \mathbb{C}$ the Chiodo spectral curve \eqref{chiodocurve} is locally admissible. As $(\omega_{0,2})_t$ extends to a meromorphic bidifferential on $\mathbb{P}^1$, the correlators $(\omega_{g,n})_t$ computed by topological recursion for $2g - 2 + n > 0$ are meromorphic $n$-differentials on $(\mathbb{P}^1)^n$, with poles at $w = e^{\frac{2{\rm i}\pi j}{r}}r^{\frac{1}{r}}t$ for $j \in [r]$. 

Now take $t \in \mathbb{C}^*$. We introduce the change of global coordinate $\tilde{w} = r^{-\frac{1}{r}}t^{-1}w$ to compare the Chiodo spectral curve to \eqref{tildecurve}. Namely, we have
\begin{equation*}
\begin{array}{rclcrcl} x_t(w) & = & - rt^{r}\,\tilde{x}(\tilde{w}) - rt^{r}\ln(r^{\frac{1}{r}}t) & \qquad & y_t(w) & = & r^{\frac{s}{r} - 1}t^{s - r} \,\tilde{y}\big(r^{-\frac{1}{r}}t^{-1}w\big)  \\[3pt]
(\omega_{0,1})_t(w) & = &  -r^{\frac{s}{r}} t^{s}\, \tilde{\omega}_{0,1}(\tilde{w}) & \qquad &  (\omega_{0,2})_t(w_1,w_2) & = & \tilde{\omega}_{0,2}(\tilde{w}_1,\tilde{w}_2) \end{array}
\end{equation*}
and
\begin{equation}
\label{Xidiff}
\tilde{\Xi}_{a,m}(\tilde{w}) =  t^{rm + a}\,(\Xi_{a,m})_t(w)\qquad {\rm where}\quad (\Xi_{a,m})_t(w) \coloneqq \frac{(-1)^m}{r} \,\partial_{x_t}^m\bigg(\frac{w^{r - a}}{t^r - w^{r}}\bigg)\,.
\end{equation}

Besides, looking at the topological recursion formula \eqref{eq:TR} for $2g - 2 + n > 0$, we see that multiplying $\omega_{0,1}$ by a scalar prefactor results in multiplying $\omega_{g,n}$ by the same prefactor to the power ``Euler characteristics $2 - 2g - n$''. Therefore
\begin{equation*}
\begin{split}
(\omega_{g,n}^{{\rm C},(r,s)})_t(w_1,\ldots,w_n) & = (-r^{\frac{s}{r}}t^{s})^{-(2g - 2 + n)} \tilde{\omega}_{g,n}(\tilde{w}_1,\ldots,\tilde{w}_n) \\
 &  = \frac{t^{-s(2g - 2 + n)}}{(r - s)^{2g - 2 + n}} \sum_{\substack{1 \leq a_1,\ldots, a_n \leq r \\ m_1,\ldots,m_n \geq 0}} \bigg(\int_{\overline{\mathcal{M}}_{g,n}} {\rm C}_{g,n}^{\frac{s}{r} - 1}(\mathbf{a})\prod_{i = 1}^n \psi_i^{m_i}\bigg) \prod_{i = 1}^n \dd \tilde{\Xi}_{a_i,m_i}(\tilde{w}_i)\,.
 \end{split}
 \end{equation*}
The integral of $\prod_{i = 1}^n \psi_i^{m_i}$ against the Chiodo class ${\rm C}_{g,n}^{\frac{s}{r} - 1}(\mathbf{a})$ only receives contributions from the components of the latter with cohomological degree $2\big(3g - 3 + n  - \sum_{i = 1}^n m_i\big)$. We do not change this integral if we replace the Chiodo class it contains by the rescaled Chern polynomial $u^{- (3g - 3 + n) + \sum_{i = 1}^n m_i}\,{\rm C}_{g,n}^{\frac{s}{r} - 1}(u;\mathbf{a})$ for an arbitrary $u \in \mathbb{C}^*$. Making the choice $u = t^{-r}$, we obtain
\begin{equation*}
\begin{split}
& \quad (\omega_{g,n}^{{\rm C},(r,s)})_t(w_1,\ldots,w_n)  \\
&  = \frac{t^{-s(2g - 2 + n) + r(3g - 3 + n) + \sum_{i = 1}^n a_i}}{(r - s)^{2g - 2 + n}}\sum_{\substack{1 \leq a_1,\ldots, a_n \leq r \\ m_1,\ldots,m_n \geq 0}} \bigg(\int_{\overline{\mathcal{M}}_{g,n}} {\rm C}_{g,n}^{\frac{s}{r} - 1}(t^{-r};\mathbf{a})\prod_{i = 1}^n \psi_i^{m_i}\bigg) \prod_{i = 1}^m (\dd \Xi_{a_i,m_i})_t(w_i)\,.  
\end{split}
\end{equation*}
Comparing with \eqref{Dgnra} we recognise the overall factor $t^{rD_{g,n}(\mathbf{a})}$, justifying \eqref{omgntch}. If $s \in [r - 1]$, we deduce from \eqref{CgnEul} that
\begin{equation*}
	{\rm C}_{g,n}^{\frac{s}{r} - 1}(t^{-r};\mathbf{a}) \mathop{=}_{t \rightarrow 0} t^{-rD_{g,n}(\mathbf{a})}\big({\rm e}_{g,n}^{\frac{s}{r} - 1}(\mathbf{a}) + O(t^{r})\big)\,,
\end{equation*}
while a straightforward consequence of \eqref{Xidiff} and $x_t'(w) = rw^{r - 1} + O(t^{r})$ is
\begin{equation*}
	(\Xi_{a,m})_t(w) \mathop{=}_{t \rightarrow 0}  -\frac{\prod_{j = 0}^{m - 1}(jr + a)}{r^m}\,\frac{1}{w^{mr + a}} + O(t^{r})\,.
\end{equation*}
This shows that in the $t \rightarrow 0$ regime, $(\omega_{g,n})_t(w_1,\ldots,w_n)$ behaves like the right-hand side of \eqref{rsmoinsEUl} up to $O(t^{r})$ correction, uniformly for $w_1,\ldots,w_n$ away from $0$.
\end{proof}

We see that $(\omega_{g,n}^{{\rm C},(r,s)})_t$ have a well-defined $t \rightarrow 0$ limit for any value of $s \in [r - 1]$. This limit is expressed in terms of the intersection indices of the Euler class of the Chiodo bundle and we denote it $\omega_{g,n}^{{\rm E},(r,s)}$. On the other hand, $\omega_{g,n}^{(r,s)}$ is well defined (i.e. symmetric in its variables if and only if $r = \pm 1\,\,{\rm mod}\,\,s$. So, when this congruence does not hold, $\omega_{g,n}^{(r,s)}$ cannot be equal to $\omega_{g,n}^{{\rm E},(r,s)}$. Even when the congruence is satisfied, the Chiodo family is not globally admissible, because the ramification points hit the logarithm cut and therefore escape the curve when $t \rightarrow 0$ (making $\mathsf{Ram}_T$ non-proper). In fact, for $s \in [r - 1]$, \textsc{Mathematica}~\cite{Mathematica} computations suggest that
\begin{equation}
	\omega_{g,n}^{{\rm E},(r,s)} = \omega_{g,n}^{(r,s)}\quad \Longleftrightarrow \quad s = r - 1\,.
\end{equation}
The results of \cite{CGG22} imply indirectly  $\omega_{g,n}^{{\rm E},(r,r - 1)} = \omega_{g,n}^{(r,r - 1)}$ because the $\Theta^r$-class coincides up to a factor with the Euler class of the Chiodo bundle. This leads to an interesting puzzle:
\begin{enumerate}
\item For $s = r - 1$, we do not know how to prove  $\omega_{g,n}^{{\rm E},(r,r - 1)} = \omega_{g,n}^{(r,r - 1)}$ without a detour through properties of the Chiodo bundle;
\item For $r = \pm 1\,\,{\rm mod}\,\,s$ and $s \in [r - 2]$, we do not know an intersection-theoretic expression for $\omega_{g,n}^{(r,s)}$;
\item For arbitrary $s \in [r - 2]$, we do not know a recursion (on $2g - 2 + n$) that computes $\omega_{g,n}^{{\rm E},(r,s)}$.
\end{enumerate}
This illustrates the shortcoming of our approach in treating global topological recursion when branchcuts of $x$ are included in the curve --- see also \cref{GWP111}.	

\subsection{Hurwitz theory}

\label{sec:Hurwitz}

Hurwitz theory, in its most general formulation, covers numerous problems of enumeration of branched coverings of a two-dimensional sphere that can equivalently be reformulated as enumeration of factorisations in the group algebra of the symmetric group. This includes original Hurwitz numbers, Hurwitz numbers with completed cycles, enumeration of various kinds of maps and constellations, fully simple maps and their analogs, etc. The goal of this Section is to illustrate how \cref{th:TRLimits} can be used to waive genericity assumptions in the general setup connecting Hurwitz theory to topological recursion. We will restrict our presentation to two sample situations, which are sufficient to illustrate the general principle.

\subsubsection{Weighted Hurwitz numbers} \label{sec:weiHur}

Let $\Lambda = \mathbb{C}[p_1,p_2,\ldots]$ be the polynomial ring in countably many variables, and $\Lambda_{\hbar} = \Lambda \otimes \mathbb{C}[\![\hbar]\!]$. If $n \in \mathbb{Z}_{> 0}$ and $\lambda \vdash n$, we denote $\mathsf{s}_\lambda \in \Lambda$ the Schur function in the (power-sum) variables $(p_k)_{k \geq 1}$. Let $\psi(\hbar^2,\theta)$ be a formal power series in its variables such that $\psi(0,0)=0$. We associate to it a $\mathbb{C}[\![\hbar]\!]$-linear operator $\cD_\psi \in {\rm End}(\Lambda_{\hbar})$ defined by
\begin{equation}
	\cD_\psi(\mathsf{s}_\lambda) = \sum_{(i,j)\in \mathbb{Y}_{\lambda}} e^{\psi(\hbar^2,\hbar(i-j))} \mathsf{s}_\lambda \,.
\end{equation}
Finally, let $ \gamma (z) = \frac{R_1(z)}{R_2(z)} = \sum_{k=1}^{\infty} \gamma_k z^k$ be a rational function vanishing at $z=0$. 

Consider the partition function $Z$ defined as 
\begin{equation}
	Z=\mathcal{D}_\psi \exp\Big(\sum_{k \geq 1} \frac{\gamma_k p_k}{\hbar k}  \Big)\,.
\end{equation}
 in a suitable completion of $\Lambda_{\hbar}$. Under some assumptions on $\psi$ one can construct meromorphic $n$-differentials $(\omega_{g,n})_{(g,n)\in\Z_{\geq 0}\times\Z_{> 0}}$ on $\Sigma^n$ for $\Sigma = \mathbb{C}$ such that
 \begin{equation} \label{eq:HurwitzCor}
	\sum_{g \geq 0} \hbar^{2g-2+n} \omega_{g,n}(z_1,\dots,z_n) = \bigotimes_{i=1}^n \Big(\sum_{k \geq 1} \dd(x_i^{k}) \partial_{p_k}\Big)\ln Z\Big|_{p_1= p_2 = \cdots = 0} + \delta_{n,2}\frac{\dd x_1\,\dd x_2}{(x_1-x_2)^2}\,,
\end{equation}
where $x_i$ is related to the uniformising coordinate $z_i \in \mathbb{C}$ as $x_i = z_ie^{-\psi(0,\gamma(z_i))}$.  Here is a sample theorem proved in \cite{BDKS20a} (it is called Family I in \emph{op.~cit.}).

\begin{theorem}[\cite{BDKS20a}]\label{th:trHurwitz}
Assume there exist polynomials $P_1,P_2,P_3$ such that $P_1(0) = 0$ and $P_2(0) =P_3(0)= 1$, and
\begin{equation}
	\psi(\hbar^2,\theta) = \ln \bigg(\frac{P_2(\theta)}{P_3(\theta)} \bigg) + \frac{e^{\frac{\hbar\partial_\theta}{2}}-e^{-\frac{\hbar\partial_\theta}{2}}}{\hbar \partial_\theta} P_1(\theta)\,.
\end{equation}
Then the $(\omega_{g,n})_{(g,n)\in\Z_{\geq 0}\times\Z_{> 0}}$ defined in \eqref{eq:HurwitzCor} form a system of correlators satisfying topological recursion on the spectral curve $\Sigma = \mathbb{C}$ with
\begin{equation}\label{eq:TfamilyHurwitz}
	x(z)
	=
	z\,\frac{P_3\big(\tfrac{R_1(z)}{R_2(z)}\big)}{P_2\big(\tfrac{R_1(z)}{R_2(z)}\big)} e^{-P_1\big(\tfrac{R_1(z)}{R_2(z)}\big)}\,,
	\quad \omega_{0,1}(z) 
	=
	\frac{R_1(z)}{R_2(z)}\,\frac{\dd x(z)}{x(z)}\,,\quad \omega_{0,2}(z_1,z_2)= \frac{\dd z_1\,\dd z_2}{(z_1-z_2)^2}.
	\end{equation}
under the assumption that $\frac{R_1(z)}{R_2(z)}$ is analytic and has order $1$ (i.e. it is unramified) at the ramification points of $x$. 
\end{theorem}

The proof of \cref{th:trHurwitz} under the extra assumption that the polynomials $P_1,P_2,P_3,R_1,R_2$ are in general position is the main content of~\cite{BDKS20a}. By general position in this particular situation we mean that all ramification points of $x$ are simple, and there are further technical assumptions that the zeros of all involved polynomials are simple. We also assume that the zeros of $P_2$ and $P_3$ and of $R_1$ and $R_2$ are different (this condition is not restriction but rather a convention in the definition of the ratio $\frac{P_2}{P_3}$ and $\frac{R_1}{R_2}$).  In order to relax the general position assumption, the approach of \cite{BDKS20a} needs the following construction. Let $(P_{1,o},P_{2,o},P_{3,o},R_{1,o},R_{2,o})$ be a reference tuple of polynomials such that the general position requirements are possibly not satisfied. Denote by $x_o(z)$ the corresponding function $x$ given by~\eqref{eq:TfamilyHurwitz}. Let $T$ be a complex manifold that is a neighbourhood of the reference tuple in the space of the tuples $(P_{1,t},P_{2,t},P_{3,t},R_{1,t},R_{2,t})$ satisfying the conditions that the degrees of all polynomials are fixed, the zeros of $P_2$ and $P_3$  and of $R_1$ and $R_2$ are different, and $P_1(0)=0$, $P_2(0)=P_3(0)=1$, $R_1(0)=0$. 

Consider all ramification points of $x_o$ given as the limit points of the ramification points of $x_t(z)$ for $t \in T$ tending to $o$, with multiplicities. Here we count each finite pole of $\frac{\dd x_o}{x_o}$ with residue $m$ for $m\in\mathbb{Z} \setminus \{0\}$ as a ramification point of multiplicity $|m| - 1$. Let $(\mathsf{D}_i)_{i = 1}^{\mathsf{k}} \subset \mathbb{P}^1$ be a sequence of discs whose closures are pairwise disjoint and  containing all branch points of $x_o$. Note that $\mathcal{D} = x_o(\mathbb{C})$ may contain $\infty \in \mathbb{P}^1$. Instead of considering a big family suggested by~\eqref{eq:TfamilyHurwitz}, we rather choose $T$ small enough so that the branch points of all $x_t$, $t\in T$, are also in $({\mathsf{D}}_i)_{i = 1}^{\mathsf{k}}$. 
Define $S_T$ be setting $\Sigma_T= \mathbb{C}\times T$ with its natural projection to $T$, and $x_t(z)$, $(\omega_{0,1})_t$, and $(\omega_{0,2})_t$ given by restrictions of formulae~\eqref{eq:TfamilyHurwitz}.

\begin{proposition}\label{pr:PQRHur}
	$\Sigma_T$ is a globally admissible family of spectral curves over $T$. 
\end{proposition}
\begin{proof}
	Recall~\cref{de:FamSC} and~\cref{de:admfamily}. All basic analytic and topological properties are manifestly satisfied. Note, however, that the constructed family of spectral curves is not proper. We have to check three properties.

\begin{itemize}

	\item \emph{Each fibre $S_t$ is an admissible spectral curve in the sense of~\cref{de:localadm}.} To this end, we note that the condition that $\frac{R_{1,t}(z)}{R_{2,t}(z)}$ is analytic and has order $1$ at the ramification points of $x_t(z)$ is an open condition hence satisfied for the whole family once satisfied at $t=o$. Furthermore, we have the following local expansions for $(\omega_{(0,1)})_t$ near a point $p$ of ramification order $r$:
		\begin{equation*}
		\begin{split}
			(\omega_{(0,1)})_t & = \big(b_0+b_1\zeta+ O(\zeta^2)\big) \frac{r a_r \zeta^r +O(\zeta^{r+1} )}{a_0+a_r \zeta^r+O(\zeta^{r+1})} \frac{\dd\zeta}{\zeta},\qquad  \text{if } x_t(p)\coloneqq a_0\not=0,\infty; \\
			(\omega_{(0,1)})_t & = \big(b_0+b_1\zeta+ O(\zeta^2)\big)\big(r+O(\zeta)\big) \frac{\dd\zeta}{\zeta},\qquad  \text{if } x_t(p)=0; \\
			(\omega_{(0,1)})_t & = \big(b_0+b_1\zeta+ O(\zeta^2)\big) \big(-r+O(\zeta)\big) \frac{\dd\zeta}{\zeta},\qquad  \text{if } x_t(p)=\infty; 
		\end{split}
		\end{equation*} 
		The condition that $\frac{R_{1,t}(z)}{R_{2,t}(z)}$ has order $1$ means that $b_1\not=0$. In the first case we assume $a_r\not=0$, and we have $r_p=r$, $\bar s_p=r$ if $b_0\not=0$ or $\bar s_p=r+1$ otherwise, and $s_p=r+1$. In the other two cases  $r_p=r$, $\bar s_p=0$ if $b_0\not=0$ or $\bar s_p=1$ otherwise, and $s_p=1$. Clearly it is admissible in the sense  of~\cref{de:localadm}.

	\item \emph{$(\omega_{0,1})_t $  separates fibres at every point $ q  \in \bigcup_{i = 1}^{\mathsf{k}} \mathsf{D}_i$  which is not a branch point of $x_t$.} Indeed, if both $x_t(p_1)=x_t(p_2)$ and $\frac{R_{1,t}(p_1)}{R_{2,t}(p_1)} = \frac{R_{1,t}(p_2)}{R_{2,t}(p_2)}$, then, using~\eqref{eq:TfamilyHurwitz} to express $ z(p_i)$ in terms of $ x(p_i)$ and $ \frac{R_{1,t}(p_i)}{R_{2,t}(p_i)}$, we find that $z(p_1)=z(p_2)$, where $z$ is the global coordinate on $\Sigma_t = \mathbb{C}$. So it is only possible if $p_1=p_2$. 
		
	\item \emph{The conditions for globalisation in \cref{th:rewriting} are satisfied for the sets $ x_t^{-1}(q) $ for every branch point $q \in \bigcup_{i = 1}^{\mathsf{k}} \mathsf{D}_i$ of $x_t$.} By the previous discussion, the values of $b_0$ are different for all $p\in x_t^{-1}(q)$. 
		
	If $q=a_0\not=0$, this means that $\bar s_p = r+1$ for at most one point in $x_t^{-1}(q)$, and for all other points in $x_t^{-1}(q)$ we have $\bar s_p = r$. So, for each pair of points $\{p_1,p_2\} \subseteq x_t^{-1}(q)$ either $\nu_1\not=\nu_2$, or $\nu_1=\nu_2=1$, $\frac{\tau_1}{r_1} \neq \frac{\tau_2}{r_2}$, and $r_1'=r_2'=1$, so each pair of points is non-resonant. Also for each pair of points condition (C-i) of~\cref{th:rewriting} is satisfied for $m=1$.
		
	In the other two cases, that is, $q=0$ or $q=\infty$, we see that $\bar s_p=1$ for at most one point in $x_t^{-1}(q)$, and for all other points in $x_t^{-1}(q)$ we have $\bar s_p = 0$. So, for each pair of points $\{p_1,p_2\} \subseteq x_t^{-1}(q)$ either $\nu_1\not=\nu_2$, or $\nu_1=\nu_2=0$, $\frac{\tau_1}{r_1} \neq \frac{\tau_2}{r_2}$, and $r_1'=r_2'=1$, so each pair of points is non-resonant.

\end{itemize}

Thus we see that indeed we have a globally admissible family of spectral curves. 
\end{proof}

Thanks to \cref{th:TRLimits} we conclude that topological recursion commutes with taking limits of spectral curves, and thus $(\omega_{g,n})_{t = o}$ still satisfy topological recursion even if the spectral curve is not in general position. Note that each $(\omega_{g,n})_t$ for $2g - 2 + n > 0$ is a meromorphic $n$-differential on $\mathbb{C}^n$ given in~\cite{BDKS20a} by an explicit closed formula that depends smoothly on $t$, so we have even a more refined statement that identifies the $(\omega_{g,n})_o$ obtained by topological recursion with the given explicit formulae.

In a number of earlier papers where special cases of Hurwitz problems are studied, this step is either not spelled out at all, or erroneously referred to~\cite{BE13}, see e.g. \cite{BDS20,BDKLM22}. On the other hand, in \cite{BDKS20a} this step is explained and explicitly refers to the present article. An illustrative example where the family and a choice of the disk $\mathsf{D}_1$ (with $\mathsf{k}=1$ in this particular case) are described in detail and everything works precisely as exposed above is the case of Bousquet-M\'elou--Schaeffer numbers discussed in~\cite{BDS20}.

\subsubsection{Combination with symplectic transformations}

Other examples of globally admissible families of spectral curves are coming from Hurwitz theory combined with generalised symplectic transformations, leading to the so-called fully simple maps~\cite{BGF20} and their generalisations, as well as the problems of enumeration of maps with internal faces. To this end, one considers a partition function $Z$ defined as 
\begin{equation}
\label{eq:GeneralPartitionZ}
Z=\mathcal{D}_{\psi_2} \exp\Big(\sum_{k=1}^{d_2} \frac{\gamma_{k,2} \partial_{p_k}}{\hbar}  \Big) \mathcal{D}_{\psi_1} \exp\Big(\sum_{k=1}^{d_1} \frac{\gamma_{k,1} p_k}{k\hbar}  \Big)\,.
\end{equation}
Here $\psi_i(\hbar^2,\theta)$ for $i \in \{1,2\}$ are given, for instance, by the same formulae as above,
\begin{align}
	\psi_i(\hbar^2,\theta) = \ln \bigg(\frac{P_{2,i}(\theta)}{P_{3,i}(\theta)} \bigg) + \frac{e^{\frac{\hbar\partial_\theta}{2}}-e^{-\frac{\hbar\partial_\theta}{2}}}{\hbar \partial_\theta} P_{1,i}(\theta),
\end{align}
for some polynomials $P_{1,i},P_{2,i},P_{3,i}$ of some fixed degrees satisfying the same conditions as above, and $T_i(z)= \sum_{k=1}^{d_1} \gamma_{k,i} z^k$, $i\in\{1,2\}$, are polynomials of fixed positive degrees $d_1,d_2$ (more general families can be treated as well, see comments in ~\cite{alexandrov2023topological}). 

One can define meromorphic $n$-differentials $(\omega_{g,n})_{(g,n)\in\Z_{\geq 0}\times\Z_{> 0}}$ by Equation~\eqref{eq:HurwitzCor}, with a parametrisation $x(z)$ by the uniformising coordinate $z \in \mathbb{C}$ which is now only implicitly defined, cf. \cite[Sections 3--4]{bychkov2023symplectic} and \cite{alexandrov2023topological}. In this situation, it is proved in \cite{bychkov2023symplectic,alexandrov2023topological} that they form a system of correlators satisfying topological recursion under the assumption that the spectral curve is in general position, in the same sense as above. 

Without this assumption, the statement on topological recursion is still true, but then it relies on exactly the same construction of a family of spectral curves $\Sigma_T$ as what we did above in \cref{sec:weiHur}. The statement that the thus constructed family is globally admissible as well as the subsequent steps on extending the limit result to the initial system of rational differentials (rather than the ones restricted to $\Sigma_t$ for $t\in T$) follow exactly the same lines as the argument in~\cref{sec:weiHur}.

\subsection{A few more examples and non-examples}

\label{sec:moreexamples}

In this section we will consider a few examples of families of spectral curves that fall outside the main scope of the article, but are interesting nonetheless. In each case, the exact statements of some results have to be adjusted, but the same proof strategy still applies.

\subsubsection{Equivariant Gromov--Witten theory of $ \P^1$}
\label{GWP111}
Fang--Liu--Zong have proved in \cite{FLZ17} that topological recursion on the spectral curve given by
\begin{equation}
  x = \mathsf{t}_0 + y + \frac{\mathsf{Q} e^{\mathsf{t}_1}}{y} + \mathsf{w}_1 \ln y + \mathsf{w}_2 \ln \frac{\mathsf{Q} e^{\mathsf{t}_1}}{y}
\end{equation}
yields the equivariant Gromov--Witten theory of $ \P^1$, where $\mathsf{w}_1$ and $\mathsf{w}_2$ are the equivariant parameters, $\mathsf{t}_0$ and $\mathsf{t}_1$ are the quantum cohomology coordinates, and $\mathsf{Q}$ is the Novikov parameter. They also consider several specialisations and limits of these parameters. As this naturally gives a family of spectral curves, let us investigate this from our point of view.\par
The ramification points are found by
\begin{equation}
  0 = \dd x = (y^2 + (\mathsf{w}_1 - \mathsf{w}_2 ) y - \mathsf{Q} e^{\mathsf{t}_1}) \frac{\dd y}{y^2}\,,
\end{equation}
which generically has two solutions, giving two simple ramification points.\par
The non-equivariant limit \cite[Section~1.2]{FLZ17}, obtained by setting $\mathsf{w}_1 = \mathsf{w}_2 = 0$, is a generic point, so topological recursion commutes with this limit. The further reduction to the stationary phase by $\mathsf{t}_0 = \mathsf{t}_1 = 0$ and $\mathsf{Q} = 1$ does not change this.\par
The large-radius limit \cite[Section~1.3]{FLZ17} on the other hand is taken by first setting $\mathsf{w}_2 = \mathsf{t}_0 = 0$ and $\mathsf{q} = \mathsf{Q}e^{\mathsf{t}_1}$ to get
\begin{equation}
  \begin{split}
    x = y + \frac{\mathsf{q}}{y} + \mathsf{w}_1 \ln y \,;
    \\
    \dd x = (y^2 + \mathsf{w}_1 y - \mathsf{q}) \frac{\dd y}{y^2}\,,
  \end{split}
\end{equation}
and then taking $\mathsf{q} \to 0$. Like for the Chiodo family in \cref{sec:Chiodo}, this limit \emph{does not fit} in the scope of our paper, as one of the ramification points moves off to infinity where it meets the logarithm cut. A proper setup to handle such situations would probably have to address the globalisation of the topological recursion for transalgebraic spectral curves defined in \cite{BKW23}. However, in this case the $\mathsf{q} = 0$ limit can be taken directly in two independent ways:
\begin{itemize}
\item[(i)] at the level of intersection-theoretic formulae for the Gromov--Witten theory, it yields simple Hodge integrals. Topological recursion is known independently to hold as well at $\mathsf{q} = 0$: this was Bouchard--Mari\~{n}o conjecture \cite{BM08} established in \cite{EMS11,DKOSS15}.
\item[(ii)] \cite{FLZ17} relies on Givental--Teleman reconstruction and its correspondence with topological recursion found in \cite{DOSS14}. The $R$ and $T$ matrices can be directly studied at $\mathsf{q} = 0$ and the correspondence of \cite{DOSS14}. This was achieved in \cite[Section 5]{FLZ17} which does give a proof of Bouchard--Mari\~{n}o conjecture logically independent of (i).
\end{itemize}

\label{sec:flyoff}

\subsubsection{Vertical component}

This example was suggested by Elba Garcia-Failde and appears as a special case of the families of spectral curves enumerating the $r$-spin graphs of \cite{BCEG21}.  Consider the family of affine curves
\begin{equation}
	x y^2 - y^2 - x^2 + (1 + 2t)x + 2t y - t^2 = 0
\end{equation}
with parameter $t \in \mathbb{C}$.  For $ t \neq 0$, it has a parametrisation
\begin{equation}
	x_t(z) = z^2 \,, \quad  y_t(z)  = z + \frac{t}{1-z} \,.
\end{equation}
It has a singularity at $ x= t+1$ and $ y = -1$: this is a double point for $ z = \pm \sqrt{ t + 1}$.
At $ t = 0$, the curve becomes reducible
\begin{equation}
	(y^2 - x)(x - 1) = 0
\end{equation}
and the two components intersect at $ x = 1$ and $ y = \pm 1$.\par
As in \cref{sec:PaulsExample}, we may ignore the node for $ t \neq 0$, as it gives unramified points in the normalisation. For $ t = 0$, the vertical component is problematic: we have explicitly excluded this from our definition of topological recursion as this does not give a ramified cover. However, in this particular situation, we see that independently of $t$, the only branch points are located at $x = 0$ and $ x= \infty$, while the vertical component is at $ x = 1$. We may take our spectral curve to be $x_T^{-1}(\mathcal{D})$ with $ \mathcal{D} = \P^1 \setminus \{ 1\}$, and then this family is globally admissible. Hence, topological recursion is analytic in this family -- if we ignore the node and the vertical component completely.

\appendix

\section{Symplectic transformations and local admissibility}
\label{ap:appendix}
The notion of local admissibilty in \cref{de:localadm} contains a surprising congruence condition $r_p = \pm 1 \,\,{\rm mod}\,\,s_p$ (when $r_p \geq 2$ and $s_p \in [r_p + 1]$) to be imposed on the local parameters at any ramification point $p$. This congruence comes from a computation within representation theory of $\mathcal{W}(\mathfrak{gl}_r)$-algebra \cite{BBCCN18}, where it was pointed out that its meaning for the geometry of the spectral curve is unclear. In this appendix, although we do not completely elucidate the question, we show that it appears naturally if we consider the action of symplectic transformations on spectral curves.

We recall that local admissibility was designed to guarantee that topological recursion is well defined, i.e. produces symmetric correlators.
Possible obstructions for this lie solely in the local parameters attached to the spectral curve. Thus they should already be seen at the level of the $(r,s)$-spectral curve
\begin{equation}
	\Sigma = \mathbb{P}^1\,,\qquad x = z^{r}\,,\qquad y = z^{s - r}\,,\qquad \omega_{0,2}(z_1,z_2) = \frac{\dd z_1\, \dd z_2}{(z_1 - z_2)^2}\,.
\end{equation}
Let us take here $r \in \mathbb{Z}\setminus \{0\}$ and $s \in \mathbb{Z}$, assuming that $s$ is not divisible by $r$ in case $|r| \geq 2$. The two points that may be ramified are found at $z = 0$ or $\infty$ and have local parameters
\begin{equation}
	(r_0,s_0) = (|r|,{\rm sgn}(r)s),\qquad (r_{\infty},s_{\infty}) = (|r|,-{\rm sgn}(r)s)\,.
\end{equation}
Topological recursion is well-defined on this spectral curve if and only if
\begin{itemize}
\item  either $|r| = 1$ (there are no ramification points, the $\omega_{g,n}$ all vanish);
\item or $|r| \geq 2$ and $|s| \in [|r| + 1]$ and $r = \pm 1 \,\,{\rm mod}\,\,s$.
\end{itemize}
The inequality $|s| > |r|+ 1$ means that $\dd x$ and $\dd y$ have a common zero, and it was already observed in \cite{EO07} that topological recursion is not well-defined in such cases. 

By symplectic transformations we mean birational maps $(x,y) \mapsto (\tilde{x},\tilde{y})$ with unit Jacobian, i.e. such that $\dd x \wedge \dd y = \dd \tilde{x} \wedge \dd \tilde{y}$. Given a spectral curve $\mathcal{S} = (\Sigma,x,y \dd x,\omega_{0,2})$ and a symplectic transformation, we can define a new spectral curve $\tilde{\mathcal{S}} = (\Sigma,\tilde{x},\tilde{y}\dd \tilde{x},\omega_{0,2})$. Such transformations play an important role in the theory of topological recursion: one can associate to a spectral curve (using the correlators computed by topological recursion) a sequence of free energies $(F_g)_{g \geq 0}$ which is invariant (in a suitable way) under symplectic transformations \cite{EO2MM,EO13}. For this invariance to even make sense, topological recursion should be well-defined both for $\mathcal{S}$ and $\tilde{\mathcal{S}}$.

An interesting set of transformations to look at are monomial transformations $(x,y) \mapsto (x^{a}y^b,x^{c}y^{d})$. Such transformations are symplectic (meaning that they preserve $\dd x \wedge \dd y$) if and only if  $a = 1 +b$, $c = -b$ and $d = 1 -b$. Indeed, we have
\begin{equation*}
	\dd(x^{a}y^b) \wedge \dd(x^c y^d) = (ad-bc) x^{a+c-1} y^{b+d-1} dx \wedge dy,
\end{equation*} and we must choose $a = 1+ b$, $c = -b$ and $d = 1 -b$ in order to obtain the symplectic form $dx \wedge dy$. We call these transformations
\begin{equation}
\Phi_b(x,y) = (x^{1+b}y^b,x^{-b}y^{1-b})\,.
\end{equation}

\begin{lemma}
Let $\mathcal{Y} \subset \mathbb{Z}^2$ be a set parametrising $(r,s)$-type spectral curves, such that for any $(\tilde{r},\tilde{s}) \in \mathcal{Y}$ either $|\tilde{r}| = 1$ or $|\tilde{s}| \leq |\tilde{r}| + 1$ holds. If $\mathcal{Y}$ is stable under $(\Phi_b)_{b \in \mathbb{Z}}$,\footnote{That is, if the geometrically understood condition ``$|r| = 1$ or $|s| \leq |r| + 1$''  is preserved by monomial symplectic transformations.}  then we have $\tilde{r} = \pm 1\,\,{\rm mod}\,\,|\tilde{s}|$ for any $(\tilde{r},\tilde{s}) \in \mathcal{Y}$.
\end{lemma}
\begin{proof}
Take $(r,s) \in \mathcal{Y}$ and write $r = k + \beta s$ for $\beta \in \mathbb{Z}$ and $k \in \{1,\ldots,|s|-1\}$. We see that $\Phi_{-\beta}$ takes the $(r,s)$-spectral curve to the $(k,s)$-spectral curve. If $k = 1$ the $(k,s)$-curve is not ramified; if $k = |s| - 1$ it satisfies $|s| \leq k + 1$; in all other cases the $(k,s)$-curve is ramified but satisfies $|s| > k + 1$. This entails the claim.
\end{proof}

In other words, the mysterious congruence condition $r = \pm 1\,\,{\rm mod}\,\,s$ can be generated from the geometrically understood condition ``$r = 1$ or $s \leq r + 1$'' under the action of monomial symplectic transformations.

\printbibliography

\end{document}